\documentclass[11pt]{preprint}
\usepackage[full]{textcomp}
\usepackage[osf]{newtxtext} 
\usepackage[cal=boondoxo]{mathalfa}
\usepackage{colortbl}

\usepackage{comment}

\usepackage{amssymb}
\usepackage{mathtools}
\usepackage{hyperref}
\usepackage{breakurl}
\usepackage{mhenvs}
\usepackage{mhequ} 
\usepackage{mhsymb}
\usepackage{booktabs}
\usepackage{tikz}
\usepackage{tcolorbox}
\usepackage{mathrsfs}
\usepackage[utf8]{inputenc}
\usepackage{longtable}
\usepackage{wrapfig}
\usepackage{subcaption}
\usepackage{mathrsfs}
\usepackage{epsfig}
\usepackage{microtype}
\usepackage{comment}
\usepackage{wasysym}
\usepackage{centernot}
\usepackage{enumitem}
\usepackage{bm}
\usepackage{stackrel}
\usepackage{graphicx}

\makeatletter
\newcommand{\globalcolor}[1]{%
  \color{#1}\global\let\default@color\current@color
}
\makeatother

\usetikzlibrary{calc}
\usetikzlibrary{decorations}
\usetikzlibrary{positioning}
\usetikzlibrary{shapes}
\usetikzlibrary{external}

\definecolor{blush}{rgb}{0.87, 0.36, 0.51}
	\definecolor{brightcerulean}{rgb}{0.11, 0.67, 0.84}
	\definecolor{greenryb}{rgb}{0.4, 0.69, 0.2}

\newif\ifdark
\darkfalse

\ifdark
\definecolor{darkred}{rgb}{0.9,0.2,0.2}
\definecolor{darkblue}{rgb}{0.7,0.3,1}
\definecolor{darkgreen}{rgb}{0.1,0.9,0.1}
\definecolor{franck}{rgb}{0,0.8,1}
\definecolor{pagebackground}{rgb}{.15,.21,.18}
\definecolor{pageforeground}{rgb}{.84,.84,.85}
\pagecolor{pagebackground}
\AtBeginDocument{\globalcolor{pageforeground}}
\definecolor{symbols}{rgb}{0,0.7,1}
\colorlet{connection}{red!80!black}
\colorlet{boxcolor}{blue!50}

\else

\definecolor{darkred}{rgb}{0.7,0.1,0.1}
\definecolor{darkblue}{rgb}{0.4,0.1,0.8}
\definecolor{darkgreen}{rgb}{0.1,0.7,0.1}
\definecolor{franck}{rgb}{0,0,1}
\definecolor{pagebackground}{rgb}{1,1,1}
\definecolor{pageforeground}{rgb}{0,0,0}
\colorlet{symbols}{blue!90!black}
\colorlet{connection}{red!30!black}
\colorlet{boxcolor}{blue!50!black}

\fi

\def\slash{\leavevmode\unskip\kern0.18em/\penalty\exhyphenpenalty\kern0.18em}
\def\dash{\leavevmode\unskip\kern0.18em--\penalty\exhyphenpenalty\kern0.18em}

\DeclareMathAlphabet{\mathbbm}{U}{bbm}{m}{n}

\DeclareFontFamily{U}{BOONDOX-calo}{\skewchar\font=45 }
\DeclareFontShape{U}{BOONDOX-calo}{m}{n}{
  <-> s*[1.05] BOONDOX-r-calo}{}
\DeclareFontShape{U}{BOONDOX-calo}{b}{n}{
  <-> s*[1.05] BOONDOX-b-calo}{}
\DeclareMathAlphabet{\mcb}{U}{BOONDOX-calo}{m}{n}
\SetMathAlphabet{\mcb}{bold}{U}{BOONDOX-calo}{b}{n}
\setlist{noitemsep,topsep=4pt,leftmargin=1.5em}

\DeclareMathAlphabet{\mathbbm}{U}{bbm}{m}{n}

\DeclareMathAlphabet{\mcb}{U}{BOONDOX-calo}{m}{n}
\SetMathAlphabet{\mcb}{bold}{U}{BOONDOX-calo}{b}{n}
\DeclareFontFamily{U}{mathx}{\hyphenchar\font45}
\DeclareFontShape{U}{mathx}{m}{n}{
      <5> <6> <7> <8> <9> <10>
      <10.95> <12> <14.4> <17.28> <20.74> <24.88>
      mathx10
      }{}
\DeclareSymbolFont{mathx}{U}{mathx}{m}{n}
\DeclareMathSymbol{\bigtimes}{1}{mathx}{"91}

\newcommand{\z}{\mathsf{z}}

\def\reg{\mathord{\mathrm{reg}}}

\providecommand{\figures}{false}
{ \ifthenelse{\equal{\figures}{false}} {#1}{\[ {\rm Figure \ missing !} \]} }{}
\def\id{\mathrm{id}}
\def\Id{\mathrm{Id}}
\def\proj{\mathbf{p}}
\def\reg{\mathop{\mathrm{reg}}}

\tikzstyle{tinydots}=[dash pattern=on \pgflinewidth off \pgflinewidth]
\tikzstyle{superdense}=[dash pattern=on 4pt off 1pt]

\newcommand{\mcM}{\mathcal{M}}

\newcommand{\mcA}{\mathcal{A}}

\newcommand{\mcR}{\mathcal{R}}
\newcommand{\mcC}{\mathcal{C}}

\newcommand{\mcS}{\CS}

\newcommand{\mcL}{\mathcal{L}}
\newcommand{\mcI}{\mathcal{I}}

\newcommand{\mcN}{\mathcal{N}}

\newcommand{\mcD}{\mathcal{D}}

\newcommand{\mcP}{\mathcal{P}}

\newcommand{\beq}{\begin{equation}}
\newcommand{\eeq}{\end{equation}}

\usepackage{empheq}

\newcommand{\mbpartial}{\boldsymbol{\partial}}

\newcommand{\mfT}{\mathfrak{T}}

\newcommand{\mfL}{\mathfrak{L}}

\newcommand{\mfA}{\mathfrak{A}}

\newcommand{\mfl}{\mathfrak{l}}

\newcommand{\mfg}{\mathfrak{g}}

\def\${|\!|\!|}

\def\proj{\mathfrak{p}}

\newcounter{theorem}

\newtheorem{ex}[theorem]{Example}

\newenvironment{DIFnomarkup}{}{} % see man latexdiff

\newcommand{\rmT}{{\rm T}}

\newtheorem{assumption}{Assumption}
 
\theorembodyfont{\rmfamily}

\newfont{\indic}{bbmss12}

\def\Nabla_#1{\nabla_{\!#1}}

\makeatletter
\pgfdeclareshape{crosscircle}
{
  \inheritsavedanchors[from=circle] % this is nearly a circle
  \inheritanchorborder[from=circle]
  \inheritanchor[from=circle]{north}
  \inheritanchor[from=circle]{north west}
  \inheritanchor[from=circle]{north east}
  \inheritanchor[from=circle]{center}
  \inheritanchor[from=circle]{west}
  \inheritanchor[from=circle]{east}
  \inheritanchor[from=circle]{mid}
  \inheritanchor[from=circle]{mid west}
  \inheritanchor[from=circle]{mid east}
  \inheritanchor[from=circle]{base}
  \inheritanchor[from=circle]{base west}
  \inheritanchor[from=circle]{base east}
  \inheritanchor[from=circle]{south}
  \inheritanchor[from=circle]{south west}
  \inheritanchor[from=circle]{south east}
  \inheritbackgroundpath[from=circle]
  \foregroundpath{
    \centerpoint%
    \pgf@xc=\pgf@x%
    \pgf@yc=\pgf@y%
    \pgfutil@tempdima=\radius%
    \pgfmathsetlength{\pgf@xb}{\pgfkeysvalueof{/pgf/outer xsep}}%  
    \pgfmathsetlength{\pgf@yb}{\pgfkeysvalueof{/pgf/outer ysep}}%  
    \ifdim\pgf@xb<\pgf@yb%
      \advance\pgfutil@tempdima by-\pgf@yb%
    \else%
      \advance\pgfutil@tempdima by-\pgf@xb%
    \fi%
    \pgfpathmoveto{\pgfpointadd{\pgfqpoint{\pgf@xc}{\pgf@yc}}{\pgfqpoint{-0.707107\pgfutil@tempdima}{-0.707107\pgfutil@tempdima}}}
    \pgfpathlineto{\pgfpointadd{\pgfqpoint{\pgf@xc}{\pgf@yc}}{\pgfqpoint{0.707107\pgfutil@tempdima}{0.707107\pgfutil@tempdima}}}
    \pgfpathmoveto{\pgfpointadd{\pgfqpoint{\pgf@xc}{\pgf@yc}}{\pgfqpoint{-0.707107\pgfutil@tempdima}{0.707107\pgfutil@tempdima}}}
    \pgfpathlineto{\pgfpointadd{\pgfqpoint{\pgf@xc}{\pgf@yc}}{\pgfqpoint{0.707107\pgfutil@tempdima}{-0.707107\pgfutil@tempdima}}}
  }
}
\makeatother

\def\symbol#1{\textcolor{symbols}{#1}}

\def\decorate#1#2{
        \ifnum#2>0
    		\foreach \count in {1,...,#2}{
	       	let
				\p1 = (sourcenode.center),
                \p2 = (sourcenode.east),
				\n1 = {\x2-\x1},
				\n2 = {1mm},
				\n3 = {(1.3+0.6*(\count-1))*\n1},
				\n4 = {0.7*\n1}
			in 
        		node[rectangle,fill=symbols,rotate=30,inner sep=0pt,minimum width=0.2*\n2,minimum height=\n2] at ($(sourcenode.center) + (\n3,\n4)$) {}
				}
		\fi
        \ifnum#1>0
    		\foreach \count in {1,...,#1}{
	       	let
				\p1 = (sourcenode.center),
                \p2 = (sourcenode.east),
				\n1 = {\x2-\x1},
				\n2 = {1mm},
				\n3 = {(1.3+0.6*(\count-1))*\n1},
				\n4 = {0.7*\n1}
			in 
        		node[rectangle,fill=symbols,rotate=-30,inner sep=0pt,minimum width=0.2*\n2,minimum height=\n2] at ($(sourcenode.center) + (-\n3,\n4)$) {}
				}
		\fi
}

\tikzset{
    dectriangle/.style 2 args={
        triangle,
        alias=sourcenode,
        append after command={\decorate{#1}{#2}}
    },
    dectriangle/.default={0}{0},
}

\tikzset{
	cross/.style={path picture={ 
  		\draw[symbols]
			(path picture bounding box.south east) -- (path picture bounding box.north west) (path picture bounding box.south west) -- (path picture bounding box.north east);
		}},
root/.style={circle,fill=green!50!black,inner sep=0pt, minimum size=1.2mm},
        dot/.style={circle,fill=pageforeground,inner sep=0pt, minimum size=1mm},
        dotred/.style={circle,fill=pageforeground!50!pagebackground,inner sep=0pt, minimum size=2mm},
        var/.style={circle,fill=pageforeground!10!pagebackground,draw=pageforeground,inner sep=0pt, minimum size=3mm},
        kernel/.style={semithick,shorten >=2pt,shorten <=2pt},
        kernels/.style={snake=zigzag,shorten >=2pt,shorten <=2pt,segment amplitude=1pt,segment length=4pt,line before snake=2pt,line after snake=5pt,},
        rho/.style={densely dashed,semithick,shorten >=2pt,shorten <=2pt},
           testfcn/.style={dotted,semithick,shorten >=2pt,shorten <=2pt},
        renorm/.style={shape=circle,fill=pagebackground,inner sep=1pt},
        labl/.style={shape=rectangle,fill=pagebackground,inner sep=1pt},
        xic/.style={very thin,circle,draw=symbols,fill=symbols,inner sep=0pt,minimum size=1.2mm},
        g/.style={very thin,rectangle,draw=symbols,fill=symbols!10!pagebackground,inner sep=0pt,minimum width=2.5mm,minimum height=1.2mm},
        xi/.style={very thin,circle,draw=symbols,fill=symbols!10!pagebackground,inner sep=0pt,minimum size=1.2mm},
	xies/.style={very thin,rectangle,fill=green!50!black!25,draw=symbols,inner sep=0pt,minimum size=1.1mm},
	xiesf/.style={very thin,rectangle,fill=green!50!black,draw=symbols,inner sep=0pt,minimum size=1.1mm},
        xix/.style={very thin,crosscircle,fill=symbols!10!pagebackground,draw=symbols,inner sep=0pt,minimum size=1.2mm},
        X/.style={very thin,cross,rectangle,fill=pagebackground,draw=symbols,inner sep=0pt,minimum size=1.2mm},
	xib/.style={thin,circle,fill=symbols!10!pagebackground,draw=symbols,inner sep=0pt,minimum size=1.6mm},
	xie/.style={thin,circle,fill=green!50!black,draw=symbols,inner sep=0pt,minimum size=1.6mm},
	xid/.style={thin,circle,fill=symbols,draw=symbols,inner sep=0pt,minimum size=1.6mm},
	xibx/.style={thin,crosscircle,fill=symbols!10!pagebackground,draw=symbols,inner sep=0pt,minimum size=1.6mm},
	kernels2/.style={very thick,draw=connection,segment length=12pt},
	keps/.style={thin,draw=symbols,->},
	kepspr/.style={thick,draw=connection,->},
	krho/.style={thin,draw=symbols,superdense,->},
	krhopr/.style={thick,draw=connection,superdense},
	triangle/.style = { regular polygon, regular polygon sides=3},
	not/.style={thin,circle,draw=connection,fill=connection,inner sep=0pt,minimum size=0.5mm},
	diff/.style = {very thin,draw=symbols,triangle,fill=red!50!black,inner sep=0pt,minimum size=1.6mm},
	diff1/.style = {very thin,dectriangle={1}{0},fill=red!50!black,draw=symbols,inner sep=0pt,minimum size=1.6mm},
	diff2/.style = {very thin,dectriangle={1}{1},fill=red!50!black,draw=symbols,inner sep=0pt,minimum size=1.6mm},
		diffmini/.style = {very thin,rectangle,fill=black,draw=black,inner sep=0pt,minimum size=0.75mm},
	 kernelsmod/.style={very thick,draw=connection,segment length=12pt},
	 rec/.style = {very thin,rectangle,fill=black,draw=black,inner sep=0pt,minimum size=2mm},
	cerc/.style={very thin,circle,draw=black,fill=symbols,inner sep=0pt,minimum size=2mm},
	stars/.style={very thin,star,star points=6,star point ratio=0.5, draw=black,fill=red,inner sep=0pt,minimum size=0.7mm},
	>=stealth,
        }
        \tikzset{
root/.style={circle,fill=black!50,inner sep=0pt, minimum size=3mm},
        circ/.style={circle,fill=white,draw=black,very thin,inner sep=.5pt, minimum size=1.2mm},
        round1/.style={fill=white,outer sep = 0,inner sep=2pt,rounded corners=1mm,draw,text=black,thin,minimum size=1.2mm},
          circ1/.style={circle,fill=red!10,draw=red,very thin,inner sep=.5pt, minimum size=1.2mm},
        rect/.style={fill=white,outer sep = 0,inner sep=2pt,rectangle,draw,text=black,thin,minimum size=1.2mm},
        rect1/.style={fill=white,outer sep = 0,inner sep=2pt,rectangle,draw,text=black,thin,minimum size=1.2mm},
        round2/.style={fill=red!10,outer sep = 0,inner sep=2pt,rounded corners=1mm,draw,text=black,thin,minimum size=1.2mm},
       round3/.style={fill=blue!10,outer sep = 0,inner sep=2pt,rounded corners=1mm,draw,text=black,thin,minimum size=1.2mm}, 
        rect2/.style={fill=black!10,outer sep = 0,inner sep=2pt,rectangle,draw,text=black,thin,minimum size=1.2mm},
        dot/.style={circle,fill=black,inner sep=0pt, minimum size=1.2mm},
        dotred/.style={circle,fill=black!50,inner sep=0pt, minimum size=2mm},
        var/.style={circle,fill=black!10,draw=black,inner sep=0pt, minimum size=3mm},
        kernel/.style={semithick,shorten >=2pt,shorten <=2pt},
         diag/.style={thin,shorten >=4pt,shorten <=4pt},
        kernel1/.style={thick},
        kernels/.style={snake=zigzag,shorten >=2pt,shorten <=2pt,segment amplitude=1pt,segment length=4pt,line before snake=2pt,line after snake=5pt,},
		kernels1/.style={snake=zigzag,segment amplitude=0.5pt,segment length=2pt},
		rho1/.style={densely dotted,semithick},
        rho/.style={densely dashed,semithick,shorten >=2pt,shorten <=2pt},
           testfcn/.style={dotted,semithick,shorten >=2pt,shorten <=2pt},
           visible/.style={draw, circle, fill, inner sep=0.25ex},
        renorm/.style={shape=circle,fill=white,inner sep=1pt},
        labl/.style={shape=rectangle,fill=white,inner sep=1pt},
        xic/.style={very thin,circle,fill=symbols,draw=black,inner sep=0pt,minimum size=1.2mm},
        xi/.style={very thin,circle,fill=blue!10,draw=black,inner sep=0pt,minimum size=1.2mm},
	xib/.style={very thin,circle,fill=blue!10,draw=black,inner sep=0pt,minimum size=1.6mm},
	xie/.style={very thin,circle,fill=green!50!black,draw=black,inner sep=0pt,minimum size=1mm},
	xid/.style={very thin,circle,fill=symbols,draw=black,inner sep=0pt,minimum size=1.6mm},
	edgetype/.style={very thin,circle,draw=black,inner sep=0pt,minimum size=5mm},
	nodetype/.style={very thick,circle,draw=black,inner sep=0pt,minimum size=5mm},
	kernels2/.style={very thick,draw=connection,segment length=12pt},
clean/.style={thin,circle,fill=black,inner sep=0pt,minimum size=1mm},	not/.style={thin,circle,fill=symbols,draw=connection,fill=connection,inner sep=0pt,minimum size=0.8mm},
	>=stealth,
        }

\makeatletter
\def\DeclareSymbol#1#2#3{%
	\expandafter\gdef\csname MH@symb@#1\endcsname{\tikzsetnextfilename{symbol#1}%
	\tikz[baseline=#2,scale=0.15,draw=symbols,line join=round]{#3}}%
	\expandafter\gdef\csname MH@symb@#1s\endcsname{\scalebox{0.75}{\tikzsetnextfilename{symbol#1}%
	\tikz[baseline=#2,scale=0.15,draw=symbols,line join=round]{#3}}}%
	\expandafter\gdef\csname MH@symb@#1ss\endcsname{\scalebox{0.65}{\tikzsetnextfilename{symbol#1}%
	\tikz[baseline=#2,scale=0.15,draw=symbols,line join=round]{#3}}}%
	}
\def\<#1>{\ifthenelse{\boolean{mmode}}{\mathchoice{\csname MH@symb@#1\endcsname}{\csname MH@symb@#1\endcsname}{\csname MH@symb@#1s\endcsname}{\csname MH@symb@#1ss\endcsname}}{\csname MH@symb@#1\endcsname}}
\makeatother

\DeclareSymbol{Xi22}{0.5}{\draw (0,0) node[xi] {} -- (-1,1) node[not] {} -- (0,2) node[xi] {};} % 1 not used in the text
\DeclareSymbol{Xi2}{-2}{\draw (-1,-0.25) node[xi] {} -- (0,1) node[xi] {};} % 2
\DeclareSymbol{Xi2b}{-2}{\draw (-1,-0.25) node[xic] {} -- (0,1) node[xic] {};} % 2
\DeclareSymbol{Xi2g}{-2}{\draw (-1,-0.25) node[xies] {} -- (0,1) node[xi] {};} % 2
\DeclareSymbol{Xi2g2}{-2}{\draw (-1,-0.25) node[xi] {} -- (0,1) node[xies] {};} % 2
\DeclareSymbol{cXi2}{-2}{\draw (0,-0.25) node[xi] {} -- (-1,1) node[xic] {};}%3
\DeclareSymbol{Xi3}{0}{\draw (0,0) node[xi] {} -- (-1,1) node[xi] {} -- (0,2) node[xi] {};}%4
\DeclareSymbol{XiIIXi}{0}{\draw (0,0) node[xi] {} -- (-1,1); \draw[kernels2] (-1,1) node[not] {} -- (0,2) node[xi] {};}%4

\DeclareSymbol{Xi4}{2}{\draw (0,0) node[xi] {} -- (-1,1) node[xi] {} -- (0,2) node[xi] {} -- (-1,3) node[xi] {};}%5
\DeclareSymbol{Xi4_1}{2}{\draw (0,0) node[xic] {} -- (-1,1) node[xic] {} -- (0,2) node[xi] {} -- (-1,3) node[xi] {};}%6
\DeclareSymbol{Xi4_2}{2}{\draw (0,0) node[xic] {} -- (-1,1) node[xi] {} -- (0,2) node[xi] {} -- (-1,3) node[xic] {};}%7
\DeclareSymbol{Xi2X}{-2}{\draw (0,-0.25) node[xi] {} -- (-1,1) node[xix] {};}%8
\DeclareSymbol{XXi2}{-2}{\draw (0,-0.25) node[xix] {} -- (-1,1) node[xi] {};}%9
\DeclareSymbol{IIXi}{0}{\draw (0,-0.25) node[not] {} -- (-1,1) node[xi] {} -- (0,2) node[xi] {};}%10
\DeclareSymbol{IXi^2}{-1}{\draw (-1,1) node[xi] {} -- (0,0) node[not] {} -- (1,1) node[xi] {};}%11
\DeclareSymbol{IIXi^2}{-4}{\draw (0,-1.5) node[not] {} -- (0,0);
\draw[kernels2] (-1,1) node[xi] {} -- (0,0) node[not] {} -- (1,1) node[xi] {};}%12
\DeclareSymbol{XiX}{-2.8}{\node[xibx] {};}%13
\DeclareSymbol{tauX}{-2.8}{ \node[X] {};}%14
\DeclareSymbol{Xi}{-2.8}{\node[xib] {};}%15

\DeclareSymbol{IXiX}{-1}{\draw (0,-0.25) node[not] {} -- (-1,1) node[xix] {};}%18
\DeclareSymbol{IXi3}{2}{\draw (0,-0.25) node[not] {} -- (-1,1) node[xi] {} -- (0,2) node[xi] {} -- (-1,3) node[xi] {};}%20
\DeclareSymbol{IXi}{-2}{\draw (0,-0.25) node[not] {} -- (-1,1) node[xi] {};}%21
\DeclareSymbol{XiI}{-2}{\draw (0,-0.25) node[xi] {} -- (-1,1) node[not] {};}%22

\DeclareSymbol{Xi4b}{0}{\draw(0,1.5) node[xi] {} -- (0,0); \draw (-1,1) node[xi] {} -- (0,0) node[xi] {} -- (1,1) node[xi] {};}%23
\DeclareSymbol{Xi4b'}{0}{\draw(0,1.5) node[xi] {} -- (0,-0.2); \draw (-1,1) node[xi] {} -- (0,-0.2) node[not] {} -- (1,1) node[xi] {};}%24 not used
\DeclareSymbol{Xi4c}{0}{\draw (0,1) -- (0.8,2.2) node[xi] {};\draw (0,-0.25) node[xi] {} -- (0,1) node[xi] {} -- (-0.8,2.2) node[xi] {};}%25
\DeclareSymbol{Xi4d}{-4.5}{\draw (0,-1.5) node[not] {} -- (0,0); \draw (-1,1) node[xi] {} -- (0,0) node[xi] {} -- (1,1) node[xi] {};}%26 not used
\DeclareSymbol{Xi4e}{0}{\draw (0,2) node[xi] {} -- (-1,1) node[xi] {} -- (0,0) node[xi] {} -- (1,1) node[xi] {};}%27
\DeclareSymbol{Xi4e'}{0}{\draw (0,2) node[xi] {} -- (-1,1) node[xi] {} -- (0,-0.2) node[not] {} -- (1,1) node[xi] {};}%28 not used

\DeclareSymbol{Xitwo}%29
{0}{\draw[kernels2] (0,0) node[not] {} -- (-1,1) node[not] {}
-- (-2,2) node[not]{} -- (-3,3) node[xi]  {};
\draw[kernels2] (0,0) -- (1,1) node[xi] {};
\draw[kernels2] (-1,1) -- (0,2) node[xi] {};
\draw[kernels2] (-2,2) -- (-1,3) node[xi] {};}

\DeclareSymbol{IXitwo}%30
{0}{\draw (-.7,1.2) node[xi] {} -- (0,-0.2) -- (.7,1.2) node[xi] {};}
\DeclareSymbol{I1Xitwo}%30
{0}{\draw[kernels2] (0,0) node[not] {} -- (-1,1) node[xi] {};
\draw[kernels2] (0,0) -- (1,1) node[xi] {};}

\DeclareSymbol{I1Xitwobis}%30
{0}{\draw[kernels2] (0,0) node[not] {} -- (-1,1) node[xies] {};
\draw[kernels2] (0,0) -- (1,1) node[xies] {};}

\DeclareSymbol{I1Xitwog}%30
{0}{\draw[kernels2] (0,0) node[not] {} -- (-1,1) node[xies] {};
\draw[kernels2] (0,0) -- (1,1) node[xi] {};}

\DeclareSymbol{cI1Xitwo}%31
{0}{\draw[kernels2] (0,0) node[not] {} -- (-1,1) node[xic] {};
\draw[kernels2] (0,0) -- (1,1) node[xi] {};}

\DeclareSymbol{I1IXi3}{0}{\draw (0,0) node[xi] {} -- (-1,1) ; %32
\draw[kernels2] (-1,1) node[not] {} -- (0,2) node[xi] {};
\draw[kernels2] (-1,1) node[not] {} -- (-2,2) node[xi] {};}

\DeclareSymbol{I1Xi3c}{-1}{\draw[kernels2](0,1.5) node[xi] {} -- (0,0) node[not] {}; \draw (-1,1) node[xi] {} -- (0,0) ; \draw[kernels2] (0,0) -- (1,1) node[xi] {};}%33

\DeclareSymbol{I1Xi3cbis}{-1}{\draw[kernels2](0,1.5) node[xies] {} -- (0,0) node[not] {}; \draw (-1,1) node[xies] {} -- (0,0) ; \draw[kernels2] (0,0) -- (1,1) node[xies] {};}%33

\DeclareSymbol{I1IXi3b}{0}{\draw[kernels2] (0,0) node[not] {} -- (-1,1) ; \draw[kernels2] (0,0)   -- (1,1) node[xi] {} ;
\draw (-1,1) node[xi] {} -- (0,2) node[xi] {};
}%34

\DeclareSymbol{I1IXi3c}{0}{\draw[kernels2] (0,0) node[not] {} -- (-1,1) ; \draw[kernels2] (0,0)   -- (1,1) node[xi] {} ;
\draw[kernels2] (-1,1) node[not] {} -- (0,2) node[xi] {};
\draw[kernels2] (-1,1) node[not] {} -- (-2,2) node[xi] {};}%35

\DeclareSymbol{I1IXi3cbis}{0}{\draw[kernels2] (0,0) node[not] {} -- (-1,1) ; \draw[kernels2] (0,0)   -- (1,1) node[xies] {} ;
\draw[kernels2] (-1,1) node[not] {} -- (0,2) node[xies] {};
\draw[kernels2] (-1,1) node[not] {} -- (-2,2) node[xies] {};}%35

\DeclareSymbol{I1Xi}{0}{\draw[kernels2] (0,0) node[not] {} -- (-1,1)  node[xi] {} ;}%36

\DeclareSymbol{I1Xi4a}{2}{\draw[kernels2] (0,0) node[not] {} -- (-1,1) ; \draw[kernels2] (0,0) node[not] {} -- (1,1) node[xi] {} ;%37
\draw (-1,1) node[xi] {} -- (0,2) node[xi] {} -- (-1,3) node[xi] {};}
\DeclareSymbol{cI1Xi4a}{2}{\draw[kernels2] (0,0) node[not] {} -- (-1,1) ; \draw[kernels2] (0,0) node[not] {} -- (1,1) node[xic] {} ;%39
\draw (-1,1) node[xic] {} -- (0,2) node[xi] {} -- (-1,3) node[xi] {};}
\DeclareSymbol{I1Xi4b}{2}{\draw (0,0) node[xi] {} -- (-1,1) node[xi] {} -- (0,2) ; \draw[kernels2] (0,2) node[not] {} -- (-1,3) node[xi] {};\draw[kernels2] (0,2)  -- (1,3) node[xi] {};
}%41

\DeclareSymbol{cI1Xi4b}{2}{\draw (0,0) node[xic] {} -- (-1,1) node[xic] {} -- (0,2) ; \draw[kernels2] (0,2) node[not] {} -- (-1,3) node[xi] {};\draw[kernels2] (0,2)  -- (1,3) node[xi] {};
}%42

\DeclareSymbol{I1Xi4c}{2}{\draw (0,0) node[xi] {} -- (-1,1) node[not] {}; \draw[kernels2] (-1,1) -- (0,2) ; 
\draw[kernels2] (-1,1) -- (-2,2) node[xi] {} ;
\draw (0,2) node[xi] {} -- (-1,3) node[xi] {};}%43

\DeclareSymbol{cI1Xi4c}{2}{\draw (0,0) node[xic] {} -- (-1,1) node[not] {}; \draw[kernels2] (-1,1) -- (0,2) ; 
\draw[kernels2] (-1,1) -- (-2,2) node[xic] {} ;
\draw (0,2) node[xi] {} -- (-1,3) node[xi] {};}%44

\DeclareSymbol{I1Xi4ab}{2}{\draw[kernels2] (0,0) node[not] {} -- (-1,1) ; \draw[kernels2] (0,0) node[not] {} -- (1,1) node[xi] {};\draw (-1,1) node[xi] {} -- (0,2) ; \draw[kernels2] (0,2) node[not] {} -- (-1,3) node[xi] {};\draw[kernels2] (0,2)  -- (1,3) node[xi] {}; }%45

\DeclareSymbol{cI1Xi4ab}{2}{\draw[kernels2] (0,0) node[not] {} -- (-1,1) ; \draw[kernels2] (0,0) node[not] {} -- (1,1) node[xic] {};\draw (-1,1) node[xic] {} -- (0,2) ; \draw[kernels2] (0,2) node[not] {} -- (-1,3) node[xi] {};\draw[kernels2] (0,2)  -- (1,3) node[xi] {}; }%46

\DeclareSymbol{I1Xi4bc}{2}{\draw (0,0) node[xi] {} -- (-1,1) node[not] {}; \draw[kernels2] (-1,1) -- (0,2) ; %47
\draw[kernels2] (-1,1) -- (-2,2) node[xi] {} ; \draw[kernels2] (0,2) node[not] {} -- (-1,3) node[xi] {};\draw[kernels2] (0,2)  -- (1,3) node[xi] {};
}

\DeclareSymbol{cI1Xi4bc}{2}{\draw (0,0) node[xic] {} -- (-1,1) node[not] {}; \draw[kernels2] (-1,1) -- (0,2) ; %48
\draw[kernels2] (-1,1) -- (-2,2) node[xic] {} ; \draw[kernels2] (0,2) node[not] {} -- (-1,3) node[xi] {};\draw[kernels2] (0,2)  -- (1,3) node[xi] {};
}

\DeclareSymbol{I1Xi4abcc1}{2}{\draw[kernels2] (0,0) node[not] {} -- (-1,1) node[not] {}%50
-- (-2,2) node[not]{} -- (-3,3) node[xic]  {};
\draw[kernels2] (0,0) -- (1,1) node[xic] {};
\draw[kernels2] (-1,1) -- (0,2) node[xi] {};
\draw[kernels2] (-2,2) -- (-1,3) node[xi] {};
}

\DeclareSymbol{I1Xi4abcc1b}{2}{\draw[kernels2] (0,0) node[not] {} -- (-1,1) node[not] {}%50
-- (-2,2) node[not]{} -- (-3,3) node[xi]  {};
\draw[kernels2] (0,0) -- (1,1) node[xic] {};
\draw[kernels2] (-1,1) -- (0,2) node[xic] {};
\draw[kernels2] (-2,2) -- (-1,3) node[xi] {};
}

\DeclareSymbol{I1Xi4abcc2}{2}{\draw[kernels2] (0,0) node[not] {} -- (-1,1) node[not] {}%51
-- (-2,2) node[not]{} -- (-3,3) node[xic]  {};
\draw[kernels2] (0,0) -- (1,1) node[xi] {};
\draw[kernels2] (-1,1) -- (0,2) node[xi] {};
\draw[kernels2] (-2,2) -- (-1,3) node[xic] {};
}

\DeclareSymbol{I1Xi4ac}{2}{\draw[kernels2] (0,0) node[not] {} -- (-1,1) ; \draw[kernels2] (0,0) node[not] {} -- (1,1) node[xi] {}; 
\draw[kernels2] (-1,1) node[not] {} -- (0,2) ; %52
\draw[kernels2] (-1,1) -- (-2,2) node[xi] {} ;
\draw (0,2) node[xi] {} -- (-1,3) node[xi] {};}

\DeclareSymbol{cI1Xi4ac}{2}{\draw[kernels2] (0,0) node[not] {} -- (-1,1) ; \draw[kernels2] (0,0) node[not] {} -- (1,1) node[xic] {}; 
\draw[kernels2] (-1,1) node[not] {} -- (0,2) ; %53
\draw[kernels2] (-1,1) -- (-2,2) node[xic] {} ;
\draw (0,2) node[xi] {} -- (-1,3) node[xi] {};}

\DeclareSymbol{I1Xi4acc1}{2}{\draw[kernels2] (0,0) node[not] {} -- (-1,1) ; \draw[kernels2] (0,0) node[not] {} -- (1,1) node[xic] {}; 
\draw[kernels2] (-1,1) node[not] {} -- (0,2) ; %54
\draw[kernels2] (-1,1) -- (-2,2) node[xi] {} ;
\draw (0,2) node[xic] {} -- (-1,3) node[xi] {};}

\DeclareSymbol{I1Xi4acc2}{2}{\draw[kernels2] (0,0) node[not] {} -- (-1,1) ; \draw[kernels2] (0,0) node[not] {} -- (1,1) node[xic] {}; 
\draw[kernels2] (-1,1) node[not] {} -- (0,2) ; %55
\draw[kernels2] (-1,1) -- (-2,2) node[xi] {} ;
\draw (0,2) node[xi] {} -- (-1,3) node[xic] {};}

\DeclareSymbol{2I1Xi4}{2}{\draw[kernels2] (0,0) node[not] {} -- (-1,1) node[not] {};
\draw[kernels2] (0,0) -- (1,1) node[not] {};%56
\draw[kernels2] (-1,1) -- (-1.5,2.5) node[xi] {};
\draw[kernels2] (-1,1) -- (-0.5,2.5) node[xi] {};
\draw[kernels2] (1,1) -- (0.5,2.5) node[xi] {};
\draw[kernels2] (1,1) -- (1.5,2.5) node[xi] {};
}

\DeclareSymbol{2I1Xi4dis}{2}{\draw[kernels2] (0,0) node[not] {} -- (-1,1) node[not] {};
\draw[kernels2] (0,0) -- (1,1) node[not] {};%56
\draw[kernels2] (-1,1) -- (-1.5,2.5) node[xies] {};
\draw[kernels2] (-1,1) -- (-0.5,2.5) node[xies] {};
\draw[kernels2] (1,1) -- (0.5,2.5) node[xies] {};
\draw[kernels2] (1,1) -- (1.5,2.5) node[xies] {};
}

\DeclareSymbol{2I1Xi4c1}{2}{\draw[kernels2] (0,0) node[not] {} -- (-1,1) node[not] {};%57
\draw[kernels2] (0,0) -- (1,1) node[not] {};
\draw[kernels2] (-1,1) -- (-1.5,2.5) node[xic] {};
\draw[kernels2] (-1,1) -- (-0.5,2.5) node[xi] {};
\draw[kernels2] (1,1) -- (0.5,2.5) node[xic] {};
\draw[kernels2] (1,1) -- (1.5,2.5) node[xi] {};
}

\DeclareSymbol{2I1Xi4c2}{2}{\draw[kernels2] (0,0) node[not] {} -- (-1,1) node[not] {};%58
\draw[kernels2] (0,0) -- (1,1) node[not] {};
\draw[kernels2] (-1,1) -- (-1.5,2.5) node[xic] {};
\draw[kernels2] (-1,1) -- (-0.5,2.5) node[xic] {};
\draw[kernels2] (1,1) -- (0.5,2.5) node[xi] {};
\draw[kernels2] (1,1) -- (1.5,2.5) node[xi] {};
}

\DeclareSymbol{2I1Xi4b}{2}{\draw[kernels2] (0,0) node[not] {} -- (-1,1) ;%59
\draw[kernels2] (0,0) -- (1,1);
\draw (-1,1) node[xi] {} -- (-1,2.5) node[xi] {};
\draw (1,1)  node[xi] {} -- (1,2.5) node[xi] {};
}

\DeclareSymbol{2I1Xi4bb}{2}{\draw[kernels2] (0,0) node[not] {} -- (-1,1) ;%59
\draw[kernels2] (0,0) -- (1,1);
\draw (-1,1) node[xi] {} -- (-1,2.5) node[xiesf] {};
\draw (1,1)  node[xi] {} -- (1,2.5) node[xic] {};
}

\DeclareSymbol{2I1Xi4c}{2}{\draw[kernels2] (0,0) node[not] {} -- (-1,1);%60
\draw[kernels2] (0,0) -- (1,1) node[not] {};
\draw (-1,1)  node[xi] {} -- (-1,2.5) node[xi] {};
\draw[kernels2] (1,1) -- (0.4,2.5) node[xi] {};
\draw[kernels2] (1,1) -- (1.6,2.5) node[xi] {};
}

\DeclareSymbol{2I1Xi4cc1}{2}{\draw[kernels2] (0,0) node[not] {} -- (-1,1);%61
\draw[kernels2] (0,0) -- (1,1) node[not] {};
\draw (-1,1)  node[xic] {} -- (-1,2.5) node[xi] {};
\draw[kernels2] (1,1) -- (0.4,2.5) node[xic] {};
\draw[kernels2] (1,1) -- (1.6,2.5) node[xi] {};
}

\DeclareSymbol{2I1Xi4cc2}{2}{\draw[kernels2] (0,0) node[not] {} -- (-1,1);%62
\draw[kernels2] (0,0) -- (1,1) node[not] {};
\draw (-1,1)  node[xic] {} -- (-1,2.5) node[xic] {};
\draw[kernels2] (1,1) -- (0.4,2.5) node[xi] {};
\draw[kernels2] (1,1) -- (1.6,2.5) node[xi] {};
}

\DeclareSymbol{Xi4ba}{0}{\draw(-0.5,1.5) node[xi] {} -- (0,0); \draw (-1.5,1) node[xi] {} -- (0,0) node[not] {}; \draw[kernels2] (0,0) -- (1.5,1) node[xi] {};
\draw[kernels2] (0,0) -- (0.5,1.5) node[xi] {} ;}%63

\DeclareSymbol{Xi4badis}{0}{\draw(-0.5,1.5) node[xies] {} -- (0,0); \draw (-1.5,1) node[xies] {} -- (0,0) node[not] {}; \draw[kernels2] (0,0) -- (1.5,1) node[xies] {};
\draw[kernels2] (0,0) -- (0.5,1.5) node[xies] {} ;}%63

\DeclareSymbol{Xi4ba1}{0}{\draw(-0.5,1.5) node[xi] {} -- (0,0); \draw (-1.5,1) node[xi] {} -- (0,0) node[not] {}; \draw[kernels2] (0,0) -- (1.5,1) node[xic] {};
\draw[kernels2] (0,0) -- (0.5,1.5) node[xic] {} ;}%64

\DeclareSymbol{Xi4ba1b}{0}{\draw(-0.5,1.5) node[xic] {} -- (0,0); \draw (-1.5,1) node[xic] {} -- (0,0) node[not] {}; \draw[kernels2] (0,0) -- (1.5,1) node[xi] {};
\draw[kernels2] (0,0) -- (0.5,1.5) node[xi] {} ;}%64

\DeclareSymbol{Xi4ba1bdiff}{0}{\draw(-0.5,1.5) node[xic] {} -- (0,0); \draw (-1.5,1) node[xic] {} -- (0,0) node[not] {}; \draw (0,0) -- (1.5,1) node[xi] {};
\draw (0,0) -- (0.5,1.5) node[xi] {};
\draw(0,0) node[diff] {};}%64

\DeclareSymbol{Xi4ba1bb}{0}{\draw(-0.5,1.5) node[xic] {} -- (0,0); \draw (-1.5,1) node[xiesf] {} -- (0,0) node[not] {}; \draw[kernels2] (0,0) -- (1.5,1) node[xi] {};
\draw[kernels2] (0,0) -- (0.5,1.5) node[xi] {} ;}%64

\DeclareSymbol{Xi4ba2}{0}{\draw(-0.5,1.5) node[xi] {} -- (0,0); \draw (-1.5,1) node[xic] {} -- (0,0) node[not] {}; \draw[kernels2] (0,0) -- (1.5,1) node[xi] {};
\draw[kernels2] (0,0) -- (0.5,1.5) node[xic] {} ;}%65

\DeclareSymbol{Xi4ba2b}{0}{\draw(-0.5,1.5) node[xi] {} -- (0,0); \draw (-1.5,1) node[xic] {} -- (0,0) node[not] {}; \draw[kernels2] (0,0) -- (1.5,1) node[xi] {};
\draw[kernels2] (0,0) -- (0.5,1.5) node[xiesf] {} ;}%65

\DeclareSymbol{Xi4ca}{0}{\draw (0,1) -- (-1,2.2) node[xi] {};\draw (0,-0.25) node[xi] {} -- (0,1) ; \draw[kernels2] (0,1) node[not] {} -- (1,2.2) node[xi] {};%67
\draw[kernels2] (0,1) {} -- (0,2.7) node[xi] {};
}

\DeclareSymbol{Xi4cb}{0}{\draw (-1,1) -- (-2,2) node[xi] {};\draw[kernels2] (0,0)  -- (-1,1) node[xi] {} ; \draw[kernels2] (0,0) node[not] {} -- (1,1) node[xi] {} ; 
\draw (-1,1) node[xi] {} -- (0,2) node[xi] {};}

\DeclareSymbol{Xi4cbb}{0}{\draw (-1,1) -- (-2,2) node[xiesf] {};\draw[kernels2] (0,0)  -- (-1,1) node[xi] {} ; \draw[kernels2] (0,0) node[not] {} -- (1,1) node[xi] {} ; 
\draw (-1,1) node[xi] {} -- (0,2) node[xic] {};}

\DeclareSymbol{Xi4cbc1}{0}{\draw (-1,1) -- (-2,2) node[xic] {};\draw[kernels2] (0,0)  -- (-1,1) node[xic] {} ; \draw[kernels2] (0,0) node[not] {} -- (1,1) node[xi] {} ; 
\draw (-1,1) node[xic] {} -- (0,2) node[xi] {};}

\DeclareSymbol{Xi4cbc2}{0}{\draw (-1,1) -- (-2,2) node[xi] {};\draw[kernels2] (0,0)  -- (-1,1) node[xi] {} ; \draw[kernels2] (0,0) node[not] {} -- (1,1) node[xic] {} ; 
\draw (-1,1) node[xic] {} -- (0,2) node[xi] {};}

\DeclareSymbol{Xi4cab}{0}{\draw (-1,1) -- (-2,2) node[xi] {};\draw[kernels2] (0,0)  -- (-1,1); \draw[kernels2] (0,0) node[not] {} -- (1,1) node[xi] {} ; 
\draw[kernels2] (-1,1)  {} -- (0,2) node[xi] {};
\draw[kernels2] (-1,1) node[not] {} -- (-1,2.5) node[xi] {};
}%69

\DeclareSymbol{Xi4cabdis}{0}{\draw (-1,1) -- (-2,2) node[xies] {};\draw[kernels2] (0,0)  -- (-1,1); \draw[kernels2] (0,0) node[not] {} -- (1,1) node[xies] {} ; 
\draw[kernels2] (-1,1)  {} -- (0,2) node[xies] {};
\draw[kernels2] (-1,1) node[not] {} -- (-1,2.5) node[xies] {};
}%69

\DeclareSymbol{Xi4cabc1}{0}{\draw (-1,1) -- (-2,2) node[xi] {};\draw[kernels2] (0,0)  -- (-1,1); \draw[kernels2] (0,0) node[not] {} -- (1,1) node[xic] {} ; 
\draw[kernels2] (-1,1)  {} -- (0,2) node[xic] {};
\draw[kernels2] (-1,1) node[not] {} -- (-1,2.5) node[xi] {};
}%69

\DeclareSymbol{Xi4cabc2}{0}{\draw (-1,1) -- (-2,2) node[xic] {};\draw[kernels2] (0,0)  -- (-1,1); \draw[kernels2] (0,0) node[not] {} -- (1,1) node[xic] {} ; 
\draw[kernels2] (-1,1)  {} -- (0,2) node[xi] {};
\draw[kernels2] (-1,1) node[not] {} -- (-1,2.5) node[xi] {};
}%69

\DeclareSymbol{Xi4ea}{1.5}{\draw (-1,2.5) node[xi] {} -- (-1,1) node[xi] {} -- (0,0); %71
 \draw[kernels2] (0,0)  -- (1,1) node[xi] {};
\draw[kernels2] (0,0) node[not] {} -- (0,1.5) node[xi] {}; }

\DeclareSymbol{Xi4eac1}{1.5}{\draw (-1,2.5) node[xic] {} -- (-1,1) node[xi] {} -- (0,0); %72
 \draw[kernels2] (0,0)  -- (1,1) node[xic] {};
\draw[kernels2] (0,0) node[not] {} -- (0,1.5) node[xi] {}; }

\DeclareSymbol{Xi4eac1b}{1.5}{\draw (-1,2.5) node[xic] {} -- (-1,1) node[xi] {} -- (0,0); %72
 \draw[kernels2] (0,0)  -- (1,1) node[xiesf] {};
\draw[kernels2] (0,0) node[not] {} -- (0,1.5) node[xi] {}; }

\DeclareSymbol{Xi4eac2}{1.5}{\draw (-1,2.5) node[xic] {} -- (-1,1) node[xic] {} -- (0,0); %73
 \draw[kernels2] (0,0)  -- (1,1) node[xi] {};
\draw[kernels2] (0,0) node[not] {} -- (0,1.5) node[xi] {}; }

\DeclareSymbol{Xi4eact1}{1.5}{\draw (-1,2.5) node[xic] {} -- (-1,1) node[xi] {} -- (0,0); %74
 \draw (0,0)  -- (1,1) node[xic] {};
\draw[rho] (0,0) node[not] {} -- (0,1.5) node[xi] {}; }

\DeclareSymbol{Xi4eact2}{1.5}{\draw[rho] (-1,2.5) node[xic] {} -- (-1,1) node[xi] {} -- (0,0); %75
 \draw (0,0)  -- (1,1) node[xic] {};
\draw (0,0) node[not] {} -- (0,1.5) node[xi] {}; }

\DeclareSymbol{Xi4eabis}{1.5}{\draw (-1,2.5) node[xi] {} -- (-1,1) ; \draw[kernels2] (-1,1) node[xi] {} -- (0,0); 
 \draw (0,0)  -- (1,1) node[xi] {};%76
\draw[kernels2] (0,0) node[not] {} -- (0,1.5) node[xi] {}; }

\DeclareSymbol{Xi4eabisc1}{1.5}{\draw (-1,2.5) node[xic] {} -- (-1,1) ; \draw[kernels2] (-1,1) node[xi] {} -- (0,0); 
 \draw (0,0)  -- (1,1) node[xi] {};%77
\draw[kernels2] (0,0) node[not] {} -- (0,1.5) node[xic] {}; }

\DeclareSymbol{Xi4eabisc1b}{1.5}{\draw (-1,2.5) node[xic] {} -- (-1,1) ; \draw[kernels2] (-1,1) node[xi] {} -- (0,0); 
 \draw (0,0)  -- (1,1) node[xi] {};%77
\draw[kernels2] (0,0) node[not] {} -- (0,1.5) node[xiesf] {}; }

\DeclareSymbol{Xi4eabisc1bis}{1.5}{\draw (-1,2.5) node[xi] {} -- (-1,1) ; \draw[kernels2] (-1,1) node[xi] {} -- (0,0); 
 \draw (0,0)  -- (1,1) node[xi] {};%78
\draw[kernels2] (0,0) node[not] {} -- (0,1.5) node[xi] {};
\draw (-2,1) node[] {\tiny{$i$}};
\draw (-2,2.5) node[] {\tiny{$\ell$}};
\draw (2,1) node[] {\tiny{$k$}};
\draw (0,2.5) node[] {\tiny{$j$}};
 }

\DeclareSymbol{Xi4eabisc1tris}{1.5}{\draw (-1,2.5) node[xi] {} -- (-1,1) ; \draw[kernels2] (-1,1) node[xi] {} -- (0,0); 
 \draw (0,0)  -- (1,1) node[xi] {};%78
\draw[kernels2] (0,0) node[not] {} -- (0,1.5) node[xi] {};
\draw (-2,1) node[] {\tiny{i}};
\draw (-2,2.5) node[] {\tiny{j}};
\draw (2,1) node[] {\tiny{j}};
\draw (0,2.5) node[] {\tiny{i}};
 }

\DeclareSymbol{Xi4eabisc1quater}{1.5}{\draw (-1,2.5) node[xic] {} -- (-1,1) ; \draw[kernels2] (-1,1) node[xi] {} -- (0,0); 
 \draw (0,0)  -- (1,1) node[xic] {};%78
\draw[kernels2] (0,0) node[not] {} -- (0,1.5) node[xi] {};
 }

\DeclareSymbol{Xi4eabisc2}{1.5}{\draw (-1,2.5) node[xic] {} -- (-1,1) ; \draw[kernels2] (-1,1) node[xi] {} -- (0,0); %79
 \draw (0,0)  -- (1,1) node[xic] {};
\draw[kernels2] (0,0) node[not] {} -- (0,1.5) node[xi] {}; }

\DeclareSymbol{Xi4eabisc2l}{1.5}{\draw (-1,2.5) node[xiesf] {} -- (-1,1) ; \draw[kernels2] (-1,1) node[xi] {} -- (0,0); %79
 \draw (0,0)  -- (1,1) node[xic] {};
\draw[kernels2] (0,0) node[not] {} -- (0,1.5) node[xi] {}; }

\DeclareSymbol{Xi4eabisc2r}{1.5}{\draw (-1,2.5) node[xic] {} -- (-1,1) ; \draw[kernels2] (-1,1) node[xi] {} -- (0,0); %79
 \draw (0,0)  -- (1,1) node[xiesf] {};
\draw[kernels2] (0,0) node[not] {} -- (0,1.5) node[xi] {}; }

\DeclareSymbol{Xi4eabisc3}{1.5}{\draw (-1,2.5) node[xic] {} -- (-1,1) ; \draw[kernels2] (-1,1) node[xic] {} -- (0,0); %80
 \draw (0,0)  -- (1,1) node[xi] {};
\draw[kernels2] (0,0) node[not] {} -- (0,1.5) node[xi] {}; }

\DeclareSymbol{Xi4eb}{0}{%81
\draw[kernels2] (0,2) node[xi] {} -- (-1,1) ; \draw[kernels2] (-2,2)  node[xi] {} -- (-1,1) ; \draw (-1,1)  node[not] {} -- (0,0); 
 \draw (0,0) node[xi] {}  -- (1,1) node[xi] {};
}

\DeclareSymbol{Xi4eab}{1.5}{\draw[kernels2] (-1,2.5) node[xi] {} -- (-1,1) ; \draw[kernels2] (-2,2)  node[xi] {} -- (-1,1) ; \draw (-1,1)  node[not] {} -- (0,0); %82
 \draw[kernels2] (0,0)  -- (1,1) node[xi] {};
\draw[kernels2] (0,0) node[not] {} -- (0,1.5) node[xi] {}; 
}

\DeclareSymbol{Xi4eabdis}{1.5}{\draw[kernels2] (-1,2.5) node[xies] {} -- (-1,1) ; \draw[kernels2] (-2,2)  node[xies] {} -- (-1,1) ; \draw (-1,1)  node[not] {} -- (0,0); %82
 \draw[kernels2] (0,0)  -- (1,1) node[xies] {};
\draw[kernels2] (0,0) node[not] {} -- (0,1.5) node[xies] {}; 
}

\DeclareSymbol{Xi4eabc1}{1.5}{\draw[kernels2] (-1,2.5) node[xic] {} -- (-1,1) ; \draw[kernels2] (-2,2)  node[xi] {} -- (-1,1) ; \draw (-1,1)  node[not] {} -- (0,0); %83
 \draw[kernels2] (0,0)  -- (1,1) node[xic] {};
\draw[kernels2] (0,0) node[not] {} -- (0,1.5) node[xi] {}; 
}

\DeclareSymbol{Xi4eabc2}{1.5}{\draw[kernels2] (-1,2.5) node[xi] {} -- (-1,1) ; \draw[kernels2] (-2,2)  node[xi] {} -- (-1,1) ; \draw (-1,1)  node[not] {} -- (0,0); %84
 \draw[kernels2] (0,0)  -- (1,1) node[xic] {};
\draw[kernels2] (0,0) node[not] {} -- (0,1.5) node[xic] {}; 
}

\DeclareSymbol{Xi4eabbis}{1.5}{\draw[kernels2] (-1,2.5) node[xi] {} -- (-1,1) ; \draw[kernels2] (-2,2)  node[xi] {} -- (-1,1) ; \draw[kernels2] (-1,1)  node[not] {} -- (0,0); %85
 \draw (0,0)  -- (1,1) node[xi] {};
\draw[kernels2] (0,0) node[not] {} -- (0,1.5) node[xi] {}; 
}

\DeclareSymbol{Xi4eabbisc1}{1.5}{\draw[kernels2] (-1,2.5) node[xic] {} -- (-1,1) ; \draw[kernels2] (-2,2)  node[xi] {} -- (-1,1) ; \draw[kernels2] (-1,1)  node[not] {} -- (0,0); %86
 \draw (0,0)  -- (1,1) node[xic] {};
\draw[kernels2] (0,0) node[not] {} -- (0,1.5) node[xi] {}; 
}

\DeclareSymbol{Xi4eabbisc1perm}{1.5}{\draw[kernels2] (-1,2.5) node[xi] {} -- (-1,1) ; \draw[kernels2] (-2,2)  node[xic] {} -- (-1,1) ; \draw[kernels2] (-1,1)  node[not] {} -- (0,0); %86
 \draw (0,0)  -- (1,1) node[xic] {};
\draw[kernels2] (0,0) node[not] {} -- (0,1.5) node[xi] {}; 
}

\DeclareSymbol{Xi4eabbisc2}{1.5}{\draw[kernels2] (-1,2.5) node[xi] {} -- (-1,1) ; \draw[kernels2] (-2,2)  node[xi] {} -- (-1,1) ; \draw[kernels2] (-1,1)  node[not] {} -- (0,0); %87
 \draw (0,0)  -- (1,1) node[xic] {};
\draw[kernels2] (0,0) node[not] {} -- (0,1.5) node[xic] {}; 
}

\DeclareSymbol{Xi2cbis}{0}{\draw[kernels2] (0,1) -- (0.8,2.2) node[xi] {};\draw[kernels2] (0,-0.25) node[not] {} -- (0,1); \draw[kernels2] (0,1) node[not] {} -- (-0.8,2.2) node[xi] {};}%88

\DeclareSymbol{Xi2cbis1}{0}{\draw (0,1) -- (-0.8,2.2) node[xi] {};\draw[kernels2] (0,-0.25) node[not] {} -- (0,1) node[xi] {}; }%89

\DeclareSymbol{Xi2Xbis}{-2}{\draw[kernels2] (0,-0.25)  -- (-1,1) ; \draw (-1,1) node[xix] {};%91
\draw[kernels2] (0,-0.25) node[not] {} -- (1,1) node[xi] {};}

\DeclareSymbol{XXi2bis}{-2}{\draw[kernels2] (0,-0.25) -- (-1,1) node[xi] {};%92
\draw[kernels2] (0,-0.25) node[X] {} -- (1,1) node[xi] {};}

\DeclareSymbol{I1XiIXi}{0}{\draw[kernels2] (0,-0.25) -- (1,1) node[xi] {};%93
\draw (0,-0.25) node[not] {} -- (-1,1) node[xi] {};}

\DeclareSymbol{I1XiIXib}{0}{\draw  (0,-0.25) node[xi] {} -- (0,1) node[not] {};
\draw[kernels2] (0,1) -- (0,2.25) ; \draw (0,2.25) node[xi]{}; }%94

\DeclareSymbol{I1XiIXic}{0}{%95
\draw[kernels2] (0,0) -- (1,1) node[xi] {} ; 
\draw[kernels2] (0,0) node[not] {}  -- (-1,1) node[not] {} -- (0,2) node[xi] {};
}

\DeclareSymbol{thin}{1.4}{\draw[pagebackground] (-0.3,0) -- (0.3,0); \draw  (0,0) -- (0,2);}
\DeclareSymbol{thin2}{1.4}{\draw[pagebackground] (-0.3,0) -- (0.3,0); \draw[tinydots]  (0,0) -- (0,2);}

\DeclareSymbol{thick}{1.4}{\draw[pagebackground] (-0.3,0) -- (0.3,0); \draw[kernels2]  (0,0) -- (0,2);}

\DeclareSymbol{thick2}{1.4}{\draw[pagebackground] (-0.3,0) -- (0.3,0); \draw[kernels2,tinydots]  (0,0) -- (0,2);}

\DeclareSymbol{Xi4ind}{2}{\draw (0,0) node[xi,label={[label distance=-0.2em]right: \scriptsize  $ i $}]  { } -- (-1,1) node[xi,label={[label distance=-0.2em]left: \scriptsize  $ j $}] {} -- (0,2) node[xi,label={[label distance=-0.2em]right: \scriptsize  $ k $}] {} -- (-1,3) node[xi,label={[label distance=-0.2em]left: \scriptsize  $ \ell $}] {};}%115

\DeclareSymbol{Xi4c1}{2}{\draw (0,0) node[xic] {} -- (-1,1) node[xi] {} -- (0,2) node[xic] {} -- (-1,3) node[xi] {};} %119
\DeclareSymbol{IXi2ex}{0}{\draw (0,-0.25) node[xie] {} -- (-1,1) node[xi] {} ; \draw (0,-0.25)-- (1,1) node[xi] {};}
\DeclareSymbol{IXi2ex1}{0}{\draw (0,-0.25) node[xie] {} -- (-1,1) node[xi] {} -- (0,2) node[xi] {};}

\DeclareSymbol{Xi4b1}{0}{\draw(0,1.5) node[xic] {} -- (0,0); \draw (-1,1) node[xic] {} -- (0,0) node[xi] {} -- (1,1) node[xi] {};}

\DeclareSymbol{Xi4ec1}{0}{\draw (0,2) node[xi] {} -- (-1,1) node[xic] {} -- (0,0) node[xic] {} -- (1,1) node[xi] {};}
\DeclareSymbol{Xi4ec2}{0}{\draw (0,2) node[xic] {} -- (-1,1) node[xi] {} -- (0,0) node[xic] {} -- (1,1) node[xi] {};}
\DeclareSymbol{Xi4ec3}{0}{\draw (0,2) node[xic] {} -- (-1,1) node[xic] {} -- (0,0) node[xi] {} -- (1,1) node[xi] {};}

\DeclareSymbol{I1Xi4ac1}{2}{\draw[kernels2] (0,0) node[not] {} -- (-1,1) ; \draw[kernels2] (0,0) node[not] {} -- (1,1) node[xic] {} ;
\draw (-1,1) node[xi] {} -- (0,2) node[xic] {} -- (-1,3) node[xi] {};}%161

\DeclareSymbol{I1Xi4ac2}{2}{\draw[kernels2] (0,0) node[not] {} -- (-1,1) ; \draw[kernels2] (0,0) node[not] {} -- (1,1) node[xic] {} ;
\draw (-1,1) node[xi] {} -- (0,2) node[xi] {} -- (-1,3) node[xic] {};}

\DeclareSymbol{I1Xi4bp}{2}{\draw (0,0) node[not] {} -- (-1,1) node[xi] {} -- (0,2) ; \draw[kernels2] (0,2) node[not] {} -- (-1,3) node[xi] {};\draw[kernels2] (0,2)  -- (1,3) node[xi] {};
}

\DeclareSymbol{I1Xi4bc1}{2}{\draw (0,0) node[xic] {} -- (-1,1) node[xi] {} -- (0,2) ; \draw[kernels2] (0,2) node[not] {} -- (-1,3) node[xi] {};\draw[kernels2] (0,2)  -- (1,3) node[xic] {};
}%165

\DeclareSymbol{I1Xi4bc2}{2}{\draw (0,0) node[xic] {} -- (-1,1) node[xi] {} -- (0,2) ; \draw[kernels2] (0,2) node[not] {} -- (-1,3) node[xic] {};\draw[kernels2] (0,2)  -- (1,3) node[xi] {};
}

\DeclareSymbol{I1Xi4cp}{2}{\draw (0,0) node[not] {} -- (-1,1) node[not] {}; \draw[kernels2] (-1,1) -- (0,2) ; 
\draw[kernels2] (-1,1) -- (-2,2) node[xi] {} ;
\draw (0,2) node[xi] {} -- (-1,3) node[xi] {};}

\DeclareSymbol{I1Xi4cc1}{2}{\draw (0,0) node[xic] {} -- (-1,1) node[not] {}; \draw[kernels2] (-1,1) -- (0,2) ; 
\draw[kernels2] (-1,1) -- (-2,2) node[xi] {} ;
\draw (0,2) node[xic] {} -- (-1,3) node[xi] {};}%169

\DeclareSymbol{I1Xi4cc2}{2}{\draw (0,0) node[xic] {} -- (-1,1) node[not] {}; \draw[kernels2] (-1,1) -- (0,2) ; 
\draw[kernels2] (-1,1) -- (-2,2) node[xi] {} ;
\draw (0,2) node[xi] {} -- (-1,3) node[xic] {};}

\DeclareSymbol{I1Xi4abc1}{2}{\draw[kernels2] (0,0) node[not] {} -- (-1,1) ; \draw[kernels2] (0,0) node[not] {} -- (1,1) node[xic] {};\draw (-1,1) node[xi] {} -- (0,2) ; \draw[kernels2] (0,2) node[not] {} -- (-1,3) node[xic] {};\draw[kernels2] (0,2)  -- (1,3) node[xi] {}; }

\DeclareSymbol{I1Xi4abc2}{2}{\draw[kernels2] (0,0) node[not] {} -- (-1,1) ; \draw[kernels2] (0,0) node[not] {} -- (1,1) node[xic] {};\draw (-1,1) node[xi] {} -- (0,2) ; \draw[kernels2] (0,2) node[not] {} -- (-1,3) node[xi] {};\draw[kernels2] (0,2)  -- (1,3) node[xic] {}; }%173

\DeclareSymbol{R1}{0}{\draw (-1,1) node[xi] {} -- (0,0) node[not] {};%131
\draw[kernels2] (0,1.5) node[xic] {} -- (0,0) -- (1,1) node[xic] {};}
\DeclareSymbol{R2}{0}{\draw (-1,1) node[xic] {} -- (0,0) node[not] {};
\draw[kernels2] (0,1.5)  {} -- (0,0) -- (1,1)  {};
\draw (0,1.5) node[xi] {};
\draw (1,1) node[xic] {};
}%132
\DeclareSymbol{R3}{1}{\draw[kernels2] (-1,1.5)  {} -- (0,0) node[not] {} -- (1,1.5);%133
\draw (-1,1.5) node[xi] {};
\draw[kernels2] (0,3) {} -- (1,1.5) -- (2,3)  {};
\draw  (0,3) node[xic] {} ;
\draw (2,3) node[xic] {};}
\DeclareSymbol{R4}{1}{\draw[kernels2] (-1,1.5) node[xic] {} -- (0,0) node[not] {} -- (1,1.5);%134
\draw[kernels2] (0,3) {} -- (1,1.5) -- (2,3) node[xic] {};
\draw (0,3) node[xi] {};}

\DeclareSymbol{I1Xi4bcp}{2}{\draw (0,0) node[not] {} -- (-1,1) node[not] {}; \draw[kernels2] (-1,1) -- (0,2) ; %175
\draw[kernels2] (-1,1) -- (-2,2) node[xi] {} ; \draw[kernels2] (0,2) node[not] {} -- (-1,3) node[xi] {};\draw[kernels2] (0,2)  -- (1,3) node[xi] {};
}

\DeclareSymbol{I1Xi4bcc1}{2}{\draw (0,0) node[xic] {} -- (-1,1) node[not] {}; \draw[kernels2] (-1,1) -- (0,2) ; 
\draw[kernels2] (-1,1) -- (-2,2) node[xi] {} ; \draw[kernels2] (0,2) node[not] {} -- (-1,3) node[xi] {};\draw[kernels2] (0,2)  -- (1,3) node[xic] {};
}

\DeclareSymbol{I1Xi4bcc2}{2}{\draw (0,0) node[xic] {} -- (-1,1) node[not] {}; \draw[kernels2] (-1,1) -- (0,2) ; 
\draw[kernels2] (-1,1) -- (-2,2) node[xi] {} ; \draw[kernels2] (0,2) node[not] {} -- (-1,3) node[xic] {};\draw[kernels2] (0,2)  -- (1,3) node[xi] {};
} %177

\DeclareSymbol{2I1Xi4bc1}{2}{\draw[kernels2] (0,0) node[not] {} -- (-1,1) ;
\draw[kernels2] (0,0) -- (1,1);
\draw (-1,1) node[xic] {} -- (-1,2.5) node[xi] {};
\draw (1,1)  node[xic] {} -- (1,2.5) node[xi] {};
}

\DeclareSymbol{2I1Xi4bc2}{2}{\draw[kernels2] (0,0) node[not] {} -- (-1,1) ;
\draw[kernels2] (0,0) -- (1,1);
\draw (-1,1) node[xi] {} -- (-1,2.5) node[xic] {};
\draw (1,1)  node[xic] {} -- (1,2.5) node[xi] {};
}

\DeclareSymbol{diff2I1Xi4bc2}{2}{\draw (0,0) node[diff] {} -- (-1,1) ;
\draw (0,0) -- (1,1);
\draw (-1,1) node[xi] {} -- (-1,2.5) node[xic] {};
\draw (1,1)  node[xic] {} -- (1,2.5) node[xi] {};
}

\DeclareSymbol{2I1Xi4bc3}{2}{\draw[kernels2] (0,0) node[not] {} -- (-1,1) ;
\draw[kernels2] (0,0) -- (1,1);
\draw (-1,1) node[xic] {} -- (-1,2.5) node[xic] {};
\draw (1,1)  node[xi] {} -- (1,2.5) node[xi] {};
}

\DeclareSymbol{Xi41}{0}{\draw (0,1) -- (0.8,2.2) node[xic] {};\draw (0,-0.25) node[xi] {} -- (0,1) node[xi] {} -- (-0.8,2.2) node[xic] {};} 

\DeclareSymbol{Xi42}{0}{\draw (0,1) -- (0.8,2.2) node[xi] {};\draw (0,-0.25) node[xic] {} -- (0,1) node[xi] {} -- (-0.8,2.2) node[xic] {};}

\DeclareSymbol{Xi4ca1}{0}{\draw (0,1) -- (-1,2.2) node[xic] {};\draw (0,-0.25) node[xi] {} -- (0,1) ; \draw[kernels2] (0,1) node[not] {} -- (1,2.2) node[xic] {};
\draw[kernels2] (0,1) {} -- (0,2.7) node[xi] {};
}

\DeclareSymbol{Xi4ca2}{0}{\draw (0,1) -- (-1,2.2) node[xi] {};\draw (0,-0.25) node[xi] {} -- (0,1) ; \draw[kernels2] (0,1) node[not] {} -- (1,2.2) node[xic] {};
\draw[kernels2] (0,1) {} -- (0,2.7) node[xic] {};
}

\DeclareSymbol{Xi4cap}{0}{\draw (0,1) -- (-1,2.2) node[xi] {};\draw (0,-0.25) node[not] {} -- (0,1) ; \draw[kernels2] (0,1) node[not] {} -- (1,2.2) node[xi] {};
\draw[kernels2] (0,1) {} -- (0,2.7) node[xi] {};
}

\DeclareSymbol{Xi3a}{0}{
 \draw (-1,1)  node[xi] {} -- (0,0); 
 \draw (0,0) node[xi] {}  -- (1,1) node[xi] {};
}

\DeclareSymbol{Xi4ebc1}{0}{
\draw[kernels2] (0,2) node[xi] {} -- (-1,1) ; \draw[kernels2] (-2,2)  node[xic] {} -- (-1,1) ; \draw (-1,1)  node[not] {} -- (0,0); 
 \draw (0,0) node[xic] {}  -- (1,1) node[xi] {};
}

\DeclareSymbol{Xi4ebc2}{0}{
\draw[kernels2] (0,2) node[xi] {} -- (-1,1) ; \draw[kernels2] (-2,2)  node[xi] {} -- (-1,1) ; \draw (-1,1)  node[not] {} -- (0,0); 
 \draw (0,0) node[xic] {}  -- (1,1) node[xic] {};
}

\DeclareSymbol{Xi2cbispex}{0}{\draw[kernels2] (0,1) -- (0.8,2.2) node[xi] {};\draw (0,-0.25) node[xie] {} -- (0,1); \draw[kernels2] (0,1) node[not] {} -- (-0.8,2.2) node[xi] {};}

\DeclareSymbol{Xi2cbis1p}{0}{\draw (0,1) -- (-0.8,2.2) node[xi] {};\draw (0,-0.25) node[not] {} -- (0,1) node[xi] {}; }

\DeclareSymbol{Xi2Xp}{-2}{\draw (0,-0.25) node[not] {} -- (-1,1) node[xix] {};} % 229 not used

\DeclareSymbol{I1XiIXib}{0}{\draw  (0,-0.25) node[xi] {} -- (0,1) node[not] {};
\draw[kernels2] (0,1) -- (0,2.25) ; \draw (0,2.25) node[xi]{}; }

\DeclareSymbol{IXi2b}{0}{\draw  (0,-0.25) node[xi] {} -- (0,1) node[not] {};
\draw (0,1) -- (0,2.25) ; \draw (0,2.25) node[xi]{}; }

\DeclareSymbol{IXi2bex}{0}{\draw  (0,-0.25) node[xi] {} -- (0,1) node[xie] {};
\draw (0,1) -- (0,2.25) ; \draw (0,2.25) node[xi]{}; }

 \def\1{\mathbf{\symbol{1}}}

\DeclareSymbol{diff}{0}{
\draw (0,0.5) node[diff] {};
}

\DeclareSymbol{diff1}{0}{
\draw (0,0.5) node[diff1] {};
}

\DeclareSymbol{diff2}{0}{
\draw (0,0.5) node[diff2] {};
}

\DeclareSymbol{geo}{0}{
\draw (0,0) node[diff] {};
\draw (0.3,0) node[diff] {};
}

\DeclareSymbol{generic}{0}{
\draw (0,0.6) node[xi] {};
}

\DeclareSymbol{g}{0}{
\draw (0,0.6) node[g] {};
}

\DeclareSymbol{Ito}{0}{
\draw (0,0.6) node[xies] {};
}

\DeclareSymbol{Itob}{0}{
\draw (0,0.6) node[xiesf] {};
}

\DeclareSymbol{greycirc}{0}{
\draw (0,0.3) node[xi] {};
}

\DeclareSymbol{not}{0}{
\draw (0,0.6) node[not] {};
\draw[tinydots] (0,0.6) circle (0.8);
}

\DeclareSymbol{genericb}{0}{
\draw (0,0.6) node[xic] {};
}

\DeclareSymbol{bluecirc}{0}{
\draw (0,0.3) node[xic] {};
}

\DeclareSymbol{genericxix}{0}{
\draw (0,0.6) node[xix] {};
}

\DeclareSymbol{genericX}{0}{
\draw (0,0.6) node[X] {};
}

\DeclareSymbol{diffIto}{1}{
\draw  (0,2.5) -- (0,0) ;
\draw (0,-0.1) node[diff] {};
\draw (0,2.5) node[xies] {};
}
\DeclareSymbol{Itodiff}{2}{
\draw(0,2.9) -- (0,-0.2);
\draw (0,2.9) node[diff] {};
\draw (0,-0.1) node[xies] {};
}

\DeclareSymbol{diffgeneric}{1}{
\draw  (0,2.5) -- (0,0) ;
\draw (0,-0.1) node[diff] {};
\draw (0,2.5) node[xi] {};
}

\DeclareSymbol{genericdiff}{2}{
\draw(0,2.9) -- (0,-0.2);
\draw (0,2.9) node[diff] {};
\draw (0,-0.1) node[xi] {};
}

\DeclareSymbol{diffdot}{2}{
\draw  (0,3) -- (0,-0.1) ;
\draw (0,3) node[not] {};
\draw (0,-0.1) node[diff] {};
}

\DeclareSymbol{diffdotmini}{0}{
\draw  (0,0) -- (0,1.2) ;
\draw (0,1.2) node[not] {};
\draw (0,0) node[diffmini] {};
}

\DeclareSymbol{dotdiff}{2}{
\draw[kernelsmod]  (0,3) -- (0,-0.1) ;
\draw (0,3) node[diff] {};
\draw (0,-0.1) node[not] {};
}

\DeclareSymbol{dotdiff1}{2}{
\draw[kernelsmod]  (0,3) -- (0,-0.1) ;
\draw (0,3) node[diff1] {};
\draw (0,-0.1) node[not] {};
}

\DeclareSymbol{dotdiff1mini}{0}{
\draw[kernelsmod]  (0,1.2) -- (0,0) ;
\draw (0,1.2) node[diffmini] {};
\draw (0,0) node[not] {};
}

\DeclareSymbol{dotdiff2}{2}{
\draw (0,3) -- (0,-0.1) ;
\draw (0,3) node[diff] {};
\draw (0,-0.1) node[not] {};
}

\DeclareSymbol{dotdiff2mini}{0}{
\draw (0,1.2) -- (0,0) ;
\draw (0,1.2) node[diffmini] {};
\draw (0,0) node[not] {};
}

\DeclareSymbol{dotdiffstraight}{0}{
\draw  (0,3) -- (0,-0.1) ;
\draw (0,3) node[diff] {};
\draw (0,-0.1) node[not] {};
}

\DeclareSymbol{arbre1}{0}{
\draw  (0,0) -- (1.5,1.5) ;
\draw (1.5,1.5) node[not] {};
\draw (0,0) node[not] {};
}

\DeclareSymbol{arbre2}{0}{
\draw  (0,0) -- (1.5,1.5) ;
\draw[kernelsmod] (0,0) -- (-1.5,1.5);
\draw (1.5,1.5) node[not] {};
\draw (0,0) node[not] {};
\draw (-1.5,1.5) node[xi] {};
}

\DeclareSymbol{arbre3}{0}{
\draw  (0,0) -- (1.5,1.5) ;
\draw[kernelsmod] (1.5,1.5) -- (0,3);
\draw (0,0) node[not] {};
\draw (1.5,1.5) node[not] {};
\draw (0,3) node[xi] {};
}

\DeclareSymbol{treeeval}{0}{
\draw (0,0) -- (1,1);
\draw (0,0) node[xi] {};
\draw (1.25,1.25) node[xi] {};
\draw (-0.6,0.6) node[]{\tiny{$i$}};
\draw (0.65,1.85) node[]{\tiny{$j$}};
}

\DeclareSymbol{testeval}{0}{
\draw (0,0) -- (1,1);
\draw (0,0) -- (-1,1);
\draw (0,0) node[xi] {};
\draw (1.25,1.25) node[xi] {};
\draw (-1.25,1.25) node[xi] {};
\draw (-0.6,-0.6) node[]{\tiny{$i$}};
\draw (0.65,1.85) node[]{\tiny{$j$}};
\draw (-1.95,1.85) node[]{\tiny{$k$}};
}

\DeclareSymbol{treeeval2}{0}{
\draw[kernelsmod] (-0.25,-1) -- (1,0.5) ;
\draw[kernelsmod] (1,0.5) -- (-0.25,2);
\draw (1,0.5) node[diff2] {};
\draw (-0.25,-1) node[not] {};
\draw (-0.25,2) node[xi] {};
\draw (-0.6,1.2) node[]{\tiny{1}};
}

\DeclareSymbol{arbreact}{1}{
\draw (0,0) node[not] {};
\draw[kernelsmod] (0,0) -- (1,1);
\draw[kernelsmod] (0,0) -- (-1,1);
\draw (-1,1) node[xic] {};
\draw  (0,2) -- (1,1) ;
\draw (0,2) node[xic] {};
\draw (1,1) node[xi] {};
}

\DeclareSymbol{arbreact1}{0}{
\draw (0,-1.5) -- (0,0);
\draw[kernelsmod] (0,0) -- (1,1);
\draw[kernelsmod] (0,0) -- (-1,1);
\draw  (0,2) -- (1,1) ;
\draw (0,-1.5) node[diff] {};
\draw (0,0) node[not] {};
\draw (-1,1) node[xic] {};
\draw (0,2) node[xic] {};
\draw (1,1) node[xi] {};
}

\DeclareSymbol{arbreact2}{0}{
\draw (0,-0.75) -- (-1,0.5); 
\draw (0,-0.75) -- (1,0.5);
\draw (0,1.5) -- (1,0.5);
\draw (0,1.5) node[xic] {};
\draw (1,0.5) node[xi] {};
\draw (-1,0.5) node[xic] {};
\draw (0,-0.75) node[diff] {};
}

\DeclareSymbol{arbreact3}{0}{
\draw[kernelsmod] (0,-0.75) -- (-1,0.5); 
\draw[kernelsmod] (0,-0.75) -- (1,0.5);
\draw (0,1.75) -- (1,0.5);
\draw (2,1.75) -- (1,0.5);
\draw (0,1.75) node[xic] {};
\draw (1,0.5) node[diff] {};
\draw (-1,0.5) node[xic] {};
\draw (2,1.75) node[xi] {};
\draw (0,-0.75) node[not] {};
}

\DeclareSymbol{pre_im_I1Xitwo}{0}{
\draw[kernels2] (0,-0.3) node[not] {} -- (-0.6,0.7) ;
\draw[kernels2] (0,-0.3) -- (0.6,0.7);
\draw (0,0.9) node[g] {};
}

\DeclareSymbol{pre_im_cI1Xi4ab}{2}{
\draw[kernels2] (0,-1) node[not] {} -- (-0.6,0) ;
\draw[kernels2] (0,-1) -- (0.6,0);
\draw (0,0.2) node[g] {};
\draw (0,0.6) -- (0,1.5);
\draw[kernels2] (0,1.5) node[not] {} -- (-0.6,2.5) ;
\draw[kernels2] (0,1.5) -- (0.6,2.5);
\draw (0,2.7) node[g] {};
}%46

\DeclareSymbol{pre_im_I1Xi4acc2}{0}{
\draw[kernels2] (-1,-0.5) node[not] {} -- (-1.6,0.5) ;
\draw[kernels2] (-1,-0.5) -- (-0.4,0.5);
\draw[kernels2] (-1,-0.5) -- (0.2,-1.5) node[not] {} ;
\draw (-1,1.1) -- (-1,2);
\draw[kernels2] (0.2,-1.5) -- (0.2,2);
\draw (-1,0.7) node[g] {};
\draw (-0.3,2.2) node[g] {};
}

\DeclareSymbol{pre_im_I1Xi4abcc2}{2}{
\draw[kernels2] (0,-1) node[not] {} -- (-1,0) node[not] {};
\draw[kernels2] (-1,1.2) node[not] {} -- (-1,0);
\draw[kernels2] (-1,1.2) -- (-1.5,2.5);
\draw[kernels2] (-1,1.2) -- (-0.5,2.5);
\draw[kernels2] (-1,0) -- (0.7,2.5);
\draw[kernels2] (0,-1) -- (1.5,2.5);
\draw (-1,2.7) node[g] {};
\draw (1,2.7) node[g] {};
}

\DeclareSymbol{pre_im_2I1Xi4c1}{2}{
\draw[kernels2] (0,-0.5) node[not] {} -- (-1,0.5) node[not] {};
\draw[kernels2] (0,-0.5) -- (1,0.5) node[not] {};
\draw[kernels2] (-1,0.5) node[not] {}-- (-1.7,2);
\draw[kernels2]  (-1,2) -- (1,0.5);
\draw[kernels2] (-1,0.5) -- (1,2);
\draw[kernels2] (1,0.5) -- (1.7,2);
\draw (-1.2,2.2) node[g] {};
\draw (1.2,2.2) node[g] {};
}

\DeclareSymbol{pre_im_Xi4eabisc2}{2}{
\draw[kernels2] (1.2,-0.5) node[not] {} -- (-0.7,0.8) ;
\draw[kernels2] (1.2,-0.5) -- (0.4,0.8);
\draw (0,1.4)  -- (0,2.2);
\draw (1.2,2.2) -- (1.2,-0.6);
\draw (0,1) node[g] {};
\draw (0.6,2.4) node[g] {};
}

\DeclareSymbol{pre_im_Xi4eabisc22}{2}{
\draw (1.2,-0.5) node[not] {} -- (-0.7,0.8) ;
\draw[kernels2] (1.2,-0.5) -- (0.4,0.8);
\draw (0,1.4)  -- (0,2.2);
\draw[kernels2] (1.2,2.2) -- (1.2,-0.6);
\draw (0,1) node[g] {};
\draw (0.6,2.4) node[g] {};
}

\DeclareSymbol{pre_im_Xi4eabisc222}{2}{
\draw[kernels2] (0.4,-0.5) node[not] {} -- (-0.6,1) ;
\draw[kernels2] (1.2,0) -- (0.3,1);
\draw (0,1.1)  -- (0,2.5);
\draw[kernels2] (1.2,2.5) -- (1.2,0) node[not] {} -- (0.4,-0.6);
\draw (0,1.2) node[g] {};
\draw (0.6,2.5) node[g] {};
}

\DeclareSymbol{pre_im_Xi4eabc2}{2}{
\draw (0,-0.5) node[not] {} -- (-1,0.5) node[not] {};
\draw[kernels2] (-1,0.5) -- (-1.5,2);
\draw[kernels2] (0,-0.5)  -- (0.7,2);
\draw[kernels2] (-1,0.5) -- (-0.5,2);
\draw[kernels2] (0,-0.5) -- (1.5,2);
\draw (-1,2.2) node[g] {};
\draw (1,2.2) node[g] {};
}

\DeclareSymbol{pre_im_Xi4eabbisc2}{2}{
\draw[kernels2] (0,-0.5) node[not] {} -- (-1,0.5) node[not] {};
\draw[kernels2] (-1,0.5) -- (-1.5,2);
\draw[kernels2] (0,-0.5)  -- (0.7,2);
\draw[kernels2] (-1,0.5) -- (-0.5,2);
\draw (0,-0.5) -- (1.5,2);
\draw (-1,2.2) node[g] {};
\draw (1,2.2) node[g] {};
}

\DeclareSymbol{pre_im_I1Xi4abcc1}{2}{
\draw[kernels2] (0,-1) node[not] {} -- (-1,0) node[not] {};
\draw[kernels2] (0,1.1) node[not] {} -- (-1,0);
\draw[kernels2] (-1,0) -- (-1.5,2.5);
\draw[kernels2] (0,1.1) node[not] {} -- (-0.5,2.5);
\draw[kernels2] (0,1.1) -- (0.5,2.5);
\draw[kernels2] (0,-1) -- (1.5,2.5);
\draw (-1,2.7) node[g] {};
\draw (1,2.7) node[g] {};
}

\DeclareSymbol{pre_im_Xi4eabc1}{2}{
\draw (0,-0.5) node[not] {} -- (-1,0.5) node[not] {};
\draw[kernels2] (-1,0.5) -- (-1.5,2);
\draw[kernels2] (0,-0.5)  -- (-0.5,2);
\draw[kernels2] (-1,0.5) -- (0.8,2);
\draw[kernels2] (0,-0.5) -- (1.5,2);
\draw (-1,2.2) node[g] {};
\draw (1,2.2) node[g] {};
}

\DeclareSymbol{pre_im_Xi4ba1b}{2}{
\draw[kernels2] (0,0) node[not] {}  -- (1.8,1.5);
\draw[kernels2] (0,0) -- (0.8,1.5);
\draw (0,-0.1) -- (-1.8,1.5);
\draw (0,-0.1) -- (-0.8,1.5);
\draw (-1,1.7) node[g] {};
\draw (1,1.7) node[g] {};
}

\DeclareSymbol{pre_im_Xi4ba2}{2}{
\draw (0,-0.1) node[not] {}  -- (1.8,1.5);
\draw[kernels2] (0,0) -- (0.8,1.5);
\draw (0,-0.1) -- (-1.8,1.5);
\draw[kernels2] (0,0) -- (-0.8,1.5);
\draw (-1,1.7) node[g] {};
\draw (1,1.7) node[g] {};
}

\DeclareSymbol{pre_im_Xi4cabc2}{2}{
\draw[kernels2] (0,-0.5) node[not] {} -- (-1,0.5) node[not] {};
\draw[kernels2] (-1,0.5) -- (-1.5,2);
\draw (-1,0.5)  -- (0.7,2);
\draw[kernels2] (-1,0.5) -- (-0.5,2);
\draw[kernels2] (0,-0.5) -- (1.5,2);
\draw (-1,2.2) node[g] {};
\draw (1,2.2) node[g] {};
}

\DeclareSymbol{pre_im_Xi4cabc1}{2}{
\draw[kernels2] (0,-0.5) node[not] {} -- (-1,0.5) node[not] {};
\draw (-1,0.5) -- (-1.5,2);
\draw[kernels2] (-1,0.5)  -- (0.7,2);
\draw[kernels2] (-1,0.5) -- (-0.5,2);
\draw[kernels2] (0,-0.5) -- (1.5,2);
\draw (-1,2.2) node[g] {};
\draw (1,2.2) node[g] {};
}

\DeclareSymbol{pre_im_Xi4eabbisc1}{2}{
\draw[kernels2] (0,-0.5) node[not] {} -- (-1,0.5) node[not] {};
\draw[kernels2] (-1,0.5) -- (-1.5,2);
\draw[kernels2] (0,-0.5)  -- (-0.5,2);
\draw[kernels2] (-1,0.5) -- (0.8,2);
\draw (0,-0.5) -- (1.5,2);
\draw (-1,2.2) node[g] {};
\draw (1,2.2) node[g] {};
}

\DeclareSymbol{pre_im_1}{0}{
\draw[kernels2] (0,-0.5) node[not] {} -- (-0.6,0.5) ;
\draw[kernels2] (0,-0.5) -- (0.6,0.5);
\draw (0,1.1)  -- (-0.55,2);
\draw (0,1.1)  -- (0.55,2);
\draw (0,0.7) node[g] {};
\draw (0,2.2) node[g] {};
}

\DeclareSymbol{disconnect}{0}{
\draw[kernels2] (0,-0.5) node[not] {} -- (-0.6,0.5) ;
\draw[kernels2] (0,-0.5) -- (0.6,0.5);
\draw (-0.55,1.1)  -- (-0.55,2.3);
\draw (0.55,2.3) -- (0.55,1.5) -- (1.2,1.5) -- (1.2,3.5) -- (0.55,3.5) -- (0.55,2.7);
\draw (0,0.7) node[g] {};
\draw (0,2.5) node[g] {};
}

\DeclareSymbol{pre_im_2}{2}{\draw[kernels2] (0,0) node[not] {} -- (-1,1) node[not] {};
\draw[kernels2] (0,0) -- (1,1) node[not] {};
\draw[kernels2] (-1,1) -- (-1.5,2.5);
\draw[kernels2] (-1,1) -- (-0.5,2.5);
\draw[kernels2] (1,1) -- (0.5,2.5);
\draw[kernels2] (1,1) -- (1.5,2.5);
\draw (-1,2.7) node[g] {};
\draw (1,2.7) node[g] {};
}

\DeclareSymbol{CX_rec}{0}{
\draw [black] (-0.3,1) to (-0.3,-0.3);
\draw [black] (0.3,1) to (0.3,-0.3);
\draw [black] (-0.3,1) to (-0.3,2.3);
\draw [black] (0.3,1) to (0.3,2.3);
\draw (0,1) node[rec] {};
}

\DeclareSymbol{CX_cerc}{0}{
\draw [black] (0,1) to (0,-0.3);
\draw (0,1) node[cerc] {};
}

\DeclareMathAlphabet{\mathpzc}{OT1}{pzc}{m}{it}

\def\eqref#1{(\ref{#1})}

\makeatletter % Stolen from the internet to make a fat \cdot which isn't as fat as a \bullet
\newcommand*{\bigcdot}{}% Check if undefined
\DeclareRobustCommand*{\bigcdot}{%
  \mathbin{\mathpalette\bigcdot@{}}%
}
\newcommand*{\bigcdot@scalefactor}{.5}
\newcommand*{\bigcdot@widthfactor}{1.15}
\newcommand*{\bigcdot@}[2]{%
  \sbox0{$#1\vcenter{}$}% math axis
  \sbox2{$#1\cdot\m@th$}%
  \hbox to \bigcdot@widthfactor\wd2{%
    \hfil
    \raise\ht0\hbox{%
      \scalebox{\bigcdot@scalefactor}{%
        \lower\ht0\hbox{$#1\bullet\m@th$}%
      }%
    }%
    \hfil
  }%
}
\makeatother

\tcbset
{colframe=boxcolor,colback=symbols!7!pagebackground,coltext=pageforeground,
fonttitle=\bfseries,nobeforeafter,center title,size=fbox,boxsep=1.5pt,
top=0mm,bottom=0mm,boxsep=0mm,tcbox raise base}

\def\two{{\<generic>\kern0.05em\<genericb>}}
\def\twoI{{\<Ito>\kern0.05em\<Itob>}}

\def\mail#1{\burlalt{#1}{mailto:#1}}

\newcommand{\cop}{\textnormal{cop}\,}
\newcommand{\msfD}{\mathsf{D}}
\newcommand{\msfZ}{\mathsf{Z}}
\newcommand{\rmU}{{\rm U}}

\newcommand{\lnh}{(\hspace{-5pt}(\hspace{1pt}}
\newcommand{\rnh}{\hspace{1pt})\hspace{-5pt})}
\newcommand{\lspan}{\textnormal{span}\,}
\newcommand{\tL}{\tilde{L}}
\newcommand{\populated}{\mathcal{T}} %set of populated multi-indices
\newcommand{\Dcop}{\Delta_\bullet} % coproduct dual to the concatenation product
\newcommand{\Dcom}{\Delta_\rho} % coaction dual to the representation \rho
\newcommand{\mcMp}{{\mcM^+}}
\newcommand{\mcMm}{{\mcM^-}}
\newcommand{\bfPi}{\bm{\Pi}}
\newcommand{\bfpi}{\bm{\pi}}
\newcommand{\bfsigma}{\bm{\sigma}}
\newcommand{\bflambda}{\bm{\lambda}}

\newcommand{\regsol}{\overline{\reg}}
\newcommand{\length}[1]{\mathcal{l}({#1})}

\newcommand{\alphamax}{\overline{\alpha}}

\newcommand{\lpol}{|}
\newcommand{\rpol}{|_p}
\newcommand{\bff}{\mathbf{f}}

\newcommand{\mcNmin}{{\hat{\mcN}}}
\newcommand{\coord}{{\mcR}}
\newcommand{\coordind}{{\mathcal{r}}}
\newcommand{\counterterms}{{T_\mcC^*}}
\newcommand{\trees}{B}
\newcommand{\treespace}{\mathcal{B}}
\newcommand{\mhyphen}{\mbox{-}}

\newcommand{\referee}[1]{{ #1}}%For comments to referees. Changes highlighted in blue.

\usepackage{thmtools}

\begin{document}

\title{A top-down approach to algebraic renormalization in regularity structures based on multi-indices}
\author{Yvain Bruned$^1$, Pablo Linares $^2$}
\institute{ IECL (UMR 7502), Université de Lorraine \and
 Universidad Autónoma de Madrid \\
Email:\ \begin{minipage}[t]{\linewidth}
\mail{yvain.bruned@univ-lorraine.fr},\\ \mail{pablo.linaresb@uam.es}.
\end{minipage}}

%%%% Spell check
% !TeX spellcheck = en_US 

\maketitle 

\begin{abstract}
	We provide an algebraic framework to describe renormalization in regularity structures based on multi-indices for a large class of semi-linear stochastic PDEs. This framework is ``top-down", in the sense that we postulate the form of the counterterm and use the renormalized equation to build a canonical smooth model for it. The core of the construction is a generalization of the Hopf algebra of derivations in \cite{LOT}, which is extended beyond the structure group to describe the model equation via an exponential map: This allows to implement a renormalization procedure which resembles the preparation map approach in our context.  
\end{abstract}

\setcounter{tocdepth}{2}
\setcounter{secnumdepth}{4}
\tableofcontents

\section{Introduction}
\label{Introduction}
%-----------------------%

It is now a decade since regularity structures \cite{reg,BHZ,CH16,BCCH} were introduced to solve a large class of singular stochastic partial differential equations (SPDEs). Loosely speaking, the theory of regularity structures is a theory of calculus for local jets built upon homogeneous nonlinear functionals of the driving noises of the corresponding SPDE. A systematic construction of such functionals is motivated by Picard iterations and leads to a tree-based algebraic description. The operations of recentering (i.~e. changing the base point of the local jet) and renormalization (i.~e. removing divergences from ill-posed products) are algebraically characterized using Hopf algebras of trees. Indeed, the recentering Hopf algebra in \cite{reg,BHZ} is a variant of the Butcher-Connes-Kreimer Hopf algebra \cite{Butcher72,CK1,CK2}; see \cite{BM22} for a construction via a deformation of the grafting pre-Lie product, and \cite{BK} via a post-Lie product. The renormalization Hopf algebra is related to the extraction/contraction Hopf algebra \cite{CEFM} and was first introduced in \cite{BHZ}; see once more \cite{BM22} for a construction via a deformation of the insertion pre-Lie product. These two Hopf algebras are in cointeraction in the sense of \cite[Theorem 8]{CEFM} (see \cite{CHV} for a review of these results in the context of numerical analysis); however, obtaining the cointeraction property requires extending the decorations of the trees, in such a way that the triangularity properties of recentering and renormalization do not clash. This algebraic construction allows to implement a renormalization procedure, inspired by the BPHZ renormalization of Feynman diagrams \cite{BP57,KH69,WZ69}, for which convergence of renormalized models can be proven \cite{CH16}. \referee{Later}, a recursive formulation of the algebraic renormalization problem was introduced in \cite{BR18,BB21} using preparation maps. This is a way of localizing the renormalization of a tree at its root, which avoids the difficulties of the cointeraction property sacrificing a robust group structure; it turns out that this construction is still useful for obtaining convergence results for renormalized models (cf. \cite{BN,BN2,BB23}), while also being well-suited for situations in which translation invariance is lost, cf. \cite{BB21b}.

\medskip

More recently, in the context of quasi-linear SPDEs, \cite{OSSW} introduced an alternative index set to describe local solutions. Instead of Picard iterations, their approach is based on a sort of infinite perturbative expansion, and naturally leads to multi-indices, which encode products of derivatives of the nonlinearity, as the basic index set. A Hopf-algebraic construction of the structure group within this setup was developed in \cite{LOT}, based on a pre-Lie algebra of derivations in a power series algebra; see \cite{BK,JZ23} for its corresponding post-Lie perspective. The construction of a renormalized model (and in turn of the renormalized equation) was done in \cite{LOTT}, under the assumption of a spectral gap inequality which allows for a recursive formulation of the renormalization problem (see \cite{HS23} for the extension of the spectral gap method to a large class of semi-linear SPDEs in the tree-based approach); later \cite{Tem23} established the convergence of this renormalized model. No explicit algebraic structure for renormalization (and in particular no renormalization group) is built in \cite{LOTT}. The reader can find an introduction to this approach in \cite{LO22,OST}.

\medskip

The goal of the current article is twofold. On the one hand, we will generalize the multi-index approach of \cite{OSSW}, and more specifically the algebraic constructions in \cite{LOT}, from the quasi-linear SPDE considered there to the class of subcritical semi-linear SPDEs of the form
\begin{equation}\label{set01}
	\mcL u = \sum_{\mfl\in\mfL^-\cup\{0\}} a^\mfl (\mathbf{u}) \xi_\mfl,
\end{equation}
where $\mcL$ is a linear operator with regularizing properties; $\mfL^-$ is a finite set which indexes the driving noises $\xi_\mfl$, with the additive term given in $\xi_0 = 1$; and $\{a^\mfl\}_{\mfl\in\mfL^-\cup \{0\}}$ are smooth nonlinearities depending on the solution $ u $ and its derivatives, denoted by $\mathbf{u}$ (see below for more precise assumptions). On the other hand, we give an algebraic characterization of the renormalization procedure performed in \cite{LOTT, Tem23}. Our construction does not take the form of the Hopf algebras in cointeraction from \cite{BHZ}. However, in \cite{Lin23}, the second author applies the multi-index approach to rough paths (where polynomials are not required) and builds the corresponding algebraic renormalization (\textit{translation}, cf. \cite{BCFP}) group over an insertion-like pre-Lie algebra, so it is conceivable that a Hopf-algebraic multi-index approach with extended decorations will lead to a renormalization group analogous to the one in \cite{BHZ}. We do not use preparation maps as such either; this is because the construction in \cite{BR18} relies on the tree grafting pre-Lie algebra being free, cf. \cite{CL}, as in particular it requires identifying the root of a tree. This piece of information is lost in the multi-index description\footnote{ See Subsection \ref{subsec::2.5} for the connection between trees and multi-indices.}. However, the philosophy of preparation maps is still observed in the inductive construction (already present in \cite{LOTT}); we explain the similarities in Subsection \ref{subsec::preparation} below.

\medskip

The reason to restrict \eqref{set01} to scalar equations, unlike \cite{BHZ,BCCH} which also incorporates systems of SPDEs, is that multi-indices are better suited for the scalar case. Since multi-indices encode products of derivatives of the nonlinearities, a description based on multi-indices becomes useful in one dimension by the algebra structure of $\R$, but is not so helpful for vector-valued nonlinearities where the pre-Lie algebra of vector fields is a more natural underlying structure. Of course, one could still treat vector-valued nonlinearities component-wise as one-dimensional and use an algebra structure of an enlarged set of variables (now incorporating each component of each nonlinearity individually, i.~e. forgetting their vectorial nature). This turns out to not be very efficient: Translating into trees, it would be the equivalent of incorporating kernel types not only on edges but also as noise decorations themselves, to keep track of the system component every noise comes from, and then only considering trees which are \textit{consistent}, in the sense that the kernel type attached to a noise should coincide with that of its incoming edge. Such a structure is not necessary in the tree-based description: The pre-Lie algebra of tree grafting passes from trees to vectorial nonlinearities in a canonical way (since the  pre-Lie algebra of tree grafting is free, \referee{cf. \cite{CL}}). However, it is conceivable that other non-free pre-Lie algebras can be used for a more efficient bookkeeping (compared to the tree-based) of regularity structures in higher dimension; this is not the subject of the current paper, but we believe some of our ideas could be extended to such a situation, as we never need to work with trees.

\medskip

Our algebraic construction is what in \cite{LOTT} was called \textit{top-down} in the following sense: Instead of changing the model and then studying the consequences of this operation at the level of the equation (\textit{bottom-up}), we consider the family of \textit{admissible modified equations} and use them to construct the corresponding algebraically renormalized model. This means in particular that for us there is no distinguished \textit{canonical model} (this would correspond to having no counterterm at all), but we rather work with all the algebraically renormalized models at once. This is partially motivated by the construction in \cite{LOTT}, where the model is constructed globally using the PDE, and the choice of renormalization constants solves an infrared problem which cannot be solved without counterterms (in that regard, the choice of renormalization in \cite{LOTT} is unique). Our main result can be roughly formulated as follows (see Theorem \ref{th:main} below for the more rigorous statement).
\begin{theorem}\label{ref:thintro}
	Assume that all the noises $\{\xi_\mfl\}_{\mfl\in\mfL^-}$ are smooth. For every admissible counterterm $c$, there exists a smooth model $(\Pi_x,\Gamma_{xy})$ based on multi-indices for the equation
	\begin{equation*}
		\mcL u = \sum_{\mfl\in\mfL^-\cup\{0\}} a^\mfl (\mathbf{u}) \xi_\mfl + c(\mathbf{u}).
	\end{equation*}
\end{theorem}	

\medskip

\referee{
\subsection{Outline}	
	Our starting point is an Ansatz for a local expansion that describes the solution of \eqref{set01}: We think of the nonlinearities $a_\mfl$ as parameterized by coefficients
	\begin{equation}\label{ref:coef01}
		z_{(\mfl,k)} = \frac{1}{k!} \partial_\mathbf{u}^k a_\mfl (0).
	\end{equation}
	The evaluation point is momentarily chosen at $0$ for simplicity. These coefficients, provided they converge, serve as coordinates in the infinite-dimensional space of analytic functions around $0$. We shall think of the solution $u$, or rather an approximation to it, as a polynomial of $z_{(\mfl,k)}$: This leads to expansions of the form
	\begin{equation*}
		\sum_\beta \Pi_\beta z^\beta,
	\end{equation*}  
	where the sum runs through multi-indices of $(\mfl,k)\in (\mfL^-\cup \{0\}) \times \N_0^d$  and the $ \Pi_{\beta}$ are space-time functions or distributions.
	
	\medskip
	
	This naive approach has several limitations. The first one is that the coefficients \eqref{ref:coef01} are insufficient to characterize $u$, as they are blind to the effect of initial conditions, boundary values or any other constraints that guarantee the well-posedness of \eqref{set01}. These may be captured with additional coefficients
	\begin{equation*}
		z_{\mathbf{n}} = \frac{1}{\mathbf{n}!} \partial^\mathbf{n} p (0)
	\end{equation*}
	which serve as coordinates for space-time polynomials. In terms of bookkeeping, this only enlarges our index set and leads us to consider multi-indices over $((\mfL^-\cup \{0\}) \times \N_0^d) \cup \N_0^d$. A second limitation is that the fixed origin does not allow for enough freedom to cover local expansions of functions or distributions. To solve this, we interpret the above coefficients as functionals of the nonlinearities $a^\mfl$, of space-time polynomials $\mathbf{p}$, and space-time points $x$, i.~e.
	\begin{equation*}
		\z_{(\mfl,k)} [\mathbf{a},\mathbf{p},x] := \tfrac{1}{k!}\partial^k a^\mfl(\mathbf{p}(x)),\;\;\; \z_\mathbf{n} [\mathbf{a},\mathbf{p},x] := \tfrac{1}{\mathbf{n}!}\partial^\mathbf{n} p(x),
	\end{equation*} 
	see \eqref{bb03} below. This is already an extension of the ideas of \cite{OSSW,LOT} for two reasons. On the one hand, we incorporate the space-time dependence into the functional, instead of fixing an origin. On the other hand, and more importantly, we do not fix an origin for the evaluation of $a^\mfl$: Instead, we define its evaluation in terms of the polynomial and the space-time point in the \textit{nested} form seen above. With this notation, we would formally express the solution $u$ locally around $x\in \R^d$ as an expansion of the form
	\begin{equation*}
		u = \sum_\beta \Pi_{x \beta} \z^\beta[\mathbf{a},\mathbf{p},x],
	\end{equation*} 
	where $\mathbf{p}$ is some polynomial (yet to be determined) and for every multi-index $\beta$ the term $\Pi_{x \beta}$ is a space-time distribution which is homogeneous around the base point $x$. These distributions constitute (part of) the \textit{model}, and the expansion suggests that they take the form
	\begin{equation*}
		\Pi_{x \beta} = \frac{1}{\beta!}\big(\partial_{\z[\mathbf{a},\mathbf{p},x]}^\beta u\big)(\z = 0).
	\end{equation*}
	In order to characterize these functions and distributions, we may appeal to equation \eqref{set01}, first noting that the nonlinearities may be expanded around $\mathbf{p}(x)$ using \eqref{bb03}:
	\begin{equation}\label{ref:coef02}
		a^\mfl (\mathbf{u}) = \sum_{k\in\N_0^d} \z_{(\mfl,k)}[\mathbf{a},\mathbf{p},x] \big(\mathbf{u} - \mathbf{p}(x)\big)^k.
	\end{equation}
	We then may take derivatives of $u$ with respect to the coefficients $\{\z_{(\mfl,k)}[\mathbf{a},\mathbf{p},x]\}_{(\mfl,k)}$ $\cup$ $\{\z_{\mathbf{n}}[\mathbf{a},\mathbf{p},x]\}_\mathbf{n}$, leading to a hierarchy of linear equations (cf. \eqref{bb02}) that can be inductively solved with analytic bounds around the base point $x$ (as we show in Section \ref{section::4}). 
	
	\medskip
	
	However, this still does not solve the problem of changing the evaluation points of our functional, as in the procedure outlined above we are implicitly fixing $\mathbf{p}$ and $x$ once and for all. To free ourselves from this restriction, we follow \cite{LOT} and seek a set of transformations which arise from exponentials of the infinitesimal generators of shifts of the form
	\begin{equation}\label{ref:coef03}
		\mathbf{p}(x) \mapsto \mathbf{p}(x) + \mathbf{q}[\mathbf{a},\mathbf{p},x],\,\,\,\,x\mapsto x + y,
	\end{equation}
	seen as algebraic operations on the formal power series algebra generated by $\z_{(\mfl,k)}$ and $\z_{\mathbf{n}}$. Here $\mathbf{q}[\mathbf{a},\mathbf{p},x]$ is an $(\mathbf{a},\mathbf{p},x)$-dependent polynomial; the reason why we need this kind of shifts will become transparent later in Section \ref{section::3}, see also \cite[Subsection 3.7]{LOT}. The infinitesimal generators of these shifts form a Lie algebra whose universal envelope is a Hopf algebra isomorphic to the symmetric algebra equipped with a Grossman-Larson-type product: This is a consequence of the theorem of Guin and Oudom \cite{Guin1,Guin2}, although our Lie algebra is slightly weaker than a pre-Lie algebra, see Subsection \ref{ref:subsecgo} for the details. Restricting the generators to those which satisfy the triangular constraint \eqref{sg17} we may construct the Hopf algebra of recentering and, ultimately, the \textit{structure group}: The latter takes the form of a group of exponential-type maps $\Gamma_{\bfpi}^{+*}$ associated to functionals $\bfpi$ of the graded dual of the aforementioned universal enveloping algebra. This was already done in \cite{LOT} for quasi-linear SPDEs; in Subsection \ref{subsec::3.3}, we extend the techniques to generic equations of the form \eqref{set01}.
	
	\medskip
	
	One of the novelties of this work is that we use the structure of shifts for another purpose, namely the reformulation of the hierarchy of model equations. The idea is motivated by the following observation: Assuming $u$ is smooth, we may express the evaluation of the nonlinearity as
	\begin{equation*}
		a^\mfl(\mathbf{u}(x)) = \z_{(\mfl,0)} [\mathbf{a},\mathbf{u},x].
	\end{equation*} 
	Thanks to the local identification $u = \Pi_x = \sum_\beta \Pi_{x\beta} \z^\beta$, we can take advantage of the shifts \eqref{ref:coef03} to re-express this evaluation. Instead of considering the triangular generators as for the structure group, we restrict to shifts that only affect the derivatives appearing on the r.~h.~s. of the equation. The number of derivatives is bounded by the order of the operator thanks to the semi-linearity condition, and leads to the index set \eqref{mod01}; this is crucial to guarantee the correct finiteness properties to make the procedure of \cite{LOT} work. As a consequence, we can associate to the model $\Pi_x$ a functional $\bfPi_x$ and an exponential-type map $\Gamma_{\bfPi_{x}}^{*}$ so that\footnote{ This is a simplified version: We would rather fix the origin at the function-like part of the solution.} 
	\begin{equation*}
		a^\mfl(\mathbf{u}) \equiv \Gamma_{\bfPi_{x}}^{*}\z_{(\mfl,0)}[\mathbf{a},0,0].
	\end{equation*}
	We use this to rewrite the hierarchy of model equations as a PDE of the form
	\begin{equation*}
		\mcL \Pi_x = \Gamma_{\bfPi_x}^{*}\sum_{\mfl\in \mfL^-\cup \{0\}} \xi_\mfl \z_{(\mfl,0)},
	\end{equation*}
	see \eqref{bb02bis} or Theorem \ref{th:main} below for more precise versions. Lemma \ref{lem:exp01} shows that both formulations of the model equations \eqref{bb02} and \eqref{bb02bis} are equivalent: The latter can be interpreted as the PDE version of the characterizing ODE (Cartan's development) of the Hopf-algebraic smooth rough paths \cite[proof of Theorem 4.2]{BFPP} (see \cite[(3.16)]{Lin23} for a closer connection).
	
	\medskip
	
	The reader should note that the equation above is meaningless in the singular case, because the map $\Gamma_{\bfPi_{x}}^{-*}$ contains products of distributions which are ill-posed. However, it is useful if we assume the noises are mollified, since it allows us to introduce finite counterterms in a very simple way, namely by shifting
	\begin{equation}\label{ref:coef04}
		\z_{(0,0)}\mapsto \z_{(0,0)} + c,
	\end{equation}
	where $c = \sum_\beta c_\beta \z^\beta$ represents the counterterm in terms of the renormalization constants $c_\beta$. This simple transformation is a consequence of a few reasonable assumptions that we extensively discuss in Subsection \ref{subsec::3.5} below. If we fix $c$, then the model equations turn into
	\begin{equation*}
		\mcL \Pi_x = \Gamma_{\bfPi_x}^{*}\Big(\sum_{\mfl\in \mfL^-\cup \{0\}} \xi_\mfl \z_{(\mfl,0)} + c\Big),
	\end{equation*}
	which under the qualitative smoothness assumption on $\xi_\mfl$ yields a model (Theorem \ref{th:main}). 
	
	\medskip
	
	Fixing $c$ beforehand is a limitation for two reasons. On the one hand, we cannot remove the qualitative smoothness assumption: Our result should rather be seen as the construction of a \textit{canonical smooth model for the modified equation}. The analogue in rough paths takes the form of the \textit{translated rough paths} in \cite{BCFP}, although we incorporate the necessary triangularity properties \eqref{cou01b} to preserve the homogeneity at small scales. Removing the smoothness assumption is not in the scope of this paper, but we refer to \cite{LOTT,Tem23} for the full construction of the renormalized model within our framework in the quasi-linear case. On the other hand, if we wanted to use our result for the proper construction of renormalized models, the counterterm $c$ should be chosen \textit{parallel to the construction of the model}, and not fixed from the beginning. We address this issue in Subsection \ref{subsec::4.4}, where we derive an order in the set of multi-indices that allows for an inductive construction of the model and the renormalization constants at the same time. 
	
	\medskip
	
	An important feature of our method is that we do not build a renormalization group, or even algebraic renormalization maps. Instead, we take the simple translation \eqref{ref:coef04} at the level of the model equations and build the \textit{concrete} model associated to the counterterm $c$. This is another reason to formulate Theorem \ref{ref:thintro} for \textit{any} admissible counterterm: We need to show the flexibility of our approach at the analytic level. The recursive procedure is reminiscent of \textit{preparation maps}, but avoids the use of an abstract (algebraic) integration map; see Subsection \ref{subsec::preparation} for a detailed discussion about the similarities and differences between the two approaches.
	
	\medskip
	
	We conclude the article implementing the renormalization procedure and finding the structure of the counterterms for three classical examples of singular SPDEs: the $\Phi_3^4$ model (Subsection \ref{subsec::5.1}), the multiplicative stochastic heat equation (Subsection \ref{subsec::5.2}) and the generalized KPZ equation (which serves as an example for concrete computations throughout the whole text, but is concluded in Subsection \ref{subsec::5.3}). We use these examples to show how under certain additional assumptions (mostly related to symmetries) we can restrict the counterterm \textit{a priori} and work with a smaller set of constants. We also show that in all these cases multi-indices generate fewer renormalization constants than trees; see also Subsection \ref{subsec::2.5} for the comparison between trees and multi-indices as index sets.
}

\subsection*{Acknowledgements}

{\small
	Y.B. thanks the Max Planck Institute for Mathematics in the Sciences (MiS) in Leipzig for having supported his research via a long stay in Leipzig from January to June 2022. Y. B.  gratefully acknowledges funding support from the European Research Council (ERC) through the ERC Starting Grant Low Regularity Dynamics via Decorated Trees (LoRDeT), grant agreement No.\ 101075208 and the ANR via the project LoRDeT (Dynamiques de faible régularité via les arbres décorés) from the projects call T-ERC\_STG. Views and opinions
	expressed are however those of the author(s) only and do not necessarily reflect those of
	the European Union or the European Research Council Executive Agency. Neither the
	European Union nor the granting authority can be held responsible for them. P.L. thanks Rishabh Gvalani, Francesco Pedullà, Rhys Steele and Markus Tempelmayr for helpful discussions, and Xue-Mei Li for financial support via the EPSRC grant EP/V026100/1. Finally, we thank the anonymous referees for
	their remarks, which led to an improvement in the presentation.
}   

\subsection*{Data Availibility} Data sharing not applicable to this article as no datasets were
generated or analysed during the current study.

\subsection*{Conflict of interest} The authors declare that they have no conflict of interest.

\section{Regularity structures based on multi-indices}
\label{section::2}

In this section we build the set of multi-indices necessary for the treatment of the scalar SPDE \eqref{set01}.
\subsection{Notation for multi-indices}
We begin by setting some notation that we will use consistently throughout the paper. 
\begin{definition}
	Let $\mathsf{I}$ be a countable set. 
	\begin{itemize}
		\item A multi-index over $\mathsf{I}$ is a map $m:\mathsf{I}\to \N_0$ such that $m(i)= 0$ for all but finitely many $i\in\mathsf{I}$. We denote the set of multi-indices over $\mathsf{I}$ as $M(\mathsf{I})$.
		\item The \textit{length} is the function 
		\begin{equation*}
			\begin{array}{cccl}
				\mathcal{l}: &M(\mathsf{I}) &\longrightarrow &\N_0\\
				\mbox{} & m & \longmapsto & \length{m} := \sum_{i\in\mathsf{I}} m(i).
			\end{array}
		\end{equation*}
		\item Length-one multi-indices will be denoted by $e_i$, $i\in \mathsf{I}$, where $e_i (j) = \delta_i^j$.
		\item The multi-index factorial is defined as
		\begin{equation*}
			m! := \prod_{i\in\mathsf{I}} m(i)!.
		\end{equation*}
		\item Let $V$ be a commutative algebra, and let $v: \mathsf{I} \to V$. For a multi-index $m\in M(\mathsf{I})$, the $m$-th power is given by
		\begin{equation}\label{powers}
			v^m := \prod_{i\in\mathsf{I}} v(i)^{m(i)}.
		\end{equation}
		We denote by $\R[\mathsf{I}]$ the free commutative algebra over variables indexed by $\mathsf{I}$; as a vector space, it is generated by \eqref{powers}. Similarly, we denote by $\R[[\mathsf{I}]]$ the corresponding (non-truncated) algebra of formal power series.
		\item Let $f : \R^{\mathsf{I}} \to \R$. For a multi-index $m\in M(\mathsf{I})$, the $m$-th derivative is given by
		\begin{equation*}
			\partial^{\referee{m}} f := \big(\prod_{i\in\mathsf{I}} \partial_i^{m(i)}\big) f.
		\end{equation*}
	\end{itemize}
\end{definition}	

\subsection{\referee{Basics on regularity structures}}
The theory of regularity structures is, at its core, a theory of calculus for general local expansions of functions and distributions. As for the expansions themselves, in words of \cite{Hairer16}, they involve an algebraic \textit{skeleton} (i.~e. the underlying algebraic structure of the expansions) and analytic \textit{flesh} (i.~e. the form of the generalized monomials involved in the expansions). The algebraic skeleton is what we call \textit{regularity structure}, cf. \cite[Definition 2.1]{reg}.
\begin{definition}\label{ref:defrg}
	A \textit{regularity structure} $ 
 (A,T,G) $ consists of the following elements:
	\begin{itemize}
	\item  a set of homogeneities, which is a set $ A \subset \R$ bounded from below and locally finite;
	\item a model space, which is a graded vector space $T = \bigoplus_{\nu \in A} T_\nu$;  
	\item a structure group $ G $, which is a group of linear endomorphisms of $ T $  such that, for every $ \Gamma \in G $,
	every  $ \nu \in A$, and every $ a \in T_{\nu} $ one has $ \Gamma a - a \in \bigoplus_{\nu' < \nu} T_{\nu'} $.
	\end{itemize}
\end{definition}	
The analytic flesh is provided by the \textit{model}, cf. \cite[Definition 2.17]{reg}.
\begin{definition}
	Given a regularity structure, a model consists of a collection of linear maps $\Pi_x : T\to \mcS'(\R^d)$ and of elements of the structure group $\Gamma_{xy}\in G$ such that they satisfy the algebraic properties
	\begin{equation*}
		\Pi_y = \Pi_x \Gamma_{xy},\,\,\,\,\Gamma_{xy} = \Gamma_{xz}\Gamma_{zy}
	\end{equation*}
	as well as the estimates for all compact subsets\footnote{ We refrain from giving a complete and rigorous definition at this stage; the reader can find the details in \cite[Section 2]{reg}.} $ K $ of $ \R^d $
	\begin{equation*}
		\referee{|}\langle \Pi_x(a),\varphi_x^\lambda \rangle \referee{|}\lesssim \lambda^{\nu},\,\,\,\, |\referee{\proj_{\nu'}}(\Gamma_{xy} a)| \lesssim |y-x|^{\nu-\nu'},
	\end{equation*}
	uniformly over all $a\in T_\nu$, $x,y\in K$, $\lambda\in(0,1)$ and all localized test functions $\varphi_x^\lambda$\referee{, where $\proj_{\nu'}$ denotes the projection onto $T_{\nu'}$.}	
\end{definition}	
\begin{remark}
	As we will always work with the smooth case in this paper, and to make the construction easier later, we reformulate the definition using \cite[Definition 6.7]{BHZ}, i.~e. the \textit{smooth model}, and the dual perspective in \referee{line} with \cite{OSSW,LOT,LOTT}: A model is a collection of maps $\Pi_x : \R^d \to T^*$ and elements of the structure group $\Gamma_{xy}\in G$ such that
	\begin{equation*}
		\Pi_y = \Gamma_{yx}^*\Pi_x,\,\,\,\,\Gamma_{xy}^* = \Gamma_{xz}^*\Gamma_{zy}^*
	\end{equation*}
	and
	\begin{equation*}
		|\Pi_{x} (y)(a)|\lesssim |y-x|^{\nu},\,\,\,\, |\Gamma_{xy} a|_{\nu'} \lesssim |y-x|^{\nu-\nu'}
	\end{equation*}
	for all $a\in T_\nu$ and $|x-y|\leq 1$.
\end{remark}	

\medskip

\referee{Given a regularity structure and a model, we are able to characterize the functions or distributions which locally describe the solution: These are called \textit{modelled distributions}. A modelled distribution $f$ is a $T$-valued function which satisfies certain analytic constraints, see \cite[Definition 3.1]{reg}. When tested against $\Pi_x$, $f$ gives rise to a local expansion around the base point $x$:
\begin{equation*}
	\Pi_x(f) = \sum_{a\in T} f^{(a)}(x) \Pi_x(a),
\end{equation*}	
where $f^{(a)} = \langle a, f \rangle$ are vanishing functions except for finitely many basis elements $a\in T$. The statement that such an expression really is the local expansion of some function or distribution is nontrivial and comes in form of Hairer's Reconstruction Theorem, cf. \cite[Theorem 3.10]{reg}. We refrain from giving more details about these analytic results, which we will not need for our purposes.

\medskip

In the upcoming pages, we will make an Ansatz for expansions of the above form describing the solution of \eqref{set01} and use it to motivate the index set (and thus the model space) and the form of the model components.
}

\medskip

\subsection{Warm up}\label{ref:subsecwarm}
\referee{We follow some of the main ideas of \cite{LOT,LOTT,OSSW}}. We start by providing the space of nonlinearities with a set of coordinates. In \eqref{set01}, for every $\mfl\in \mfL^-\cup\{0\}$, the nonlinearity $a^\mfl$ is assumed to be a smooth function $a^\mfl : \R^{\mathcal{O}} \to \R$ where $\mathcal{O}$ $\subset$ $\N_0^d$ is a finite subset; we identify $\mathbf{n}\in\N_0^d$ with the component of the function depending on the $\mathbf{n}$-th derivative of the solution, and write
\begin{equation*}
	\mathbf{u} = \big(\tfrac{1}{\mathbf{n}!} \partial^\mathbf{n} u\big)_{\mathbf{n}\in\N_0^d}.
\end{equation*}
For every $\mfl\in\mfL^-\cup\{0\}$ and every multi-index $k\in M(\N_0^d)$, we define
\begin{equation}\label{warm20}
	z_{(\mfl,k)} := \tfrac{1}{k!}\partial^k a^\mfl (0).
\end{equation}
The fact that we choose $0$ as the evaluation point is arbitrary\footnote{As we shall see later, one may change the origin by shift.}, but already generates a parameterization of analytic nonlinearities. With this notation, we formally have
\begin{equation}\label{warm03}
	a^\mfl(\mathbf{u}) = \sum_{k\in M(\N_0^d)} z_{(\mfl,k)} \mathbf{u}^k;
\end{equation} 
in particular, this yields the also formal
\begin{equation*} \label{equation_monomials}
	\mcL u = \sum_{(\mfl,k)}  z_{(\mfl,k)} \mathbf{u}^k \xi_\mfl.
\end{equation*}

\referee{
\begin{exam} \label{example_1}
	Consider the generalized KPZ equation (cf. e.~g. \cite[(KPZ)]{reg}), which is posed in space-time dimension $1+1$:
	\begin{equation}\label{kpz01_example}
		(\partial_t - \partial_x^2) u = f(u) + g(u)\partial_x u + h(u)(\partial_x u)^2 + \sigma(u) \xi,
	\end{equation}
with $u: \R^{1+1} \to \R \ni u(t,x)$ and $\xi$ being space-time white noise. We can formally rewrite this equation in terms of \eqref{equation_monomials}. We represent the multiplicative nonlinearity with the same symbol as that of the noise, i.~e., we take $ \mathfrak{L}^- = \left\lbrace \xi \right\rbrace  $ and write
\begin{align*}
		 a^{0}(\mathbf{u}) &= a^0 (u,\partial_x u) = f(u) + g(u)\partial_x u + h(u)(\partial_x u)^2,\\
		 a^{\xi}(\mathbf{u}) &= a^{\xi}(u) = \sigma(u).
		 \end{align*}
Then, one formally has
\begin{equation*}
	(\partial_t - \partial_x^2) u= \sum_{k}  z_{(0,k)} \mathbf{u}^k + \sum_{k}  z_{(\xi,k)} \mathbf{u}^k \xi.
\end{equation*}
Since $a^0$ only depends on the solution (to any power) and its first derivative (at most quadratically), in the first sum $k\in M(\N_0^{1+1})$ actually runs through $k$ of the form
\begin{equation*}
	\{k_{\mathbf{0}} e_\mathbf{0} + k_{(0,1)} e_{(0,1)}\, | \, k_{\mathbf{0}} \in \N_0,\, k_{(0,1)}=0,1,2  \};
\end{equation*}
in $\mathbf{n} = (n_1,n_2)\in \N_0^{2}$, the first component refers to time while the second refers to space, and $\mathbf{0} = (0,0)$. Similarly, $a^\xi$ depends only on $u$, and therefore $k$ in the second sum runs through
\begin{equation*}
	\{k_{\mathbf{0}} e_\mathbf{0}\, |\, k_{\mathbf{0}} \in \N_0 \}.
\end{equation*}
\end{exam}
}

However, as observed in \cite{OSSW,LOT}, the nonlinearity does not by itself determine the solution, but rather defines a family of solutions indexed by e.~g. initial values or boundary conditions. We capture this effect with a local parameterization of the manifold of solutions in terms of polynomials $p$, to which we can also give coordinates in terms of their derivatives, namely for every $\mathbf{n}\in\N_0^d$
\begin{equation}\label{warm11}
	z_{\mathbf{n}}:=\tfrac{1}{\mathbf{n}!}\partial^\mathbf{n} p(0).
\end{equation}
We consider the set
\begin{equation*}
	\coord := ((\referee{\mfL^-}\cup\{0\})\times M(\N_0^d)) \sqcup \N_0^d,
\end{equation*}
which contains the indices of the coordinates \eqref{warm20}, \eqref{warm11}. Given these coordinates, a (formal) Taylor-like expansion of $u$ would look like
\begin{equation}\label{warm02}
	u = \sum_{\beta \in M(\coord)} \tfrac{1}{\beta!} \big(\big(\partial^\beta\big)|_{z = 0} u\big)\, z^\beta.
\end{equation}
This equation already points in the direction of an expansion for the solution, assuming we can characterize the terms $\frac{1}{\beta!}\big(\partial^\beta\big)|_{z=0} u$ as the components of a model. Here is where \eqref{warm03} is convenient: By the Leibniz rule\footnote{ Here the sums of multi-indices are understood component-wise. Note that all the sums are finite, since $k\in M(\N_0^d)$ and thus $k(\mathbf{n})=0$ for all but finitely many $\mathbf{n}$'s.},
\begin{align*}
	&\tfrac{1}{\beta!}\referee{(\partial^\beta)|_{z=0}} a^\mfl(\mathbf{u}) \\
	&\quad= \sum_{k\in M(\N_0^d)} \sum_{\tilde{\beta} + \sum_\mathbf{n} \sum_{j=1}^{k(\mathbf{n})} \beta_\mathbf{n}^j = \beta} \tfrac{1}{\tilde{\beta}!}\big(\partial^{\tilde{\beta}}\big)|_{z=0} z_{(\mfl,k)} \prod_{\mathbf{n}\in\N_0^d} \prod_{j=1}^{k(\mathbf{n})} \tfrac{1}{\beta_\mathbf{n}^j !}\big(\partial^{\beta_\mathbf{n}^j}\big)|_{z=0}\tfrac{1}{\mathbf{n}!}\partial^\mathbf{n} u \\
	&\quad= \sum_{k\in M(\N_0^d)} \sum_{e_{(\mfl,k)} + \sum_\mathbf{n} \sum_{j=1}^{k(\mathbf{n})} \beta_\mathbf{n}^j = \beta}  \prod_{\mathbf{n}\in\N_0^d} \prod_{j=1}^{k(\mathbf{n})} \tfrac{1}{\mathbf{n}!}\partial^\mathbf{n}\big(\tfrac{1}{\beta_\mathbf{n}^j !}\big(\partial^{\beta_\mathbf{n}^j}\big)|_{z=0} u\big),
\end{align*}
so that feeding this into \eqref{set01} we obtain for every $\beta\in M(\coord)$
\begin{equation}\label{warm08}
	\mcL \big(\tfrac{1}{\beta!}\big(\partial^\beta\big)|_{z=0} u\big) = \sum_{(\mfl,k)} \sum_{e_{(\mfl,k)} + \sum_\mathbf{n} \sum_{j=1}^{k(\mathbf{n})} \beta_\mathbf{n}^j = \beta}  \prod_\mathbf{n} \prod_{j=1}^{k(\mathbf{n})} \tfrac{1}{\mathbf{n}!}\partial^\mathbf{n}\big(\tfrac{1}{\beta_\mathbf{n}^j !}\big(\partial^{\beta_\mathbf{n}^j}\big)|_{z=0} u\big) \xi_\mfl. 
\end{equation}
This seemingly complicated expression is, for a fixed $\beta$, a linear SPDE: Indeed, note that thanks to the presence of $e_{(\mfl,k)}$, all $\beta_\mathbf{n}^j$ on the r.~h.s. are of strictly smaller length than $\beta$, and thus if the linear equation is well-posed we may define $\tfrac{1}{\beta!}\big(\partial^\beta\big)|_{z=0} u$ inductively as the solution with r.~h.~s. given only in terms of lower levels. Obviously, we cannot expect uniqueness to hold unless we impose some offline conditions\footnote{ These are given in \cite{LOTT,LO22} as the combination of a local vanishing and a polynomial growth conditions, both determined by the homogeneity of the model, but extended globally. In this paper, we instead use a mild formulation closer to \cite{reg}, cf. Subsection \ref{subsec::4.1} below.}.

\medskip

When recentering, it is required to subtract a Taylor polynomial (assuming smoothness) $q_\beta$ to each\footnote{ The degree of the polynomial depends on the homogeneity of the corresponding term.} $\tfrac{1}{\beta!}\big(\partial^\beta\big)|_{z=0} u$, which then suggests to replace the r.~h.~s. of \eqref{warm08} by
\begin{equation*}
	\sum_{(\mfl,k)} \sum_{e_{(\mfl,k)} + \sum_\mathbf{n} \sum_{j=1}^{k(\mathbf{n})} \beta_\mathbf{n}^j = \beta}  \prod_\mathbf{n} \prod_{j=1}^{k(\mathbf{n})} \tfrac{1}{\mathbf{n}!}\partial^\mathbf{n}\big(\tfrac{1}{\beta_\mathbf{n}^j !}\big(\partial^{\beta_\mathbf{n}^j}\big)|_{z=0} u - q_{\beta_\mathbf{n}^j}\big) \xi_\mfl.
\end{equation*}
Let us incorporate all the polynomials $q_{\beta}$ in a formal power series in $z$ as in \eqref{warm02}; then it is easy to see that $u-q$ solves the same equation as $u$, i.~e. \eqref{set01}, up to a polynomial and with nonlinearities given by
\begin{equation}\label{warm14}
	\tilde{a}^\mfl := a^\mfl (\cdot + \mathbf{q}),\,\,\mathbf{q} = \big(\tfrac{1}{\mathbf{n}!} \partial^{\mathbf{n}} q\big)_{\mathbf{n}\in\mathbf{N}_0^d}.
\end{equation}
At the same time, this action is seen in the space of solutions considering
\begin{equation}\label{warm13}
	\tilde{p} := p + q.
\end{equation}
Thus, the algebraic structure obtained by the naive approach \eqref{warm08} is actually preserved when we want to define the \textit{centered model}, only up to a (infinite-dimensional) shift in the space of solutions. This would ultimately allow us to express $u$ locally around a base point $x$ as
\begin{equation}\label{warm15}
	u = \sum_{\beta\in M(\coord)} \tfrac{1}{\beta!}\big(\partial^\beta\big)_{z=0} u_x z^\beta,
\end{equation}
where now $z^\beta$ is defined in terms of \eqref{warm14}, \eqref{warm13}. Our goal later will be to identify $\tfrac{1}{\beta!}\big(\partial^\beta\big)|_{z=0} u_x$ with a model component, where the model space is indexed by multi-indices $\beta\in M(\coord)$ (or rather a subset).

\medskip

Note now that our family of SPDEs includes, as a subcase, the homogeneous PDE $\mcL u = 0$. In such a situation, it is natural to assume that \eqref{warm15} takes the form of a Taylor polynomial. We incorporate this assumption by giving the multi-indices independent of $a^\mfl$, i.~e. those depending only on the variables $z_\mathbf{n}$, a special role: For all $\beta\in M(\coord)$ such that $\beta(\mfl,k) = 0$ for all $(\mfl,k)$ $\in$ $(\mfL^- \cup\{0\})\times M(\N_0^d)$,
\begin{equation}\label{warm16}
	\tfrac{1}{\beta!}\big(\partial^\beta\big)|_{z=0} u_x = \left\{\begin{array}{ll}
		(\cdot - x)^\mathbf{n} & \mbox{if }\beta = e_\mathbf{n},\\
		0 & \mbox{otherwise.}
	\end{array} \right.\;
\end{equation}
Separating these multi-indices from the rest of the expansion, \eqref{warm15} takes the form
\begin{equation*}
	u = \sum_{\substack{\beta \in M(\coord)\\\beta \neq e_\mathbf{n}}} \tfrac{1}{\beta!}\big(\partial^\beta\big)|_{z=0} u_x z^\beta + \sum_{\mathbf{n}\in \N_0^d} (\cdot - x)^\mathbf{n} z_\mathbf{n},
\end{equation*}
so that, in the homogeneous case $a^\mfl \equiv 0$, $z_\mathbf{n}$ indeed correspond to the Taylor coefficients at the point $x$; it is therefore natural to identify
\begin{equation}\label{warm17}
	z_\mathbf{n} = \tfrac{1}{\mathbf{n}!}\partial^\mathbf{n} p(x).
\end{equation} 
Roughly speaking, we are parameterizing our space of solutions locally in terms of $\mcL$-harmonic/caloric polynomials. The passage from \eqref{warm11} to \eqref{warm17} involves a shift in space-time,
\begin{equation*}
	\tilde{p} := p(\cdot + x),
\end{equation*}	
which is as well propagated to $a^\mfl$, since the latter depends on the solution itself.

\subsection{The building blocks}\label{subsec::2.3}
Let us now present the main objects which will play a role in our construction. For this we will take a more abstract perspective: Since the actions of recentering and renormalization require that we consider multiple transformations in the space of (equations, solutions, space-time), we replace the coordinates \eqref{warm20}, \eqref{warm11} by nested functionals in $(\mathbf{a},\mathbf{p},x)$, namely
\begin{equation}\label{bb03}
	\z_{(\mfl,k)} [\mathbf{a},\mathbf{p},x] := \tfrac{1}{k!}\partial^k a^\mfl(\mathbf{p}(x)),\;\;\; \z_\mathbf{n} [\mathbf{a},\mathbf{p},x] := \tfrac{1}{\mathbf{n}!}\partial^\mathbf{n} p(x).
\end{equation}
\referee{
\begin{exam}
For the generalized KPZ equation \eqref{kpz01_example} one has for all $k_\mathbf{0}\in \N_0$ the following functionals: 
\begin{align}
	\z_{(\xi,k_\mathbf{0} e_\mathbf{0})} [\mathbf{a},\mathbf{p},x] &= \frac{1}{k_\mathbf{0}!} \sigma^{(k_\mathbf{0})} \big(p(x)\big),\label{kpz02_example}\\
	\z_{(0,k_\mathbf{0} e_\mathbf{0})} [\mathbf{a},\mathbf{p},x] &= \frac{1}{k_\mathbf{0}!} \Big( f^{(k_\mathbf{0})} \big(p(x)\big) +  g^{(k_\mathbf{0})}\big(p(x)\big) \partial_x p(x)\nonumber\\
	& \quad \quad \quad \quad + h^{(k_\mathbf{0})}\big(p(x)\big) \big(\partial_x p (x)\big)^2\Big)\label{kpz03_example}\\
	\z_{(0,k_\mathbf{0} e_\mathbf{0} + e_{(0,1)})} [\mathbf{a},\mathbf{p},x] &= \frac{1}{k_\mathbf{0}!} \Big( g^{(k_\mathbf{0})} \big(p(x)\big) + 2 h^{(k_\mathbf{0})}\big(p(x)\big) \partial_x p(x)\Big)\label{kpz04_example}\\
	\z_{(0,k_\mathbf{0} e_\mathbf{0} + 2 e_{(0,1)})} [\mathbf{a},\mathbf{p},x] &= \frac{1}{k_\mathbf{0}!} h^{(k_\mathbf{0})} \big(p(x)\big).\label{kpz05_example}
\end{align}
\end{exam}
}
\referee{Equipped with \eqref{bb03}}, the formal expression \eqref{warm15} is generalized to a formal power series in $\R[[\coord]]$, which is later evaluated at a specific triple $(\mathbf{a},\mathbf{p},x)$. Note that we are still discussing at a formal level; in particular, there is no reason why the formal power series, when evaluated at $(\mathbf{a},\mathbf{p},x)$, should converge, thus defining a proper functional. With this new interpretation we are considering all possible \textit{origins} in the $(\mathbf{a},\mathbf{p},x)$-space, which is a way of saying that we are parameterizing all equations and solutions at all space-time points; the choice of a specific origin is adapted to each situation, and we will have freedom to pass from one to another translating the shifts into algebraic operations on $\{\z_\coordind\}_{\coordind\in \coord}$. In order to do so, we first study their infinitesimal generators. 

\medskip

Let us start with \eqref{warm14}, \eqref{warm13}. We consider shifts given by monomials of space-time centered at $x$: Given $\mathbf{n}'\in\N_0^d$,
\begin{equation*}
	\mathbf{p} \mapsto  \mathbf{p} + t(\cdot - x)^{\mathbf{n}'} = \big(\tfrac{1}{\mathbf{n}!} \partial^\mathbf{n} p + \tbinom{\mathbf{n'}}{\mathbf{n}}  t(\cdot - x)^{\mathbf{n}'- \mathbf{n}} \big)_{\mathbf{n}\in\N_0^d}.
\end{equation*}
It holds
\begin{alignat*}{3}
		&\left.\frac{d}{dt}\right|_{t=0} \z_{(\mfl,k)}[\mathbf{a},\mathbf{p} + t(\cdot - x)^{\mathbf{n}'},x] &&= \left.\frac{d}{dt}\right|_{t=0} \tfrac{1}{k!}\partial^k a^\mfl ((\mathbf{p} + t(\cdot - x)^{\mathbf{n}'})) \\
		& &&= \tfrac{1}{k!} \partial^{k + e_{\mathbf{n}'}} a^\mfl (\mathbf{p}(x)) \\
		& &&= (k(\mathbf{n}')+1)\z_{(\mfl,k+e_{\mathbf{n}'})}[\mathbf{a},\mathbf{p},x],\\
	&\left.\frac{d}{dt}\right|_{t=0} \z_{\mathbf{n}}[\mathbf{a} ,\mathbf{p} + t(\cdot - x)^{\mathbf{n}'},x] &&=\delta_\mathbf{n}^\mathbf{n'},
\end{alignat*}
and postulating that it is a derivation the infinitesimal generator is given by
\begin{equation}\label{bb16}
	D^{(\mathbf{n})}:= \sum_{(\mfl,k)\in (\mfL^-\cup\{0\})\times M(\N_0^d)}\hspace*{-30pt} (k(\mathbf{n})+1) \z_{(\mfl,k+e_\mathbf{n})} \partial_{\z_{(\mfl,k)}} + \partial_{\z_\mathbf{n}}.
\end{equation}
\referee{
\begin{exam}
	In the case of the generalized KPZ equation, we have for instance
	\begin{equation*}
		D^{(\mathbf{0})} \z_{(\xi, e_{\mathbf{0}})} =  2 \z_{(\xi, 2e_{\mathbf{0}} )}, \quad 
		D^{(0,1)} \z_{(0,3 e_{\mathbf{0}})} =   \z_{(0, 3 e_{\mathbf{0}} + e_{(0,1)})}. 
	\end{equation*}
\end{exam}
}
\begin{lemma}\label{lem:Dn}
	For every $\mathbf{n}\in\N_0^d$, the map $D^{(\mathbf{n})}\in \textnormal{End}(\R[[\coord]])$ is well-defined.
\end{lemma}
\begin{proof}
	We look at the matrix representation\footnote{\referee{By \textit{matrix representation} of a map $A\in \textnormal{End}(\R[\coord])$ we mean the coefficients $(A)_\beta^\gamma$, where
\begin{equation*}
	A \z^\gamma = \sum_\beta (A)_\beta^\gamma \z^\beta.
\end{equation*}	
}} of \eqref{bb16}: For fixed $\beta,\gamma$ $\in$ $M(\coord)$,
	\begin{equation}\label{bb16bis}
		(D^{(\mathbf{n})})_\beta^\gamma = \sum_{(\mfl,k)} (k(\mathbf{n})+1)\gamma(\mfl,k)\delta_\beta^{\gamma - e_{(\mfl,k)} + e_{(\mfl,k+e_\mathbf{n})}} + \gamma(\mathbf{n})\delta_\beta^{\gamma - e_\mathbf{n}}.
	\end{equation}
	The first summand imposes the condition $\beta + e_{(\mfl,k)} = \gamma + e_{(\mfl,k+e_{\mathbf{n}})}$, and thus for a fixed $\beta$ there are only finitely many possible $\gamma$ and $(\mfl,k+e_{\mathbf{n}})$; the latter means that the sum over $(\mfl,k)$ is effectively finite, and thus meaningful. In the second summand, we have $\beta + e_\mathbf{n} = \gamma$, which only allows for one choice of $\gamma$. Altogether, this implies the finiteness property, for every $\mathbf{n}$ $\in$ $\N_0^d$,
	\begin{equation}\label{ref:fin01}
		\mbox{for all }\beta\in M(\coord),\,\# \{\gamma\in M(\coord)\,|\,(D^{(\mathbf{n})})_\beta^\gamma\neq 0\} <\infty.
	\end{equation}
	\referee{Now for a power series $\pi\in \R[[\coord]]$, the coefficients of which are denoted by $\{\pi_\beta\}_{\beta\in M(\coord)}\subset \R$, we may write
	\begin{equation*}
		(D^{(\mathbf{n})}\pi)_\beta = \sum_\gamma (D^{(\mathbf{n})})_\beta^\gamma \pi_\gamma,
	\end{equation*}
	which is a finite sum for fixed $\beta$ due to \eqref{ref:fin01}.} This in particular implies that \eqref{bb16} is well-defined as a map in $\R[[\coord]]$.
\end{proof}	

\medskip

We now address the space-time shift: We fix a direction $e_i$, $i=1,...,d$, and consider the transformation
\begin{equation*}
	x \mapsto x + t e_i.
\end{equation*}
We have
\begin{align*}
	&\left.\frac{d}{dt}\right|_{t=0} \hspace*{-5pt}\z_{(\mfl,k)}[\mathbf{a},\mathbf{p},x+te_i] = \sum_\mathbf{n} (k(\mathbf{n})+1)\z_{(\mfl,k+e_\mathbf{n})}[\mathbf{a},\mathbf{p},x] (\mathbf{n}(i) + 1)\z_{\mathbf{n}+e_i}[\mathbf{a},\mathbf{p},x],\\
	&\left.\frac{d}{dt}\right|_{t=0} \hspace*{-5pt}\z_{\mathbf{n}}[\mathbf{a},\mathbf{p},x+te_i] =(\mathbf{n}(i) + 1)\z_{\mathbf{n}+e_i}[\mathbf{a},\mathbf{p},x],
\end{align*}
which leads to considering
\begin{equation}\label{bb17}
	\mbpartial_i := \sum_{\mathbf{n}\in \N_0^d} (\mathbf{n}(i)+1)\z_{\mathbf{n}+e_i} D^{(\mathbf{n})}.
\end{equation}
\referee{
\begin{exam}
	Once more for generalized KPZ, we have
	\begin{equation*}
		\mbpartial_2 \z_{(0,e_\mathbf{0})} = 2 \z_{(0,2e_{\mathbf{0}})} \z_{(0,1)} + 2 \z_{(0,e_{\mathbf{0}} + e_{(0,1)})} \z_{(0,2)}+...
	\end{equation*}
	The r.~h.~s. is formal because $\mbpartial_2 \z_{(0,e_\mathbf{0})}$ is a \textit{bona fide} power series. In particular, this shows that, in general, $\mbpartial_i$ maps $\R[\coord]$ to $\R[[\coord]]$.
\end{exam}	
}
\begin{lemma}
	For every $i=1,...,d$, the map $\mbpartial_i\in\textnormal{End}(\R[[\coord]])$ is well-defined.
\end{lemma}	
\begin{proof}
	Again, we look at the matrix representation of \eqref{bb17}. To this end, we first fix $\mathbf{n}'\in \N_0^{d}$ and note that for every $\beta,\gamma,\gamma'$ $\in$ $M(\coord)$
	\referee{\begin{equation}\label{ref:new1}
		(\z^{\gamma'}D^{(\mathbf{n}')})_\beta^\gamma = \left\{\begin{array}{cl}
			(D^{(\mathbf{n}')})_{\beta - \gamma'}^{\gamma} & \mbox{if }\gamma'\leq \beta,\\
			0 & \mbox{otherwise,}
		\end{array}\right.
	\end{equation}
	where $\leq$ denotes the component-wise partial ordering
	\begin{equation}\label{comp02}
		\gamma\leq \beta \iff \gamma(\coordind) \leq \beta(\coordind)\,\mbox{for all }\coordind\in\coord.
	\end{equation}
	In particular,}
	\begin{equation*}
		(\z^{\gamma'}D^{(\mathbf{n}')})_\beta^\gamma \neq 0 \implies (D^{(\mathbf{n}')})_{\beta-\gamma'}^\gamma \neq 0,
	\end{equation*}
	and thus \eqref{bb16bis} implies
	\begin{align}\label{bb16bis2}
		(\z^{\gamma'}D^{(\mathbf{n}')})_\beta^\gamma = \sum_{(\mfl,k)} \gamma(\mfl,k)(k(\mathbf{n}')+1)\delta_\beta^{\gamma-e_{(\mfl,k)}+ e_{(\mfl,k+ e_{\mathbf{n}'})}+\gamma'}\hspace*{-4pt} + \gamma(\mathbf{n}')\delta_\beta^{\gamma-e_{\mathbf{n}'}+\gamma'},
	\end{align} 
	which in turn means
	\begin{align}
		&(\z^{\gamma'}D^{(\mathbf{n}')})_\beta^\gamma \neq 0 \nonumber\\
		&\quad \implies \left\{\begin{array}{l}
			\beta + e_{(\mfl,k)} = \gamma + e_{(\mfl,k+e_{\mathbf{n}'})} + \gamma'\,\,\mbox{for some }(\mfl,k),\,\mbox{or}\\
			\beta + e_{\mathbf{n}'}  = \gamma + \gamma'.
		\end{array}\right.\label{pol05}
	\end{align}
	We fix $\gamma' = e_{\mathbf{n}'+e_i}$. The first item in \eqref{pol05}, for a fixed $\beta$, yields only finitely many $\mathbf{n}'$, $(\mfl,k)$ and $\gamma$. The second item again fixes finitely many $\mathbf{n}'$ and $\gamma$. As a consequence, the sum over $\mathbf{n}$ in \eqref{bb17} is effectively finite, and we furthermore have the finiteness property
	\begin{equation*}
		\mbox{for every }\beta\in M(\coord),\,\#\{\gamma\in M(\coord)\,|\,(\mbpartial_i)_\beta^\gamma\neq 0\}<\infty.
	\end{equation*}
	The proof concludes as in Lemma \ref{lem:Dn}.
\end{proof}	
\begin{remark}\label{rem:fin01}
	The same argument shows as well the following finiteness property: For all $(\gamma',\mathbf{n}')$ $\in$ $M(\coord)\times \N_0^d$,
	\begin{equation}\label{fin50}
		\mbox{for all }\beta\in M(\coord),\,\#\{\gamma\in M(\coord)\,|\,(\z^{\gamma'}D^{(\mathbf{n}')})_\beta^\gamma\referee{\neq 0}\}<\infty.
	\end{equation}
	Note in addition that
	\begin{equation}\label{comp01}
		\big(\z^{\gamma'}D^{(\mathbf{n}')}\big)_\beta^\gamma \neq 0\,\implies \left\{\begin{array}{ll}
			\gamma' = \beta & \mbox{if }\gamma = e_{\mathbf{n}'}\\
			\gamma'<\beta & \mbox{otherwise. }
		\end{array}\right.
	\end{equation}	
\end{remark}	

\medskip

Once we have fixed an interpretation of the $\z$-variables, which carry the $(\mathbf{a},\mathbf{p},x)$-dependence, we build the analogue of $\tfrac{1}{\beta!}\big(\partial^\beta\big)|_{z=0} u_x$ in \eqref{warm15}. We will construct the \textit{centered model} as the inductive solution to the hierarchy of linear SPDEs described in \eqref{warm08}. Motivated by \eqref{warm16}, we set
\begin{equation}\label{bb01}
	\Pi_{x e_\mathbf{n}} = (\cdot - x)^\mathbf{n}.
\end{equation}
For $\beta\neq e_\mathbf{n}$, we consider the equation
\begin{equation}\label{bb02}
	\left\{\begin{array}{l}
		\mcL\Pi_{x\beta} = \Pi_{x\beta}^-,\\
		\Pi_{x\beta}^- = \sum_{(\mfl,k)} \sum_{e_{(\mfl,k)} + \sum_\mathbf{n} \sum_{j=1}^{k(\mathbf{n})} \beta_\mathbf{n}^j = \beta} \prod_\mathbf{n} \prod_{j=1}^{k(\mathbf{n})} \tfrac{1}{\mathbf{n}!}\partial^\mathbf{n} \Pi_{x \beta_\mathbf{n}^j} \xi_\mfl.
	\end{array}\right.
\end{equation}
\referee{
	\begin{exam}\label{example_05}
		Going back to the generalized KPZ equation, we have for the multi-index $\beta = e_{(\xi,0)}$
		\begin{equation*}
			(\partial_t - \partial_x) \Pi_{x\,e_{(\xi,0)}} = \xi.
		\end{equation*} 
		For $\beta = e_{(\xi,0)} + e_{(\xi,e_\mathbf{0})}$,
		\begin{equation*}
			(\partial_t - \partial_x) \Pi_{x\,e_{(\xi,0)} +  e_{(\xi,e_\mathbf{0})}} = \Pi_{x\, e_{(\xi,0)}}\xi.
		\end{equation*}
		For $\beta = 2 e_{(\xi,0)} + e_{(0,2 e_{(0,1)})}$,
		\begin{equation*}
			(\partial_t - \partial_x) \Pi_{x\,2 e_{(\xi,0)} + e_{(0,2 e_{(0,1)})}} = \big(\partial_x\Pi_{x\, e_{(\xi,0)}}\big)^2\xi.
		\end{equation*}
	\end{exam}
}
Assuming that \referee{equation \eqref{bb02}} can be solved uniquely, we can see that not all multi-indices are relevant. In particular, the following holds.
\begin{lemma}
	If
	\begin{equation*}
		\mcL v = 0\;\implies v\equiv0,
	\end{equation*}
	then the unique solution $\{\Pi_{x\beta}\}_{\beta\in M(\coord)}$ of \eqref{bb01}, \eqref{bb02} is such that
	\begin{equation}\label{bb06}
		\Pi_{x\beta} \not\equiv 0 \implies
			\left[\beta\right] := \sum_{(\mfl,k)} (1 - \length{k})\beta(\mfl,k) + \sum_\mathbf{n} \beta(\mathbf{n}) = 1.
	\end{equation}
\end{lemma}
\begin{proof}
	We will show \eqref{bb06} in its negated form
	\begin{equation*}
		[\beta]\neq 1 \implies \Pi_{x\beta} \equiv 0
	\end{equation*}
	by induction in $\length{\beta}$. For $\length{\beta}=0$ we have $\beta=0$, and thus the r.~h.~s. of the second item in \eqref{bb02} is vanishing (note the presence of $e_{(\mfl,k)}$ under the sum); uniqueness then implies $\Pi_{x 0} \equiv 0$. For the induction step, we note by additivity of \eqref{bb06} that if
	$ e_{(\mfl,k)}$ $+$ $\sum_\mathbf{n} \sum_{j=1}^{k(\mathbf{n})} \beta_\mathbf{n}^j$ $=$ $\beta$ it holds
	\begin{equation*}
		[\beta] = 1-\length{k} + \sum_\mathbf{n} \sum_{j=1}^{k(\mathbf{n})} [\beta_\mathbf{n}^j].
	\end{equation*}
	Therefore, 
	\begin{equation*}
		[\beta]\neq 1 \iff \sum_\mathbf{n} \sum_{j=1}^{k(\mathbf{n})} [\beta_\mathbf{n}^j] \neq \length{k},
	\end{equation*}
	which in turn implies that $[\beta_\mathbf{n}^j]\neq 1$ for at least a pair $(j,\mathbf{n})$. By the induction hypothesis, the r.~h.~s. of the second item in \eqref{bb02} is vanishing, and the uniqueness assumption implies $\Pi_{x \beta} \equiv 0$.
\end{proof}	
\referee{ \begin{exam}\label{example_06}
		In the case of the generalized KPZ equation, it holds
		\begin{align*}
			[\beta] = &\sum_{k_\mathbf{0}\in \N_0} (1-k_\mathbf{0}) \beta(\xi,k_\mathbf{0} e_\mathbf{0}) + \sum_{k_\mathbf{0}\in \N_0} (1-k_\mathbf{0}) \beta(0,k_\mathbf{0} e_\mathbf{0})\\
			&- \sum_{k_\mathbf{0}\in \N_0} k_\mathbf{0} \beta(0,k_\mathbf{0} e_\mathbf{0} +e_{(0,1)}) \\
			&- \sum_{k_\mathbf{0}\in \N_0} (1+k_\mathbf{0}) \beta(0,k_\mathbf{0} e_\mathbf{0} + 2 e_{(0,1)})\\
			&+ \sum_{\mathbf{n}\in\N_0^2} \beta(\mathbf{n}).
		\end{align*}
		The reader is invited to check that condition \eqref{bb06} is satisfied for the multi-indices of Example \ref{example_05}.
	\end{exam}}
Note that the additivity of $[\cdot]$ means that the linear subspace of $\R[[\coord]]$ generated by multi-indices $\beta$ satisfying $[\beta]=1$ is not an algebra, and thus it does not make sense to speak about derivation properties (i.~e. the Leibniz rule) in this subspace. Moreover, $D^{(\mathbf{n})}$ does not preserve multi-indices satisfying $[\beta]=1$, but $\mbpartial_i$ does, as shown in the next lemma.
\begin{lemma}
	It holds:
	\begin{align}
		(D^{(\mathbf{n})})_\beta^\gamma \neq 0 &\implies \left[\beta\right] = \left[\gamma\right] -1 \label{bb20},\\
		(\mbpartial_i)_\beta^\gamma \neq 0 &\implies \left[\beta\right] = \left[\gamma\right]\label{bb21}. 
	\end{align}
\end{lemma}
\begin{proof}
	We note that
	\begin{equation}\label{bb19}
		(\partial_{\z_{(\mfl,k)}})_\beta^\gamma \neq 0 \implies \left[\beta\right] = \left[\gamma\right] + \length{k}-1. 
	\end{equation}
	Then \eqref{bb20} follows from \eqref{bb19} and definition \eqref{bb16}. In turn, \eqref{bb21} follows from \eqref{bb20} and formula \eqref{bb17}.
\end{proof}	
As an immediate consequence of \eqref{bb20}, we have that
\begin{equation}\label{bb22}
	[\gamma'] = 1,\,\,(\z^{\gamma'}D^{(\mathbf{n}')})_\beta^\gamma \neq 0 \implies \left[\beta\right] = \left[\gamma\right].
\end{equation}
Therefore, while $D^{(\mathbf{n})}$ does not preserve $[\cdot] = 1$, the bilinear map $(\z^{\gamma'},\z^\gamma)\mapsto \sum_\beta (\z^{\gamma'}D^{(\mathbf{n})})_\beta^\gamma \z^\beta$ does. This crucial observation is algebraically translated into the fact that, even if $D^{(\mathbf{n})}$ is no longer a derivation, it still acts as a pre-Lie product.
\referee{\begin{definition}
	A \textit{(left) pre-Lie product} $\triangleright$  is a bilinear operation that satisfies
	\begin{equation}\label{ref:assoc}
		(x \triangleright y) \triangleright z - x \triangleright (y \triangleright z) = (y \triangleright x) \triangleright z - y \triangleright (x \triangleright z).
	\end{equation}
	A vector space equipped with a pre-Lie product is called \textit{pre-Lie algebra}.
\end{definition}
We refer the reader to \cite{Man11} for a short survey on pre-Lie algebras.}
\begin{remark}
	Actually, the pre-Lie products $D^{(\mathbf{n})}$ generate what has been called in the literature a \textit{multi-pre-Lie algebra}, cf. \cite[Proposition 4.21]{BCCH}. A multi-pre-Lie algebra is an extension of a pre-Lie algebra, for which the associator identity \eqref{ref:assoc} is imposed pairwise in different pre-Lie products: For example, in our case, for $\mathbf{n},\mathbf{n}'\in\N_0^d$,
	\begin{equation*}
	(\pi D^{(\mathbf{n})} \pi') D^{(\mathbf{n'})} \pi'' - \pi D^{(\mathbf{n})} (\pi' D^{(\mathbf{n}')} \pi'') = (\pi' D^{(\mathbf{n}')} \pi) D^{(\mathbf{n})} \pi'' - \pi' D^{(\mathbf{n}')} (\pi D^{(\mathbf{n})} \pi'').	
	\end{equation*}	
	This multi-pre-Lie algebra can be summarized in a single pre-Lie algebra on the space of elements $\pi D^{(\mathbf{n})}$, which will be the starting point of the algebraic construction in Section \ref{section::3} below.
\end{remark}	

\medskip

Under a stronger uniqueness assumption of the form\footnote{The reader might find \eqref{pol01} too strong to be reasonable, but it can be established in some concrete situations: In \cite{reg}, by means of a decomposition of the solution kernel in a singular part that annihilates polynomials (which is used in integration) and a regular part that produces smooth remainders (which we neglect, as they are again locally parameterized by polynomials); in \cite{LOTT}, by interpreting $v$ only as a solution modulo polynomials, and imposing additional conditions on the degree of the polynomial to argue for uniqueness via a Liouville principle. See Subsection \ref{subsec::4.1} below, where we construct the model following the analytic strategy of \cite{reg}, and \eqref{pol01} is interpreted as neglecting all polynomial contributions from the r.~h.~s. of the model equation.} 
\begin{equation}\label{pol01}
	\mcL v = \mbox{polynomial} \implies v\equiv0,
\end{equation}
and recalling $\xi_0 = 1$, condition \eqref{bb06} is complemented by 
\begin{equation}\label{bb07}
	\left\{\begin{array}{l}
			\beta = e_\mathbf{n}\; \mbox{for some }\mathbf{n}\in\N^{d}, \mbox{ or}\\
			\lnh \beta\rnh :=\sum_{\mfl\in\mfL^-}\sum_k \beta(\mfl,k)>0.
		\end{array}\right.
\end{equation}
We call $\lnh\beta\rnh$ \textit{noise homogeneity}; note that we are not considering $\mfl = 0$ in the sum. Multi-indices satisfying \eqref{bb06} and \eqref{bb07} are the ones that potentially give non-vanishing contributions to $\Pi_x$ under \eqref{pol01}, but the same is not true for $\Pi_x^-$, since the latter incorporates contributions with $\lnh\beta\rnh = 0$ which are polynomials. This will be reflected in the different index sets considered in Definition \ref{def:pop} below.

\medskip

We now focus on the homogeneity. In regularity structures, one of the main properties of the model is that it is \textit{locally homogeneous} of a certain order. We now guess the homogeneity of a model component $\Pi_{x\beta}$ in terms of the multi-index $\beta$ by a scaling argument. For this, we first fix a scaling $\mathfrak{s}= \{\mathfrak{s}_i\}_{i=1,..,d}$ $\subset$ $[1,\infty)$ of the Euclidean space which is compatible with the operator, and use it to measure all regularity properties in this \textit{inhomogeneous} space. More precisely, we consider
\begin{itemize}
	\item the scaled degree, i.~e. for $\mathbf{n}\in\N_0^d$ 
	\begin{equation}\label{scal01}
	 |\mathbf{n}|_{\mathfrak{s}} := \sum_{i=1}^d \mathfrak{s}_i \mathbf{n}(i);
\end{equation}
	\item the scaled Carnot-Carathéodory distance, i.~e.
	\begin{equation*}
		|y-x|_{\mathfrak{s}} := \sum_{i=1}^d |y_i-x_i|^\frac{1}{\mathfrak{s}_i};
	\end{equation*}
	\item the rescaling operator $\mcS_\lambda^\mathfrak{s} : \R^d\to\R^d$,
	\begin{equation*}
		\mcS_\lambda^\mathfrak{s}(x_1,...,x_d) = (\lambda^{\mathfrak{s}_1}x_1,...,\lambda^{\mathfrak{s}_d}x_d).
	\end{equation*}
\end{itemize}
Note that $|\mcS_\lambda^\mathfrak{s}(y)-\mcS_\lambda^\mathfrak{s}(x)|_\mathfrak{s} = \lambda |y-x|_\mathfrak{s}$. In the sequel, we will remove the $\mathfrak{s}$ from the notation and assume all distances are inhomogeneous\footnote{ In particular, regularity will always be measured in the inhomogeneous distance}. By compatibility with the operator we mean that $\mcL (f\circ\mcS_\lambda) = \lambda^{\eta} (\mcL f)\circ \mcS_\lambda$ for some $\eta>0$, i.~e. $\mcL$ is homogeneous. Moreover, we will assume that $\mcL$ satisfies a Schauder estimate of degree $\eta$; for our current purposes, it is enough to think of this as \textit{increasing the homogeneity by $\eta$ when solving the PDE}. Similarly, we impose for $\{\xi_\mfl\}_{\mfl\in\mfL^-}$ a scale-invariance\footnote{This is usually assumed in law.}, i.~e. for every $\mfl\in\mfL^-$ there exists $\alpha_\mfl\in\R$ such that
\begin{equation*}
	\xi_\mfl\circ \mcS_\lambda = \lambda^{\alpha_\mfl} \xi_\mfl;
\end{equation*}
we accordingly set $\alpha_0 = 0$ for $\xi_0 = 1$.  This scaling property, when assumed locally around any point\footnote{ If $\alpha_\mfl <0$, this means that testing $\xi_\mfl$ against a test function localized at any space-time point $x\in \R^d$, the convolution scales as $\alpha_\mfl$. If $\alpha_\mfl >0$, the same holds, but only for $\xi_\mfl$ up to subtracting a polynomial centered at $x\in \R^d$ and of degree $<\alpha_\mfl$. Both conditions can be combined into one if one chooses test functions that annihilate polynomials of a certain order; cf. e.~g. \cite[Definition 4.12, Proposition 4.15]{FH}.}  implies a Hölder regularity condition of $\xi$. 
We now consider the rescaled solution $\tilde{u} := \lambda^{-\eta} u\circ\mcS_\lambda$, and note that
\begin{equation}\label{sca01}
	\partial^\mathbf{n} \tilde{u} = \lambda^{-\eta + |\mathbf{n}|} (\partial^\mathbf{n}u)\circ\mcS_\lambda.
\end{equation}
This rescaled solution then solves
\begin{equation*}
	\mcL \tilde{u} =  (\mcL u)\circ \mcS_\lambda = \sum_{\mfl\in\mfL^-\cup\{0\}} a^\mfl(\mathbf{u}\circ\mcS_\lambda) (\xi_\mfl\circ\mcS_\lambda) = \sum_{\mfl\in\mfL^-\cup\{0\}}\tilde{a}^\mfl(\tilde{\mathbf{u}})\xi_\mfl,
\end{equation*}
where
\begin{equation}\label{sca02}
	\tilde{a}^\mfl (\{\partial^\mathbf{n}v\}_{\mathbf{n}\in\N_0^d}) = \lambda^{\alpha_\mfl} a^\mfl(\{\lambda^{\eta-|\mathbf{n}|} \partial^\mathbf{n} v\}).
\end{equation}
Now \eqref{sca01} and \eqref{sca02} impose a rescaling of $(\mathbf{a},\mathbf{p})$\footnote{At least locally around $x=0$, although one can generalize the argument to fix the origin at any point $x\in \R^d$.} which associates
\begin{equation*}
	|e_\mathbf{n}| = - \eta + |\mathbf{n}|,\,\,\, |e_{(\mfl,k)}| = \alpha_\mfl + \sum_{\mathbf{n}\in\N_0^d} (\eta - |\mathbf{n}|)k(\mathbf{n}).
\end{equation*}
One can then extend this multiplicatively via \eqref{bb02} to see that $\Pi_{x\beta}^-$ for $\beta\neq e_\mathbf{n}$ should be locally homogeneous of degree
\begin{equation}\label{bb04}
	|\beta| := \sum_{(\mfl,k)} \Big(\alpha_\mfl + \sum_\mathbf{n} (\eta-|\mathbf{n}|)k(\mathbf{n})\Big) \beta(\mfl,k) + \sum_\mathbf{n}  (|\mathbf{n}|-\eta)\beta(\mathbf{n}).
\end{equation}
Assuming an $\eta$-regularizing property for the inverse of $\mcL$, which is compatible with its $\eta$-scaling, $\Pi_{x\beta}$ for $\beta\neq e_\mathbf{n}$ should be homogeneous of degree $|\beta|+\eta$. This can be extended to purely polynomial multi-indices: Indeed, since $\Pi_{x e_\mathbf{n}}(y) = (y-x)^\mathbf{n}$, $\Pi_{x e_\mathbf{n}}(y)$ is homogeneous of degree $|\mathbf{n}| = |e_{\mathbf{n}}| + \eta$.
\referee{ \begin{exam}
		Let us specify the homogeneity in the case of the generalized KPZ equation. We first fix the parabolic scaling, associated to the heat operator: Recalling that the first component represents time and the second component represents space, for $\mathbf{n} \in \N_0^2$ we set
		\begin{equation*}
			|\mathbf{n}| = 2 \mathbf{n}(1) + \mathbf{n}(2).
		\end{equation*}
		Under this scaling, the heat operator regularizes by $\eta = 2$. On the other hand, the scale invariance of space-time white noise implies $\xi \circ \mcS_\lambda = \lambda^{-\frac{3}{2}} \xi$ in law. By Kolmogorov's criterion\footnote{ It is convenient to think in a pathwise manner in preparation for Section \ref{section::4}, but see also \cite{LOTT} where the estimates are annealed and do not carry an infinitesimal loss.} we set $\alpha_\xi = -\frac{3}{2}\mhyphen$, by which we mean $\alpha_\xi<-\frac{3}{2}$ and $\alpha_\xi$ is aribtrarily close to $-\frac{3}{2}$. In addition we set $\alpha_0 = 0$.
		
		\medskip
		
		With these choices, \eqref{bb04} leads to
		\begin{align*}
			|\beta| = &\sum_{k_\mathbf{0}} \big(-\tfrac{3}{2}\mhyphen + 2 k_\mathbf{0}\big) \beta(\xi,k_\mathbf{0} e_\mathbf{0}) + \sum_{k_\mathbf{0}}  2 k_\mathbf{0} \beta(0,k_\mathbf{0} e_\mathbf{0}) \\
			&+ \sum_{k_\mathbf{0}}  \big(2 k_\mathbf{0} + 1\big) \beta(0,k_\mathbf{0} e_\mathbf{0} + e_{(0,1)})\\
			& + \sum_{k_\mathbf{0}}  \big(2 k_\mathbf{0} + 2\big) \beta(0,k_\mathbf{0} e_\mathbf{0} + 2 e_{(0,1)})\\
			& + \sum_{\mathbf{n}\in \N_0^2} (|\mathbf{n}| - 2)\beta(\mathbf{n}).
		\end{align*}
		Next to condition $[\beta] = 1$ (see Example \ref{example_06}), we can express it as
		\begin{align}
			|\beta|= &\big(\tfrac{1}{2}\mhyphen\big)\sum_{k_\mathbf{0}} \beta (\xi,k_\mathbf{0} e_\mathbf{0}) + 2 \sum_{k_\mathbf{0}} \beta(0,k_\mathbf{0} e_\mathbf{0}) \nonumber\\
			&+ \sum_{k_\mathbf{0}} \beta (0,k_\mathbf{0} e_\mathbf{0} + e_{(0,1)}) + \sum_{\mathbf{n}\in \N_0^2} |\mathbf{n}| \beta(\mathbf{n}) - 2.\label{ref:hom_kpz}
		\end{align}
		For instance, the multi-index $\beta = 3 e_{(\xi,0)} + e_{(0,e_\mathbf{0} + e_{(0,1)})} + e_{(0,2e_{(0,1)})}$ has homogeneity
		\begin{align*}
			|\beta|= \big(\tfrac{1}{2}\mhyphen \big) \times 3 + 1 -2 = \tfrac{1}{2}\mhyphen.
		\end{align*}
	\end{exam}
}
\subsection{Subcritical multi-indices}\label{subsec::2.4}
The set of multi-indices described up to this point covers a large class of semi-linear equations, which goes beyond the subcritical class covered by regularity structures. In order to algebraically exploit subcriticality, we will follow the approach of \cite[Subsection 5.2]{BHZ}, adapting it to our language. We assume there exists a map $\reg : \mfL^-\cup \{0\} \to \R$ such that
\begin{equation}\label{sub01}
	\reg(\mfl) < \alpha_\mfl
\end{equation}
and a value $\regsol\in \R$. Here, $\reg(\mfl)$ and $\regsol$ are placeholders for the regularity of $\xi_\mfl$ and $u$, respectively. Recall that subcriticality, roughly speaking, means that the solution $u$ is a perturbation of the solution to the linear problem by more regular terms, so that an expansion of the form \eqref{warm15} is a reasonable Ansatz. Let us fix $\mfl$ and recall the Taylor expansion \eqref{warm03}; it is then natural to associate to a pair $(\mfl,k)$ the \textit{expected regularity} as 
\begin{equation*}
\referee{	\min \Big\{\reg ( \mfl ), \min_{0<k'\leq k}\sum_{\mathbf{n} \in \N_0^d} ( \regsol - |\mathbf{n}| ) k'(\mathbf{n}), \reg ( \mfl ) + \min_{0<k'\leq k}\sum_{\mathbf{n} \in \N_0^d} ( \regsol - |\mathbf{n}| ) k'(\mathbf{n})\Big\}.}
\end{equation*}
Expecting the integrated version of this term to be better than the solution yields the following condition.
\begin{definition}\label{def:subcrit}
	A pair $(\mfl,k)\in (\mfL^-\cup \{0\})\times M(\N_0^{d})$ is \textit{subcritical} if
	\referee{\begin{align}
		\regsol< \eta 
		+ \min \Big\{&\reg ( \mfl ), \min_{0<k'\leq k}\sum_{\mathbf{n} \in \N_0^d} ( \regsol - |\mathbf{n}| ) k'(\mathbf{n}),\nonumber\\
		& \reg ( \mfl ) + \min_{0<k'\leq k}\sum_{\mathbf{n} \in \N_0^d} ( \regsol - |\mathbf{n}| ) k'(\mathbf{n})\Big\}.\label{sub02}
	\end{align}}
A multi-index $\beta$ is subcritical if
\begin{equation*}
	\beta(\mfl,k)\neq 0 \implies (\mfl,k)\mbox{ is subcritical}.
\end{equation*}
\end{definition}	
\referee{Note that in order for an equation to be subcritical, the ``additive noise'' terms should satisfy this subcriticality condition (in our language, $(\mfl,0)$ must be subcritical for all $\mfl\in \mfL^-\cup\{0\}$), which in turn implies
	\begin{equation}\label{ref:sub_10}
		\regsol < \eta + \reg(\mfl)\,\,\,\mbox{for all }\mfl\in\mfL^-\cup\{0\}.
	\end{equation}}
Subcriticality imposes some more restrictions on $k $.
\begin{lemma}\label{lemsub03}
	\mbox{}
	\begin{itemize}
		\item Maximal $\mathbf{n}$: Let $(\mfl,k)$ be a subcritical pair. Then
		\begin{equation}\label{sub12}
			k(\mathbf{n})\neq 0 \implies |\mathbf{n}| < \eta.
		\end{equation}
		\item Maximal $k(\mathbf{n})$: 
		\begin{itemize}
			\item If $|\mathbf{n}| < \regsol$, then for every $N\in\N_0$ there exists a subcritical $(\mfl,k)$ such that $k(\mathbf{n})>N$.
			\item If $|\mathbf{n}|>\regsol$ and $(\mfl,k)$ is a subcritical pair, then
			\begin{equation*}
				\referee{k(\mathbf{n}) < \frac{\eta+\min \{\reg(\mfl),0\} - \regsol}{|\mathbf{n}|- \regsol}.}
			\end{equation*} 
		\end{itemize}
	\end{itemize}
\end{lemma}	
\begin{proof}
	For the first item, $k(\mathbf{n})\neq 0$ implies that $e_\mathbf{n} \leq k$; then \eqref{sub12} follows choosing $k'=e_\mathbf{n}$ on the r.~h.~s. of \eqref{sub02}. For the second item, \referee{we fix $\mathbf{n}$ and distinguish between the two cases.
	\begin{itemize}
		\item If $|\mathbf{n}|<\regsol$, then the minimum 
		\begin{equation*}
			\min_{0<k'\leq k} \sum_{\mathbf{n}\in\N_0^d} (\regsol - |\mathbf{n}|)k'(\mathbf{n})
		\end{equation*}
		is attained at some $k'$ with $k'(\mathbf{n})=0$. Therefore, since $(\mfl,0)$ is subcritical, the pair $(\mfl, k(\mathbf{n}) e_\mathbf{n})$ is also subcritical for all $k(\mathbf{n})$.
		\item If $\mathbf{n}>\regsol$, then
		\begin{equation*}
			\min_{k'\leq k} \sum_{\mathbf{n}\in\N_0^d} (\regsol - |\mathbf{n}|)k'(\mathbf{n})<0.
		\end{equation*}
		Condition \eqref{sub02} reduces to
		\begin{equation*}
			\regsol < \eta + \min \{\reg(\mfl),0\} + \min_{k'\leq k} \sum_{\mathbf{n}\in\N_0^d} (\regsol - |\mathbf{n}|)k'(\mathbf{n}).
		\end{equation*}
		In particular, this implies
		\begin{equation*}
			\regsol<\eta + \min \{\reg(\mfl),0\} + (\regsol - |\mathbf{n}|)k(\mathbf{n}),
		\end{equation*}
		which in turn yields the desired inequality.
\end{itemize}}
\end{proof}

\medskip

\referee{
	\begin{exam}
		Let us study the subcriticality conditions in the case of the generalized KPZ equation. We start by fixing some arbitrarily small $\kappa>0$ and
		\begin{equation}\label{ref:sub_01}
			\reg(\xi) = -\tfrac{3}{2}-2\kappa,\,\, \reg(0)=0-2\kappa,\,\, \regsol = \tfrac{1}{2}-3\kappa.
		\end{equation}
		The reason why we subtract $2\kappa$ from the regularity of the noises and $3\kappa$ from the regularity of the solution is to guarantee \eqref{ref:sub_10}, but there is of course flexibility in this choice. We look for the subcritical pairs under these assumptions. Condition \eqref{sub12} implies that $(\mfl,k)$ can only be subcritical if $k(\mathbf{n}) = 0$ for all $|\mathbf{n}|\geq 2$, which means that necessarily
		\begin{equation*}
			k = k_\mathbf{0} e_\mathbf{0} + k_{(0,1)} e_{(0,1)},\,\, k_\mathbf{0},k_{(0,1)} \in \N_0^d.
		\end{equation*}
		Now, on the one hand, the second part of Lemma \ref{lemsub03} implies that in order for $(\xi,k)$ to be subcritical, $k(\mathbf{0})$ can be arbitrarily large, whereas
		\begin{align*}
			k_{(0,1)} < \frac{2-\frac{3}{2}-2\kappa - \frac{1}{2} + 3\kappa}{1-\frac{1}{2}+3\kappa} = \frac{\kappa}{\frac{1}{2} + 3\kappa} <1,
		\end{align*}
		and therefore $k_{(0,1)}=0$. On the other hand, in order for $(0,k)$ to be subcritical, $k_\mathbf{0}$ can again be arbitrarily large, whereas
		\begin{align*}
			k_{(0,1)} < \frac{2 - 2 \kappa - \frac{1}{2} + 3 \kappa}{1-\frac{1}{2} + 3\kappa} = \frac{3 + 2 \kappa}{1+6\kappa} < 3,
		\end{align*} 
		and therefore $k_{(0,1)}\leq 2$. Consequently, all subcritical pairs are of the form
		\begin{align*}
			&\{(\xi,k_\mathbf{0}e_\mathbf{0})\,|\, k_\mathbf{0}\in \N_0^2\} \\
			&\quad \cup \{(0, k_\mathbf{0}e_\mathbf{0} + k_{(0,1)}e_{(0,1)})\,|\, k_\mathbf{0}\in\N_0^2,\, k_{(0,1)}=0,1,2\}.
		\end{align*}
		These correspond to all nonlinearities present in the generalized KPZ equation \eqref{kpz01_example}. The multiplicative noise term contains a function of the solution, which is consistent with the first class of subcritical pairs. The additive noise term is a function of the solution and a quadratic polynomial in the first derivative, which agrees with the second class of subcritical pairs. Therefore, the generalized KPZ equation is subcritical.
		
		\medskip
		
		If we had a cubic term in the first derivative, we should also consider pairs of the form
		\begin{equation*}
			(0, k_\mathbf{0}e_\mathbf{0} + 3e_{(0,1)}).
		\end{equation*}
		Now the r.~h.~s. of \eqref{sub02} reduces to
		\begin{equation*}
			2 - 2\kappa + 3 \big(-\tfrac{1}{2} - 3\kappa\big) = \tfrac{1}{2} - 11\kappa,
		\end{equation*}
		which is obviously smaller than $\frac{1}{2} - 3 \kappa$. The reader is invited to check that varying \eqref{ref:sub_01}, as long as $\regsol<\eta + \reg(\xi)$, does not make \eqref{sub02} hold. Therefore, we conclude that $(0,k_\mathbf{0} e_\mathbf{0} + 3e_{(0,1)})$ is not subcritical.
	\end{exam}
}

\medskip
	
The second item \referee{of Lemma \ref{lemsub03} reflects} the following fact: Since we think of $\regsol$ as the expected regularity of the solution, the derivatives $\partial^\mathbf{n} u$ for $|\mathbf{n}|< \regsol$ are functions, and in particular the r.~h.~s. of \eqref{set01} is allowed to contain arbitrarily large powers of these derivatives without breaking subcriticality. A rather unpleasant consequence is that this creates terms with arbitrarily large powers of the constant $1$, in form of the model component $\partial^\mathbf{n} \Pi_{x e_\mathbf{n}}$. To avoid these redundancies, we make the following restriction in our set of multi-indices:
\begin{equation}\label{bb05}
	\mbox{for every }\mathbf{n}\in\N_0^{d}\mbox{ with }|\mathbf{n}|< \regsol,\,\,\beta(\mathbf{n}) = \left\{\begin{array}{ll}
		1 & \mbox{if } \beta =  e_\mathbf{n},\\
		0 & \mbox{otherwise}.
	\end{array}\right.
\end{equation}
We are not losing any information by imposing this restriction. This is because if the solution is supposed to be a $C^{l,\alpha}$ function, we may substract its $l$-th Taylor polynomial and build the model starting from the homogeneous remainder. This is the approach taken in \cite{OSSW}, where the main goal is to obtain an a priori estimate; since the solution is Hölder continuous, the authors already think of the model up to constants, and thus the reason to exclude $\mathbf{n} = \referee{\mathbf{0}}$ from the polynomial sector. In our case, we do not mod out the lowest polynomial contributions completely, but restrict them to appear \textit{only} in the polynomial sector, which is essentially equivalent but allows for cleaner formulations of the recentering properties (as compared to \cite[(9)]{OSSW} and \cite[(2.30)]{LOTT}). See Subsection \ref{subsec::5.2} for an implementation of this in the one-dimensional multiplicative stochastic heat equation.
\begin{definition}\label{def:pop}
We define the following sets of multi-indices:
\begin{align*}
	\mcNmin& := \{\beta\;|\; [\beta]=1,\,\beta\,\mbox{subcritical},\, \beta(\mathbf{n}) = 0\; \mbox{for all }|\mathbf{n}|< \regsol\},\\
	\mcN &:= \{\beta\in\mcNmin\;|\; \lnh\beta\rnh>0\},\\
	\mcP &:= \{e_\mathbf{n}\}_{\mathbf{n}\in \N_0^d},\\ 
	\populated &:= \mcP \sqcup \mcN.
\end{align*}
\end{definition}
The following finiteness property is the analogue of \cite[Proposition 5.15]{BHZ}.
\begin{lemma}\label{lemsub01}
	For every $a\in \R$,
	\begin{equation}\label{sub06}
		\# \{\beta\in\populated\cup\mcNmin\,|\,|\beta|<a\}<\infty.
	\end{equation}
\end{lemma}	
\begin{proof}
	The statement is clear for $\beta \in \mcP$, and thus we focus on $\beta \in \mcNmin$. We rewrite \eqref{bb04} as
	\begin{align*}
		|\beta| &= \sum_{(\mfl,k)} \big(\reg(\mfl) + \sum_{\mathbf{n}} (\regsol - |\mathbf{n}|)k(\mathbf{n})\big)\beta (\mfl,k) + \sum_{\mathbf{n}} (\regsol - \eta)\beta(\mathbf{n})\\
		&\quad + \sum_{(\mfl,k)} \big(\alpha_\mfl - \reg(\mfl)\big)\beta(\mfl,k) + (\eta - \regsol)\sum_{(\mfl,k)}\length{k}\beta(\mfl,k) + \sum_{\mathbf{n}} (|\mathbf{n}|-\regsol)\beta(\mathbf{n}).
	\end{align*}
	We use \eqref{sub02} to bound
	\begin{equation*}
		\sum_{(\mfl,k)} \big(\reg(\mfl) + \sum_{\mathbf{n}} (\regsol - |\mathbf{n}|)k(\mathbf{n})\big)\beta (\mfl,k) > (\regsol - \eta) \sum_{(\mfl,k)} \beta(\mfl,k),
	\end{equation*}
	which combined with $[\beta]=1$, cf. \eqref{bb06}, yields
	\begin{equation*}
		|\beta| > \regsol - \eta + \sum_{(\mfl,k)} (\alpha_\mfl - \reg(\mfl))\beta(\mfl,k) + \sum_{\mathbf{n}} (|\mathbf{n}|-\regsol)\beta(\mathbf{n}).
	\end{equation*}
	Now, on the one hand, by \eqref{sub01} (note that the inequality is strict) there exists a small constant $\referee{\kappa>0}$ such that
	\begin{align*}
		\reg(\mfl) + \kappa\leq \alpha_\mfl,\;\mfl\in\mfL^-\cup\{0\}.
	\end{align*}
	On the other hand, since $|\N_0^d|\subset \R$ is locally finite, $\kappa$ can be chosen small enough so that
	\begin{equation*}
		|\mathbf{n}| \geq \regsol + \referee{2}\kappa
	\end{equation*}
	for all $|\mathbf{n}|>\regsol$. As a consequence, 
	\begin{equation}\label{sub04}
		|\beta| > \regsol - \eta + \kappa \length{\beta} +\referee{\tfrac{1}{2}} \sum_{\mathbf{n}} (|\mathbf{n}| - \regsol)\beta(\mathbf{n}),
	\end{equation}
	where the last term is nonnegative by \eqref{bb05}. This means that the homogeneity of subcritical multi-indices bounds their length and the size of their polynomial contributions (meaning that there are only finitely many of them allowed). Note that condition \eqref{bb02} implies that $\length{k}\leq \length{\beta} - 1$; combined with \eqref{sub12}, this means there are only finitely many pairs $(\mfl,k)$ allowed in a multi-index of fixed length. 
\end{proof}	
\begin{remark}
	Note that \eqref{sub04} implies the lower bound $|\beta|>\regsol - \eta$ for $\beta\in \mcNmin$. As a consequence,
	\begin{equation}\label{sub05}
		\beta\in\populated\cup\mcNmin \implies |\beta| > (\regsol \wedge 0) - \eta.
	\end{equation}
\end{remark}	
\begin{corollary}
	The set
	\begin{equation}\label{hom50}
		A:= \{|\beta|\,|\, \beta\in\mcNmin\}\cup \{|\beta| + \eta\,|\, \beta\in\populated\}\subset \R
	\end{equation}
is a well-defined set of homogeneities.
\end{corollary}	
\begin{definition}
	For every index set $\mathcal{V}\subset M(\coord)$, we define the linear subspace of $\R[[\coord]]$:
	\begin{equation*}
		T_\mathcal{V}^* := \big\{\pi = \sum_{\beta\in \mathcal{V}} \pi_\beta \z^\beta\big\}.
	\end{equation*}
	We also introduce the projection $\proj_\mathcal{V} : \R[[\coord]] \to T_\mathcal{V}^*$.
\end{definition}
Note that with this notation
	\begin{equation*}
		T^*_\populated = T^*_\mathcal{P} \oplus T^*_\mathcal{N}.
	\end{equation*}
	The model components together with the functionals \eqref{bb03} allow us to consider the formal power series
\begin{align*}
	\Pi_x &:= \sum_{\beta\in\populated} \Pi_{x\beta} \z^\beta\\
	\Pi_x^- &:= \sum_{\beta\in \mcNmin} \Pi_{x\beta}^- \z^\beta.
\end{align*}
More precisely, assuming smoothness, we will think of
\begin{equation*}
	(\Pi_x,\Pi_x^-): \R^d \to  T_\populated^* \times T_{\mcNmin}^*,
\end{equation*}
where the first component is equipped with the shifted homogeneity $|\cdot| + \eta$, whereas the second is equipped with the plain homogeneity $|\cdot|$; this is why we consider both in \eqref{hom50}.

\medskip

\referee{\begin{remark}\label{rem:primal}The notation $T_\mathcal{V}^*$ is chosen so that we think of these spaces as the algebraic dual of the vector space $T_\mathcal{V}$ given by
	\begin{equation*}
		T_\mathcal{V} := \lspan\{\z_\beta\,|\,\beta\in \mathcal{V}\}
	\end{equation*}
	where $\{\z_\beta\}_{\beta\in \mathcal{V}}$ is a basis dual to the monomials $\{\z^\beta\}_{\beta\in\mathcal{V}}$ via the canonical pairing. This way, the model space in the sense of \cite[Definition 2.1]{reg} is a copy of $T_\populated \oplus T_\mcNmin$, and the polynomial sector can be identified with $T_\mathcal{P}$.\end{remark}}
	
	\medskip
				
\begin{remark}
	Restricting to the set of subcritical multi-indices is equivalent to restricting to subcritical equations: Removing a pair $(\mfl,k)$ is the same as fixing that $\partial^k a^\mfl \equiv 0$, according to \eqref{bb03}. Thus, a model with values in $T_\populated^*\times T_\mcNmin^*$ as above is a model for a subcritical equation.
\end{remark}	

\medskip

We note now that the derivation $D^{(\mathbf{n})}$ defined in \eqref{bb16} does not preserve the subcriticality property: In particular, the infinite sum stays subcritical only for $|\mathbf{n}|<\regsol$, according to Lemma \ref{lemsub03}. In order for $\z^\gamma D^{(\mathbf{n})}$ to self-map $T_\populated^*\oplus T_\mcNmin^*$, we need to redefine \eqref{bb16} introducing a projection onto subcritical multi-indices, setting
\begin{equation}\label{sub13}
	D^{(\mathbf{n})} :=  \Big(\hspace*{-10pt}\sum_{\substack{(\mfl,k)\\ (\mfl,k+e_\mathbf{n})\,\text{subcritical}}} \hspace*{-10pt}(k(\mathbf{n})+1)\z_{(\mfl,k+e_\mathbf{n})} \partial_{\z_{(\mfl,k)}} + \partial_{\z_\mathbf{n}}\Big)\proj_{\populated\cup\mcNmin}.
\end{equation}
Obviously, the projection $\proj_{\populated\cup\mcNmin}$ is immaterial when we think of $D^{(\mathbf{n})}$ $\in$ $\textnormal{End}(T_{\populated\cup \mcNmin}^*)$, but it is useful to write it down as a reminder for later, when we transpose these maps and we still want to look only at multi-indices in $\populated\cup \mcNmin$. Note that by \eqref{sub12} only finitely many $\mathbf{n}$ $\in$ $\N_0^d$ produce the first contribution in \eqref{sub13}; more precisely, as endomorphisms of $T_{\populated\cup \mcNmin}^*$,
\begin{equation*}
	\mbox{for }|\mathbf{n}|>\eta ,\,\,\,D^{(\mathbf{n})} = \partial_{\z_\mathbf{n}}.
\end{equation*}
Similarly, we redefine $\mbpartial_i$, cf. \eqref{bb17}, according to the new $D^{(\mathbf{n})}$ defined in \eqref{sub13}. Note that the restriction to subcritical $(\mfl,k)$ makes the sum over $\mathbf{n}$ in \eqref{bb17} finite, thanks to condition \eqref{sub12}.
\subsection{Connection to regularity structures based on trees}\label{subsec::2.5}
This subsection is independent of the rest of the paper, but might give more intuition to the reader already familiar with tree-based regularity structures. It was observed in \cite[Sections 6 and 7]{LOT} that one can establish a connection between trees and multi-indices: Roughly speaking, a multi-index represents the \textit{fertility} of a tree, in the sense that it encodes the amount of nodes with a specific configuration of outgoing edges. We will now describe a map from trees to multi-indices in the general situation given by the equation \eqref{set01}.

\medskip

Let us first describe the set of trees under consideration. The nodes of these trees are decorated by $\mfL^-\cup\{0\}$ and  $\N_0^d$. The former are placeholders of the noises $\xi_\mfl$, whereas $\mathbf{n}\in \N_0^d$ represents the $\mathbf{n}$-th monomial. We denote these nodes as
\begin{equation}\label{tree12}
	\{ \Xi_\mfl \}_{\mfl\in\mfL^-\cup\{0\}} \cup \{ X^\mathbf{n}\}_{\mathbf{n}\in \N_0^d}.
\end{equation}
Edges are decorated by $ \N_0^d $; these decorations index the derivatives of the solution kernel, which does not need an additional type because we are dealing with the scalar case. Given a tree $\tau$, we denote by $\mcI_{\mathbf{m}}(\tau)$ the operation of growing an incoming edge decorated by $\mathbf{m}$ at the root. We assume that all inner nodes have a noise decoration $\mfl$, and that decorations in $ \N_0^d $ appear only on some leaves (i.~e. nodes with no outgoing edges). Under these restrictions, a generic tree $\tau\neq X^{\mathbf{n}}$ takes the form
\begin{equation}\label{tree01}
	\tau = \Xi_\mfl \prod_i \mcI_{\mathbf{m}_i} (\tau_i).
\end{equation} 
We call $\trees$ the set of trees generated from \eqref{tree12} via the recursive construction \eqref{tree01}, and $\treespace$ the free vector space spanned by $\trees$. Note that $\treespace$ corresponds to the space introduced in \cite[Subsection 4.1]{BCCH}, where we only have one edge type, and where we in principle allow for polynomial contributions $\mcI_{\mathbf{m}}(X^\mathbf{n})$ for which $\mathbf{m}\not< \mathbf{n}$. 

\medskip

We define a linear map $ \Psi : \treespace \to \R[\coord]$ setting
\begin{equation}\label{tree02}
	\Psi[X^{\mathbf{n}}]  =  \mathbf{n}! \z_{\mathbf{n}}  
\end{equation}
and then inductively for a tree \eqref{tree01}
\begin{equation}\label{tree04}
	\Psi \left[  \tau  \right] = k! \z_{(\mfl, k)} \prod_{i} \tfrac{1}{\mathbf{m}_i!} \Psi \left[ \tau_i \right],\quad k = \sum_i e_{\mathbf{m}_i} \in M(\N_0^d).
\end{equation}
In particular,
\begin{equation}\label{tree14}
	\Psi \left[ \Xi_\mfl \right] = \z_{(\mfl,0)}.
\end{equation}
\begin{lemma}\label{lem:trees}
	$\Psi \left[ \tau \right] = C(\beta) \z^\beta$, where 
	\begin{itemize}
		\item $\beta(\mfl,k)$ is the number of nodes $\Xi_\mfl$ of $\tau$ with outgoing edges given by the multi-index $k$, i.~e. $k(\mathbf{n})$ edges decorated by $\mathbf{n}$,
		\item $\beta(\mathbf{n})$ is the number of leaves $X^{\mathbf{n}}$ of $\tau$,
		\item and $C(\beta)$ $:=$ $\prod_{(\mfl,k)} \left(\frac{k!}{\prod_\mathbf{m} (\mathbf{m}!)^{k(\mathbf{m})}}\right)^{\beta(\mfl,k)}$ $\prod_\mathbf{n} (\mathbf{n}!)^{\beta(\mathbf{n})}$, \referee{which is the combinatorial factor that compensates our normalization of the coefficients \eqref{bb03}.}
	\end{itemize}
\end{lemma}	
Note however that the information of which subtree $\tau_j$ is attached to the edge $\mathbf{n}_j$ is lost when applying $\Psi$. 
\begin{proof}
	This is clear for $\tau = X^{\mathbf{n}}$ by \eqref{tree02}. For a general $\tau$ of the form \eqref{tree01}, it easily follows inductively, noting that the root of $\tau$, which is identified with $\z_{(\mfl,k)}$ in \eqref{tree04}, is a node $\Xi_\mfl$ with outgoing edges given by the multi-index $k$.
\end{proof}	
Under this interpretation, condition $[\beta] = 1$ now becomes transparent: $\length{\beta}$ is the number of nodes, while $\sum_{(\mfl,k)} \length{k} \beta(\mfl,k)$ is the number of edges, and the difference for a rooted tree is always $1$ (all nodes have one incoming edge except the root). Also the subcriticality conditions of Definition \ref{def:subcrit} can be connected to those of \cite{BHZ}: Since a pair $(\mfl,k)$ represents a node $\Xi_\mfl$ with $k$ giving its outgoing edges, fixing the admissible pairs is essentially equivalent to fixing a \textit{rule}, cf. \cite[Definition 5.7]{BHZ}. Condition \cite[(5.10)]{BHZ} for noise types is satisfied by assumption \eqref{sub01}, whereas \eqref{sub02} adapts \cite[(5.10)]{BHZ} for a specific set in the rule, given by $(\mfl,k)$. The fact that we take the minimum over all $k'\leq k$ reflects the normality of the rule, cf. \cite[Definition 5.7]{BHZ}. Finally, subcritical multi-indices correspond to trees that conform to a normal subcritical rule. Closely related to this is the homogeneity \eqref{bb04}, which by the same reasoning is consistent with e.~g. the scaling map on \cite[p.~419]{reg}. As anticipated at the end of Subsection \ref{subsec::2.3}, $|\beta|$ represents the homogeneity of the \textit{rooted} tree, and thus $\Pi_\beta$, which is integrated once, is homogeneous of degree $|\beta| + \eta$.
\begin{remark}
	The map $\Psi$ can be seen as an extension of the transposition of \cite[(7.8)]{LOT}, but with a modification: We do not include the symmetry factor of the tree. This is more convenient in terms of the upcoming pre-Lie morphisms: Our definition is consistent with the duality pairing \cite[(4.2)]{BCCH}, and thus we may formulate pre-Lie morphism properties with respect to the usual grafting product (instead of the normalization of \cite[(7.11)]{LOT}). 
\end{remark}	
\medskip

The generators \eqref{bb16} and \eqref{bb17} can be connected to the algebraic operations on trees defined in \cite{BCCH}. Consider first the family of grafting products $ \curvearrowright^\mathbf{n} : \treespace \times \treespace \to \treespace $, $ \mathbf{n} \in \N_0^d $ introduced in \cite[Def. 4.7]{BCCH}. This product is defined for a tree $\sigma\in \trees$ by
\begin{equation}\label{tree03}
	\sigma \curvearrowright^{\mathbf{n}} X^{\mathbf{n}'} = \sigma \delta_\mathbf{n}^{\mathbf{n}'}
\end{equation}
and then inductively for $\tau$ as in \eqref{tree01} by
\begin{align}
	\sigma \curvearrowright^\mathbf{n} \tau &	= \Xi_{\mathfrak{l}} \mathcal{I}_{\mathbf{n}}( \sigma  )  \prod_i \mcI_{\mathbf{m}_i} (\tau_i)  \nonumber
	\\ & \quad +	\Xi_{\mathfrak{l}}
	\sum_{j} \mathcal{I}_{\mathbf{m}_j}( \sigma \curvearrowright^{\mathbf{n}} \tau_j ) 
	\prod_{i \neq j} \mathcal{I}_{\mathbf{m}_i}( \tau_i ) .\label{tree05}
\end{align}
\begin{proposition}
	For every $\mathbf{n}\in \N_0^{d}$ and every $\tau,\sigma \in \treespace$, it holds
	\begin{equation}\label{tree10}
		\Psi \left[\sigma \curvearrowright^{\mathbf{n}} \tau\right] = \tfrac{1}{\mathbf{n}!} \Psi \left[ \sigma \right] D^{(\mathbf{n})} \Psi \left[ \tau \right].
	\end{equation}
\end{proposition}	
\begin{proof}
	Let $\tau = X^{\mathbf{n}'}$. Then the combination of \eqref{bb16}, \eqref{tree02} and \eqref{tree03} yields
	\begin{align*}
		\Psi \big[ \sigma \curvearrowright^{\mathbf{n}} X^{\mathbf{n}'} \big] = \delta_\mathbf{n}^{\mathbf{n}'} \Psi \big[ \sigma  \big] = \Psi \big[ \sigma  \big] D^{(\mathbf{n})} \z_{\mathbf{n}'} = \tfrac{1}{\mathbf{n}!} \Psi \big[ \sigma  \big] D^{(\mathbf{n})} \Psi \big[ X^{\mathbf{n}'} \big].
	\end{align*}
	For a generic $\tau$ of the form \eqref{tree01}, we combine \eqref{tree04} and \eqref{tree05} to the effect of
	\begin{align*}
		&\Psi \left[ \sigma \curvearrowright^{\mathbf{n}} \tau \right] \\
		&\quad = \Psi \left[ \Xi_\mfl  \mcI_\mathbf{n} (\sigma) \prod_i \mcI_{\mathbf{m}_i} (\tau_i)\right] \\
		&\quad \quad + \sum_j \Psi \left[ \Xi_\mfl  \mcI_{\mathbf{m}_j} (\sigma \curvearrowright^{\mathbf{n}} \tau_j) \prod_{i\neq j}  \mcI_{\mathbf{m}_i} (\tau_i) \right]\\
		&\quad = (k + e_\mathbf{n})! \z_{(\mfl,k+e_\mathbf{n})} \tfrac{1}{\mathbf{n}!} \Psi \left[ \sigma \right] \prod_i \tfrac{1}{\mathbf{m}_i!} \Psi \left[ \tau_i \right] \\
		& \quad \quad + k! \sum_j \z_{(\mfl,k)} \tfrac{1}{\mathbf{m}_j} \tfrac{1}{\mathbf{n}!}\big(\Psi \left[ \sigma \right] D^{(\mathbf{n})} \Psi \left[ \tau_j \right] \big) \prod_{i\neq j} \tfrac{1}{\mathbf{m}_i} \Psi \left[ \tau_i \right]. 
	\end{align*}
	Rearranging the factors, and using \eqref{bb16} in the form $D^{(\mathbf{n})}\z_{(\mfl,k)}$ $=$ $(k(\mathbf{n}) + 1) \z_{(\mfl,k+e_\mathbf{n})}$, this expression equals
	\begin{align*}
		&\tfrac{1}{\mathbf{n}!} (\Psi \left[ \sigma \right] D^{(\mathbf{n})} k! \z_{(\mfl,k)}) \prod_i \tfrac{1}{\mathbf{m}_i!} \Psi \left[ \tau_i \right]  \\
		&\quad + \tfrac{1}{\mathbf{n}!}  \sum_j k! \z_{(\mfl,k)} \tfrac{1}{\mathbf{m}_j} \tfrac{1}{\mathbf{n}!}\big(\Psi \left[ \sigma \right] D^{(\mathbf{n})} \Psi \left[ \tau_j \right] \big) \prod_{i\neq j} \tfrac{1}{\mathbf{m}_i} \Psi \left[ \tau_i \right],
	\end{align*}
	which by the Leibniz rule reduces to
	\begin{equation*}
		\tfrac{1}{\mathbf{n}!}\Psi \left[ \sigma \right] D^{(\mathbf{n})} \big( k! \z_{(\mfl,k)} \prod_i \tfrac{1}{\mathbf{m}_i!} \Psi \left[ \tau_i \right]\big).
	\end{equation*}
	Now \eqref{tree04} concludes the proof.
\end{proof}	
The graftings $\curvearrowright^\mathbf{n}$, as the derivations $D^{(\mathbf{n})}$, generate a multi pre-Lie algebra, which in this case is free, cf. \cite[Proposition 4.21]{BCCH}. As a consequence of \eqref{tree10}, $\Psi : \treespace \to \R[\coord]$ is a multi pre-Lie morphism; by \cite[Proposition 4.21]{BCCH}, it is the unique extension of the map defined by \eqref{tree02} and \eqref{tree14}.

\medskip

We now consider $ \uparrow^{i} : \treespace \rightarrow \treespace^* $ as defined in \cite[Definition 4.7]{BCCH}. It is constructed by
\begin{equation}\label{tree11}
	\uparrow^i X^\mathbf{n} = X^{\mathbf{n} + e_i}
\end{equation}
and recursively for a tree $\tau$ of the form \eqref{tree01}
\begin{align}
	\uparrow^i \tau &= \sum_{\mathbf{n}\in\N_0^d} \Xi_\mfl \mcI_{\mathbf{n}} (X^{\mathbf{n} + e_i}) \prod_{j} \mcI_{\mathbf{m}_j} (\tau_j) \nonumber\\
	& \quad + \sum_j \Xi_\mfl \mcI_{\mathbf{m}_j} (\uparrow^i \tau_j) \prod_{j'\neq j} \mcI_{\mathbf{m}_{j'}} (\tau_j') .\label{tree08}
\end{align}
As with $\mbpartial_i$, subcriticality makes the sum over $\mathbf{n}$ finite.
\begin{proposition}
	For every $i=1,...,d$ and every $\tau\in\treespace$,
	\begin{equation}\label{tree09}
		\Psi \left[ \uparrow^i  \tau \right] = \mbpartial_i \Psi \left[ \tau \right].
	\end{equation}
\end{proposition}	
\begin{proof}
	By \eqref{tree03} and \eqref{tree05}, we may rewrite \eqref{tree11} and \eqref{tree08} as
	\begin{equation*}
		\uparrow^i \tau = \sum_{\mathbf{n}\in \N_0^d} X^{\mathbf{n} + e_i} \curvearrowright^\mathbf{n} \tau.
	\end{equation*}
	Then \eqref{tree09} follows from \eqref{tree10} in the form
	\begin{equation*}
		\Psi [X^{\mathbf{n} + e_i} \curvearrowright^\mathbf{n} \tau] = (\mathbf{n}(i) + 1) \z_{\mathbf{n} + e_i} D^{(\mathbf{n})} \Psi[\tau].
	\end{equation*}
\end{proof}	
\begin{remark} The map $ \Psi $ can be extended to a post-Lie morphism in the sense of \cite[Remark 5.6]{BK}. 
\end{remark}
%
%
%%%%%%%%%%%
%
\section{Algebraic renormalization of multi-indices}
\label{section::3}
In this section, starting from the building blocks of Section \ref{section::2}, we follow the algebraic route of \cite{LOT} to construct actions by shift. Our application goes beyond the construction of the structure group: We use this technique to also write a new formulation of the model equations which is suitable for implementing algebraic renormalization procedures.
\subsection{\referee{The Lie algebra of generators and its universal enveloping algebra}}\label{subsec::3.1}
Let us denote $\mcM$ $:=$ $\mcN\times \N_0^{d}$. We begin by studying some properties of the generators
\begin{equation}\label{mod51}
	\mcD:= \{\z^\gamma D^{(\mathbf{n})}\}_{(\gamma,\mathbf{n})\in\mcM}\sqcup \{\mbpartial_i\}_{i=1,...,d}
\end{equation}
as linear endomorphisms of $\R[[\coord]]$ and, more precisely, as derivations in that space. We first show their mapping properties in the sets of multi-indices of Definition \ref{def:pop}.
\begin{lemma}\label{lem:218}
	\mbox{}
	\begin{itemize}
		\item For all $(\gamma',\mathbf{n}')\in \mcM$,
		\begin{align}
			\z^{\gamma'}D^{(\mathbf{n}')} T_\mcNmin^* &\subset T_\mcN^*,\label{map01}\\
			\z^{\gamma'}D^{(\mathbf{n}')} T_\populated^* &\subset T_\mcN^*.\label{map02}
		\end{align} 
		\item For all $i=1,...,d$, 
		\begin{align}
			\mbpartial_i T_\mcNmin^* &\subset T_\mcNmin^*,\label{map03}\\
			\mbpartial_i T_\mcP^* &\subset T_\mcP^*,\label{map04}\\
			\mbpartial_i T_\mcN^* &\subset T_\mcN^*.\label{map05}
		\end{align}	
	\end{itemize}
\end{lemma}	
\begin{proof}
	We start with \eqref{map01} and \eqref{map02}, which reduces to showing that for all $\gamma\in \mcP\cup\mcNmin$
	\begin{equation}\label{pro01}
		(\z^{\gamma'}D^{(\mathbf{n}')})_\beta^\gamma\neq 0 \implies \beta\in\mcN.
	\end{equation}
	From \eqref{bb16bis2} and \eqref{sub13} we have the following representation of the coefficients:
	\begin{align}\label{ind04}
		(\z^{\gamma'}D^{(\mathbf{n}')})_\beta^\gamma = \hspace*{-20pt}\sum_{\substack{(\mfl,k) \\ (\mfl,k+e_{\mathbf{n}'})\,\text{subcritical}}}\hspace*{-20pt} \gamma(\mfl,k)(k(\mathbf{n}')+1)\delta_\beta^{\gamma-e_{(\mfl,k)}+ e_{(\mfl,k+ e_{\mathbf{n}'})}+\gamma'} + \gamma(\mathbf{n}')\delta_\beta^{\gamma-e_{\mathbf{n}'}+\gamma'}.
	\end{align}  
	We distinguish the cases $\gamma\in\mcP$ and $\gamma\in\mcNmin$. In the former, we note that for $\gamma = e_\mathbf{n}$
	\begin{equation*}
		(\z^{\gamma'}D^{(\mathbf{n}')})_\beta^{e_\mathbf{n}} = \delta_\mathbf{n}^{\mathbf{n}'} \delta_\beta^{\gamma'}
	\end{equation*}
	and thus \eqref{pro01} holds. For $\gamma\in\mcNmin$, we first note by \eqref{bb22} that $1 = [\gamma] = [\beta]$; from \eqref{ind04} we learn
	\begin{equation}\label{map06}
		(\z^{\gamma'}D^{(\mathbf{n}')})_\beta^\gamma\neq 0 \implies \lnh\beta\rnh = \lnh\gamma'\rnh + \lnh\gamma\rnh \geq \lnh\gamma'\rnh>0;
	\end{equation} 
	and finally if $\gamma(\mathbf{n}) = \gamma'(\mathbf{n}) = 0$ then $\beta(\mathbf{n})=0$, thus proving that \eqref{bb05} is preserved.
	
	\medskip
	
	We now show \eqref{map03}, \eqref{map04} and \eqref{map05}. Property \eqref{map04} is trivial from definition \eqref{bb17} (the subcritical projection does not play any role), so we focus on the other two. We first prove
	\begin{equation*}
		(\mbpartial_i)_\beta^{\gamma\in\mcNmin}\neq 0 \implies \beta\in\mcNmin.
	\end{equation*}	 
	Let us fix $\mathbf{n}\in\N_0^{d}$ in \eqref{bb17}. The problem reduces to
	\begin{equation*}
		(\z_{\mathbf{n}+e_i}D^{(\mathbf{n})})_\beta^{\gamma\in\mcNmin}\neq 0 \implies \beta\in\mcNmin.
	\end{equation*}
	Again by \eqref{bb22} it holds $1=[\gamma]=[\beta]$; \eqref{bb05} is trivially preserved (polynomial contributions only increase); and finally the subcriticality condition is guaranteed.	If we now want to obtain
	\begin{equation*}
		(\z_{\mathbf{n}+e_i}D^{(\mathbf{n})})_\beta^{\gamma\in\mcN}\neq 0 \implies \beta\in\mcN,
	\end{equation*}
	we furthermore need to show $\lnh\gamma\rnh >0 \implies \lnh\beta\rnh>0$. This follows in the same way as \eqref{map06}, but with $\lnh\gamma'\rnh = 0$ and $\lnh\gamma\rnh>0$.
\end{proof}	
As anticipated in Subsection \ref{subsec::2.3}, since $D^{(\mathbf{n})}$ is a derivation in $\R[[\coord]]$, it generates a natural pre-Lie algebra with pre-Lie product
\begin{align}
	\z^{\gamma'}D^{(\mathbf{n}')} \triangleright \z^{\gamma}D^{(\mathbf{n})} = \sum_\beta (\z^{\gamma'}D^{(\mathbf{n}')})_\beta^\gamma \z^{\beta}D^{(\mathbf{n})}.\label{sg11}
\end{align} 
We can extend this structure using definition \eqref{bb17}:
\begin{align}
	\z^{\gamma'}D^{(\mathbf{n}')} \triangleright \mbpartial_i &= \mathbf{n}'(i)\z^{\gamma'} D^{(\mathbf{n}' - e_i)},\label{sg12}\\
	\mbpartial_i\triangleright \z^{\gamma}D^{(\mathbf{n})} &= \sum_\beta (\mbpartial_i)_\beta^\gamma \z^{\beta}D^{(\mathbf{n})}.\label{sg13}
\end{align}
However, as noted in \cite{LOT}, terms of the form $\mbpartial_i\triangleright\mbpartial_j \notin \lspan \mcD$, and thus there is no closed pre-Lie structure of $\mcD$. This turns out to be of little effect, since $\mbpartial_i$ and $\mbpartial_j$ commute as endomorphisms, so we can define a Lie bracket as
\begin{align}
	[D,\z^\beta D^{(\mathbf{n})}] &= D\triangleright\z^\beta D^{(\mathbf{n})} - \z^\beta D^{(\mathbf{n})}\triangleright D,\,D\in \mcD,\label{sg14}\\
	[\mbpartial_i,\mbpartial_j] &= 0.\label{sg15}
\end{align}
The mapping properties \eqref{map02} and \eqref{map05} allow us to consider the following.
\begin{definition}
	We denote by $L$ the Lie algebra generated by $\mcD$ with the Lie bracket \eqref{sg14}, \eqref{sg15}. Furthermore, we denote by $\tL$ the Lie sub-algebra generated by $\{\z^\gamma D^{(\mathbf{n})}\}_{(\gamma,\mathbf{n})\in\mcM}$.
\end{definition}
\begin{remark}
	Note that \eqref{sg11} implies that $(\tL,\triangleright)$ is a pre-Lie algebra.
\end{remark}		

\medskip

We will now see that the pre-Lie product \eqref{sg11}, \eqref{sg12} and \eqref{sg13}, and as a consequence the Lie bracket \eqref{sg14} and \eqref{sg15}, preserves a natural grading. Let us define
\begin{align}
	|(\gamma,\mathbf{n})| &:= |\gamma| + \eta - |\mathbf{n}|,\label{sg05bis}\\
	|i| &:= |e_i|;\label{sg06bis}
\end{align}
here $|e_i| = \mathfrak{s}_i$, cf. \eqref{scal01}.
\begin{lemma}\label{lem:big}
	It holds:
	\begin{align}
		(\z^{\gamma'}D^{(\mathbf{n}')})_\beta^\gamma \neq 0 &\implies |\beta| = |\gamma| + |\gamma'| + \eta - |\mathbf{n}'|,\label{sg09}\\
		(\mbpartial_i)_\beta^\gamma \neq 0 &\implies |\beta| = |\gamma| + |e_i|.\label{sg10} 
	\end{align}
\end{lemma}
\begin{proof}
	We start with \eqref{sg09}. Recall \eqref{pol05}; from the first item we have for some $(\mfl,k)$
	\begin{align*}
		|\beta| &= |\gamma| - \alpha_\mfl - \sum_{\mathbf{n}} (\eta - |\mathbf{n}|)k(\mathbf{n}) + \alpha_\mfl +  \sum_{\mathbf{n}} (\eta - |\mathbf{n}|)(k + e_{\mathbf{n}'})(\mathbf{n}) + |\gamma'| \\
		&= |\gamma| + \eta - |\mathbf{n}'| + |\gamma'|.
	\end{align*}
	From the second item of \eqref{pol05},
	\begin{equation*}
		|\beta| = |\gamma| - (|\mathbf{n}'| - \eta) + |\gamma'|.
	\end{equation*}
	For \eqref{sg10}, we perform the same argument taking $\gamma' = e_{\mathbf{n}' + e_i}$ in \eqref{sg09}.
\end{proof}	
\begin{corollary}\label{ref:corgrad}
	$L$ is graded with respect to $|\cdot|$ \referee{defined in \eqref{sg05bis}, \eqref{sg06bis}}.
\end{corollary}	
\referee{\begin{proof}Recall that a Lie algebra $\mfg$ is graded if there exists a set $\mfA\subset \R$ such that $\mfg = \bigoplus_{\kappa\in \mfA} \mfg_\kappa$ and $[\mfg_\kappa,\mfg_{\kappa'}]$ $\subset$ $\mfg_{\kappa + \kappa'}$. In the case of $L$, we may decompose it in terms of the basis \eqref{mod51} according to \eqref{sg05bis} and \eqref{sg06bis}:
\begin{equation*}
	\z^\gamma D^{(\mathbf{n})} \in L_{|(\gamma,\mathbf{n})|},\,\,\,\, \mbpartial_i\in L_{|i|}.
\end{equation*}
The claim follows from \eqref{sg11}, \eqref{sg12}, \eqref{sg13}, \eqref{sg14}, \eqref{sg15} and the grading properties \eqref{sg09} and \eqref{sg10}.
\end{proof}
 }

\medskip

As in \cite{LOT}, from now on we shall see the set $\mathcal{D}$ as a set of symbols, preserving the pre-Lie and Lie algebra relations via the structure constants (as given in \eqref{sg11} to \eqref{sg15}). Each symbol is then identified with the corresponding endomorphism via a representation $\rho: \mathcal{D}\to \textnormal{End}(\R[[\coord]])$, which is trivially a Lie algebra morphism that additionally preserves the pre-Lie relations \eqref{sg11} to \eqref{sg13}. This representation gives the structure of an action
	\begin{equation*}
		L\otimes \R[[\coord]] \ni D\otimes \pi \longmapsto \rho(D)\pi \in \R[[\coord]].
	\end{equation*}

\medskip

Our next step is to consider the universal envelope of $L$, which we denote by ${\rm U}(L)$. \referee{We first recall the following definition.
	\begin{definition}
		Let $\mfg$ be a Lie algebra. The \textit{universal enveloping algebra} of $\mfg$, denoted ${\rm U}(\mfg)$, is $T(\mfg)/\mcC$, where
		\begin{equation*}
			T(\mfg) := \bigoplus_{n=0}^\infty \mfg^{\otimes n}
		\end{equation*} 
		is the tensor algebra over the vector space $\mfg$, and $\mcC$ is the ideal generated by 
		\begin{equation*}
			\{x\otimes y - y\otimes x - [x,y]\,|\,x,y\in \mfg\}.
		\end{equation*}
	\end{definition}
In other words, the universal enveloping algebra is the free algebra modulo the commutator generated by the Lie bracket. The universal enveloping algebra posseses the following universality property (cf. e.~g. \cite[(U), p.~29]{Abe80}): Given an algebra $A$ and a Lie algebra morphism $\phi : \mfg \to A$, where $A$ is endowed with the commutator of the product as a bracket, there exists a unique extension $\phi: {\rm U}(\mfg) \to A$ which is an algebra morphism. As a consequence, the map $\rho$ defined above extends to an algebra morphism $\rho : {\rm U}(L) \to \textnormal{End}(\R[[\coord]])$ with respect to composition of endomorphisms.
}

\medskip

\referee{The universal enveloping algebra of a Lie algebra carries the structure of a Hopf algebra: For completeness, we recall its definition here.
\begin{definition}
	\mbox{}
	\begin{itemize}
		\item $(A,\nabla,\eta)$ is a \textit{unital associative algebra} if $A$ is a vector space, $\nabla: A\otimes A \to A$ is a bilinear map which is associative, i.~e.
		\begin{equation*}
			\nabla \circ (\Id_A \otimes \nabla) = \nabla \circ (\nabla \otimes \Id_A),
		\end{equation*}
		and $\eta: \R \to A$ is a linear map such that
		\begin{equation*}
			\nabla \circ (\Id_A \otimes \eta) = \Id_A = \nabla \circ (\eta \otimes \Id_A).
		\end{equation*}
		Here $\nabla$ is called \textit{product} and $\eta$ is called \textit{unit}.
		\item $(C,\Delta,\varepsilon)$ is a \textit{counital coassociative coalgebra} if $C$ is a vector space, $\Delta: C \to C \otimes C$ is a linear map which is coassociative, i.~e. 
		\begin{equation*}
			(\Id_C \otimes \Delta)\circ \Delta = (\Delta \otimes \Id_C)\circ \Delta,
		\end{equation*}
		and $\varepsilon : C \to \R$ is a linear map such that
		\begin{equation*}
			(\Id_C \otimes \varepsilon) \circ \Delta = \Id_C = (\varepsilon \otimes \Id_C) \circ \Delta.
		\end{equation*}
		Here $\Delta$ is called \textit{coproduct} and $\varepsilon$ is called \textit{counit}.
		\item $(B,\nabla,\eta,\Delta,\varepsilon)$ is a \textit{bialgebra} if $(B,\nabla,\eta)$ is a unital associative algebra, $(B,\Delta,\varepsilon)$ is a counital coassociative coalgebra, and one of the following equivalent conditions holds:
		\begin{itemize}
			\item $\nabla$ and $\eta$ are coalgebra morphisms;
			\item $\Delta$ and $\varepsilon$ are algebra morphisms.
		\end{itemize}
		\item $(H,\nabla,\eta,\Delta,\varepsilon, S)$ is a \textit{Hopf algebra} if $(H,\nabla,\eta,\Delta,\varepsilon)$ is a bialgebra and $S:H \to H$ is a linear map which satisfies
		\begin{equation*}
			\nabla \circ (\Id_H \otimes S) \circ \Delta = \Id_H = \nabla \circ (S \otimes \Id_H) \circ \Delta.
		\end{equation*}
		We call $S$ \textit{antipode}.
	\end{itemize}
\end{definition}	
Given a Lie algebra $\mfg$, its universal enveloping algebra ${\rm U}(\mfg)$ is a Hopf algebra. The product is defined as the concatenation of tensors up to the Lie bracket: In particular,
	\begin{equation*}
		\nabla (x \otimes y) = x \otimes y = y \otimes x + [x,y],\,\,\,\,x,y\in \mfg.
	\end{equation*}
	The unit $\eta: \R \to {\rm U}(\mfg)$ is given by
	\begin{equation*}
		\eta (k) = k \in \R = \mfg^{\otimes 0}.
	\end{equation*}
	For the coproduct $\Delta$, we take the map
	\begin{equation*}
		\begin{array}{rrcl}
			\Delta :  & \mfg & \longrightarrow & {\rm U}(\mfg)\\
			 & x & \longmapsto & 1 \otimes x + x \otimes 1
		\end{array}
	\end{equation*}
	and extend it to an algebra morphism. The counit, in turn, is extended from the projection $\varepsilon: \mfg \to 0$. Finally, the antipode $S: {\rm U}(\mfg) \to {\rm U}(\mfg)$ is the antiautomorphism extended from
	\begin{equation*}
		\begin{array}{rrcl}
			S :  & \mfg & \longrightarrow & \mfg\\
			& x & \longmapsto & -x.
		\end{array}
	\end{equation*}
	See \cite[Examples 2.5 and 2.8]{Abe80} for more details.

\medskip

In the sequel, we will represent the product in ${\rm U}(L)$ with a dot $\bullet$, but we will not spell it out in the notation and instead will represent it as the concatenation of symbols (without the tensor product $\otimes$): For example
\begin{align*}
	\bullet \big(\z^{\gamma} D^{(\mathbf{n})} \otimes \mbpartial_i\big) &= \z^\gamma D^{(\mathbf{n})} \mbpartial_i \\
	&= \mbpartial_i \z^\gamma D^{(\mathbf{n})} + \mathbf{n}(i) \z^{\gamma} D^{(\mathbf{n}-e_i)} - \sum_\beta (\mbpartial_i)_\beta^\gamma \z^\beta D^{(\mathbf{n})}.
	\end{align*}
The coproduct will be denoted by $\cop$: For example
\begin{align*}
	\cop \big(\z^{\gamma} D^{(\mathbf{n})} \mbpartial_i\big) = &1 \otimes \z^{\gamma} D^{(\mathbf{n})} \mbpartial_i + \z^{\gamma} D^{(\mathbf{n})} \otimes \mbpartial_i \\
	&\quad + \mbpartial_i \otimes \z^{\gamma} D^{(\mathbf{n})} + \z^{\gamma} D^{(\mathbf{n})} \mbpartial_i \otimes 1.
\end{align*}
We will not introduce any notation for the unit, counit and antipode, as we will make little use of them.
}

\medskip	
	
Since $L$ is graded with respect to $|\cdot|$ in \eqref{sg05bis}, \eqref{sg06bis}, \referee{cf. Corollary \ref{ref:corgrad},} so is ${\rm U}(L)$, i.~e. there exists a decomposition
\begin{equation*}
	{\rm U}(L) = \bigoplus_{\nu} {\rm U}_\mathbf{\nu},
\end{equation*}
where $\nu\in \R$, such that
\begin{alignat*}{2}
	\bullet &: {\rm U}_{\nu'} \otimes {\rm U}_{\nu''} &&\longrightarrow {\rm U}_{\nu'+\nu''},\\
	\cop &: {\rm U}_{\nu} &&\longrightarrow \bigoplus_{\nu' + \nu'' = \nu} {\rm U}_{\nu'}\otimes {\rm U}_{\nu''}.
\end{alignat*}
However, ${\rm U}(L)$ is in general not $|\cdot|$-connected, \referee{i.~e. it is not true that ${\rm U}_{\nu = 0} = \R$.}
\subsection{\referee{The Guin-Oudom procedure}}\label{ref:subsecgo}
\referee{Guin and Oudom \cite{Guin1,Guin2} showed that, given a pre-Lie algebra $(\mfg,\triangleright)$, the universal envelope ${\rm U}(\mfg)$ of its associated Lie algebra is isomorphic as a Hopf algebra to the symmetric algebra $S(\mfg)$ over $\mfg$ equipped with a noncommutative product which is built using the pre-Lie product; see \cite[Theorem 2.12]{Guin2}.}
\referee{If $\mfg$ as a vector space has a countable basis $\{x_i\}$, a weaker version of the Guin-Oudom result, which is enough for our purposes, can be reformulated as follows: There exists an order-independent basis of ${\rm U}(\mfg)$. Recall that, by the Poincaré-Birkhoff-Witt theorem (cf. \cite[Theorem 1.9.6]{HGK}), the ordered concatenation of elements of the basis $\{x_i\}$ generates a basis of ${\rm U}(\mfg)$, but it crucially depends on a fixed ordering of $\{x_i\}$. Via the Guin-Oudom procedure, we can build a basis of ${\rm U}(\mfg)$ such that each basis element is associated to an \textit{unordered} collection of basis elements\footnote{ The construction of the basis relies on Poincaré-Birkhoff-Witt, as will become apparent in Lemma \ref{ref:lembasul} below. This is why our version is weaker than the original \cite{Guin1,Guin2}, and explains why the countable basis assumption, while not strictly necessary, is useful.} $\{x_i\}$; the pre-Lie product plays a crucial role in the construction of this basis. In our situation, even though we do not have a pre-Lie algebra (recall that $\mbpartial_i\triangleright \mbpartial_j$ is not well-defined), the Guin-Oudom procedure applies and allows to choose an order-independent basis. This was originally shown in \cite{LOT} and is the construction which we reproduce below; we will often refer to \cite[Section 4]{LOT} for some omitted details\footnote{See also \cite{Lin23} for the simpler pre-Lie algebra case, with applications to the theory of rough paths.}, but it is conceptually self-contained and does not require any knowledge of Guin-Oudom. Alternatively, we could use the approach of \cite{BK,JZ23} via post-Lie algebras, which is based on a more general version of the Guin-Oudom procedure.} 
	
\medskip	

\referee{ \begin{remark}
		Experts in the field should note that our method is only ``half'' of the Guin-Oudom procedure, because the construction of the order-independent basis only requires a one-sided operation, namely the analogue of the extension described in \cite[Proposition 2.7]{Guin2}. Then we may build an isomorphism with the symmetric algebra identifying the \textit{Guin-Oudom basis} \eqref{sg25} with the canonical basis of $S(\mfg)$, and the concatenation product (up to bracket corrections) in ${\rm U}(\mfg)$ with the product of $S(\mfg)$. The missing step in the construction of the product of $S(\mfg)$ is the analogue of \cite[Definition 2.9]{Guin2}, namely combining it with the coproduct: This would take the form of an extension of property \eqref{sg23} below. 
		\end{remark}}

\medskip

The guiding principle of the construction in \cite{LOT} is the following observation: Since $\tilde{L}\subset L$ is an ideal\footnote{ Note that $[L,L]\subset \tilde{L}$, cf. \eqref{sg11} to \eqref{sg15}.}, the quotient Lie algebra $L/\tilde{L}$ is Abelian (see \cite[Lemma 1.2.5]{HGK}) and thus isomorphic to $(\{\mbpartial_i\}_{i=1,...,d},[\cdot,\cdot])$. The Lie algebra morphism $L\to L/\tilde{L} \simeq \{\mbpartial_i\}_{i=1,...,d}$, by the universality property, extends to an algebra morphism ${\rm U}(L) \to \{\mbpartial^\mathbf{m}\}_{\mathbf{m}\in \N_0^{d}}$, which in turn induces a decomposition
\begin{equation*}
	{\rm U}(L) = \bigoplus_{\mathbf{m}\in \N_0^{d}} {\rm U}_\mathbf{m}.
\end{equation*}				
Then each of the subspaces ${\rm U}_\mathbf{m}$ is isomorphic to ${\rm U}(\tilde{L})$, to which Guin-Oudom can be applied; more precisely, the pre-Lie product $\triangleright$, or rather an extension of it, provides a natural isomorphism.

\medskip

Let us note first that, by the Leibniz rule, given $\z^\gamma D^{(\mathbf{n})}$ for all $D\in\mathcal{D}$ (here seen again as endomorphisms) it holds
\begin{equation}\label{uni01}
	\z^\gamma D D^{(\mathbf{n})} = D \z^\gamma D^{(\mathbf{n})} - D\triangleright \z^\gamma D^{(\mathbf{n})},
\end{equation}
so it makes sense to define a map, for every $\tilde{D}\in\tilde{L}$,
\begin{equation*}
	L\ni D \mapsto D\tilde{D} - D\triangleright \tilde{D} \in {\rm U}(L).
\end{equation*}
Taking $\tilde{D} = \z^\gamma D^{(\mathbf{n})}$, this map extends to 
\begin{equation*}
	{\rm U}(L) \ni U\longmapsto \z^\gamma U D^{(\mathbf{n})} \in {\rm U}(L)
\end{equation*}
by the trivial $\z^\gamma 1 D^{(\mathbf{n})} = \z^\gamma D^{(\mathbf{n})}$ and then inductively, for every $D\in L$ and $U\in{\rm U}(L)$, via
\begin{equation}\label{sg20}
	\z^\gamma DU D^{(\mathbf{n})} = D\z^\gamma U D^{(\mathbf{n})} - \sum_\beta D_\beta^\gamma \z^\beta U D^{(\mathbf{n})},
\end{equation}
which is again consistent with the Leibniz rule. 
\begin{lemma}
	The following properties hold:
	\begin{enumerate}[label=(\roman*)]
		\item\label{enumprop1} Commutativity: For every $(\gamma,\mathbf{n})$, $(\gamma',\mathbf{n}')$ $\in \mcM$, and every $U\in{\rm U}(L)$,
		\begin{equation}\label{sg21}
			\z^\gamma\z^{\gamma'}U D^{(\mathbf{n})} D^{(\mathbf{n}')} = \z^{\gamma'}\z^{\gamma}U D^{(\mathbf{n}')} D^{(\mathbf{n})};
		\end{equation}
		\item\label{enumprop2} Coalgebra morphism: For every $(\gamma,\mathbf{n})\in\mcM$ and $U\in{\rm U}(L)$,
		\begin{equation*}
			\cop \z^\gamma U D^{(\mathbf{n})} = \sum_{(U)} \big(\z^\gamma U_{(1)} D^{(\mathbf{n})}\otimes U_{(2)} + U_{(1)}\otimes \z^\gamma U_{(2)} D^{(\mathbf{n})}\big),
		\end{equation*}
		where we used Sweedler's notation
		\begin{equation*}
			\cop U =  \sum_{(U)} U_{(1)}\otimes U_{(2)};
		\end{equation*}
		\item\label{enumprop3} Generalized Leibniz rule: For every $(\gamma,\mathbf{n})\in\mcM$ and $U\in{\rm U}(L)$,
		\begin{equation}\label{sg23}
			U\z^\gamma D^{(\mathbf{n})} = \sum_{(U),\beta} (U_{(1)})_\beta^\gamma \z^\beta U_{(2)} D^{(\mathbf{n})};
		\end{equation}
		\item\label{enumprop4} Intertwining with $\mbpartial_i$: For every $(\gamma,\mathbf{n})\in\mcM$, $U\in{\rm U}(L)$ and $i=1,...,d$,
		\begin{equation}\label{sg24}
			\z^\gamma U D^{(\mathbf{n})}\mbpartial_i = \z^\gamma U \mbpartial_i D^{(\mathbf{n})} + \mathbf{n}(i) \z^\gamma U D^{(\mathbf{n}- e_i)}.
		\end{equation}
	\end{enumerate}
\end{lemma}
\begin{proof}
	All these properties follow from \eqref{sg20} by induction in $U$, as in \cite[Subsection 4.2]{LOT}. More precisely, \ref{enumprop1} follows in the same way as \cite[Lemma 4.1]{LOT}; \ref{enumprop2} as \cite[Lemma 4.2]{LOT}; \ref{enumprop3} as \cite[Lemma 4.3]{LOT}; and \ref{enumprop4} as \cite[Lemma 4.4]{LOT}.
\end{proof}	

\medskip

An iterative application of these maps to an element $\mbpartial^\mathbf{m}:=$ $\prod_{i} \mbpartial_i^{\mathbf{m}(i)}$, $\mathbf{m}\in \N_0^{d}$ (in any order, thanks to \eqref{sg21}) allows us to define
\begin{equation}\label{sg25}
	\msfD_{(J,\mathbf{m})} := \frac{1}{J!\mathbf{m}!} \prod_{(\gamma,\mathbf{n})\in\mcM} \hspace*{-10pt}(\z^\gamma)^{J(\gamma,\mathbf{n})} \hspace*{3pt}\mbpartial^\mathbf{m}  \hspace*{-8pt}\prod_{(\gamma,\mathbf{n})\in\mcM} \hspace*{-8pt}(D^{(\mathbf{n})})^{J(\gamma,\mathbf{n})},
\end{equation}			
where $J\in M(\mcM)$.
\begin{lemma}\label{ref:lembasul}
	\mbox{}
	\begin{enumerate}[label=(\roman*)]
		\item\label{enumbas1} The set $\{\msfD_{(J,\mathbf{m})}\}_{(J,\referee{\mathbf{m}})\in M(\mcM)\times \N_0^{d}}$ is a basis of ${\rm U}(L)$.
		\item\label{enumbas2} Representation of the coproduct: 
		\begin{equation}\label{sg26}
			\cop \msfD_{(J,\mathbf{m})} = \sum_{(J',\mathbf{m}') + (J'',\mathbf{m}'') = (J,\mathbf{m})} \hspace*{-15pt}\msfD_{(J',\mathbf{m}')}\otimes \msfD_{(J'',\mathbf{m}'')}.
		\end{equation}
	\item\label{enumbas3} Rank one projection of the product: For any $\mathbf{n}\in\N_0^{d}$, let $\iota_\mathbf{n}:\rmU(L)\to T_\mcN^*$ be the projection defined in coordinates as
	\begin{equation*}
		\iota_\mathbf{n} (\msfD_{(J,\mathbf{m})}) = \left\{\begin{array}{cl}
			\z^\gamma & \mbox{if } (J,\mathbf{m}) = (e_{(\gamma,\mathbf{n})},0),\\
			0 & \mbox{otherwise, }
		\end{array}\right.
	\end{equation*}
	and let $\epsilon_\mathbf{n}:\rmU(L)\to\R$ be the generalized counit defined in coordinates as
	\begin{equation*}
		\epsilon_\mathbf{n} (\msfD_{(J,\mathbf{m})}) = \left\{\begin{array}{cl}
			1& \mbox{if } (J,\mathbf{m}) = (0,\mathbf{n}),\\
			0 & \mbox{otherwise. }
		\end{array}\right.
	\end{equation*}
	Then \referee{for any $U_1,U_2\in \rmU(L)$}
	\begin{equation}
		\iota_\mathbf{n}(U_1 U_2) = \rho(U_1)\iota_\mathbf{n} (U_2) + \sum_\mathbf{m} \tbinom{\mathbf{n}+\mathbf{m}}{\mathbf{m}}\iota_{\mathbf{n}+\mathbf{m}}(U_1)\epsilon_\mathbf{m}(U_2)\referee{.}\label{sg35}
	\end{equation}
	\item\label{enumbas4} Homogeneity: $\msfD_{(J,\mathbf{m})}\in {\rm U}_{|(J,\mathbf{m})|}$, where extending \eqref{sg05bis} and \eqref{sg06bis} we define
	\begin{equation}\label{sg30}
		|(J,\mathbf{m})| := \sum_{(\gamma,\mathbf{n})\in\mcM} J(\gamma,\mathbf{n}) |(\gamma,\mathbf{n})| + \sum_i \mathbf{m}(i) |i|.
	\end{equation}
	\end{enumerate}
\end{lemma}	
\begin{proof}
	\ref{enumbas1} follows from Poincaré-Birkhoff-Witt (cf. e.~g. \cite[Theorem 1.9.6]{HGK}) and the triangular structure (in the length of the basis elements) of the ``change of basis" generated by \eqref{sg20}; cf. \cite[Lemma 4.5]{LOT} for details. \ref{enumbas2} and \ref{enumbas3} can be shown inductively in $\length{J,\mathbf{m}}$ via \eqref{sg20}, \eqref{sg23} and \eqref{sg24}, cf. \cite[Lemma 4.6]{LOT} and \cite[Lemma 4.7]{LOT}, respectively. Finally, \ref{enumbas4} is a consequence of Lemma \ref{lem:big} propagated inductively via \eqref{sg20}, cf. \cite[Lemma 4.8]{LOT}. 
\end{proof}	

\medskip

Property \eqref{sg26} implies that, as a coalgebra, $\rmU(L)$ is the transposition of the free commutative algebra $\R[\mcM \sqcup \{1,...,d\}]$. More precisely, given the generators 
	\begin{equation}\label{mod08}
		\{\msfZ_{(\gamma,\mathbf{n})}\}_{(\gamma,\mathbf{n})\in \mcM} \sqcup \{\msfZ_i\}_{i=1,...,d},
	\end{equation}
	\referee{which give rise to the monomials
	\begin{equation*}
		\msfZ^{(J,\mathbf{m})} = \prod_{(\gamma,\mathbf{n})\in \mcM} \msfZ_{(\gamma,\mathbf{n})}^{J(\gamma,\mathbf{n})} \,\,\prod_{i=1}^d \msfZ_i^{\mathbf{m}(i)},
	\end{equation*}	
	we equip $(\R[\mcM\cup {1,...,d}], \rmU(L))$ with} the canonical duality pairing \referee{given by}
	\begin{equation*}
		\langle \msfD_{(J,\mathbf{m})}, \msfZ^{(J',\mathbf{m}')} \rangle = \delta_{(J,\mathbf{m})}^{(J',\mathbf{m}')}.
	\end{equation*}
	\referee{Then} the following relation holds: \referee{For all $M_1,M_2\in \R[\mcM\cup {1,...,d}]$ and $U\in \rmU(L)$,}
	\begin{equation*}
		\langle U, M_1 M_2\rangle = \langle \cop U, M_1\otimes M_2 \rangle.
	\end{equation*}
For later purpose, let us write the coordinate representation of the action $\rho$ and the concatenation product with respect to the basis \eqref{sg25}, i.~e.
\begin{align}
	\rho(\msfD_{(J,\mathbf{m})})\z^\gamma &= \sum_\beta (\Dcom)_{(J,\mathbf{m}),\beta}^\gamma \z^\beta,\label{sg42}\\
	\mathsf{D}_{(J',\mathbf{m}')} \mathsf{D}_{(J'',\mathbf{m}'')} &= \sum_{(J,\mathbf{m})} (\Dcop)_{(J',\mathbf{m}'),\,(J'',\mathbf{m}'')}^{(J,\mathbf{m})} \msfD_{(J,\mathbf{m})}.\nonumber	
\end{align}
Note that
\begin{equation}\label{cha01}
	(\Dcom)_{(J,\mathbf{m}),\beta}^\gamma = (\msfD_{(J,\mathbf{m})})_\beta^\gamma.
\end{equation}
An iterative application of the mapping properties \eqref{map01} to \eqref{map05} imply that
\begin{align}
	(\Dcom)_{(J,\mathbf{m}),\beta}^{\gamma\in\mcNmin} \neq 0\,&\implies \beta\in \mcNmin;\nonumber\\
	\mbox{for all }J\neq0,\quad(\Dcom)_{(J,\mathbf{m}),\beta}^{\gamma\in\mcNmin} \neq 0\,&\implies \beta\in \mcN; \nonumber\\
	(\Dcom)_{(J,\mathbf{m}),\beta}^{\gamma\in\mcN} \neq 0\,&\implies \beta\in \mcN;\label{map13}\\
	(\Dcom)_{(J,\mathbf{m}),\beta}^{\gamma\in\populated} \neq 0\,&\implies \beta\in \populated.\label{map14}
\end{align}
Similarly, iterating \eqref{sg09} and \eqref{sg10} yields
\begin{equation}\label{mod20}
	(\Dcom)_{(J,\mathbf{m}),\beta}^\gamma\neq 0 \implies |\beta| = |\gamma| + |(J,\mathbf{m})|.
\end{equation}
In addition, from \eqref{sg35} we have for every $(\gamma,\mathbf{n})\in\mcM$ and every $i=1,...,d$
\begin{equation}\label{mod40}
	\begin{split}
		&(\Dcop)_{(J',\mathbf{m}'),(J'',\mathbf{m}'')}^{(e_{(\gamma,\mathbf{n})}, 0)} \\&\quad= \sum_{\gamma'\in\mcN} (\msfD_{(J',\mathbf{m}')})_{\gamma}^{\gamma'} \delta_{(J'',\mathbf{m}'')}^{(e_{(\gamma',\mathbf{n})},0)}\hspace*{5pt} + \tbinom{\mathbf{n}+\mathbf{m}''}{\mathbf{m}''} \delta_{(J',\mathbf{m}')}^{(e_{(\gamma,\mathbf{n}+\mathbf{m}'')},0)} \delta_{J''}^{0}.
	\end{split}
\end{equation}
\subsection{The canonical model as an exponential map}\label{subsec:hopfmodel}
For the sake of this discussion, let us assume that the solution to \eqref{set01} is smooth, and thus all derivatives make classical sense. Following \eqref{bb03}, we may express the r.~h.~s. as 
\begin{equation}\label{cor04}
	\sum_{\mfl\in\mfL^-\cup\{0\}} \z_{(\mfl,0)}[\mathbf{a},\mathbf{u},\cdot] \xi_\mfl;
\end{equation}
thus, it is natural to seek a representation of the model equation \eqref{bb02} in terms of a shift. By the identification \eqref{warm15}, the nonlinearity is shifted by the model, which at the infinitesimal level means that we need to consider the generators in $\mcD$, which contain multi-indices\footnote{ Giving a special role to the polynomials, encoded in $\mbpartial_i$.} in $\populated$. In addition, the subcriticality condition imposes that the derivatives of $u$ on the r.~h.~s. cannot be of arbitrarily high order; \referee{in particular, \eqref{sub12} implies that the derivatives on the r.~h.~s. are of order at most $\eta$}. We shall focus on the smaller set of indices
\begin{equation}\label{mod01}
	\mcM^- := \{(\gamma,\mathbf{n})\in\mcM\,|\, |\mathbf{n}| < \eta\},
\end{equation}
and accordingly on the set of generators 
\begin{equation}\label{mod50}
	\mcD^- := \{\z^\gamma D^{(\mathbf{n})}\}_{(\gamma,\mathbf{n})\in \mcMm} \sqcup \{\mbpartial_i\}_{i=1,...,d}.
\end{equation}
We first note the following strengthening of the finiteness property \eqref{fin50}:
\begin{lemma}\label{lemfin01}
	For every $\beta\in M(\coord)$,
	\begin{equation}\label{mod04}
		\#\{\big(\gamma,(\gamma',\mathbf{n}')\big)\in M(\coord)\times\mcMm\,|\, (\z^{\gamma'}D^{(\mathbf{n}')})_\beta^\gamma\neq 0\} <\infty.
	\end{equation}
\end{lemma}
\begin{proof}
	Recall \referee{\eqref{ref:new1}}. Fixing $\beta$, there are finitely many $\gamma'$ to consider, since it must hold that $\gamma' \leq \beta$. In addition, the restriction $|\mathbf{n}'|<\eta$ implies there are finitely many $\mathbf{n}'$. For fixed $(\gamma',\mathbf{n}')$, we appeal to \eqref{fin50} and obtain finitely many $\gamma$.
\end{proof}		 	
This restricted set of generators forms a Lie sub-algebra, as can be deduced from \eqref{sg11}, \eqref{sg12} and \eqref{sg13} and recalling the mapping properties \eqref{map02}, \eqref{map04} and \eqref{map05}. More precisely, the following holds.
\begin{lemma}
	Let $L^-$ be the Lie algebra generated by $\mcD^-$ with the Lie bracket \eqref{sg14}, \eqref{sg15}; then $L^-$ is a Lie sub-algebra of $L$. In addition, let $\tL^-$ denote the Lie sub-algebra generated by $\{\z^\gamma D^{(\mathbf{n})}\}_{(\gamma,\mathbf{n})\in \mcMm}$; then $(\tL^-,\triangleright)$ is a pre-Lie sub-algebra of $(\tL,\triangleright)$. 
\end{lemma}	
We consider the universal envelope $\rmU (L^-)$, which is the Hopf sub-algebra of $\rmU (L)$ generated by the basis elements $\{\msfD_{(J,\mathbf{m})}\}_{(J,\mathbf{m})\in M(\mcMm)\times \N_0^{d}}$. The structure constants of the action, cf. \eqref{sg42} inherit the finiteness property \eqref{mod04} in the following way.
\begin{lemma}\label{lemfin02}
	For every $\beta\in \populated\cup \mcNmin$,
		\begin{equation}\label{mod06}
			\# \{\big(\gamma,(J,\mathbf{m})\big) \in (\populated \cup \mcNmin)\times(M(\mcMm)\times \N_0^{d})\,|\, (\Dcom)_{(J,\mathbf{m}),\beta}^\gamma \neq 0\} <\infty.
		\end{equation}
\end{lemma}	
\begin{proof}
	Recall \eqref{cha01}. Note that
	\begin{equation}\label{e00}
		(\msfD_{(J,\mathbf{m})})_\beta^\gamma \neq 0 \implies \beta \geq \sum_{\tilde{\gamma}} J(\tilde{\gamma},\tilde{\mathbf{n}}) \tilde{\gamma},
	\end{equation}
	which can be read off \eqref{sg25} as an endomorphism. As a consequence, for a fixed $\beta$, there are finitely many $\tilde{\gamma}$ allowed, and $J(\tilde{\gamma},\tilde{\mathbf{n}})$ is bounded. In addition, \eqref{mod01} restricts $\tilde{\mathbf{n}}$ to only finitely many. The combination of these observations implies that for a fixed $\beta$ there exist only finitely many $J\in M(\mcMm)$ giving non-vanishing contributions. Moreover, from \eqref{mod20} we learn
	\begin{equation*}
		|\mathbf{m}| = |\beta| - |\gamma| - \sum_{(\tilde{\gamma},\tilde{\mathbf{n}})\in \mcMm} J(\tilde{\gamma},\tilde{\mathbf{n}}) (|\tilde{\gamma}| + \eta -|\tilde{\mathbf{n}}|).
	\end{equation*}
	Since the homogeneity of $\gamma\in\populated\cup\mcNmin$ is bounded from below (cf. \eqref{sub05}), we obtain an upper bound for $|\mathbf{m}|$ and thus finitely many $\mathbf{m}$.
	
	\medskip
	
	It only remains to show that for fixed $\beta$ and $(J,\mathbf{m})$
	\begin{equation*}
		\# \{\gamma \in \populated \cup \mcNmin\,|\, (\Dcom)_{(J,\mathbf{m}),\beta}^\gamma \neq 0\} <\infty.
	\end{equation*}
	This can be obtained recursively from \eqref{mod04} via
	\begin{equation*}
		\mbpartial^\mathbf{m} =  \mbpartial^{\mathbf{m}-e_i} \mbpartial_i
	\end{equation*}
	and
	\begin{align}
		\msfD_{(J,\mathbf{m})} &= \tfrac{1}{J(\tilde{\gamma},\tilde{\mathbf{n}})}  \msfD_{(J-e_{(\tilde{\gamma},\tilde{\mathbf{n}})},\mathbf{m})} \z^{\tilde{\gamma}}D^{(\tilde{\mathbf{n}})}\nonumber \\
		&\quad- \sum_{\substack{(J',\mathbf{m}') + (J'',\mathbf{m}'') = (J - e_{(\tilde{\gamma},\tilde{\mathbf{n}})},\mathbf{m}) \\ (J',\mathbf{m}')\neq 0}}\sum_{\tilde{\beta}} \tfrac{J''(\tilde{\beta},\tilde{\mathbf{n}})+1}{J(\tilde{\gamma},\tilde{\mathbf{n}})} (\msfD_{(J',\mathbf{m}')})_{\tilde{\beta}}^{\tilde{\gamma}} \msfD_{(J'' + e_{(\tilde{\beta},\tilde{\mathbf{n}})},\mathbf{m}'')},\label{sg40}
	\end{align}
	which follows from \eqref{sg23} and \eqref{sg26}.
\end{proof}
As a consequence of the finiteness property \eqref{mod06}, we may transpose the action \eqref{sg42} of $\referee{\rmU}(L^-)$ over $T_\mcNmin^*$.
\begin{proposition}
	Let\footnote{ Careful! $T^-$ should not be confused with $\mfT_-$ in \cite[Definition 5.3]{BHZ}, i.~e. trees of negative degree conforming to a rule.} $T^- := \R[\mcMm \sqcup \{1,...,d\}]$. Then
	\begin{equation*}
		\Dcom^- \z_\beta = \sum_{\substack{(J,\mathbf{m})\in M(\mcMm)\times \N_0^{d}\\ \gamma\in\mcNmin}} (\Dcom^-)_{(J,\mathbf{m}),\beta}^\gamma \msfZ^{(J,\mathbf{m})}\otimes\z_\gamma.
	\end{equation*}
	defines a map $\Dcom^- : T_\mcNmin \to \referee{ T^-\otimes T_\mcNmin}$, \referee{where we recall the notation of Remark \ref{rem:primal}}.
\end{proposition}
Consider the set of characters $\textnormal{Alg}(T^-,\R)$, i.~e. multiplicative functionals $\bff$ over $T^-$; these are characterized by their action on the elements \eqref{mod08}, i.~e.
\begin{equation}\label{mod09}
	f_{\gamma}^{(\mathbf{n})}:= \bff(\msfZ_{(\gamma,\mathbf{n})}),\;\; f_i := \bff(\msfZ_i),\;\; \bff^{(J,\mathbf{m})} = \prod_{i=1}^d (f_i)^{\mathbf{m}(i)}\hspace{-6pt} \prod_{(\gamma,\mathbf{n})\in \mcMm} (f_{\gamma}^{(\mathbf{n})})^{J(\gamma,\mathbf{n})} .
\end{equation}
Following \cite[(8.17)]{reg}, with help of the map $\Dcom^-$ we may define $\Gamma_\bff^-:T_\mcNmin\to T_\mcNmin$ by
\begin{equation}\label{mod11}
	\Gamma_\bff^- = (1\otimes \bff)\Dcom^-;
\end{equation}
taking the dual perspective, $\referee{\Gamma_\bff^{-*}}:T_\mcNmin^*\to T_\mcNmin^*$ is given by
\begin{equation}\label{mod33}
	\Gamma_\bff^{-*} = \sum_{(J,\mathbf{m})} \bff^{(J,\mathbf{m})}\rho(\msfD_{(J,\mathbf{m})}),
\end{equation}
which is effectively a finite sum\footnote{\referee{By which we mean that for every matrix component $(\Gamma_\bff^{-*})_\beta^\gamma$ $=$ $\sum_{(J\mathbf{m})} \bff^{(J,\mathbf{m})} \big(\msfD_{(J,\mathbf{m})}\big)_\beta^\gamma$ is a finite sum in $(J,\mathbf{m})$.}} thanks to the finiteness property \eqref{mod06}. Note that by \eqref{map13} and \eqref{map14} we also have
\begin{align*}
	\Gamma_\bff^{-*} T_\mcN^* &\subset T_\mcN^*\\
	\Gamma_\bff^{-*} T_\populated^* &\subset T_\populated^*.
\end{align*}

\medskip

We now use these maps to rephrase the model equation \eqref{bb02}. Formally, \eqref{cor04} and the identification \eqref{warm15} suggest to write
\begin{equation*}
	\Pi_x^- = \Gamma_\bff^{-*}\big(\sum_{\mfl\in\mfL^-} \z_{(\mfl,0)}\xi_\mfl\big),
\end{equation*}
where $\bff\in\textnormal{Alg}(T^-,\R)$ is generated from $\Pi_x$ itself. Note that $\{e_{(\mfl,0)}\}_{\mfl\in\mfL^-\cup\{0\}}\subset \mcNmin$, thus the action is well-defined. There is however one issue to be solved: $\Pi_x$ takes values in $T_\populated^*$, so we need to connect the characters $\textnormal{Alg}(T^-,\R)$ with power series $T_\populated^*$. To this end, we first derive the following exponential formula, which is the analogue to the explicit formulas used in \cite{OSSW,LOT,LOTT}, cf. e.~g. \cite[(5.16)]{LOT}.
\begin{lemma}\label{lem:exp01}
	Let $\bff\in\textnormal{Alg}(T^-,\R)$. For every $\mathbf{n}$, consider 
	\begin{equation*}
		f^{(\mathbf{n})} := \sum_{\gamma\in\mcN} f_\gamma^{(\mathbf{n})} \z^\gamma + \sum_{\mathbf{m}\in \N_0^d} f_{e_\mathbf{m}}^{(\mathbf{n})} \z_{\mathbf{m}}\in T_\populated^*
	\end{equation*}
	given by
	\begin{align}
		f_\gamma^{(\mathbf{n})} &= \left\{\begin{array}{cl}
			\bff_\gamma^{(\mathbf{n})} & \mbox{if }(\gamma,\mathbf{n})\in \mcMm\,\mbox{and}\\
			0 & \mbox{otherwise,}			
		\end{array}\right.\nonumber\\
		f_{e_\mathbf{m}}^{(\mathbf{n})} &= \left\{\begin{array}{ll}
			\tbinom{\mathbf{m}}{\mathbf{n}} \bff^{(0,\mathbf{m}-\mathbf{n})}, &\mbox{if }\mathbf{n}< \mathbf{m}\,\,\,\mbox{and}\\
			0 & \mbox{otherwise.} \end{array}\right.\label{mod10}
	\end{align}
Then
\begin{equation}\label{exp01}
	\Gamma_{\bff}^{-*} = \sum_{l\geq 0} \frac{1}{l!}\sum_{\mathbf{n}_1,...,\mathbf{n}_l\in\N_0^d} f^{(\mathbf{n}_1)}\cdots f^{(\mathbf{n}_l)} D^{(\mathbf{n}_l)}\cdots D^{(\mathbf{n}_1)}.
\end{equation}
\end{lemma}	
Note that by resummation, cf. \cite[Lemma A.2]{LOT}, \eqref{exp01} may be rewritten as
	\begin{equation*}
		\Gamma_{\bff}^{-*} = \sum_{k\in M(\N_0^d)} \tfrac{1}{k!}f^k D^k.
	\end{equation*}
	In particular, by definition \eqref{bb16},
	\begin{equation*}
			\Gamma_{\bff}^{-*} \big(\sum_{\mfl\in\mfL^-\cup \{0\}} \z_{(\mfl,0)}\xi_\mfl\big) = \sum_{(\mfl,k)} \z_{(\mfl,k)} f^k \xi_\mfl.
	\end{equation*}
	This coincides with the r.~h.~s. of the model equation \eqref{bb02} (when projected onto $\beta\in\mcN$), which suggests that we should take $f^{(\mathbf{n})} = \frac{1}{\mathbf{n}!} \partial^\mathbf{n}\Pi_x$. We denote the corresponding character as $\bfPi_x\in \textnormal{Alg}(T^-,\R)$, given by
	\begin{equation}\label{mod60}
		\bfPi_{x\gamma}^{(\mathbf{n})} := \tfrac{1}{\mathbf{n}!} \partial^\mathbf{n} \Pi_{x\gamma},\,\,\,\,\bfPi_{x i} := (\cdot - x)_i,
	\end{equation}
	and rewrite the model equation \eqref{bb02} for $\beta\in\mcN$ as
	\begin{equation}\label{bb02bis}
		\left\{\begin{array}{l}
			\mcL\Pi_{x\beta} = \Pi_{x\beta}^-,\\
			\Pi_{x\beta}^- = \big(\Gamma_{\bfPi_x}^{*}\sum_{\mfl\in\mfL^-\cup\{0\}}\z_{(\mfl,0)}\xi_\mfl\big)_\beta,
		\end{array}\right.
	\end{equation}
	where for \referee{abbreviation} we denote $\Gamma_{\bfPi_x}^{*}$ the map associated to \referee{$\bfPi_x$ in \eqref{mod60} via \eqref{mod33}}.
\begin{remark}	
Note that by \eqref{map01}, \eqref{map03} and \eqref{map05} we have that for every $\beta\in \mcNmin \setminus \mcN$, the component $\big(\Gamma_{\bfPi_x}^{*}\sum_{\mfl\in\mfL^-\cup\{0\}}\z_{(\mfl,0)}\big)_{\beta}$ is a polynomial: Indeed,
\begin{equation*}
	\big(\Gamma_{\bfPi_x}^{*}\sum_{\mfl\in\mfL^-\cup\{0\}}\z_{(\mfl,0)}\big)_{\beta\in\mcNmin\setminus \mcN} = \big(\Gamma_{(\cdot- x)}^{-*} \z_{(0,0)}\big)_{\beta\in \mcNmin\setminus \mcN},
\end{equation*}
where $\Gamma_{(\cdot-x)}^{-*}$ is generated via \eqref{mod09} choosing $\bff_\gamma^{(\mathbf{n})}=0$ and $\bff_i = (\cdot - x)_i$.
\end{remark}
\begin{proof}[of Lemma \ref{lem:exp01}]	
	We claim that 
	\begin{equation}\label{exp03}
		\tfrac{1}{\mathbf{m}!}\mbpartial^\mathbf{m} = \sum_{l\geq 1}\tfrac{1}{l!} \hspace*{-5pt}\sum_{\substack{\mathbf{m}_i,\mathbf{n}_i\in\N_0^d,\,\mathbf{m}_i \neq 0 \\ \mathbf{m}_1 +...+\mathbf{m}_l = \mathbf{m}}} \hspace*{-20pt}\tbinom{\mathbf{n}_1 + \mathbf{m}_1}{\mathbf{n}_1}\hspace*{-2pt} \cdots \hspace*{-2pt}\tbinom{\mathbf{n}_l + \mathbf{m}_l}{\mathbf{n}_l} \z_{\mathbf{n}_1 + \mathbf{m}_1}\hspace*{-2pt}\cdots \hspace*{-2pt}\z_{\mathbf{n}_l + \mathbf{m}_l} D^{(\mathbf{n}_l)}\hspace*{-2pt}\cdots \hspace*{-2pt}D^{(\mathbf{n}_1)}
	\end{equation}
	and postpone the proof. Then by multiplicativity \eqref{mod09} and the identification \eqref{mod10} we have
	\begin{align*}
		&\sum_{\mathbf{m}\in\N_0^d}\tfrac{1}{\mathbf{m}!}\bff^{(0,\mathbf{m})}\mbpartial^\mathbf{m} \\
		&\quad = \sum_{\mathbf{m}\in\N_0^d}\sum_{l\geq 1}\tfrac{1}{l!} \hspace*{-5pt}\sum_{\substack{\mathbf{m}_i,\mathbf{n}_i\in\N_0^d,\,\mathbf{m}_i \neq 0 \\ \mathbf{m}_1 +...+\mathbf{m}_l = \mathbf{m}}} \hspace*{-20pt} f_{e_{\mathbf{n}_1+\mathbf{m}_1}}^{(\mathbf{n}_1)}\z_{\mathbf{n}_1 + \mathbf{m}_1}\hspace*{-2pt}\cdots \hspace*{-2pt}f_{e_{\mathbf{n}_l+\mathbf{m}_l}}^{(\mathbf{n}_l)}\z_{\mathbf{n}_l + \mathbf{m}_l} D^{(\mathbf{n}_l)}\hspace*{-2pt}\cdots \hspace*{-2pt}D^{(\mathbf{n}_1)}\\
		&\quad = \sum_{l\geq 0}\tfrac{1}{l!} \hspace*{-5pt}\sum_{\substack{\mathbf{m}_i,\mathbf{n}_i\in\N_0^d \\ \mathbf{m}_i> \mathbf{n}_i}} \hspace*{-5pt} f_{e_{\mathbf{m}_1}}^{(\mathbf{n}_1)}\z_{\mathbf{m}_1}\hspace*{-2pt}\cdots \hspace*{-2pt}f_{e_{\mathbf{m}_l}}^{(\mathbf{n}_l)}\z_{\mathbf{m}_l} D^{(\mathbf{n}_l)}\hspace*{-2pt}\cdots \hspace*{-2pt}D^{(\mathbf{n}_1)}.
	\end{align*}
	On the other hand, the summation lemma \cite[Lemma A.2]{LOT} yields
	\begin{align*}
		&\sum_{J}\sum_\mathbf{m} \bff^{(J,\mathbf{m})} \msfD_{(J,\mathbf{m})}\\ 
		&\quad= \sum_{l\geq 0}\tfrac{1}{l!}\sum_{(\gamma_1,\mathbf{n}_1),...,(\gamma_l,\mathbf{n}_l)\in \mcMm}f_{\gamma_1}^{(\mathbf{n}_1)} \z^{\gamma_1}\cdots f_{\gamma_l}^{(\mathbf{n}_l)}\z^{\gamma_l } \big(\sum_\mathbf{m}\tfrac{1}{\mathbf{m}!}\mbpartial^\mathbf{m}\big) D^{(\mathbf{n}_l)}\cdots D^{(\mathbf{n}_1)}.
	\end{align*}
	Combining both identities we get, as desired,
	\begin{align*}
		&\sum_{(J,\mathbf{m})} \bff^{(J,\mathbf{m})} \msfD_{(J,\mathbf{m})}\\
		&\quad = \sum_{l\geq 0} \tfrac{1}{l!} \sum_{(\gamma_1,\mathbf{n}_1),...,(\gamma_l,\mathbf{n}_l)\in \populated\times \N_0^d} f_{\gamma_1}^{\referee{(\mathbf{n}_1)}}\z^{\gamma_1}\cdots f_{\gamma_l}^{\referee{(\mathbf{n}_l)}}\z^{\gamma_l} D^{(\mathbf{n}_l)}\cdots D^{(\mathbf{n}_1)}\\
		&\quad = \sum_{l\geq 0} \tfrac{1}{l!} \sum_{\mathbf{n}_1,...,\mathbf{n}_l\in\N_0^d} f^{(\mathbf{n}_1)}\cdots f^{(\mathbf{n}_l)} D^{(\mathbf{n}_l)}\cdots D^{(\mathbf{n}_1)}.
	\end{align*}
	We now show \eqref{exp03} inductively in $\mathbf{m}$. The case $\length{\mathbf{m}}=1$ follows from \eqref{bb17}, we now assume it true for $\mathbf{m}$ and aim to prove it for $\mathbf{m}+e_i$ for some $i=1,...,d$. Recall that
	\begin{equation}\label{dm02}
		\mbpartial_i \tfrac{1}{\mathbf{m}!} \mbpartial^{\mathbf{m}} = (\mathbf{m}(i)+1) \tfrac{1}{(\mathbf{m}+e_i)!} \mbpartial^{\mathbf{m}+e_i}.
	\end{equation}
	On the other hand, using the Leibniz rule and the induction hypothesis, we write
	\begin{align*}
		&\mbpartial_i \tfrac{1}{\mathbf{m}!} \mbpartial^{\mathbf{m}}\\ 
		&= \sum_{l\geq 1} \tfrac{1}{l!} \hspace*{-9pt}\sum_{\mathbf{n}_1,...,\mathbf{n}_{l+1}\in\N_0^d} \sum_{\substack{\mathbf{m}_i,...,\mathbf{m}_{l+1}\in \N_0^d\setminus \{0\} \\ \mathbf{m}_1 + ... + \mathbf{m}_{l+1} = \mathbf{m} + e_i \\ \mathbf{m}_{l+1}= e_i}} \tbinom{\mathbf{n}_1 + \mathbf{m}_1}{\mathbf{n}_1} \cdots  \tbinom{\mathbf{n}_{l+1} + \mathbf{m}_{l+1}}{\mathbf{n}_{l+1}} \\
		& \quad \quad \quad \times \z_{\mathbf{n}_1 + \mathbf{m}_1} \cdots \z_{\mathbf{n}_{l+1} + \mathbf{m}_{l+1}} D^{(\mathbf{n}_{l+1})}\cdots D^{(\mathbf{n}_1)}\\
		&\,\,+ \sum_{l\geq 1} \tfrac{1}{l!} \hspace*{-9pt}\sum_{\mathbf{n}_1,...,\mathbf{n}_{l}\in\N_0^d} 
		\sum_{\substack{\mathbf{m}_i,...,\mathbf{m}_{l}\in \N_0^d\setminus \{0\} \\ \mathbf{m}_1 + ... + \mathbf{m}_{l} = \mathbf{m}}}  
		\sum_{j=1}^l \hspace*{-2pt}\tbinom{\mathbf{n}_1 + \mathbf{m}_1}{\mathbf{n}_1}\hspace*{-1pt} \cdots \hspace*{-1pt}\tbinom{\mathbf{n}_j + \mathbf{m}_j + e_i}{\mathbf{n}_j}\hspace*{-1pt} (\mathbf{m}_j(i) + 1)  \hspace*{-1pt}\cdots\hspace*{-1pt}\tbinom{\mathbf{n}_{l} + \mathbf{m}_{l}}{\mathbf{n}_{l}}\\
		& \quad \quad \quad \times \z_{\mathbf{n}_1 + \mathbf{m}_1} \cdots \z_{\mathbf{n}_j + \mathbf{m}_j + e_i} \cdots \z_{\mathbf{n}_l + \mathbf{m}_l} D^{(\mathbf{n}_l)}\cdots D^{(\mathbf{n}_j)} \cdots D^{(\mathbf{n}_1)}.
	\end{align*}
	In the first r.~h.~s. term we note that we may run the outer sum over $l\geq 0$ since the case $l=0$ is empty. On the other hand, we may by symmetry rewrite
	\begin{equation*}
		\sum_{\substack{\mathbf{m}_i,...,\mathbf{m}_{l+1}\in \N_0^d\setminus \{0\} \\ \mathbf{m}_1 + ... + \mathbf{m}_{l+1} = \mathbf{m} + e_i \\ \mathbf{m}_{l+1}= e_i}} =\tfrac{1}{l+1} \sum_{j=1}^{l+1} \sum_{\substack{\mathbf{m}_i,...,\mathbf{m}_{l+1}\in \N_0^d\setminus \{0\} \\ \mathbf{m}_1 + ... + \mathbf{m}_{l+1} = \mathbf{m} + e_i \\ \mathbf{m}_{j}= e_i}}.
	\end{equation*}
	Shifting the sum in $l$ by $l\mapsto l-1$, we then rewrite the whole term as
	\begin{align*}
		\sum_{l\geq 1} \tfrac{1}{l!}\hspace*{-8pt} \sum_{\mathbf{n}_1,...,\mathbf{n}_{l}\in\N_0^d}\sum_{j=1}^l \sum_{\substack{\mathbf{m}_i,...,\mathbf{m}_{l}\in \N_0^d\setminus \{0\} \\ \mathbf{m}_1 + ... + \mathbf{m}_{l} = \mathbf{m} + e_i \\ \mathbf{m}_{j}= e_i}} \hspace*{-13pt}\tbinom{\mathbf{n}_1 + \mathbf{m}_1}{\mathbf{n}_1}\hspace*{-2pt} \cdots \hspace*{-2pt} \tbinom{\mathbf{n}_{l} + \mathbf{m}_{l}}{\mathbf{n}_{l}} \z_{\mathbf{n}_1 + \mathbf{m}_1} \hspace*{-2pt}\cdots \z_{\mathbf{n}_{l} + \mathbf{m}_{l}} D^{(\mathbf{n}_{l})}\hspace*{-2pt}\cdots D^{(\mathbf{n}_1)}.
	\end{align*}
	In the second r.~h.~s. term, we shift the third sum by $\mathbf{m}_j \mapsto \mathbf{m}_j - e_i$ so that
	\begin{equation*}
		\sum_{\substack{\mathbf{m}_i,...,\mathbf{m}_{l}\in \N_0^d\setminus \{0\} \\ \mathbf{m}_1 + ... + \mathbf{m}_{l} = \mathbf{m}}}  
		\sum_{j=1}^l = \sum_{j=1}^l \sum_{\substack{\mathbf{m}_i,...,\mathbf{m}_{l}\in \N_0^d\setminus \{0\} \\ \mathbf{m}_1 + ... + \mathbf{m}_{l} = \mathbf{m}+ e_i \\ \mathbf{m}_j > e_i}}.
	\end{equation*}
	The sum of the two terms then yields
	\begin{align*}
		\sum_{l\geq 1} \tfrac{1}{l!}\hspace*{-9pt} \sum_{\mathbf{n}_1,...,\mathbf{n}_{l}\in\N_0^d}\hspace*{-1pt}\sum_{j=1}^l \hspace*{-2pt}\sum_{\substack{\mathbf{m}_1,...,\mathbf{m}_{l}\in \N_0^d\setminus \{0\} \\ \mathbf{m}_1 +...+ \mathbf{m}_{l} = \mathbf{m} + e_i \\ \mathbf{m}_{j}\geq e_i}} \hspace*{-25pt}\mathbf{m}_j(i)\hspace*{-1pt} \tbinom{\mathbf{n}_1 + \mathbf{m}_1}{\mathbf{n}_1} \hspace*{-2pt}\cdots\hspace*{-2pt}  \tbinom{\mathbf{n}_{l} + \mathbf{m}_{l}}{\mathbf{n}_{l}}\hspace*{-1pt} \z_{\mathbf{n}_1 + \mathbf{m}_1}\hspace*{-3pt} \cdots\hspace*{-1pt} \z_{\mathbf{n}_{l} + \mathbf{m}_{l}} \hspace*{-1pt}D^{(\mathbf{n}_{l})}\hspace*{-1pt}\cdots\hspace*{-1pt} D^{(\mathbf{n}_1)}\hspace*{-3pt}.
	\end{align*}
	Note now that the presence of the factor $\mathbf{m}_j (i)$ implies that unless $\mathbf{m}_j \geq e_i$ the corresponding term vanishes, so we can actually remove the condition $\mathbf{m}_j \geq e_i$ and write
	\begin{align*}
		&\sum_{l\geq 1} \tfrac{1}{l!}\hspace*{-9pt} \sum_{\mathbf{n}_1,...,\mathbf{n}_{l}\in\N_0^d}\hspace*{-1pt}\sum_{j=1}^l \hspace*{-2pt}\sum_{\substack{\mathbf{m}_1,...,\mathbf{m}_{l}\in \N_0^d\setminus \{0\} \\ \mathbf{m}_1 +...+ \mathbf{m}_{l} = \mathbf{m} + e_i}} \hspace*{-25pt}\mathbf{m}_j(i) \tbinom{\mathbf{n}_1 + \mathbf{m}_1}{\mathbf{n}_1} \hspace*{-3pt}\cdots\hspace*{-2pt}  \tbinom{\mathbf{n}_{l} + \mathbf{m}_{l}}{\mathbf{n}_{l}}\hspace*{-1pt} \z_{\mathbf{n}_1 + \mathbf{m}_1}\hspace*{-4pt} \cdots\hspace*{-1pt} \z_{\mathbf{n}_{l} + \mathbf{m}_{l}} \hspace*{-1pt}D^{(\mathbf{n}_{l})}\hspace*{-2pt}\cdots\hspace*{-1pt} D^{(\mathbf{n}_1)}\\
		&= (\mathbf{m}(i)+1) \sum_{l\geq 1} \tfrac{1}{l!} \sum_{\mathbf{n}_1,...,\mathbf{n}_{l}\in\N_0^d} \sum_{\substack{\mathbf{m}_i,...,\mathbf{m}_{l}\in \N_0^d\setminus \{0\} \\ \mathbf{m}_1 + ... + \mathbf{m}_{l} = \mathbf{m} + e_i}} \tbinom{\mathbf{n}_1 + \mathbf{m}_1}{\mathbf{n}_1} \cdots  \tbinom{\mathbf{n}_{l} + \mathbf{m}_{l}}{\mathbf{n}_{l}}\\
		&\quad \quad \quad \times \z_{\mathbf{n}_1 + \mathbf{m}_1} \cdots \z_{\mathbf{n}_{l} + \mathbf{m}_{l}} D^{(\mathbf{n}_{l})}\cdots D^{(\mathbf{n}_1)}.
	\end{align*}
	Then \eqref{exp03} follows from \eqref{dm02}.
\end{proof}	
%%%%%%%%%%%
%
\subsection{The structure group}\label{subsec::3.3}
\referee{In regularity structures, model components of positive homogeneity are centered at a base point $x$ around which they are homogeneous (i.~e. they locally vanish to order its homogeneity). In practice, this requires subtracting a Taylor polynomial anchored at $x$ of order the homogeneity of the model component in question\footnote{ There is an easy analogy in terms of Hölder continuous functions: If $f\in C^{k,\alpha}(\R)$ for $k\in \N_0$ and $\alpha\in(0,1)$, then
		\begin{equation*}
			f(y) - \sum_{k'\leq k} \frac{1}{k'!}\frac{d^{k'} f}{dx^{k'}}(x) (y-x)^{k'} = O(|y-x|^{k+\alpha}).
		\end{equation*}.}. For example, for some given $\gamma$ such that $|\gamma|>-\eta$, the model component $\Pi_{x\gamma}$ is homogeneous of order $|\gamma| + \eta$ thanks to the subtraction of a polynomial of degree $\lfloor|\gamma|+\eta\rfloor$. As a consequence, \textit{recentering}, i.~e. changing the base point from $x$ to $y$, forces us to consider shifts by Taylor-like expansions of order fixed by the homogeneity of the shifted component. This leads us to the set of indices\footnote{ Experts in regularity structures will recognize this restriction, as it is the same appearing in the planted trees of positive homogeneity required for recentering, e.~g. \cite[(8.7)]{reg}.}
\begin{equation}\label{sg17}
	\mcMp := \{(\gamma,\mathbf{n})\in\mcM\,|\, |\mathbf{n}| < \eta + |\gamma| \}
\end{equation}
as well as the set of generators
\begin{equation*}
	\mcD^+ := \{\z^\gamma D^{(\mathbf{n})}\}_{(\gamma,\mathbf{n})\in\mcMp} \sqcup \{\mbpartial_i\}_{i=1,...,d};
\end{equation*}
these will ultimately give rise to the \textit{structure group}.}

\medskip

The analogue of the finiteness property \eqref{mod04} holds for $\mcMp$.
\begin{lemma}
	For every $\beta\in M(\coord)$,
	\begin{equation}\label{sg41}
		\#\{\big(\gamma, (\gamma',\mathbf{n}')\big)\in M(\coord) \times\mcMp\,|\, (\z^{\gamma'}D^{(\mathbf{n}')})_\beta^\gamma \neq 0\}<\infty.
	\end{equation}
\end{lemma}
\begin{proof}
	Recall \referee{\eqref{ref:new1}}. As in the proof of Lemma \ref{lemfin01}, fixing $\beta$, there are finitely many $\gamma'$ to consider. Condition $|\mathbf{n}'|<\eta + |\gamma'|$ in turn yields finitely many $\mathbf{n}'$. Finally, we appeal to \eqref{fin50} and obtain finitely many $\gamma$.
\end{proof}	
As in the previous subsection, we now construct a Lie sub-algebra from $\mcD^+$.
\begin{lemma}
	Let $L^+$ be the Lie algebra generated by $\mcD^+$ with the Lie bracket \eqref{sg14}, \eqref{sg15}; then $L^+$ is a Lie sub-algebra of $L$. In addition, let $\tL^+$ denote the Lie sub-algebra generated by $\{\z^\gamma D^{(\mathbf{n})}\}_{(\gamma,\mathbf{n})\in\mcMp}$; then $(\tL^+,\triangleright)$ is a pre-Lie sub-algebra of $(\tL,\triangleright)$.
\end{lemma}
\begin{proof}
	Recall the mapping properties \eqref{map02}, \eqref{map04} and \eqref{map05}. It is then enough to show that identities \eqref{sg11}, \eqref{sg12} and \eqref{sg13} preserve \eqref{sg17}. Note that \eqref{sg09} implies 
	\begin{equation*}
		\mbox{for all }(\gamma',\mathbf{n}')\in\mcMp,\,\,(\z^{\gamma'}D^{(\mathbf{n}')})_\beta^\gamma \neq 0 \implies |\beta| > |\gamma|;
	\end{equation*}
	this shows closedness for \eqref{sg11}. Similarly, \eqref{sg10} implies
	\begin{equation*}
		(\mbpartial_i)_\beta^\gamma \neq 0 \implies |\beta|>|\gamma|,
	\end{equation*}
	which shows closedness for \eqref{sg13}. Finally, since $|\mathbf{n}-e_i| <|\mathbf{n}|$, \eqref{sg12} is also closed.
\end{proof}

\medskip

We consider the universal envelope $\rmU (L^+)$, which is the Hopf sub-algebra of $\rmU (L)$ generated by the basis elements $\{\msfD_{(J,\mathbf{m})}\}_{(J,\mathbf{m})\in M(\mcMp)\times \N_0^{d}}$. As in the previous subsection, the strong finiteness property \eqref{sg41} implies finiteness properties for the structure constants of the action and the product.
\begin{lemma}\label{lemfin05}
	\mbox{}
	\begin{itemize}
		\item For every $\beta\in \populated\cup \mcNmin$,
		\begin{align}
			\#\{ (\gamma, (J,\mathbf{m}))\in (\populated\cup\mcNmin)\times (M(\mcMp) \times \N_0^{d})\,|\,(\Dcom)_{(J,\mathbf{m}),\beta}^\gamma \neq 0 \}<\infty.\label{sg43}
		\end{align}
		\item For every $(J,\mathbf{m})\in M(\mcMp)\times\N_0^{d}$,
		\begin{equation}\label{sg33}
			\#\left\{(J',\mathbf{m}'), (J'',\mathbf{m}'')\in M(\mcMp)\times\N_0^{d} \,|\,(\Dcop)_{(J',\mathbf{m}'),(J'',\mathbf{m}'')}^{(J,\mathbf{m})}\neq 0 \right\}<\infty.
		\end{equation}
	\end{itemize}
\end{lemma}	
\begin{proof}
	The proof of \eqref{sg43} is a slight variation of the arguments of Lemma \ref{lemfin02}, replacing $\mcMm$ with $\mcMp$. We henceforth focus on \eqref{sg33}. By the algebra morphism property of the coproduct, together with the representation \eqref{sg26}, it is enough to show the statement for $(J,\mathbf{m})$ of length one, i.~e. either $J=0$ and $\mathbf{m}=e_i$, or $J=e_{(\gamma,\mathbf{n})}$ and $\mathbf{m} = 0$. In the first case, it can be deduced from \eqref{uni01} that necessarily $J' = J'' = 0$, so we are left with products of the form
	\begin{equation*}
		\msfD_{(0,\mathbf{m}')} \msfD_{(0,\mathbf{m}'')} = \tbinom{\mathbf{m}' + \mathbf{m}''}{\mathbf{m'}} \msfD_{(0,\mathbf{m}'+\mathbf{m}'')},
	\end{equation*}
	which in turn implies
	\begin{equation}\label{mod32}
		(\Dcop)_{(0,\mathbf{m}'),(0,\mathbf{m}'')}^{(0, e_i)} = \delta_{\mathbf{m}' + \mathbf{m}''}^{e_i}
	\end{equation}
	and trivially yields finitely many $\mathbf{m}'$, $\mathbf{m}''$. We now focus on the case $J=e_{(\gamma,\mathbf{n})}$, $\mathbf{m} = 0$; this corresponds to the $\gamma$ component of \eqref{sg35} with $U_1=\msfD_{(J',\mathbf{m}')}$ and $U_2= \msfD_{(J'',\mathbf{m}'')}$. We look at the two summands in \eqref{sg35} separately. The $\gamma$-component of the first term takes the form
	\begin{equation*}
		(\Dcom)_{(J',\mathbf{m}'),\gamma}^{\gamma'} \delta_{(J'',\mathbf{m}'')}^{(e_{(\gamma',\mathbf{n})},0)}.
	\end{equation*}
	By the finiteness property \eqref{sg41}, for fixed $\gamma$ there are finitely many $(J',\mathbf{m}')\in M(\mcMp)\times \N_0^{d}$ and $\gamma'\in\populated\cup\mcNmin$ giving non-vanishing contributions, thus yielding finitely many $(J'',\mathbf{m}'')\in M(\mcMp)\times \N_0^{d}$. The $\gamma$-component of the second term in \eqref{sg35} takes the form
	\begin{equation*}
		\sum_{\mathbf{m}} \tbinom{\mathbf{n}+\mathbf{m}}{\mathbf{m}} \delta_{(J',\mathbf{m}')}^{(e_{(\gamma,\mathbf{n} + \mathbf{m})},0)} \delta_{(J'',\mathbf{m}'')}^{(0,\mathbf{m})}.
	\end{equation*}
	By the condition $|\mathbf{n} + \mathbf{m}| < \eta + |\gamma|$ in \eqref{sg17}, only finitely many $\mathbf{m}$ are allowed, concluding the proof.
\end{proof}	
Note that \eqref{mod20} combined with condition \eqref{sg17} yields the triangularity property
	\begin{equation}\label{sg44}
		\mbox{for all }(J,\mathbf{m})\in M(\mcMp)\times \N_0^d\quad(\Dcom)_{(J,\mathbf{m}),\beta}^\gamma \neq 0 \implies |\beta|>|\gamma|.
	\end{equation}
We are now in a position to transpose the action and the product to obtain the following structure.
\begin{proposition}
	Let $T^+ := \R[\mcMp \sqcup \{1,...,d\}]$.
	\begin{itemize}
	\item Let $\Dcop^+ : T^+ \to T^+\otimes T^+$ be defined by
\begin{equation*}
	\Dcop^+ \msfZ^{(J,\mathbf{m})} = \sum_{(J',\mathbf{m}'),(J'',\mathbf{m}'')\in M(\mcMp)\times \N_0^{d}} (\Dcop)_{(J',\mathbf{m}'),(J'',\mathbf{m}'')}^{(J,\mathbf{m})}  \msfZ^{(J',\mathbf{m}')}\otimes \msfZ^{(J'',\mathbf{m}'')}.
\end{equation*}
Then there exists a map $\mcA^+:T^+\to T^+$ such that $T^+$ is a Hopf algebra with coproduct $\Dcop^+$ and antipode $\mcA^+$.
\item Let $\Dcom^+ : T_{\populated\cup\mcNmin} \to T^+ \otimes (T_{\populated\cup \mcNmin})$ be defined by
\begin{equation*}
	\Dcom^+ \z_\beta = \sum_{\substack{(J,\mathbf{m})\in M(\mcMp)\times \N_0^{d}\\ \gamma\in\populated \cup \mcNmin}} (\Dcom)_{(J,\mathbf{m}),\beta}^\gamma \msfZ^{(J,\mathbf{m})}\otimes\z_\gamma.
\end{equation*}
Then $(T,\Dcom^+)$ is a left comodule over $T^+$, i.~e.
\begin{equation*}
	(\id\otimes\Dcom^+)\Dcom^+ = (\Dcop^+\otimes\id)\Dcom^+.
\end{equation*}
\end{itemize}
\end{proposition}
\begin{proof}
	It only remains to show that the bialgebra $T^+$ is actually a Hopf algebra. We do this by a gradedness argument.
	Condition \eqref{sg17} implies for the grading \eqref{sg30}
	\begin{equation*}
		|(J,\mathbf{m})|\geq 0\;\;\mbox{and } |(J,\mathbf{m})|= 0\iff (J,\mathbf{m}) =  (0,\mathbf{0}),
	\end{equation*}
	which in turn shows that $\referee{{\rm U}(L^+)}$ is a connected graded Hopf algebra. As a consequence, $T^+$ is a connected graded bialgebra, and thus a Hopf algebra. The antipode $\mathcal{A}^+$ is the transposition of the antipode of $\rmU(L^+)$.
\end{proof}
\begin{remark}
	Scanning the proof of \eqref{sg33}, we can show the analogue in case of $\mcMm$, and thus $T^-$ can be endowed with a coproduct \referee{(leading to a composition rule of the form of the first item of \eqref{pos07} below)}. \referee{Furthermore, we believe that, with more complicated arguments, one can show that the antipode of ${\rm U}(L^-)$ can also be transposed (leading to an inverse as in the second item of \eqref{pos07} below). Note that the connectedness argument does not work in this case, because $\mcMm$ does not carry a strictly positive degree, but this is not necessary for the transposition (it only makes the existence of the antipode automatic). Since we will make no use of this Hopf algebra structure, we will not expand on this}.
\end{remark}	
\medskip

As in the previous subsection, consider the set $\textnormal{Alg}(T^+,\R)$ of multiplicative functionals $\bfpi$ over $T^+$, which as in \eqref{mod09} are characterized by their action on the elements $\mcD^+$, i.~e.
\begin{equation}\label{pos06}
	\bfpi_{\gamma}^{(\mathbf{n})}:=\bfpi(\msfZ_{(\gamma,\mathbf{n})}),\; \bfpi_i := \bfpi(\msfZ_i),\; \bfpi^{(J,\mathbf{m})} = \prod_{i=1}^d (\bfpi_i)^{\mathbf{m}(i)}\hspace{-6pt} \prod_{(\gamma,\mathbf{n})\in \mcMp} (\bfpi_{\gamma}^{(\mathbf{n})})^{J(\gamma,\mathbf{n})} .
\end{equation}
$\textnormal{Alg}(T^+,\R)$ is a group under the convolution product,
\begin{equation}\label{pos07}
	\bfpi*\bfsigma := (\bfpi\otimes \bfsigma)\Dcop^+,\;\;\; \bfpi^{-1} = \bfpi \mcA^+.
\end{equation}
Under the lens of the coaction $\Dcom^+$, this gives rise to the structure group. However, there is a subtlety: We want to see the maps of the structure group acting on $T_\populated$, but it is not true that $\Dcom^+T_\populated$ $\subset$ $T^+ \otimes T_\populated$. We therefore introduce the projected coaction
\begin{equation*}
	\overline{\Dcom^+} := (\proj_\populated\otimes\id)\Dcom^+ ,
\end{equation*}
which in coordinates takes the form
\begin{equation*}
	\overline{\Dcom^+} \z_\beta = \sum_{\substack{(J,\mathbf{m})\in M(\mcMp)\times \N_0^{d}\\ \gamma\in\populated}} (\Dcom)_{(J,\mathbf{m}),\beta}^\gamma \msfZ^{(J,\mathbf{m})}\otimes\z_\gamma .
\end{equation*}
The mapping property \eqref{map14} then implies that $(T_\populated,\overline{\Dcom^+})$ is a left comodule over $T^+$, i.~e.
\begin{equation}\label{com02}
	(\id\otimes\overline{\Dcom^+})\overline{\Dcom^+} = (\Dcop^+\otimes\id)\overline{\Dcom^+}.
\end{equation}
as maps in $T_\populated$.
\begin{lemma}\label{lem:refsg01}
	The group
	\begin{equation}\label{sg37}
		G:= \{ \Gamma_{\bfpi}^+ = (\bfpi\otimes {\rm id})\overline{\Dcom^+}\,|\, \bfpi\in \textnormal{Alg}(T^+,\R) \}
	\end{equation}
	is a well-defined structure group; in particular, for all $\bfpi\in \textnormal{Alg}(T^+,\R)$,
	\begin{equation*}
		(\Gamma_{\bfpi}^+ - {\rm id})_\beta^\gamma\neq 0\,\implies |\gamma|<|\beta|.
	\end{equation*}
\end{lemma}	
\begin{proof}
	The group structure is a consequence of the group structure of $\textnormal{Alg}(T^+,\R)$ and the comodule property \eqref{com02}. The triangularity property is a consequence of \eqref{sg44}.
\end{proof}	
\begin{remark}\label{rem:refsg02}
	Note that a character $\bfpi\in \textnormal{Alg}(T^+,\R)$ also defines a map $\Gamma_{\bfpi}^+ :T_{\mcP\cup\mcNmin}\to T_{\mcP\cup\mcNmin}$, replacing $\overline{\Dcom^+}$ by $\Dcom^+$ in \eqref{sg37}. We will use the same notation for both since, when taking the dual perspective as in the previous subsection, we have
	\begin{equation}\label{sg39}
		\Gamma_{\bfpi}^{+*} = \sum_{(J,\mathbf{m})\in M(\mcMp)\times\N_0^d} \bfpi^{(J,\mathbf{m})}\rho(\msfD_{(J,\mathbf{m})});
	\end{equation}
	now the mapping property \eqref{map14} implies
	\begin{equation*}
		\Gamma_{\bfpi}^{+*} T_\populated^* \subset T_\populated^*,
	\end{equation*}
	which means that the restriction $\Gamma_{\bfpi}^{+*}|_{T_\populated^*}$ coincides with the dual of the maps in $G$. In addition, \eqref{map13} implies
	\begin{equation}\label{map32}
		\Gamma_{\bfpi}^{+*} T_\mcN^* \subset T_\mcN^*.
	\end{equation}
\end{remark}	
\referee{\begin{remark}
	At this stage we have all the ingredients to build a regularity structure based on multi-indices in the sense of Definition \ref{ref:defrg}. In particular, we take:
	\begin{itemize}
		\item The set of homogeneities $A$ defined in \eqref{hom50}.
		\item The model space $T = T_\mcNmin \oplus T_\populated$ as described in Remark \ref{rem:primal}.
		\item The structure group $G$ given in Lemma \ref{lem:refsg01} and Remark \ref{rem:refsg02}.
	\end{itemize}
\end{remark}}	

\medskip

For later purpose, let us express the composition rule in terms of the characters.
\begin{lemma}\label{lemcomp02}
	Let $\bfpi,\bfsigma\in \textnormal{Alg}(T^+,\R)$ be defined in terms of the characters \eqref{pos06}. Then for every $(\gamma,\mathbf{n})\in \mcMp$
	\begin{equation}\label{cor01}
		(\bfpi*\bfsigma)_\gamma^{(\mathbf{n})} = \sum_{\gamma'\in\mcN} (\Gamma_{\bfpi}^{+*})_\gamma^{\gamma'}\bfsigma_\gamma^{(\mathbf{n})} + \sum_{\mathbf{m}\in\N_0^{d}} \tbinom{\mathbf{n}+\mathbf{m}}{\mathbf{n}}\bfpi_\gamma^{(\mathbf{n}+\mathbf{m})} \bfsigma^{(0,\mathbf{m})},
	\end{equation}
	and for every $i=1,...,d$
	\begin{equation}\label{cor02}
		(\bfpi*\bfsigma)_i = \bfpi_i + \bfsigma_i.
	\end{equation}
\end{lemma}	
	\begin{proof}
		Note that by \eqref{pos06} and \eqref{pos07}
		\begin{equation*}
			(\bfpi*\bfsigma)_\gamma^{(\mathbf{n})} = (\bfpi\otimes\bfsigma)\Dcop^+ \msfZ_{(\gamma,\mathbf{n})}.
		\end{equation*}
	Then \eqref{cor01} follows from \eqref{mod40} and the representation \eqref{sg39}. Similarly, \eqref{cor02} follows from \eqref{mod32}. 
	\end{proof}	
\begin{remark}
As in Lemma \ref{lem:exp01}, we define for every $\mathbf{n}\in\N_0^d$
\begin{equation*}
	\pi^{(\mathbf{n})} = \sum_{\gamma\in\mcN} \pi_\gamma^{(\mathbf{n})}\z^\gamma + \sum_{\mathbf{m}\in\N_0^d} \pi_{e_\mathbf{m}}^{(\mathbf{n})} \z_\mathbf{m}
\end{equation*}
with
\begin{align*}
	\pi_\gamma^{(\mathbf{n})} &= \left\{\begin{array}{cl}
		\bfpi_\gamma^{(\mathbf{n})} & \mbox{if }(\gamma,\mathbf{n})\in \mcMp\,\mbox{and}\\
		0 & \mbox{otherwise,}			
	\end{array}\right.\\
	\pi_{e_\mathbf{m}}^{(\mathbf{n})} &= \left\{\begin{array}{ll}
		\tbinom{\mathbf{m}}{\mathbf{n}} \bfpi^{(0,\mathbf{m}-\mathbf{n})}, &\mbox{if }\mathbf{n}\referee{<} \mathbf{m}\,\,\,\mbox{and}\\
		0 & \mbox{otherwise.} \end{array}\right.
\end{align*}	
Then
\begin{equation*}
	\Gamma_{\bfpi}^{+*} = \sum_{l\geq 0} \tfrac{1}{l!} \sum_{\mathbf{n}_1,...,\mathbf{n}_l\in\N_0^d} \pi^{(\mathbf{n}_1)} \cdots \pi^{(\mathbf{n}_l)} D^{(\mathbf{n}_l)}\cdots D^{(\mathbf{n}_1)},
\end{equation*}
which is the exponential formula found in \cite{OSSW,LOT,LOTT}.	
\end{remark}

\medskip

In the application to the model equations \eqref{bb02bis} we will need to concatenate maps of the form \eqref{mod11} and \eqref{sg37}. More precisely, and adopting the dual perspective, from \eqref{bb02bis} we see that the recentering of the model will require a composition of the form
\begin{equation*}
	\Gamma_{\bfpi}^{+*} \Gamma_{\bff}^{-*}.
\end{equation*} 
The main issue is that such a composition is not closed: Indeed, this can be seen at the level of the universal enveloping algebras, where we note that the product $\rmU (L^+) \rmU (L^-)$ is not contained in the union $\rmU (L^+) \cup \rmU (L^-)$. Nevertheless, since our interest is in the application to \eqref{bb02bis}, we effectively only need the projected product
\begin{equation}\label{proj01}
	\begin{array}{ccccc}
		\rmU (L^+)\otimes \rmU (L^-) & \longrightarrow & \rmU(L) & \longrightarrow & \rmU (L^-)\\
		U_+ \otimes U_- & \longmapsto & U_+ U_- & \longmapsto & \proj_\mcMm(U_+ U_-),
	\end{array}
\end{equation}
Note that when $\rmU (L)$ acts on $T_\mcNmin^*$ this projection is immaterial, since derivatives of $|\mathbf{n}|>\eta $ produce vanishing contributions due to the subcriticality condition. The argument for \eqref{sg33}, also applied to $\mcMm$, shows that for all $(J,\mathbf{m})\in M(\mcMm)\times \N_0^{d}$
\begin{equation*}
	\#\left\{\begin{array}{l}
		(J',\mathbf{m}')\in M(\mcMp)\times \N_0^{d} \\
		(J'',\mathbf{m}'')\in M(\mcMm)\times \N_0^{d}
	\end{array}\,\,\Big|\,(\Dcop)_{(J',\mathbf{m}'),(J'',\mathbf{m}'')}^{(J,\mathbf{m})}\neq 0\right\}<\infty,
\end{equation*}
and thus the projected concatenation product \eqref{proj01} can be transposed: We define $\Dcop^{\pm}:T^-\to T^+\otimes T^-$ in coordinates by
\begin{equation*}
	\Dcop^{\pm}\msfZ^{(J,\mathbf{m})} := \sum_{\substack{(J',\mathbf{m}')\in M(\mcMp)\times \N_0^{d} \\ (J'',\mathbf{m}'')\in M(\mcMm)\times \N_0^{d}}} (\Dcop)_{(J',\mathbf{m}'),(J'',\mathbf{m}'')}^{(J,\mathbf{m})} \msfZ^{(J',\mathbf{m}')}\otimes \msfZ^{(J'',\mathbf{m}'')},
\end{equation*}
and the corresponding convolution product $*:\textnormal{Alg}(T^+,\R)\otimes \textnormal{Alg}(T^-,\R) \to \textnormal{Alg}(T^-,\R)$ as
\begin{equation*}
	\bfpi*\bff := (\bfpi\otimes \bff)\Dcop^{\pm}.
\end{equation*}
As in Lemma \ref{lemcomp02}, \eqref{mod40} and \eqref{mod32} allow us to express the composition rule at the level of the characters by
\begin{align}
	(\bfpi*\bff)_\gamma^{(\mathbf{n})}  &= \sum_{\gamma'\in\mcN} (\Gamma_{\bfpi}^{+*})_\gamma^{\gamma'}\bff_{\gamma'}^{(\mathbf{n})} + \sum_{\mathbf{m}\in\N_0^{d}} \bfpi_\gamma^{(\mathbf{m})} \bff^{(0,\mathbf{m}-\mathbf{n})}\label{com60}\\
	(\bfpi*\bff)_i &= \bfpi_i + \bff_i.\label{com61}
\end{align}
\subsection{Admissible counterterms and the renormalized model equations}\label{subsec::3.5}
In regularity structures, renormalization takes the form of the subtraction of divergent constants to smooth approximations of the model, so that the limit when the regularization is removed is well-defined. As a consequence, the renormalized model no longer is a basis for local approximations of the solution to the original equation, but rather to a modified version of it, where the nonlinearities are shifted by divergent counterterms; this is what we may call a \textit{bottom-up} approach to renormalization. This is the approach widely used for regularity structures based on decorated trees \cite{reg,BHZ,BCCH}. In this subsection, we describe a \textit{top-down} approach to algebraic renormalization. By \textit{top-down} we mean that we adopt the opposite perspective: We postulate the presence of a counterterm in the equation and use this modified equation to deduce the form of the renormalized model, via the renormalized model equations. From the algebraic viewpoint, both approaches are essentially equivalent, but we adopt the top-down approach in line with the renormalization performed in \cite{LOTT}.

\medskip

The type of transformations we are interested in take the form 
\begin{equation}\label{shi01}
	a^\mfl \mapsto a^\mfl + c^{\mfl},
\end{equation}
where, for every $\mfl\in\mfL^-\cup\{0\}$, $c^{\mfl}$ is a priori allowed to depend on $\mathbf{a},\mathbf{p},\mathbf{u},\{\xi_\mfl\}_{\mfl\in\mfL^-}$ and space-time points $x\in \R^{d}$. The set of \textit{admissible} counterterms carries a more restrictive structure, which we may deduce as a consequence of some natural assumptions.
\begin{assumption}\label{ass:man}
	The counterterm does not depend on the parameterization of the solution space. In particular, $c^{\mfl}$ does not depend on $\mathbf{p}$ except through $\mathbf{u}$.
\end{assumption}
Roughly speaking, this means that the renormalization constants themselves should be independent of the choice of the solution or, in other words, that the renormalized equation remains the same for all the possible solutions. This could be broken in SPDEs on domains with boundaries, where the singular behavior close to the boundary is incorporated to the solution kernel and thus might produce renormalization constants which do depend on the boundary condition; cf. \cite{GH19,GH22}. \referee{However, it is not at all clear that such a dependence can be encoded in terms of a $\mathbf{p}$-dependence via a power series as we do in our approach.}
\begin{assumption}\label{ass:det}
The counterterm is deterministic, except through the randomness included in $\mathbf{u}$. This has two effects: On the one hand, it forces $c^\mfl \equiv 0$ for $\mfl\in \mfL^-$, so that we can identify $c^{\mfl = 0} = c$; on the other, the dependence of $c$ on $\xi_\mfl$ outside of $\mathbf{u}$ is only through its law.	
\end{assumption}
This is consistent with the renormalization performed in the class of singular SPDEs under consideration, where the divergent part of each nonlinear functional of the noise considered is its expectation.
\referee{\begin{assumption}\label{ass:loc}
	The counterterm is local: It only depends on $\mathbf{u}$ via its evaluation at the space-time point $x$, and on $\mathbf{a}$ via a local functional evaluated at $\mathbf{u}(x)$.
\end{assumption}
This assumption can be connected to the usual \textit{locality} assumption in QFT, and of course it is based on the fact that the nonlinearity itself is local \footnote{In fact, this locality assumption is implicit from the beginning, since our coordinates are placeholders for Taylor coefficients of the nonlinearity.}. It implies that our counterterm takes the form
\begin{equation*}
	c(\mathbf{a},\mathbf{u},x) = c(x)[\mathbf{a},\mathbf{u},x],
\end{equation*}
where the square brackets mean a local dependence of the functional $c(x)$ as in the variables \eqref{bb03}.

\begin{assumption}\label{ass:sta}
	The counterterm is independent of space-time points \referee{$x$} outside of the local dependence $\mathbf{u}(x)$.
\end{assumption}
This space-time stationarity assumption is reasonable if the noises are all stationary and the operator is translation-invariant. However, it is known that, in non-translation-invariant situations, counterterms take the form of space-time functions, cf. e.~g. \cite{BB21b}. Since these can be dealt with via preparation maps, we believe this would also be the case in our approach, but do not explore it in the sequel. This assumption further implies that the counterterm satisfies
\begin{equation*}
	c(x)[\mathbf{a},\mathbf{u},x] = c[\mathbf{a},\mathbf{u},x].
\end{equation*}}
Let us now discuss some restrictions on $c$. Since our description of solutions is based on functionals $\Pi_{x\beta}$, the population conditions of $c$ should be consistent with those of $\Pi_x$; thus we postulate
\begin{equation*}
	c\in T_\mcN^*.
\end{equation*}
Renormalization is required for nonlinear functionals of the noise, which means that we may restrict to
\begin{equation}\label{cou01c}
	c_\beta \neq 0 \implies \lnh\beta\rnh\geq 2.
\end{equation}
Finally, renormalization is required only if nonlinearities in the equation are ill-defined, which allows us to impose
	\begin{equation}\label{cou01b}
		c_\beta\neq 0 \implies |\beta|<0.
	\end{equation}
\begin{definition}
	Let
	\begin{equation}\label{setC}
		\mcC := \{\beta \in \mcN\,|\, |\beta|<0,\,\lnh \beta \rnh \geq 2\}.
	\end{equation}
	We call $\{c\in T_\mcC^*\}$ the set of \textit{admissible counterterms}
\end{definition}	
\begin{remark}
		Note that, as a consequence of \eqref{sub06}, $\counterterms\subset T^*_\mcN \subset T_\populated^*$ is a finite-dimensional linear subspace.
	
\end{remark}	

\medskip

\referee{
\begin{exam}\label{example_9}
	Let us describe the set of admissible counterterms in the case of the generalized KPZ equation \eqref{kpz01_example}. On the one hand, condition \eqref{cou01c} simply reduces to
	\begin{equation*}
		\sum_{k(0)\in \N_0} \beta(\xi, k_\mathbf{0} e_\mathbf{0}) \geq 2.
	\end{equation*}
	On the other, recalling \eqref{ref:hom_kpz}, $|\beta|<0$ reduces to the condition
	\begin{align*}
		&(\tfrac{1}{2}\mhyphen)\sum_{k_\mathbf{0}\in\N_0}\beta(\xi, k_\mathbf{0} e_\mathbf{0}) + 2 \sum_{k_\mathbf{0}\in\N_0}\beta(0,k_\mathbf{0} e_\mathbf{0}) \\
		&\quad + \sum_{k_\mathbf{0}\in\N_0} \beta(0,k_\mathbf{0} e_\mathbf{0} + e_{(0,1)}) + \sum_{\mathbf{n}\in\N_0^2} |\mathbf{n}| \beta(\mathbf{n}) < 2.
	\end{align*}
	The set of multi-indices $\mcC$ satisfying these conditions is described in Table \ref{tab:kpz} below.
	\begin{table}[h]
		\centering
		\begin{tabular}{c | c | c}
			$|\beta|$ & $\beta$ & $\#\{\beta\}$ \\
			\hline
			$-1\mhyphen$ & $e_{(\xi,0)} + e_{(\xi,e_\mathbf{0})}$, $2e_{(\xi,0)} + e_{(0,2 e_{(0,1)})}$ & 2 \\
			\hline
			& $e_{(\xi,0)} + 2 e_{(\xi, e_\mathbf{0})}$, $2 e_{(\xi,0)} + e_{(\xi, 2 e_\mathbf{0})}$, &  \\
			$-\frac{1}{2}\mhyphen$ &  $ 2 e_{(\xi,0)} + e_{(\xi, e_\mathbf{0})} + e_{(0, 2e_{(0,1)})}$, $3 e_{(\xi,0)} + 2 e_{(0, 2e_{(0,1)})}$,   & 5 \\
			& $3 e_{(\xi,0)} + e_{(0, e_\mathbf{0} + 2 e_{(0,1)})}$ & \\
			\hline
			& $e_{(\xi,0)} + 3 e_{(\xi, e_\mathbf{0})}$, $e_{(\xi,0)} + e_{(\xi, e_\mathbf{0})} + e_{(0, e_{(0,1)})}$, & \\
			& $2e_{(\xi,0)} + e_{(\xi, e_\mathbf{0})} + e_{(\xi, 2e_\mathbf{0})}$, $2 e_{(\xi,0)} + 2 e_{(\xi,e_\mathbf{0})} + e_{(0, 2 e_{(0,1)})}$, & \\
			& $2 e_{(\xi,0)} + e_{(0, e_{(0,1)})} + e_{(0, 2 e_{(0,1)})}$, $2 e_{(\xi,0)} + e_{(0, e_\mathbf{0} + e_{(0,1)})}$, & \\
			$ 0\mhyphen$ & $ 3 e_{(\xi,0)} + e_{(\xi, e_\mathbf{0})} + 2 e_{(0, 2e_{(0,1)})}$, $3 e_{(\xi,0)} + e_{(\xi,e_\mathbf{0})} + e_{(0,e_\mathbf{0} + 2 e_{(0,1)})}$, & 16\\ 
			& $3 e_{(\xi,0)} + e_{(\xi, 2 e_\mathbf{0})} + e_{(0, 2e_{(0,1)})}$, $4 e_{(\xi,0)} + e_{(0, 2e_\mathbf{0} + 2e_{(0,1)})}$, & \\
			& $4 e_{(\xi,0)} + e_{(0,2 e_{(0,1)})} + e_{(0, e_\mathbf{0} + 2e_{(0,1)})}$, $e_{(\xi,0)} + e_{(\xi, 2e_\mathbf{0})} + e_{(0,1)}$, & \\
			& $e_{(\xi,0)} + e_{(\xi, e_\mathbf{0})} + e_{(0, 2e_{(0,1)})} + e_{(0,1)}$, $2 e_{(\xi,0)} + e_{(0,1)}$, & \\
			& $2 e_{(\xi,0)} + e_{(0,e_\mathbf{0} + 2e_{(0,1)})} + e_{(0,1)}$, $2e_{(\xi,0)} + 2 e_{(0, 2e_{(0,1)})} + e_{(0,1)}$ & \\
			\hline 
		\end{tabular}
		\caption{The set of multi-indices $\mcC$, cf. \eqref{setC}, for the scalar-valued generalized KPZ equation \eqref{kpz01_example}, ordered in rows by homogeneity.}
		\label{tab:kpz}
	\end{table} 
	
	\medskip
	
	It is straightforward to obtain the form of the counterterm from the multi-index, making use of \eqref{kpz02_example} to \eqref{kpz05_example}. For example, the multi-index $\beta = 2 e_{(\xi,0)} + e_{(0,2 e_{(0,1)})}$ generates a counterterm of the form
	\begin{equation*}
		c_{2 e_{(\xi,0)} + e_{(0,2 e_{(0,1)})}} \z^{2 e_{(\xi,0)} + e_{(0,2 e_{(0,1)})}}[\mathbf{a},\mathbf{u},\cdot] = c_{2 e_{(\xi,0)} + e_{(0,2 e_{(0,1)})}} \sigma(u)^2 h(u).
	\end{equation*}
	Similarly, the multi-index $\beta = e_{(\xi,0)} + e_{(\xi,e_\mathbf{0})} + e_{(0,e_{(0,1)})}$ generates
	\begin{align*}
		c_{e_{(\xi,0)} + e_{(\xi,e_\mathbf{0})} + e_{(0,e_{(0,1)})}} \z^{e_{(\xi,0)} + e_{(\xi,e_\mathbf{0})} + e_{(0,e_{(0,1)})}}[\mathbf{a},\mathbf{u},\cdot] = \sigma(u) \sigma'(u) \big(g(u) + 2h(u)\partial_x u\big).
	\end{align*}
\end{exam}
}

\medskip

Recall now Subsection \ref{subsec:hopfmodel}, and particularly \eqref{bb02bis}, where we rewrote the model equation for \eqref{set01} in terms of a shift constructed via the multiplicative functional of $T^-$ identified with the model. Following the same principle, the translation \eqref{shi01} may be reduced to shifting \referee{$\z_{(0,0)}$ by $c$} and then applying $\Gamma_{\bfPi}^*$, so that the new model equation takes the form
\begin{equation}\label{shi02}
	\mcL\Pi =  \Gamma_{\bfPi}^*\sum_{\mfl\in\mfL^-\cup \{0\}} \xi_\mfl(\z_{(\mfl,0)} + \delta_\mfl^0 c) = \sum_{\mfl\in\mfL^-\cup \{0\}} \Gamma_{\bfPi}^*\xi_\mfl\z_{(\mfl,0)} + \Gamma_{\bfPi}^*c.
\end{equation}
\subsection{Connection to preparation maps}\label{subsec::preparation}
In this subsection, using the connection between multi-indices and trees of Subsection \ref{subsec::2.5}, we want to compare the previous construction with the algebraic renormalization generated by preparation maps. \referee{It is independent of the rest of the article and can be skipped in a first reading. Some knowledge about preparation maps is recommended, especially to understand the differences between the two approaches. We warn the reader already familiar with preparation maps that we adopt a dual perspective with respect to the standard literature.}

\medskip

\referee{Preparation maps} were introduced in \cite[Definition 3.3]{BR18} (see also \cite[Section 3.1]{BB21}) to provide a recursive construction of renormalized models in regularity structures. A preparation map (or rather its dual) is a map $R^*$ which is lower triangular with respect to the homogeneity, upper triangular in the number of noises, and furthermore satisfies a right morphism property
\begin{equation}\label{prep01}
	R^* (\sigma \curvearrowright \tau) = \sigma \curvearrowright R^* \tau,
\end{equation}
where $\curvearrowright$ here denotes the simultaneous grafting product (which is obtained via the Guin-Oudom procedure as a concatenation product corrected by the grafting pre-Lie product, cf. e.~g. \cite[Proposition 2.7 (ii)]{Guin2}). Since the tree pre-Lie algebra is free \cite{CL}, every rooted tree can be uniquely expressed as the simultaneous grafting of subtrees onto a node, namely its root. Therefore, \eqref{prep01} effectively takes the form of an action on the root of trees, i.~e.
\begin{equation*}
	R^* (\Xi_\mfl \prod_j \mcI_{\mathbf{m}_j}(\tau_j)) = (R^*\Xi_\mfl)\prod_j \mcI_{\mathbf{m}_j}(\tau_j).
\end{equation*}
The renormalization procedure is then inductively propagated by an additional \referee{comultiplicative map $M^{\circ*}$ which acts on planted trees and satisfies the relation
\begin{equation*}
	M^{\circ*} \mcI_{\mathbf{m}}(\tau) = \mcI_{\mathbf{m}}(R^*\circ M^{\circ *} \tau ),\end{equation*}
}cf. e.~g. \cite[(9)]{BR18}. \referee{This map acts on the trees which are simultaneously grafted in the description above. The proper renormalization map is the combination of both, i.~e. $M^* = R^* M^{\circ *}$, so that
\begin{equation*}
	M^* (\Xi_\mfl \prod_j \mcI_{\mathbf{m}_j}(\tau_j)) = (R^*\Xi_\mfl)\prod_j M^{\circ *}\mcI_{\mathbf{m}_j}(\tau_j) = (R^*\Xi_\mfl)\prod_j \mcI_{\mathbf{m}_j}(M^*\tau_j).
\end{equation*}}

\medskip

Similarly, we may identify $\rho : U(L) \otimes T_{\populated\cup \mcNmin}^* \to \referee{T_{\populated\cup \mcNmin}^*}$ as the simultaneous grafting of the basis elements \eqref{sg25}. Actually, it follows from the pre-Lie morphism property of $\Psi$ \eqref{tree10} that its canonical extension $\Psi_+$ satisfies a morphism property with respect to simultaneous grafting, which for a rooted tree takes the particular form
\begin{align*}
	\Psi_+ \Big(   \Xi_{\mathfrak{l}}	  \prod_{i} \mathcal{I}_{\mathbf{m}_i}( \tau_i )    \Big)	  = \Psi_+ \Big(   \prod_{i} \mathcal{I}_{\mathbf{m}_i}( \tau_i ) 	   \curvearrowright \Xi_{\mathfrak{l}} \Big)	= \rho\Big(\Psi \Big(  \prod_{i} \mathcal{I}_{\mathbf{m}_i}( \tau_i )	 \Big) \Big) \Psi \left( \Xi_{\mathfrak{l}} \right).	
\end{align*}
This suggests that a preparation map for multi-indices could be defined identifying $R^* \z_{(\mfl,0)} = R^* \Psi[\Xi_\mfl]$, and then propagating in the same way. Actually, the triangularity properties of preparation maps are consistent with the translation $\z_{(0,0)} \mapsto \z_{(0,0)} + c$ in our approach: The lower triangularity with respect to the homogeneity follows from condition \eqref{cou01b}, whereas the upper triangularity with respect to the number of noises is a consequence of $\lnh \beta \rnh \geq 1$, which in turn follows from \eqref{cou01c}.

\medskip

\referee{ However, recall from Subsection \ref{subsec::2.5} that a multi-index $\beta$ encodes the fertility of the tree, but does not identify a specific node as a \textit{root}; in other words, the root of a tree corresponding to a multi-index is an inner node for a different tree associated to the same multi-index. For example,
\begin{equation*}
	\Psi [\Xi_\mfl \mcI(\Xi_{\mfl'} \mcI(\Xi_{\mfl''}))\mcI(\Xi_{\mfl''})] = 2 \z^{2e_{(\mfl'', 0)} + e_{(\mfl', e_0)} + e_{(\mfl, 2e_0)}} = \Psi [\Xi_{\mfl'} \mcI(\Xi_\mfl (\mcI(\Xi_{\mfl''}))^2)];
\end{equation*}
on the l.~h.~s. the root is a node $\Xi_\mfl$ with fertility $2$, whereas on the r.~h.~s. the root is a node $\Xi_{\mfl'}$ with fertility $1$. In order for a preparation map to be stable in multi-indices, the contributions from these trees must be the same: This implicitly means that we not only should define the renormalization at the root, but at every node. Consequently, the right morphism property \eqref{prep01} is not expected, and the extension of $R^*$ will always take the form of a full morphism.}

\medskip

Still, the philosophy behind the preparation map approach can be seen in our construction via the model equations \eqref{shi02}. Indeed, the reformulation in terms of exponential maps \eqref{bb02bis} allows \referee{us} to keep the $\referee{{\rm U}}(L)$ and $\z_{(\mfl,0)}$ contributions virtually separated: \referee{More precisely, we have for $\mfl\in \mfL^-\cup \{0\}$,
\begin{equation*}
	\Gamma_{\bfPi_x}^* \xi_\mfl \z_{(\mfl,0)} = \rho \Big(\bfPi_x \otimes \xi_\mfl \z_{(\mfl,0)}\Big).
\end{equation*}
Our counterterm $c$ takes the form of a translation \textit{only} of $\z_{(0,0)}$. This plays the analogue role of a preparation map, but only acting on \textit{node-like components}\footnote{Note that multi-indices $e_{(\mfl,0)}$ are in one-to-one correspondence with $\{\Xi_\mfl\}$.}, i.~e. those of the form $\z_{(l,0)}$, and not extending to all $\z^\beta$. We suggestively rewrite the last equation as
\begin{equation*}
	\Gamma_{\bfPi_x}^* (\z_{(0,0)} + c) = \rho \circ (\Id \otimes R^*)\Big(\bfPi_x \otimes \z_{(0,0)}\Big),
\end{equation*} 
where $R^*: \lspan \{\z_{(\mfl,0)}\} \to T^*$ acts like
\begin{equation*}
	R^* \z_{(\mfl,0)} = \z_{(\mfl,0)} + \delta_\mfl^0 c.
\end{equation*}
However, the map $R^*$ is not complemented by a comultiplicative map $M^{\circ *}$ that would allow us to express renormalization in terms of a linear map applied to the model. Instead, we let the hierarchy of equations \eqref{shi02}, now suggestively rewritten as 
\begin{equation*}
	\mcL\Pi_x =  \sum_{\mfl\in\mfL^-\cup \{0\}} \rho \Big( \bfPi_x \otimes R^* \xi_\mfl\z_{(\mfl,0)} \Big),
\end{equation*}	
propagate the translation of $\z_{(0,0)}$. This is natural in our approach because, unlike in the tree-based setup, we do not have an abstract (algebraic) integration operation, i.~e. an analogue of the planting operation in trees, but rather two models connected by the \textit{concrete} integration kernel and indexed by the same set of multi-indices.}
\section{Renormalized equations and smooth models}
\label{section::4}
In this section we implement the previous algebraic construction to build smooth models for suitably modified versions of \eqref{set01}, i.~e.
\begin{equation*}
	\mcL u = \sum_{\mfl\in\mfL^-\cup\{0\}} a^\mfl (\mathbf{u}) \xi_\mfl + c(\mathbf{u}).
\end{equation*}
\subsection{Formulation of the main result}\label{subsec::4.1}
Before we formulate our main result, we need to establish an analytic framework. We will assume that all the noises are qualitatively smooth (of course, we do not claim that the output of Theorem \ref{th:main}, and in particular the estimates of the model, will be uniform when removing a regularization of the noise). We furthermore assume
\begin{equation}\label{ref:eqalpha}
	\alpha_\mfl <0,\,\,\mfl\in\mfL^-,\,\,\,\,\alpha_0 = 0;
\end{equation}	
this assumption avoids having to center the noises, but is only made for convenience and could in principle be removed. Regarding the solution theory of the PDE, we will use the mild formulation of the equation as was the case in Hairer's works \cite{reg}, and more specifically the integration steps of \cite[Section 5]{reg}. In this sense, we slightly deviate from the multi-index approach \cite{OSSW,LOTT}, as \cite[Section 6]{JZ23} already did in the case of an elliptic multiplicative SPDE. The analytic assumptions are then formulated in terms of kernels: We assume that
\begin{equation}\label{con01}
	u = K*\big(\hspace*{-5pt}\sum_{\mfl\in \mfL^-\cup\{0\}} \hspace*{-5pt}a^\mfl (\mathbf{u}) \xi_\mfl + c(\mathbf{u})\big) + R*\big(\hspace*{-5pt}\sum_{\mfl\in \mfL^-\cup\{0\}} \hspace*{-5pt}a^\mfl (\mathbf{u}) \xi_\mfl + c(\mathbf{u})\big) + v,
\end{equation} 
where $K$ is compactly supported, $R$ is smooth and $v$ incorporates the effect of the initial value. Since our interest is not the solution theory, but only the construction of the model, we will disregard the last two terms, as they can be parameterized by polynomials; at the same time, additional polynomial contributions (which will be required in the model equation for recentering) can be absorbed, so we allow for some freedom in that sense. Therefore, in line with the derivation of \eqref{warm08}, and the more rigorous construction of Section \ref{section::3}, for every $\beta\in\mcN$ we write
\begin{equation}\label{mod02}
	\Pi_{x\beta} = K*\big(\Gamma_{\bfPi_x}^{*}\big(\sum_{\mfl\in\mfL\cup\{0\}} \z_{(\mfl,0)}\xi_\mfl + c\big)\big)\,\,\,\textnormal{mod polynomials}.
\end{equation}
\begin{assumption}\label{ass:ker}
	There exist a smooth function $\bar{K} : \R^{d}\setminus\{0\} \to \R$ which is scale invariant in the sense that there exists $\eta>|\mathfrak{s}|$ such that
	\begin{equation*}
		\bar{K}\circ \mcS_\mathfrak{s}^\lambda =  \lambda^{|\mathfrak{s}| - \eta} \bar{K},
	\end{equation*}
	and such that $K = \bar{K}\cdot\chi$ where $\chi\in C_c^\infty$ is a smooth cutoff supported in $B_1$.
\end{assumption}
	The decomposition \eqref{con01} combined with Assumption \ref{ass:ker} corresponds to a simple case of the assumptions considered in \cite[Section 5]{reg}, as reflected in \cite[Lemma 5.5]{reg}, where translation-invariance is incorporated to the kernel. Indeed, the cited result establishes in particular that $K$ is a $\eta$-regularizing kernel, cf. \cite[Assumption 5.1]{reg}, and annihilates polynomials of some degree, cf. \cite[Assumption 5.4]{reg}. However, regarding the polynomials, we will actually make use of the property that $K$ \textit{preserves} polynomials of \textit{any} order, as stated in \cite[Lemma 2.9]{BCZ23}; i.~e.
	\begin{equation*}
		\int_{\R^{d}} K(\cdot - z)z^\mathbf{n} dz\,\,\mbox{ is a polynomial of degree }\leq|\mathbf{n}|.
	\end{equation*} 
	From the point of view of our model equation, this justifies the population conditions \eqref{bb06} and \eqref{bb07}. Indeed, \eqref{bb06} is consistent due to $K*0 = 0$. On the other hand, multi-indices with no noise components will always generate polynomials which, after convolution with $K$, can be absorbed in the polynomial we mod out in \eqref{mod02}, thus making \eqref{bb07} consistent. The $\eta$-regularizing property means that $K$ satisfies a multi-level Schauder estimate of degree $\eta$, as established in \cite[Theorem 5.17]{reg} in the context of modelled distributions. This in particular contains classical Schauder, cf. e.~g. \cite[Theorem 14.17]{FH}. In principle, the condition $\eta > |\mathfrak{s}|$ can be relaxed to $\eta\geq |\mathfrak{s}|$ with the cost of dealing with logarithmic divergences, cf. \cite[Remark 5.6]{reg}, but we keep away from these analytic technicalities. See \eqref{schau01} for the formulation we will use in our construction.

	\medskip

	The mild formulation \eqref{con01} is of course just an Ansatz which is convenient for us due to the amount of existing literature on regularity structures. In applications, other PDE arguments can be used to construct the model, as seen in the multi-index approach \cite{LOTT,LO22}. These works only deal with the heat operator, but one can generalize the techniques imposing some assumptions on the operator $\mcL$; see \cite{GT23} for a fourth order parabolic operator in the context of the stochastic thin film equation. However, there are some fundamental reasons why these approaches cannot work for our goal of constructing a model for \textit{any} counterterm. In \cite{LO22}, the well-posedness of the model equation is based on a Liouville principle combined with the local homogeneity condition, cf. \cite[Lemma 3.9]{LO22}. This works as long as the local and the global behavior of the model are \textit{within the same integer range}; in the presence of a counterterm, this is no longer expected, since at small scales the dominant term is always given by the homogeneity, but at large scales the counterterm dominates. Thus, \cite{LO22} as stated can only deal with the \textit{canonical model}. In turn, the counterterm in \cite{LOTT} deals as well with the large scale behavior, cf. \cite[Proposition 4.6]{LOTT}, imposing some sort of uniqueness; therefore a generic counterterm cannot work.

\medskip

We are now in a position to construct a model, thus showing our main result.
\begin{theorem}\label{th:main}
	Let $\xi_\mfl$ be smooth for all $\mfl\in\mfL^-$. Let $K$ satisfy Assumption \ref{ass:ker}. For every admissible counterterm $c\in\counterterms$, there exists a model $\Pi_x: \R^{d}\to T_\populated^*$, $\Gamma_{xy}\in G$ such that \eqref{bb01} holds and for every $\beta\in\mcN$
	\begin{align}
		\Pi_{x\beta} &= K*\Pi_{x\beta}^-\,\,\textnormal{mod polynomial of deg}<|\beta| + \eta, \label{est01}\\
		\Pi_{x\beta}^- &= \Big(\Gamma_{\bfPi_x}^{*}\big(\sum_{\mfl\in \mfL^-\cup \{0\}} \xi_\mfl \z_{(\mfl,0)}  + c\big)\Big)_\beta.\label{est11}
	\end{align}
	More precisely, $(\Pi_x,\Gamma_{xy})$ satisfy the algebraic constraints
	\begin{equation}\label{recentering}
		\Pi_x = \Gamma_{xy}^*\Pi_y,\,\,\,\Gamma_{xy}^* = \Gamma_{xz}^*\Gamma_{zy}^*
	\end{equation}
	as well as the analytic estimates
	\begin{align}
		|\Pi_{x\beta}^-(y)|&\lesssim |y-x|^{|\beta|},\,\,\,|\Pi_{x\beta}(y)|\lesssim |y-x|^{|\beta|+\eta},\label{est50}\\ |(\Gamma_{xy}^*)_\beta^\gamma|&\lesssim |y-x|^{|\beta|- |\gamma|},\label{gam01}
	\end{align}
	for all $\beta,\gamma\in \populated\cup \mcNmin$ and $|y-x|\leq 1$.
\end{theorem}	
\subsection{Proof of Theorem \ref{th:main}}\label{subsec::4.2}
The proof of Theorem \ref{th:main} is inductive, so we first need to set up an order in the set of multi-indices such that the $\beta$-projection of \eqref{est01}, \eqref{est11} only depends on ``previous" multi-indices. We achieve this assuming at this stage that the counterterm $c$ is given as an input; see later in Subsection \ref{subsec::4.4} how this can be modified when the renormalization constants need to be chosen within the induction.

\medskip

Looking at the r.~h.~s. of \eqref{est11}, we note that we require that $(\Gamma_{\bfPi_x}^{-*} (\xi_\mfl \z_{(\mfl,0)} + c))_\beta$ depends on $\partial^\mathbf{n} \Pi_{\beta'}$ only for ``previous" $\beta'$. This is true considering the component-wise ordering \eqref{comp02}. The inductive structure then holds as a consequence of the following general property.
\begin{lemma}\label{lemind01}
	Let $\bff\in\textnormal{Alg} (T^-,\R)$, and let $\Gamma_\bff^-$ be given by \eqref{mod11}. For every $\beta,\gamma\in\mcNmin$, the component $(\Gamma_\bff^{-*})_\beta^\gamma$ depends only on $\{\bff_i\}_{i=1,...,d}$ $\cup$ $\{\bff_{\gamma'}^{(\mathbf{n}')}\}_{(\gamma',\mathbf{n}')\in\mcMm}$ with $\gamma'<\beta$.
\end{lemma}	 
\begin{proof}
	We use representation \eqref{mod33}. By the finiteness property \eqref{mod06}, we may focus on a fixed $(J,\mathbf{m})$ $\in$ $M(\mcMm)\times\N_0^{d}$. The multiplicativity property \eqref{mod09} implies that our statement is equivalent to
	\begin{align*}
		(\msfD_{(J,\mathbf{m})})_{\beta\in\mcNmin}^{\gamma\in\mcNmin} \neq 0 \implies \gamma'<\beta\,\mbox{for all }\gamma'\,\mbox{with }J(\gamma',\mathbf{n}')\neq 0.
	\end{align*}
	Note that this does not follow from \eqref{e00}, since we need a strict inequality. We start showing it for the generators \eqref{mod51}. For $\{\mbpartial_i\}_{i=1,...,d}$ there is nothing to show; for $\z^{\gamma'} D^{(\mathbf{n}')}$, it follows from \eqref{comp01}. The proof concludes by induction on the length of $(J,\mathbf{m})$, using the recursion \eqref{sg40}.
\end{proof}	
We now describe the induction for the construction of $\Pi_x$. For every $\mathbf{n}\in\N_0^d$, we fix $\Pi_{x e_\mathbf{n}}$ by \eqref{bb01}, and the corresponding character $\bfPi_{xi} = (\cdot - x)_i$ in \eqref{mod60}. For a multi-index $\beta\in\mcN$, each induction step consists of the following:
\begin{enumerate}
	\item Construction and estimates of \eqref{est11}. Since $\sum_\mfl \xi_\mfl \z_{(\mfl,0)} + c$ $\in$ $T_\mcNmin^*$, Lemma \ref{lemind01} implies that this only depends on $\partial^\mathbf{n} \Pi_{x\beta'}$ for $\beta'\in\mcN$, $\beta'<\beta$ and the polynomials, which were fixed at the very beginning.
	\item Construction and estimates of $\partial^\mathbf{n}\Pi_{x\beta}$ via the mild formulation \eqref{est01}.
\end{enumerate}

\medskip

	It is crucial that we assume that the noises are qualitatively smooth (or at least sufficiently regular) so that multiplication is a continuous operation. This way, the estimates of $\Pi_x^-$ come out as a consequence of the estimates of $\Gamma_{\bfPi_x}$ and $\xi_\mfl$. The former, in turn, are a consequence of estimates of $\Pi_x$ in previous levels and the multiplicative structure, as reflected in the inductive procedure. More precisely, for $\bff\in\textnormal{Alg}(T^-,\R)$ and $\Gamma_\bff^-$ defined by \eqref{mod11}, it follows by the representation \eqref{mod33} and the multiplicativity property \eqref{mod09} that fixing $\beta,\gamma\in\mcNmin$,
	\begin{equation*}
		|(\Gamma_\bff^{-*})_\beta^\gamma| \leq \sum_{(J,\mathbf{m})} |\bff^{(J,\mathbf{m})}| |(\msfD_{(J,\mathbf{m})})_\beta^\gamma| = \sum_{(J,\mathbf{m})} \prod_i |\bff_i|^{\mathbf{m}(i)}\,\prod_{(\gamma',\mathbf{n'})} |\bff_{\gamma'}^{(\mathbf{n}')}|^{J(\gamma',\mathbf{n}')} |(\msfD_{(J,\mathbf{m})})_\beta^\gamma|,
	\end{equation*} 
	and by \eqref{mod20} it holds
	\begin{equation*}
		|(\Gamma_\bff^{-*})_\beta^\gamma| \lesssim \sum_{(J,\mathbf{m})} \delta_{|\beta|}^{|\gamma| + |(J,\mathbf{m})|} \prod_i |\bff_i|^{\mathbf{m}(i)}\,\prod_{(\gamma',\mathbf{n'})} |\bff_{\gamma'}^{(\mathbf{n}')}|^{J(\gamma',\mathbf{n}')}.
	\end{equation*}
	Therefore, estimates on $\Gamma_{\bfPi_x(z)}^{*}$ follow from estimates of $\partial^\mathbf{n}\Pi_x(z)$ on previous levels and the polynomial characters; more specifically, we will feed in the induction the estimate
	\begin{equation}\label{mod70}
		|\partial^\mathbf{n} \Pi_{x\beta}(z)| \lesssim |z-x|^{|\beta|+\eta - |\mathbf{n}|},
	\end{equation}
	which generalizes the second estimate in \eqref{est50}. Then, recalling \eqref{sg30}
	\begin{equation*}
		|(\Gamma_{\bfPi_{x}(z)}^{*})_\beta^\gamma| \lesssim \sum_{(J,\mathbf{m})} \delta_{|\beta|}^{|\gamma| + |(J,\mathbf{m})|} |z-x|^{|(J,\mathbf{m})|} \lesssim |z-x|^{|\beta|-|\gamma|}.
	\end{equation*}
	\begin{remark}\label{rem:rec}Of course, the multiplicativity property is useless without the smoothness assumption. Therefore, for mollified noises, we do not claim that our estimates are uniform when the regularization is removed. In the singular case, one needs to benefit from the choice of renormalization constants and potentially perform two reconstruction steps for the estimate of $\Pi_x^-$: The first one to get the estimate of $\Gamma_{\bfPi_x}$ (which could involve singular products of \textit{integrated} functionals, i.~e. \textit{planted trees}), the second to get the estimate of $\Gamma_{\bfPi_x}^{*}\z_{(\mfl,0)} \xi_\mfl$ (which would involve singular products between integrated functionals and noises, i.~e. \textit{rooted trees}).
\end{remark}	

\medskip

The so-called integration step, i.~e. passing from $\Pi_x^-$ to $\Pi_x$, involves the $\eta$-regularizing property of $K$. This more analytic side of the construction is beyond the scope of this paper, which is mostly algebraic, so we will always refer to the already existing literature on multi-level Schauder estimates in regularity structures, in particular \cite[Section 5]{reg} and \cite{BCZ23}. We will state our integration steps in a pointwise form, i.~e. 
	\begin{equation}\label{schau01}
		|h(y)|\lesssim |y-x|^\nu\,\implies |\partial^\mathbf{n} \big(K*h(y) - {\rm T}_x^{\nu + \eta} K*h(y)\big)|\lesssim |y-x|^{\nu + \eta - |\mathbf{n}|},
	\end{equation}
	where ${\rm T}_x^{\kappa}$ is the Taylor polynomial of its argument centered at $x$ and to order $\kappa$. The interested reader can consult \cite[Lemma 3.15]{BCZ23} for a clear exposition of how to obtain such an estimate under Assumption \ref{ass:ker}.

\medskip

\referee{For the recentering maps, following \eqref{sg37}, we identify $\Gamma_{xy} = \Gamma_{\bfpi_{xy}}^+$ for some character $\bfpi_{xy}\in \textnormal{Alg}(T^+,\R)$ to be defined inductively}. We have the analogue of Lemma \ref{lemind01}.
\begin{lemma}\label{lemind02}
	Let $\bfpi$ $\in$ $\textnormal{Alg}(T^+,\R)$, and let $\Gamma_{\bfpi}^+$ be given by \eqref{sg37}.
	\begin{itemize}
		\item For every $\beta\in\populated$ and $\gamma\in\mcNmin$, the component $(\Gamma_{\bfpi}^{+*})_\beta^\gamma$ depends only on $\{\bfpi_i\}_{i=1,...,d}$ $\cup$ $\{\bfpi_{\gamma'}^{(\mathbf{n}')}\}_{(\gamma',\mathbf{n}')\in\mcMp}$ with $\gamma'<\beta$.
		\item For every $\beta\in\populated$ and $\mathbf{n}\in\N_0^{d}$, the component $(\Gamma_{\bfpi}^{+*})_\beta^{e_\mathbf{n}}$ depends only on $\{\bfpi_i\}_{i=1,...,d}$ $\cup$ $\{\bfpi_{\beta}^{(\mathbf{n})}\}$.
	\end{itemize} 
\end{lemma}		
\begin{proof}
	By the finiteness property \eqref{sg43}, we may reduce the problem to a fixed element $\msfD_{(J,\mathbf{m})}$. The proof of the first item is the same as in Lemma \ref{lemind01}. For the second, it reduces to showing
	\begin{align}
		(\msfD_{(J,\mathbf{m})})_\beta^{e_\mathbf{n}} \neq 0 \implies J\leq e_{(\beta,\mathbf{n})}.\label{ind02}
	\end{align}
	We assume $J\neq 0$, otherwise there is nothing to show. It is easy to see, e.~g. inductively in $\length{J}$ via \eqref{sg40}, that $\rho(\msfD_{(J,\mathbf{m})})\z_\mathbf{n} = \sum_\beta \z^\beta \delta_{J}^{e_{(\beta,\mathbf{n})}}$; this shows \eqref{ind02}.
\end{proof}
The construction of \referee{$\Gamma_{xy} = \Gamma_{\bfpi_{xy}}^+$} follows a structure parallel to that of $\Pi_x$. We start by fixing $\bfpi_{xy i} = (y-x)_i$, $i=1,...,d$. By \eqref{com61}, this choice is already consistent with the recentering rule for the polynomial characters, since
\begin{equation*}
	\bfPi_{xi} = (\cdot-x)_i = (y-x)_i + (\cdot-y)_i = \bfpi_{xy i} + \bfPi_{yi}.
\end{equation*}
In addition, the composition rule for the polynomial characters follows from \eqref{cor02} and
\begin{equation*}
	\bfpi_{xy i} = (y-x)_i = (z-x)_i + (y-z)_i = \bfpi_{xz i} + \bfpi_{zy i}.
\end{equation*} 
For a multi-index $\beta\in\mcN$, each induction step consists of the following:
\begin{enumerate}
	\item Construction and estimates of $(\Gamma_{xy}^*)_\beta^{\gamma\in\mcNmin}$, which due to Lemma \ref{lemind02} only depends on $\bfpi_{xy \beta'}^{(\mathbf{n})}$ with $\beta'<\beta$.
	\item Construction and estimates of $\bfpi_{xy \beta}^{(\mathbf{n})}$ via the relation
	\begin{equation}\label{tpi}
		\Gamma_{xy}^*\proj_\mcN \Pi_y = \Pi_x - \sum_{\mathbf{n}\in\N_0^d} \bfpi_{xy}^\mathbf{(n)} (\cdot - y)^\mathbf{n}.
	\end{equation}
	This automatically shows the first item in \eqref{recentering}. In turn, it is a particular case of the composition rule \eqref{com60}, i.~e.
	\begin{equation}\label{ind07bis}
		\partial^\mathbf{n} \Pi_{x \beta} = \sum_{\gamma\in\mcN} (\Gamma_{xy}^*)_\beta^\gamma \partial^\mathbf{n} \Pi_{y \gamma} + \sum_{\substack{\mathbf{m}\geq\mathbf{n}\\ |\mathbf{m}|< |\beta|+\eta}} \tbinom{\mathbf{m}}{\mathbf{n}} \bfpi_{xy \beta}^{(\mathbf{m})} (\cdot-y)^{\mathbf{m}-\mathbf{n}},
	\end{equation}
	which we feed in the induction.
	\item Proof of the recentering property, which thanks to Lemma \ref{lemcomp02} reduces to showing
	\begin{equation}\label{ind07}
		\bfpi_{xy \beta}^{(\mathbf{n})} = \sum_{\gamma\in\mcN} (\Gamma_{xz}^*)_\beta^\gamma \bfpi_{zy \gamma}^{(\mathbf{n})} + \sum_{\substack{\mathbf{m}\geq\mathbf{n}\\ |\mathbf{m}|< |\beta|+\eta}} \tbinom{\mathbf{m}}{\mathbf{n}} \bfpi_{xz \beta}^{(\mathbf{m})} \referee{(y-z)^{\mathbf{m}-\mathbf{n}}},
	\end{equation}
	which we also feed in the induction.
\end{enumerate}

\medskip

As with $\Pi_x$ and $\Gamma_{\bfPi_x}$ before, although now not requiring qualitative smoothness since there are no singular products involved, the estimates of $(\Gamma_{xy}^*)_\beta^{\gamma}$ are a consequence of the estimates of its characters, namely \eqref{gam01} follows from
	\begin{equation}\label{gam04}
		|\bfpi_{xy \beta}^{(\mathbf{n})}|\lesssim |y-x|^{|\beta| + \eta - |\mathbf{n}|},
	\end{equation}
	which we also feed in the induction. The case $\gamma \in \mcP$ involves $\bfpi_{xy}^{(\mathbf{n})}$ at the current level, which is the last estimate obtained in each induction step. The construction and estimates of the characters $\bfpi_{xy}^{(\mathbf{n})}$, in turn, follow from the so-called \textit{three point argument}, cf. \cite[Proposition 4.4]{LOTT}, \cite[Lemma 4.6]{LO22}. This is based on identity \eqref{tpi}, which at level $\beta$ reads
	\begin{equation*}
		(\Pi_x - \Gamma_{xy}^*\proj_\mcN\Pi_y)_\beta = \sum_{|\mathbf{n}|<|\beta| + \eta} \bfpi_{xy \beta}^{(\mathbf{n})}(\cdot-y)^\mathbf{n};
	\end{equation*}
	the terms on the l.~h.~s. are provided by the induction hypothesis, whereas the r.~h.~s. is a polynomial of fixed degree which satisfies an $L^\infty$-bound on the ball of radius $1$, and by equivalence of norms we can deduce the estimate \eqref{gam04} of its coefficients. See \cite[Subsection 4.5]{LO22} for the details, which we will skip in the forthcoming proof.	

\medskip

\begin{proof}[of Theorem \ref{th:main}]
	We focus on the induction for multi-indices $\beta\in \mcN$. The base case is given by length-one non-polynomial multi-indices, which by conditions \eqref{bb06} and \eqref{bb07} are given by $\beta = e_{(\mfl,0)}$, $\mfl\in\mfL^-$. Then \eqref{est11} reduces to
	\begin{equation*}
		\Pi_{x e_{(\mfl,0)}}^- = \xi_\mfl.
	\end{equation*}
	Obviously, the smoothness assumption and the fact that $\alpha_\mfl <0$ imply for $|y-x|\leq 1$ the pointwise estimate
	\begin{equation*}
		|\Pi_{x e_{(\mfl,0)}}(y)^-|\lesssim |y-x|^{\alpha_\mfl} = |y-x|^{|e_{(\mfl,0)}|}.
	\end{equation*}
	We now look at $\Pi_{x e_{(\mfl,0)}}$. For simplicity, denote $\tilde{\Pi}_x := K*\Pi_x^-$. Take 
	\begin{equation*}
		\Pi_{x e_{(\mfl,0)}} = \tilde{\Pi}_{x e_{(\mfl,0)}} - \rmT_x^{\alpha_\mfl + \eta} \tilde{\Pi}_{x e_{(\mfl,0)}},
	\end{equation*}
	which is consistent with \eqref{est01}, and due to \eqref{schau01} gives the estimate
	\begin{equation}\label{con13}
		|\partial^\mathbf{n}\Pi_{x e_{(\mfl,0)}}(y)| \lesssim |y-x|^{\alpha_{\mfl} + \eta - |\mathbf{n}|}.
	\end{equation}
	We now build the $e_{(\mfl,0)}$-component of the characters for the structure group, so we implicitly assume $\alpha_\mfl + \eta >0$ (otherwise by \eqref{sg17} we have nothing to construct). Take a different base point $y$ and note that
	\begin{equation*}
		\Pi_{x e_{(\mfl,0)}} - \Pi_{y e_{(\mfl,0)}} = 0\,\,\textnormal{mod polynomial of deg}<\alpha_\mfl + \eta.
	\end{equation*}
	Expanding this polynomial around $y$, we obtain the coefficients $\bfpi_{xy e_{(\mfl,0)}}^{(\mathbf{n})}$ by
	\begin{equation}\label{con14}
		\Pi_{x e_{(\mfl,0)}}(z) = \Pi_{y e_{(\mfl,0)}}(z) + \sum_{|\mathbf{n}|<|e_{(\mfl,0)}| + \eta} \bfpi_{xy e_{(\mfl,0)}}^{(\mathbf{n})} (z-y)^\mathbf{n};
	\end{equation}
	this immediately shows the recentering rule \eqref{com60} at level $e_{(\mfl,0)}$. In addition, in \eqref{con14} we now use \eqref{con13} and the three-point argument of \cite[Subsection 4.5]{LO22}, yielding
	\begin{equation*}
		|\bfpi_{xy e_{(\mfl,0)}}^{(\mathbf{n})}| \lesssim |y-x|^{|e_{(\mfl,0)}| + \eta - |\mathbf{n}|},
	\end{equation*}
	which is \eqref{gam04} at this level, and therefore yields \eqref{gam01}. We now show the composition rule at level $e_{(\mfl,0)}$. For this we take a third base point $z$ and note that, adding and subtracting $\Pi_{z e_{(\mfl,0)}}$ in \eqref{con14}, we have that
		\begin{equation*}
			\sum_{|\mathbf{n}|<|e_{(\mfl,0)}| + \eta} \hspace*{-5pt}\bfpi_{xy e_{(\mfl,0)}}^{(\mathbf{n})} (\cdot-y)^\mathbf{n} = \sum_{|\mathbf{n}|<|e_{(\mfl,0)}| + \eta} \hspace*{-5pt}\bfpi_{xz e_{(\mfl,0)}}^{(\mathbf{n})} (\cdot-z)^\mathbf{n} + \sum_{|\mathbf{n}|<|e_{(\mfl,0)}| + \eta} \hspace*{-5pt}\bfpi_{zy e_{(\mfl,0)}}^{(\mathbf{n})} (\cdot-y)^\mathbf{n}.
		\end{equation*} 
	\referee{Reexpanding the sum around $y$ by the binomial formula, and via a change of variables, the r.~h.~s. can be rewritten as
	\begin{align*}
		\sum_{|\mathbf{n}|<|e_{(\mfl,0)}| + \eta} \Big( \sum_{\substack{\mathbf{m}\geq \mathbf{n} \\ |\mathbf{m}|<|e_{(\mfl,0)}| + \eta}} \bfpi_{xz e_{(\mfl,0)}}^{(\mathbf{m})} \tbinom{\mathbf{m}}{\mathbf{n}} (y-z)^{\mathbf{m} - \mathbf{n}} + \pi_{zy e_{(\mfl,0)}}^{(\mathbf{n})}\Big) (\cdot - y)^\mathbf{n}.
	\end{align*}
	Comparing coefficients we therefore obtain for every $\mathbf{n}$
	\begin{align*}
		\bfpi_{xy e_{(\mfl,0)}}^{(\mathbf{n})} = \sum_{\substack{\mathbf{m}\geq \mathbf{n} \\ |\mathbf{m}|<|e_{(\mfl,0)}| + \eta}} \bfpi_{xz e_{(\mfl,0)}}^{(\mathbf{m})} \tbinom{\mathbf{m}}{\mathbf{n}} (y-z)^{\mathbf{m} - \mathbf{n}} + \pi_{zy e_{(\mfl,0)}}^{(\mathbf{n})}.
	\end{align*}
	}
	Since $(\Gamma^* - \id)_{e_{(\mfl,0)}}^{\gamma\in\mcN}=0$, which can be seen from \eqref{e00}, this implies the equality
	\referee{\begin{equation*}
		\bfpi_{xy e_{(\mfl,0)}}^{(\mathbf{n})} = \sum_{\substack{\mathbf{m}\geq \mathbf{n} \\ |\mathbf{m}|<|e_{(\mfl,0)}| + \eta}} \bfpi_{xz e_{(\mfl,0)}}^{(\mathbf{m})} \tbinom{\mathbf{m}}{\mathbf{n}} (y-z)^{\mathbf{m} - \mathbf{n}} + \sum_{\gamma\in\mcN} (\Gamma_{xz}^*)_{e_{(\mfl,0)}}^{\gamma}\pi_{zy e_{\gamma}}^{(\mathbf{n})},
	\end{equation*}}
	which is \eqref{ind07}.
	
	\medskip
	
	We now turn to the induction step. We know by Lemma \ref{lemind01} that $(\Gamma_{\bfPi_x}^{*})_\beta^{\gamma\in\mcNmin}$ only depends on $\partial^\mathbf{n} \Pi_{x \beta'}$ with $\beta'<\beta$ and polynomials, and thus we can construct and estimate it by the induction hypothesis in form of \eqref{mod70} via \eqref{sg30}:
	\begin{equation*}
		|(\Gamma_{\bfPi_x(z)}^*)_\beta^{\gamma\in\mcNmin}| \lesssim |z-x|^{|\beta| - |\gamma|}.
	\end{equation*}
	Therefore,
	\begin{equation*}
		|(\Gamma_{\bfPi_x(z)}^* \xi_\mfl(z) \z_{(\mfl,0)})_\beta| \lesssim |z-x|^{|\beta| - |e_{(\mfl,0)}|}|\xi_\mfl(z)| \lesssim |z-x|^\beta,
	\end{equation*} 
	as well as
	\begin{equation*}
		|(\Gamma_{\bfPi_x(z)}^* c)_\beta| \lesssim \sum_\gamma |z-x|^{|\beta| - |\gamma|} |c_\gamma| \lesssim |z-x|^{|\beta|},
	\end{equation*}
	where we used \eqref{cou01b} and $|z-x|\leq 1$. Combining these two estimates we have the first item in \eqref{est50}, i.~e.
	\begin{equation*}
		|\Pi_{x \beta}^-(z)|\lesssim |z-x|^{|\beta|}.
	\end{equation*}
	We now consider
	\begin{align}\label{model}
		\tilde{\Pi}_{x\beta} = K*\Pi_{x\beta}^-,\,\Pi_{x \beta} = \tilde{\Pi}_{x \beta} - \rmT_x^{|\beta| + \eta} \tilde{\Pi}_{x \beta},		
	\end{align}
	which again by \eqref{schau01} implies \eqref{mod70}. Turning our attention to $\Gamma_{xy}$, we know by Lemma \ref{lemind02} that $(\Gamma_{xy}^*)_\beta^{\gamma\in\mcNmin}$ depends only on $\bfpi_{xy \beta'}^{(\mathbf{n})}$ with $\beta' < \beta$ and polynomials, so by the induction hypothesis \eqref{gam04} we have \eqref{gam01} for $\gamma\in \mcNmin$. Moreover, by the induction hypothesis in the form \eqref{ind07bis}, and \eqref{model}, it holds that
	\begin{equation*}
		(\Pi_x - \Gamma_{xy}^*\proj_\mcN\Pi_y)_\beta = 0\,\,\textnormal{mod polynomial of degree }<|\beta|+\eta.
	\end{equation*}
	Expanding this polynomial around $y$, we get the coefficients
	\begin{equation*}
		(\Pi_x - \Gamma_{xy}^*\proj_\mcN\Pi_y)_\beta = \sum_{|\mathbf{n}|<|\beta| + \eta} \bfpi_{xy \beta}^{(\mathbf{n})}(\cdot-y)^\mathbf{n}.
	\end{equation*}
	and, by the three-point argument, cf. \cite[Subsection 4.5]{LO22}, the corresponding estimate \eqref{gam04}, which in turn implies \eqref{gam01} for $\gamma\in\mcP$. It only remains to show the composition rule, i.~e. \eqref{ind07}. We first note that by the binomial formula
	\begin{align*}
		&\sum_{\mathbf{n}\in\N_0^d}\big(\bfpi_{xy}^{(\mathbf{n})} - \big(\Gamma_{xz}^* \bfpi_{zy}^{(\mathbf{n})} + \sum_{\mathbf{m}\geq \mathbf{n}} \tbinom{\mathbf{m}}{\mathbf{n}}\bfpi_{xz}^{(\mathbf{m})}(y-z)^{\mathbf{m}-\mathbf{n}}\big)\big)_\beta (\cdot - y)^\mathbf{n} \\&= \sum_{\mathbf{n}\in\N_0^d} \bfpi_{xy\beta}^{(\mathbf{n})} (\cdot-y)^\mathbf{n} - \sum_{\mathbf{n}\in\N_0^d} (\Gamma_{xz}^*\bfpi_{zy}^{(\mathbf{n})})_\beta (\cdot-y)^\mathbf{n} - \sum_{\mathbf{n}\in\N_0^d} \bfpi_{xz \beta}^{(\mathbf{n})} (\cdot - z)^\mathbf{n}.
	\end{align*}
	By construction, the r.~h.~s. is given by
	\begin{align*}
		&\Pi_{x\beta} - (\Gamma_{xy}^*\proj_\mcN\Pi_y)_\beta - (\Gamma_{xz}^*\proj_\mcN(\Pi_z - \Gamma_{zy}^*\proj_\mcN \Pi_y))_\beta - \Pi_{x\beta} + (\Gamma_{xz}^*\proj_\mcN\Pi_z)_\beta \\
		&= -(\Gamma_{xy}^*\proj_\mcN\Pi_y - \Gamma_{xz}^*\proj_\mcN \Gamma_{zy}^*\proj_\mcN \Pi_y)_\beta\\
		&= -((\Gamma_{xy}^* - \Gamma_{xz}^*\Gamma_{zy}^*)\proj_\mcN\Pi_y)_\beta,
	\end{align*}
	where in the last step we used $\proj_\mcN \Gamma_{zy}^* \proj_\mcN = \Gamma_{zy}^* \proj_\mcN$ as a consequence of \eqref{map32}. Thanks to Lemma \ref{lemind02} we can use the induction hypothesis \eqref{ind07} to conclude that this expression vanishes.
\end{proof}	

\medskip

\subsection{The induction for renormalization}\label{subsec::4.4}
The previous result allows to build a smooth model for a modified version of the original equation, but requires knowing \textit{a priori} the counterterm. This is of course not realistic, since one is expected to choose the renormalization constants while building and estimating the model, in a way that is stable when the regularization is removed. The component-wise order of multi-indices then has a flaw, which is that the constants $c_{\gamma}$ appearing on the r.~h.~s. of \eqref{est01} do not satisfy $\gamma<\beta$; indeed, for this we would need that $(\Gamma_{\bfPi_x}^{*})_\beta^\gamma$ $\neq$ $0$ implies $\gamma<\beta$, which is not true in general and can be easily seen computing 
\begin{equation*}
	(\Gamma_{\bfPi_x}^{*})_{e_{(\mfl',0)} + e_{(\mfl,e_\mathbf{0})}}^{e_{(\mfl,0)}} = \Pi_{x e_{(\mfl',0)}} \neq 0.
\end{equation*}
Thus, we need to accommodate our induction to this situation, finding a new ordering $\prec$ such that $\Pi_{x\beta}^-$ depends only on characters and renormalization constants for $\beta'\prec \beta$. The goal of this subsection is to describe a strategy which leads to such an ordering, but it is clearly not unique and might need modifications depending on the renormalization procedure; cf. e.~g. \cite[Subsection 3.5]{LOTT} for a suitable ordering in the quasi-linear case. We will focus on three quantities related to multi-indices:
\begin{itemize}
	\item The length $\length{\beta}$;
	\item the noise homogeneity $\lnh\beta\rnh\,$;
	\item the polynomial degree $\lpol\beta\rpol := \sum_\mathbf{n} |\mathbf{n}|\beta(\mathbf{n})$.
\end{itemize}
Note that we have the following lower bounds:
\begin{align*}
	&\length{\beta} \geq 1\,\mbox{ for all }\beta\in\populated\cup\mcNmin;\\
	&\lnh\beta\rnh \geq 0\,\mbox{ for all }\beta\in\populated\cup\mcNmin,\,\mbox{ and } \lnh\beta\rnh = 0 \iff \beta\in\mcP\cup\mcNmin;\\
	&\lpol\beta\rpol \geq 0\,\mbox{ for all }\beta\in\populated\cup\mcNmin.
\end{align*}
We look for a convex combination of the three quantities which allows us to get the right inductive structure; this means our goal is to find $\lambda_1$, $\lambda_2$, $\lambda_3$ $\referee{>} 0$ such that $\lambda_1+\lambda_2+\lambda_3 = 1$ and build
\begin{equation}\label{ord01}
	|\beta|_\prec := \lambda_1 \length{\beta} + \lambda_2\, \lnh\beta\rnh + \lambda_3 \lpol\beta\rpol
\end{equation}
so that the construction is triangular in $|\cdot|_\prec$.
\medskip

We start with the following observation. Let $(\gamma',\mathbf{n}')\in \mcM$, and let $\beta, \gamma\in M(\coord)$. Consider $(\z^{\gamma'}D^{(\mathbf{n}')})_\beta^\gamma$ and recall the two conditions under which the matrix entry is non-vanishing, cf. \eqref{pol05}. In the case $\beta = \gamma - e_{(\mfl,k)} + e_{(\mfl,k + e_{\mathbf{n}'})} + \gamma'$, we have
\begin{align}
	\length{\beta} &= \length{\gamma} + \length{\gamma'} \geq \length{\gamma}+1,\label{ord08}\\
	\lnh \beta\rnh &= \lnh\gamma\rnh + \lnh\gamma'\rnh \geq \lnh\gamma\rnh+1,\label{ord09}\\
	\lpol\beta\rpol &= \lpol\gamma\rpol + \lpol\gamma'\rpol \geq \lpol\gamma\rpol.\label{ord10}
\end{align}
Therefore any choice of $\lambda_i$ implies $|\gamma|_\prec<|\beta|_\prec$. In the case $\beta = \gamma-e_{\mathbf{n}'} + \gamma'$, we have
\begin{align}
	\length{\beta} &= \length{\gamma} -1 + \length{\gamma'} \geq \length{\gamma},\label{ord18}\\
	\lnh \beta\rnh &= \lnh\gamma\rnh + \lnh\gamma'\rnh \geq \lnh\gamma\rnh + 1,\label{ord19}\\
	\lpol\beta\rpol &= \lpol\gamma\rpol - |\mathbf{n}'| + \lpol\gamma'\rpol \geq \lpol\gamma\rpol-|\mathbf{n}'|.\label{ord20}
\end{align}
Thus,
\begin{equation*}
	|\beta|_\prec \geq |\gamma|_\prec + \lambda_2 - \lambda_3 |\mathbf{n}'|.
\end{equation*}
In order to get the desired triangularity property, we need
\begin{equation*}
	\lambda_2 > \lambda_3 |\mathbf{n}'|,
\end{equation*}
which of course cannot hold for every $\mathbf{n}'$. However, if we restrict to $(\gamma',\mathbf{n}')\in\mcMm$, cf. \eqref{mod01}, it is enough to have
\begin{equation}\label{ord02}
	\lambda_2 \geq \referee{\lambda_3 \eta}.
\end{equation}
We now turn to $\{\mbpartial_i\}_{i=1,...,d}$. We fix $\mathbf{n}\in\N_0^{d}$ and look at \eqref{pol05} for 
\begin{equation*}
	(\z_{\mathbf{n}+e_i} D^{(\mathbf{n})})_\beta^\gamma \neq 0.
\end{equation*} 
Similar arguments, replacing $\gamma'$ with $e_{\mathbf{n}+e_i}$, yield
\begin{align*}
	\length{\beta} &\geq \length{\gamma},\\
	\lnh\beta\rnh &= \lnh \gamma\rnh,\\
	\lpol\beta\rpol &\geq \lpol \gamma\rpol + |e_i|> \lpol \gamma \rpol,
\end{align*}
and thus any choice of $\lambda_i$'s is valid for the triangularity property. We summarize all of this in the following statement:
\begin{lemma}
	Let \eqref{ord01} be such that \eqref{ord02} holds. Then for all $D\in\mcD^-$, cf. \eqref{mod50},
	\begin{equation*}
		D_\beta^\gamma \neq 0 \implies |\beta|_\prec > |\gamma|_\prec.
	\end{equation*}
\end{lemma}
As a consequence, this triangularity is inductively propagated to $\msfD_{(J,\mathbf{m})}$ for $(J,\mathbf{m})$ $\in$ $M(\mcMm)\times \N_0^{d}$, which in turn implies the following.
\begin{corollary}
	Let $\bff\in\textnormal{Alg}(T^-,\R)$, and let $\Gamma_\bff^-$ be given by \eqref{mod11}. Then
	\begin{equation*}
		(\Gamma_\bff^{-*}-\id)_\beta^\gamma \neq 0 \implies |\beta|_\prec > |\gamma|_\prec.
	\end{equation*}
\end{corollary}	

\medskip

However, this tells nothing about the triangular structure of the elements of $G$. The main difference now is that \eqref{ord02} is not enough by itself to guarantee the triangularity, since there is no maximal $\mathbf{n}$ in $\mcMp$; we need to exploit the condition in \eqref{sg17} instead. To this end, we rewrite the homogeneity $|\gamma'|$ using \eqref{bb06} as follows:
\referee{
\begin{align*}
	|\gamma'| = \sum_{(\mfl,k)} (\alpha_\mfl + \eta) \gamma'(\mfl,k) - \sum_{(\mfl,k)} \sum_\mathbf{n} |\mathbf{n}| k(\mathbf{n}) \gamma'(\mfl,k) + \lpol \gamma' \rpol - \eta.
\end{align*}
Let us denote $\alphamax := \max_{\mfl\in \mfL^-\cup \{0\}} \alpha_\mfl$. Assuming $\alpha_0 = 0$ we have\footnote{In fact, under the stronger assumption \eqref{ref:eqalpha} used in the analytic construction, $\alphamax=0$.} $\alphamax \geq0$ and we have the bound
\begin{align*}
	|\gamma'| \leq \eta \sum_{k} \gamma'(0,k) + (\eta + \alphamax) \lnh \gamma' \rnh + \lpol \gamma' \rpol - \eta.
\end{align*}
Thus for a pair $(\gamma',\mathbf{n}')\in\mcMp$ it holds
\begin{align*}
	|\mathbf{n}'| < |\gamma'| + \eta \leq \eta \length{\gamma'} + (\eta + \alphamax) \lnh \gamma' \rnh + \lpol \gamma' \rpol.
\end{align*}
We use this together with \eqref{ord08} to \eqref{ord20}, so that for $\beta,\gamma\in M(\coord)$ and $(\gamma',\mathbf{n}')\in \mcMp$ such that $(\z^{\gamma'} D^{(\mathbf{n}')})_\beta^\gamma \neq 0$ it holds
\begin{align*}
	|\beta|_\prec &\geq |\gamma|_\prec  + \lambda_1 \length{\gamma'} + \lambda_2\hspace*{2pt} \lnh\gamma'\rnh + \lambda_3 \lpol\gamma'\rpol - \lambda_1 - \lambda_3 |\mathbf{n}'|\\
	&> |\gamma|_\prec + (\lambda_1 - \lambda_3 \eta)\length{\gamma'} + (\lambda_2 - \lambda_3(\eta+\alphamax))\hspace*{2pt}\lnh\gamma'\rnh - \lambda_1.
\end{align*}
}
To get the desired inequality $|\beta|_\prec > |\gamma|_\prec$, it is enough to choose $\lambda_i$'s such that
\referee{\begin{equation}\label{ord11}
	\left\{\begin{array}{l}
		\lambda_2-\lambda_3\eta\geq\lambda_1\\
		\lambda_1 - \lambda_3(\eta + \alphamax)\geq 0,
		\end{array}\right.
\end{equation}}
which we can do, for instance by letting
\referee{\begin{align*}
	\lambda_1 &:= \lambda_3 \eta\\
	\lambda_2 &:= \lambda_3 (2\eta + \alphamax).
\end{align*}}
\begin{lemma}
	Let \eqref{ord01} be such that \eqref{ord11} holds. Then for all $D\in\mcD^+$ 
	\begin{equation*}
		D_\beta^\gamma \neq 0 \implies |\beta|_\prec > |\gamma|_\prec.
	\end{equation*}
\end{lemma}	
As before, this triangularity is inductively propagated to $G$.
\begin{corollary}
	For every $\Gamma\in G$,
	\begin{equation*}
		(\Gamma^* - \id)_\beta^\gamma\neq 0 \implies |\beta|_\prec > |\gamma|_\prec.
	\end{equation*}
\end{corollary}	
It only remains to show that the dependence on the characters for both $\Gamma_{\bfPi_x}$ and $\Gamma_{xy}$ is triangular with respect to $|\cdot|_\prec$. This is trivial, since
\begin{equation*}
	\beta > \gamma \implies |\beta|_\prec > |\gamma|_\prec,
\end{equation*}
and we established triangularity with respect to the component-wise ordering in lemmas \ref{lemind01} and \ref{lemind02}, respectively.
\section{\referee{Examples of renormalized equations}}\label{section::5}
In this last section we implement \referee{our top-down algebraic renormalization strategy based on multi-indices} in \referee{some} classical examples of singular SPDEs of the form \eqref{set01}, namely the $\Phi_3^4$ model, the multiplicative stochastic heat equation and the generalized KPZ equation, \referee{and show that we recover the same renormalized equations with fewer renormalization constants}. In all three cases, we \referee{will construct the corresponding index set following the procedure of \eqref{ref:subsecwarm},} identify the family of renormalized equations generated by admissible counterterms \referee{as well as the model equations which generate the renormalization constants}, and discuss some possible reductions \referee{based on preservation of symmetries}. \referee{For better comparison, we recommend the reader some familiarity with the aforementioned equations and their renormalized versions, see e.~g. \cite{Hairer16} for $\Phi_3^4$, \cite{HP15} for the stochastic heat equation, and \cite{BGHZ} for the generalized KPZ equation.}
\subsection{The dynamical $\Phi_3^4$ model}\label{subsec::5.1}
The dynamical $\Phi_d^4$ model \cite{PW}, where $d$ stands for the spatial dimension, is the SPDE
\begin{equation}\label{phi43}
	(\partial_t - \Delta)\Phi = -\lambda \Phi^3 + \xi;
\end{equation}
where $\lambda>0$ and $\xi$ is space-time white noise (thus in $1+d$ dimensions). It is known that this equation is singular for $d\geq 2$ and subcritical for $d=2,3$, so we focus on the latter. For our purposes it is better to consider a generic third order polynomial nonlinearity, i.~e.
\begin{equation*}
	(\partial_t - \Delta)\Phi = \sum_{j=0}^3 \lambda_j \Phi^j + \lambda_\xi \xi.
\end{equation*}
We fix the following:
\begin{itemize}
	\item Parabolic scaling, i.~e. we take \eqref{scal01} as
	\begin{equation*}
		|\mathbf{n}| = 2 \mathbf{n}(0) + \sum_{i=1}^{d} \mathbf{n}(i),
	\end{equation*}
	under which the heat operator is $2$-homogeneous.
	\item In line with the above, $\eta = 2$.
	\item We represent $\xi$ with the same symbol, i.~e. $\mfL^- = \{\xi\}$.
	\item We set $\alpha_\xi = -\frac{5}{2}\mhyphen$.
\end{itemize}
The set of $(\mfl,k)\in \{\xi,0\}\times \N_0^{4}$ to consider for the multi-index description then reduces to five, namely $\{(\xi,0)\}$ $\cup$ $\{(0,k)\}_{k=0,...,3}$. We follow \eqref{bb03} and identify
\begin{align*}
	\z_{(\xi,0)}[\bflambda,\Phi,\cdot] &= \lambda_\xi,\\
	\z_{(0,0)}[\bflambda,\Phi,\cdot] &= \lambda_0 + \lambda_1 \Phi + \lambda_2 \Phi^2 + \lambda_3 \Phi^3,\\
	\z_{(0,1)}[\bflambda,\Phi,\cdot] &= \lambda_1 + 2\lambda_2 \Phi + 3 \lambda_3 \Phi^2,\\
	\z_{(0,2)} [\bflambda,\Phi,\cdot] &= \lambda_2 + 3 \lambda_3 \Phi,\\
	\z_{(0,3)} [\bflambda,\Phi,\cdot] &= \lambda_3,
\end{align*}
so that for a generic $\beta$ we have
\begin{align}
	\z^\beta[\bflambda,\Phi,\cdot] = &\lambda_\xi^{\beta(\xi,0)} \big(\lambda_0 + \lambda_1 \Phi + \lambda_2 \Phi^2 + \lambda_3 \Phi^3\big)^{\beta(0,0)} \big(\lambda_1 + 2\lambda_2 \Phi + 3 \lambda_3 \Phi^2\big)^{\beta(0,1)}\nonumber\\
	& \times \big(\lambda_2 + 3 \lambda_3 \Phi\big)^{\beta(0,2)} \lambda_3^{\beta(0,3)} \prod_{\mathbf{n}\in\N_0^4} (\partial^\mathbf{n}\Phi)^{\beta(\mathbf{n})}.\label{flow01}
\end{align}
The set of multi-indices satisfying condition \eqref{bb06}  is characterized by the identity
\begin{equation}\label{phi01}
	\beta(\xi,0) + \beta(0,0) - \beta(0,2) - 2\beta(0,3) + \sum_{\mathbf{n}\in\N_0^4} \beta(\mathbf{n}) = 1.
\end{equation}
Condition \eqref{bb07} simply means that we restrict to $\beta$ such that either $\beta(\xi,0)>0$ or $\beta = e_\mathbf{n}$ for some $\mathbf{n}\in\N_0^4$. Since the expected regularity of the solution $\Phi$ is $-\frac{1}{2}\mhyphen <0$, there is no need to restrict the polynomial contributions as in \eqref{bb05}. We thus have
\begin{equation*}
	\mcN = \{\beta \in M(\{(\xi,0)\}\cup \{(0,k)\}_{k=0,...,3} \cup \N_0^4)\,|\, \beta(\xi,0)>0,\, \beta\,\mbox{satisfies }\eqref{phi01}\}.
\end{equation*}
The homogeneity of a multi-index $\beta$ is given according to \eqref{bb04} by
\begin{equation}\label{phi02}
	|\beta| = (-\frac{5}{2}\mhyphen)\beta(\xi,0) + 2\beta(0,1) + 4\beta(0,2) + 6\beta(0,3) + \sum_{\mathbf{n}\in\N_0^4} (|\mathbf{n}| - 2)\beta(\mathbf{n}).
\end{equation}
\begin{remark}
	This bookkeeping of the model components has some similarities with that of the effective force coefficients in the flow approach to singular SPDEs developed in \cite{Duch21}. To establish the connection more precisely, let us momentarily fix $\lambda_\xi =1$, $\lambda_3 = \lambda$ and $\lambda_i = 0$ for $i=0,1,2$, following \cite{Duch21}. The effective force coefficients $F_{\kappa,\nu,\mu}^{i,m}$ as defined in \cite[Section 7]{Duch21} are indexed by $i\in\N_0$ and $m\in\N_0$, representing the order in $\lambda$ and $\Phi$, respectively. Looking at the definition of $\z_{(0,3)}$ and $\z_0$, one may think that $i$ is $\beta(0,3)$ in our approach, whereas $m$ is given by $\beta(0)$. This is actually not entirely true in view of \eqref{flow01}: The order in $\lambda$ and $\Phi$ also involves contributions from $e_{(0,j)}$, $j=0,1,2$. However, such contributions are clearly redundant, and this has the effect of creating different model components which look the same (see e.~g. \eqref{flow02} and \eqref{flow04} below). Thus, the index set of \cite{Duch21} is more efficient than ours. Of course, one could remove the contributions from $e_{(0,j)}$, $j=0,1,2$ from the very beginning and start the multi-index description merely based on $e_{(0,3)}$ and $e_0$ (note that then the contributions from $e_{(\xi,0)}$ are fixed by condition \eqref{phi01}), making the connection with \cite{Duch21} clear. \referee{This is the approach taken in the more recent \cite{BOT24}, leading to a treatment of the equation which is more direct and similar to \cite{LOT,LOTT,OSSW}. Such} a priori reductions, which can be made for specific equations, typically generate smaller index sets and provide a more parsimonious bookkeeping, but are not so useful for the systematic approach adopted in this paper, which is developed to cover the larger class of equations of the form \eqref{set01}.	
\end{remark}	

\medskip

We now describe admissible counterterms. For this we first look at the set $\mcC$, cf. \eqref{setC}. As a consequence of \eqref{phi01} and \eqref{phi02}, multi-indices for which $|\beta|<0$ are characterized by
\begin{equation}\label{phi03}
	5\beta(0,0) + 4\beta(0,1) + 3\beta(0,2) + 2\beta(0,3) + 2\sum_{\mathbf{n}\in\N_0^4} |\mathbf{n}|\beta(\mathbf{n}) \leq 5 - \sum_{\mathbf{n}\in\N_0^4} \beta(\mathbf{n}).
\end{equation}
We already observe that  this forces $\beta(0,k)\leq 1$ for $k=0,1,2$ and $\beta(0,3)\leq 2$. Moreover, $\beta(0,0)=0$; indeed, if $\beta(0,0)=1$, then we need $\beta(0,k)=0$ for $k=1,2,3$, which by \eqref{phi01} leaves us with $\beta = e_{(0,0)}\notin \mcN$. For the polynomial contributions, we note the following:
\begin{itemize}
	\item If $\sum_{\mathbf{n}\in\N_0^4} \beta(\mathbf{n})\geq 6$, \eqref{phi03} is impossible.
	\item If $\sum_{\mathbf{n}\in\N_0^4} \beta(\mathbf{n}) = 4,5$, then \eqref{phi01} and \eqref{phi03} are not compatible.
	\item If $\sum_{\mathbf{n}\in\N_0^4} \beta(\mathbf{n})=3$, then \eqref{phi01} and \eqref{phi03} only allow for $\beta = e_{(0,3)} + 3e_0 \notin \mcN$.	
	\item  If $\sum_{\mathbf{n}\in\N_0^4} \beta(\mathbf{n})=2$, then \eqref{phi03} gives only two possibilities:
	\begin{itemize}
		\item $\beta(0,2)=1$, $\beta(0,k)=0$ for $k=0,1,3$. Then \eqref{phi01} and \eqref{phi03} only allow for $\beta = e_{(0,2)} + 2e_0\notin\mcN$.
		\item $\beta(0,3)=1$, $\beta(0,k)=0$ for $k=0,1,2$. Then \eqref{phi01} and \eqref{phi03} only allow for $\beta = e_{(\xi,0)} + e_{(0,3)} + 2e_0$, which does not satisfy \eqref{cou01c}.
	\end{itemize}
\end{itemize}
Thus, we only need to consider $\sum_{\mathbf{n}\in\N_0^4} \beta(\mathbf{n}) = 0,1$. We now distinguish the two possible cases.
\begin{enumerate}[label= \bf{\arabic*}. ]
	\item Case $\sum_{\mathbf{n}\in\N_0^4} \beta(\mathbf{n})=0$. If $\beta(0,1)=1$, \eqref{phi01} and \eqref{phi03} imply that the only possibility is $\beta=e_{(\xi,0)} + e_{(0,1)}$, which does not satisfy \eqref{cou01c}. If $\beta(0,1)=0$ and $\beta(0,2)=1$, then \eqref{phi01} and \eqref{phi03} allow for
	\begin{equation}\label{phi*1}
		2 e_{(\xi,0)} + e_{(0,2)},\, 4e_{(\xi,0)} + e_{(0,2)} + e_{(0,3)}.
	\end{equation}
	Finally, if $\beta(0,1)= \beta(0,2) = 0$, we have
	\begin{equation}\label{phi*2}
		3 e_{(\xi,0)} + e_{(0,3)},\, 5 e_{(\xi,0)} + 2 e_{(0,3)}.
	\end{equation}
	\item Case $\sum_{\mathbf{n}\in\N_0^4} \beta(\mathbf{n}) = 1$. If $\beta(0,1)=1$, only $e_{(0,1)} + e_0\notin \mcN$ is allowed by \eqref{phi01} and \eqref{phi03}. If $\beta(0,1)=0$ and $\beta(0,2)=1$, we can only take $e_{(\xi,0)} + e_{(0,2)} + e_0$, which does not satisfy \eqref{cou01c}. Finally, for $\beta(0,1)=\beta(0,2) =0$ we only have
	\begin{equation}\label{phi*3}
		2e_{(\xi,0)} + e_{(0,3)} + e_0,\,\,4e_{(\xi,0)} + 2e_{(0,3)} + e_0,\,\,\{2 e_{(\xi,0)} + e_{(0,3)}+e_{e_i}\}_{i=1,2,3}.
	\end{equation}
\end{enumerate}
We therefore have that $\mcC$ consists of the multi-indices \eqref{phi*1}, \eqref{phi*2} and \eqref{phi*3}. An admissible counterterm $c\in\counterterms$ then produces the renormalized equation
\begin{align*}
	(\partial_t - \Delta)\Phi = & \sum_{j=0}^3 \lambda_j \Phi^j + \lambda_\xi \xi \\
	&+ c_{2e_{(\xi,0)} + e_{(0,2)}} \lambda_\xi^2 (\lambda_2 + 3\lambda_3\Phi)\\
	&+ c_{4e_{(\xi,0)} + e_{(0,2)} + e_{(0,3)}} \lambda_\xi^4 \lambda_3(\lambda_2+3\lambda_3\Phi)  \\
	&+ c_{3 e_{(\xi,0)} + e_{(0,3)}} \lambda_\xi^3 \lambda_3\\
	&+ c_{5e_{(\xi,0)} + 2e_{(0,3)}} \lambda_\xi^5 \lambda_3^2 \\
	&+ c_{2e_{(\xi,0)} + e_{(0,3)} + e_0} \lambda_\xi^2 \lambda_3 \Phi\\
	&+ c_{4e_{(\xi,0)} + 2 e_{(0,3)} + e_0} \lambda_\xi^4 \lambda_3^2 \Phi\\
	&+ \sum_{i=1}^3 c_{2e_{(\xi,0)}+e_{(0,3)} + e_{e_i}} \lambda_\xi^2 \lambda_3\partial_i \Phi.
\end{align*}
We now note from the canonical model equations ($c=0$) that
\begin{equation}\label{flow02}
	(\partial_t -\Delta)\Pi_{2 e_{(\xi,0)} + e_{(0,3)}+ e_0} = 3\Pi_{e_{(\xi,0)}}^2 = (\partial_t -\Delta)3\Pi_{2 e_{(\xi,0)} + e_{(0,2)}},
\end{equation}
so we have $\Pi_{2 e_{(\xi,0)} + e_{(0,3)}+ e_0}$ $=$ $3\Pi_{2 e_{(\xi,0)} + e_{(0,2)}}$ and may assume $c_{2e_{(\xi,0)} + e_{(0,3)} + e_0} = 3 c_{2e_{(\xi,0)} + e_{(0,2)}}$. Similarly, 
\begin{align}
	(\partial_t -\Delta)\Pi_{4 e_{(\xi,0)} + 2e_{(0,3)}+ e_0} &= 6 \Pi_{3 e_{(\xi,0)} + e_{(0,3)}} \Pi_{e_{(\xi,0)}} + 3 \Pi_{2 e_{(\xi,0)} + e_{(0,3)} + e_0} \Pi_{e_{(\xi,0)}}^2\nonumber\\
	&= 6 \Pi_{3 e_{(\xi,0)} + e_{(0,3)}} \Pi_{e_{(\xi,0)}} + 9 \Pi_{2 e_{(\xi,0)} + e_{(0,2)}} \Pi_{e_{(\xi,0)}}^2\nonumber\\
	&= (\partial_t -\Delta)3\Pi_{4 e_{(\xi,0)} + e_{(0,2)} + e_{(0,3)}},\label{flow04}
\end{align}
so we have $\Pi_{4 e_{(\xi,0)} + 2e_{(0,3)}+ e_0} = 3\Pi_{4 e_{(\xi,0)} + e_{(0,2)} + e_{(0,3)}}$ and consequently assume $c_{4 e_{(\xi,0)} + 2e_{(0,3)}+ e_0}$ $=$ $3c_{4 e_{(\xi,0)} + e_{(0,2)} + e_{(0,3)}}$. The renormalized equation simplifies to
\begin{align*}
	(\partial_t - \Delta)\Phi = & \sum_{j=0}^3 \lambda_j \Phi^j + \lambda_\xi \xi \\
	&+ c_{2e_{(\xi,0)} + e_{(0,2)}} \lambda_\xi^2 (\lambda_2 + 6\lambda_3\Phi)\\
	&+ c_{4e_{(\xi,0)} + e_{(0,2)} + e_{(0,3)}} \lambda_\xi^4 \lambda_3(\lambda_2+6\lambda_3\Phi)  \\
	&+ c_{3 e_{(\xi,0)} + e_{(0,3)}} \lambda_\xi^3 \lambda_3\\
	&+ c_{5e_{(\xi,0)} + 2e_{(0,3)}} \lambda_\xi^5 \lambda_3^2 \\
	&+ \sum_{i=1}^3 c_{2e_{(\xi,0)}+e_{(0,3)} + e_{e_i}} \lambda_\xi^2 \lambda_3\partial_i \Phi.
\end{align*}
In addition, we now include a spatial symmetry assumption in our counterterm, since both $\xi$ and $(\partial_t - \Delta)$ are reflection-invariant in space; in that case, only even polynomial contributions are required in the counterterm, making the last summand disappear:
\begin{align}
	(\partial_t - \Delta)\Phi = & \sum_{j=0}^3 \lambda_j \Phi^j + \lambda_\xi \xi \nonumber\\
	&+ c_{2e_{(\xi,0)} + e_{(0,2)}} \lambda_\xi^2 (\lambda_2 + 6\lambda_3\Phi)\nonumber\\
	&+ c_{4e_{(\xi,0)} + e_{(0,2)} + e_{(0,3)}} \lambda_\xi^4 \lambda_3(\lambda_2+6\lambda_3\Phi)\nonumber  \\
	&+ c_{3 e_{(\xi,0)} + e_{(0,3)}} \lambda_\xi^3 \lambda_3\nonumber\\
	&+ c_{5e_{(\xi,0)} + 2e_{(0,3)}} \lambda_\xi^5 \lambda_3^2. \label{phi10}
\end{align}
We are left with four renormalization constants, fixed in the following model equations:
\begin{align}
	(\partial_t - \Delta)\Pi_{2e_{(\xi,0)} + e_{(0,2)}} &= \Pi_{e_{(\xi,0)}}^2 + c_{2e_{(\xi,0)} + e_{(0,2)}},\label{phi12}\\
	(\partial_t - \Delta) \Pi_{3e_{(\xi,0)} + e_{(0,3)}} &= \Pi_{e_{(\xi,0)}}^3 + c_{3e_{(\xi,0)} + e_{(0,3)}}\nonumber\\
	&\quad + 3 c_{2e_{(\xi,0)} + e_{(0,2)}} \Pi_{e_{(\xi,0)}},\label{phi13}\\
	(\partial_t - \Delta)\Pi_{4e_{(\xi,0)} + e_{(0,2)} + e_{(0,3)}} &= 2 \Pi_{e_{(\xi,0)}}\Pi_{3e_{(\xi,0)} + e_{(0,3)}} + 3\Pi_{e_{(\xi,0)}}^2 \Pi_{2e_{(\xi,0)} + e_{(0,2)}} \nonumber \\
	&\quad+ c_{4e_{(\xi,0)} + e_{(0,2)} + e_{(0,3)}}\nonumber \\
	&\quad+ 3 c_{2 e_{(\xi,0)} + e_{(0,2)}} \Pi_{2e_{(\xi,0)} + e_{(2,0)}}\label{phi14}\\
	(\partial_t - \Delta)\Pi_{5e_{(\xi,0)} + 2e_{(0,3)}} &= 3 \Pi_{e_{(\xi,0)}}^2\Pi_{3e_{(\xi,0)} + e_{(0,3)}} + c_{5e_{(\xi,0)} + 2e_{(0,3)}}\nonumber\\
	&\quad + 3c_{4 e_{(\xi,0)} + e_{(0,2)} + e_{(0,3)}} \Pi_{e_{(\xi,0)}}\nonumber\\
	&\quad +  3 c_{2 e_{(\xi,0)} + e_{(0,2)}}\Pi_{3e_{(\xi,0)} + e_{(0,3)}}.\label{phi15}	
\end{align}
\begin{remark}
	The reader is invited to compare these four constants with the five-dimensional group in \cite[Subsection 4.5]{Hairer16}, which is required for a general non-Gaussian stationary $\xi$ in the range of regularity of space-time white noise. Note that \eqref{phi14} contains two terms from the tree-based description, so the constant $c_{4e_{(\xi,0)} + e_{(0,2)} + e_{(0,3)}}$ renormalizes the linear combination at once (although in the space-time white noise case $\Pi_{e_{(\xi,0)}}\Pi_{3e_{(\xi,0)} + e_{(0,3)}}$ does not need to be renormalized, see the proof of \cite[Theorem 10.22]{reg}). 
\end{remark}	

\medskip

\referee{If we now incorporate the invariance in law $\xi \mapsto -\xi$, which is true for white noise}, the expected value of polynomial functionals of the noise of odd order do not require a counterterm (because they are already centered in expectation); this means we could disregard new constants coming from multi-indices with $\beta(\xi,0)=$odd, namely those fixed in equations \eqref{phi13} and \eqref{phi15} above. The final form of the renormalized equation is
\begin{align*}
	(\partial_t - \Delta)\Phi = & \sum_{j=0}^3 \lambda_j \Phi^j + \lambda_\xi \xi \\
	&+ c_{2e_{(\xi,0)} + e_{(0,2)}} \lambda_\xi^2 (\lambda_2 + 6\lambda_3\Phi)\\
	&+ c_{4e_{(\xi,0)} + e_{(0,2)} + e_{(0,3)}} \lambda_\xi^4 \lambda_3(\lambda_2+6\lambda_3\Phi);
\end{align*}
here the two constants are fixed in the model equations \eqref{phi12} and \eqref{phi14}, which now reduce to
\begin{align*}
	(\partial_t - \Delta)\Pi_{2e_{(\xi,0)} + e_{(0,2)}} &= \Pi_{e_{(\xi,0)}}^2 + c_{2e_{(\xi,0)} + e_{(0,2)}},\\
	(\partial_t - \Delta)\Pi_{4e_{(\xi,0)} + e_{(0,2)} + e_{(0,3)}} &= 2 \Pi_{e_{(\xi,0)}}\Pi_{3e_{(\xi,0)} + e_{(0,3)}} + 3\Pi_{e_{(\xi,0)}}^2 \Pi_{2e_{(\xi,0)} + e_{(0,2)}}  \\
	&\quad+ c_{4e_{(\xi,0)} + e_{(0,2)} + e_{(0,3)}} \\
	&\quad+ 3 c_{2 e_{(\xi,0)} + e_{(0,2)}} \Pi_{2e_{(\xi,0)} + e_{(2,0)}}.
\end{align*}
\begin{remark}
	Let us compare this to the original tree-based regularity structures, and more specifically to the algebraic renormalization described in \cite[Subsection 9.2]{reg} and \cite[Subsection 4.5]{Hairer16}. To this end, we go back to \eqref{phi43} and set $\lambda = 1$ (i.~e. $\lambda_3 = -1$), so that \eqref{phi10} reduces to
	\begin{align}
		(\partial_t - \Delta)\Phi = & -\Phi^3 + \xi \nonumber\\
		&+6 (-c_{2e_{(\xi,0)} + e_{(0,2)}} +  c_{4e_{(\xi,0)} + e_{(0,2)} + e_{(0,3)}})\Phi. \label{phi11}
	\end{align}
	If we identify $c_{2 e_{(\xi,0)} + e_{(0,2)}}$ $=$ $-\frac{1}{2}C_1$ and $c_{4e_{(\xi,0)} + e_{(0,2)} + e_{(0,3)}}$ $=$ $-\frac{3}{2}C_2$, where $C_1$ and $C_2$ are given in \cite[Subsection 9.2]{reg}, see \cite[(9.3)]{reg}, then \eqref{phi11} coincides with the renormalized equation \cite[(9.21)]{reg}. Recall that the constant $c_{2 e_{(\xi,0)} + e_{(0,2)}}$ renormalizes both  $\Pi_{2e_{(\xi,0)} + e_{(0,2)}}$ and $\Pi_{2 e_{(\xi,0)} + e_{(0,3)} + e_0}$, cf. \eqref{flow02}, thus justifying the presence of the factor $\frac{1}{2}$; analogously, \eqref{flow04} justifies the $\frac{3}{2}$ factor for $c_{4e_{(\xi,0)} + e_{(0,2)} + e_{(0,3)}}$, since it renormalizes both $3\Pi_{e_{(\xi,0)}}^2 \Pi_{2e_{(\xi,0)} + e_{(0,2)}}$ and $3\Pi_{e_{(\xi,0)}}^2 \Pi_{2e_{(\xi,0)} + e_{(0,3)} + e_0}$.
\end{remark}	
\subsection{The multiplicative stochastic heat equation}\label{subsec::5.2}
Consider the equation
\begin{equation}\label{she01}
	(\partial_t - \partial_x^2) u = \sigma(u) \xi,
\end{equation}
where $\xi$ is again space-time white noise. The same scaling set-up as in the previous example works here in $d=1$, and  $\alpha_\xi = -\tfrac{3}{2}\mhyphen$. Note that the expected regularity of the solution this time is $\tfrac{1}{2}\mhyphen>0$, and thus in line with \eqref{bb05}, the polynomial contributions from $\mathbf{n}=(0,0)$ will only appear as purely polynomial. This is possible if we mod out constants in \eqref{she01}, which would make us rewrite the equation formally as
\begin{equation*}
	(\partial_t-\partial_x^2)(u-u(x)) = \sigma(u)\xi = \sum_{k\in \N_0} \tfrac{1}{k!}\sigma^{(k)}(u(x)) (u-u(x))^k\xi.
\end{equation*}
We consider multi-indices over $\N_0$ $\cup$ $\N_0^2\setminus \{(0,0)\}$, where we identify $k\in \N_0$ with $(\xi, ke_0)$. This means
\begin{align*}
	\z_{k} [\bfsigma,u,\cdot] &= \tfrac{1}{k!}\sigma^{(k)}(u)\\
	\z_\mathbf{n} [\bfsigma,u,\cdot] &= \tfrac{1}{\mathbf{n}!} \partial^\mathbf{n} u
\end{align*}
and consequently
\begin{equation*}
	\z^\beta [\bfsigma,u,\cdot] = \prod_{k}\big(\tfrac{1}{k!}\sigma^{(k)}(u)\big)^{\beta(k)} \prod_{\mathbf{n}\in\N_0^2\setminus (0,0)} (\tfrac{1}{\mathbf{n}!}\partial^\mathbf{n}u)^{\beta(\mathbf{n})}.
\end{equation*}
Condition \eqref{bb06} reduces to
\begin{equation}\label{rew01}
	\sum_k (1-k)\beta(k) + \sum_\mathbf{n} \beta(\mathbf{n}) = 1,
\end{equation}	
and the homogeneity \eqref{bb04} is
\begin{equation*}
	|\beta| = \sum_k ((-\tfrac{3}{2}\mhyphen) + 2k)\beta(k) + \sum_{\mathbf{n}} (|\mathbf{n}| - 2)\beta(\mathbf{n}).
\end{equation*}
Inserting \eqref{rew01} we may rewrite it as
\begin{equation*}
	|\beta| = (\tfrac{1}{2}\mhyphen)\sum_k \beta(k) + \sum_\mathbf{n} |\mathbf{n}|\beta(\mathbf{n}) - 2.
\end{equation*}

\medskip

We now characterize $\mcC$. The condition $|\beta|<0$ reduces to
\begin{equation}\label{rew02}
	(\tfrac{1}{2}\mhyphen)\sum_k \beta(k) + \sum_\mathbf{n} |\mathbf{n}|\beta(\mathbf{n}) < 2.
\end{equation}
Clearly this forces $\sum_{\mathbf{n}} |\mathbf{n}|\beta(\mathbf{n})$ $=$ $0,1$. We distinguish the two cases:
\begin{enumerate}[label = \bf{\arabic*}.]
	\item Case $\sum_{\mathbf{n}} |\mathbf{n}|\beta(\mathbf{n}) = 0$. Since \eqref{bb05} removes $\mathbf{n}=(0,0)$, this implies that there are no contributions from $\{e_\mathbf{n}\}$. In addition, \eqref{cou01c} combined with \eqref{rew02} imply $\length{\beta} = 2,3,4$. This leaves us with the following six multi-indices:
	\begin{equation}\label{rew05}
		\{e_{0} + k e_{1}\}_{k=1,2,3} \cup \{2e_0 + k e_1 + e_2\}_{k=0,1} \cup \{3e_0 + e_3\}.
	\end{equation}
	\item Case $\sum_{\mathbf{n}} |\mathbf{n}|\beta(\mathbf{n}) = 1$. The parabolic scaling combined with \eqref{cou01c} and \eqref{rew02} implies that $\beta$ is of the form $\beta' + e_{(0,1)}$, where $\beta'$ does not contain contributions from $\{e_\mathbf{n}\}$ and is of length $2$. This leaves us with
	\begin{equation}\label{rew06}
		e_0 + e_2 + e_{(0,1)}, 2e_1 + e_{(0,1)}.
	\end{equation}
\end{enumerate}
The renormalized versions of \eqref{she01} take the form
\begin{align*}
	(\partial_t - \partial_x^2) u = &\sigma(u) \xi + c_{e_0 + e_1} \sigma'(u)\sigma(u) + c_{e_0 + 2e_1} \sigma'(u)^2 \sigma(u)\\
	&  + \tfrac{1}{2}c_{2e_0 + e_2} \sigma''(u) \sigma(u)^2 + c_{e_0 + 3e_1} \sigma'(u)^3 \sigma(u)\\
	& + \tfrac{1}{2}c_{2e_0 + e_1 + e_2} \sigma''(u)\sigma'(u)\sigma(u)^2 + \tfrac{1}{6}c_{3e_0 + e_3}\sigma'''(u)\sigma(u)^3\\
	&+ \tfrac{1}{2}c_{e_0 + e_2 + e_{(0,1)}} \sigma''(u)\sigma(u)\partial_x u + c_{2e_1 + e_{(0,1)}} \sigma'(u)^2 \partial_x u,
\end{align*}
and are fixed in the model equations
\begin{align*}
	(\partial_t - \partial_x^2) \Pi_{e_0 + e_1} = &\Pi_{e_0}\xi + c_{e_0 + e_1},\\
	(\partial_t - \partial_x^2) \Pi_{e_0 + 2e_1} = &\Pi_{e_0 + e_1}\xi + c_{e_0 + 2e_1} + \Pi_{e_0} c_{e_0 + e_1},\\
	(\partial_t - \partial_x^2) \Pi_{e_0 + 3e_1} = &\Pi_{e_0 + 2e_1}\xi + c_{e_0 + 3e_1} + \Pi_{e_0} c_{e_0 + 2e_1} + \Pi_{e_0 + e_1} c_{e_0 + e_1},\\
	(\partial_t - \partial_x^2) \Pi_{2 e_0 + e_2} = &\Pi_{e_0}^2 \xi + c_{2e_0 + e_2} + 2 \Pi_{e_0} c_{e_0 + e_1},\\
	(\partial_t - \partial_x^2) \Pi_{2 e_0 + e_1 + e_2} = &2 \Pi_{e_0} \Pi_{e_0 + e_1} \xi + \Pi_{2 e_0 + e_2} \xi + c_{2 e_0 + e_1 + e_2} + 2 \Pi_{e_0} c_{2 e_0 + e_2} \\
	&\quad+ 4 \Pi_{e_0} c_{e_0 + 2 e_1} + 2 \Pi_{e_0 + e_1} c_{e_0 + e_1} + 3 \Pi_{e_0}^2 c_{e_0 + e_1},\\
	(\partial_t - \partial_x^2)\Pi_{3 e_0 + e_3} =& \Pi_{e_0}^3 \xi + c_{3e_0 + e_3} + 3 \Pi_{e_0} c_{2 e_0 + e_2} + 3 \Pi_{e_0}^2 c_{e_0 + e_1},\\
	(\partial_t - \partial_x^2) \Pi_{e_0 + e_2 + e_{(0,1)}} = &2 \Pi_{e_0} {\rm X} \xi + c_{e_0 + e_2 + e_{(0,1)}} + 2 {\rm X} c_{e_0 + e_1},\\
	(\partial_t - \partial_x^2) \Pi_{2 e_1 + e_{(0,1)}} = &\Pi_{e_1 + e_{(0,1)}} \xi + c_{2 e_1 + e_{(0,1)}} + {\rm X} c_{e_0 + e_1},
\end{align*}
where we used the shorthand ${\rm X}:= \Pi_{e_{(0,1)}}$.
\begin{remark}
	The set of multi-indices \eqref{rew05}, \eqref{rew06} should be compared with the set of trees of negative homogeneity and at least two noise components of the table \cite[(4.1)]{HP15} (contributions with a single noise come in form of $e_0$ and $e_1+ e_{(0,1)}$). We have eight multi-indices versus nine trees: This is because our multi-index $2e_0 + e_1 + e_2$ encodes a linear combination of two trees according to Lemma \ref{lem:trees}. If we only look at the three-dimensional renormalization group required for the Wong-Zakai-type theorem of \cite{HP15}, namely the one described in \cite[Subsection 4.3]{HP15}, we can identify $c_{e_0 + e_1}$ with $-c$, $c_{e_0 + 3e_1}$ with $-c^{(1)}$ and $c_{2e_0 + e_1 + e_2}$ with $-2 c^{(2)}$, while the remaining constants are set to $0$. The renormalized equation then takes the form
	\begin{align*}
		(\partial_t - \partial_x^2) u = &\sigma(u) \xi + c_{e_0 + e_1} \sigma'(u)\sigma(u) + c_{e_0 + 3e_1} \sigma'(u)^3 \sigma(u)\\
		& + \tfrac{1}{2}c_{2e_0 + e_1 + e_2} \sigma''(u)\sigma'(u)\sigma(u)^2,
	\end{align*}
	to be compared with \cite[(1.4),(1.7)]{HP15}, where the constants are fixed in
	\begin{align*}
			(\partial_t - \partial_x^2) \Pi_{e_0 + e_1} = &\Pi_{e_0}\xi + c_{e_0 + e_1},\\
			(\partial_t - \partial_x^2) \Pi_{e_0 + 3e_1} = &\Pi_{e_0 + 2e_1}\xi + c_{e_0 + 3e_1} + \Pi_{e_0 + e_1} c_{e_0 + e_1},\\
			(\partial_t - \partial_x^2) \Pi_{2 e_0 + e_1 + e_2} = &2 \Pi_{e_0} \Pi_{e_0 + e_1} \xi + \Pi_{2 e_0 + e_2} \xi \\
			&\quad + c_{2 e_0 + e_1 + e_2} + 2 \Pi_{e_0 + e_1} c_{e_0 + e_1} + 3 \Pi_{e_0}^2 c_{e_0 + e_1}.
	\end{align*}
\end{remark}	
\subsection{The generalized KPZ equation}\label{subsec::5.3}
\referee{We conclude implementing the reductions based on symmetries in the case of the generalized KPZ equation \eqref{kpz01_example}. Recall that throughout the text, and particularly in Example \ref{example_9}, we were able to identify all possible counterterms for \eqref{kpz01_example}, cf. Table \ref{tab:kpz}. We will now write} the renormalized equations \referee{only after making such reductions}. First of all, \eqref{kpz01_example} is consistent with the spatial symmetry of the heat operator only if $g \equiv 0$, which we now assume. In order for the counterterm to respect this symmetry, by \eqref{kpz03_example}, contributions from $(0, k_\mathbf{0} e_\mathbf{0} + e_{(0,1)})$ need to come in pairs; looking at Table \ref{tab:kpz}, this forces all contributions of this form to be $0$. Similarly by symmetry, multi-indices with a contribution from $e_{(0,1)}$ should not come into the counterterm. All this reduces Table \ref{tab:kpz} to Table \ref{tab:kpz2} below. As in previous examples, \referee{the invariance $\xi\mapsto -\xi$} allows us to focus only on even functionals of the noise, so we remove the constants involving $\lnh \beta \rnh = 3$. Table \ref{tab:kpz2} is further reduced to Table \ref{tab:kpz3} below.
\begin{table}[h]
	\centering
	\begin{tabular}{c | c | c}
		$|\beta|$ & $\beta$ & $\#\{\beta\}$ \\
		\hline
		$-1\mhyphen$ & $e_{(\xi,0)} + e_{(\xi,e_\mathbf{0})}$, $2e_{(\xi,0)} + e_{(0,2 e_{(0,1)})}$ & 2 \\
		\hline
		& $e_{(\xi,0)} + 2 e_{(\xi, e_\mathbf{0})}$, $2 e_{(\xi,0)} + e_{(\xi, 2 e_\mathbf{0})}$, &  \\
		$-\frac{1}{2}\mhyphen$ &  $ 2 e_{(\xi,0)} + e_{(\xi, e_\mathbf{0})} + e_{(0, 2e_{(0,1)})}$, $3 e_{(\xi,0)} + 2 e_{(0, 2e_{(0,1)})}$,   & 5 \\
		& $3 e_{(\xi,0)} + e_{(0, e_\mathbf{0}+ 2 e_{(0,1)})}$ & \\
		\hline
		& $e_{(\xi,0)} + 3 e_{(\xi, e_\mathbf{0})}$, & \\
		& $2e_{(\xi,0)} + e_{(\xi, e_\mathbf{0})} + e_{(\xi, 2e_\mathbf{0})}$, $2 e_{(\xi,0)} + 2 e_{(\xi,e_\mathbf{0})} + e_{(0, 2 e_{(0,1)})}$, & \\
		$ 0\mhyphen$ & $ 3 e_{(\xi,0)} + e_{(\xi, e_\mathbf{0})} + 2 e_{(0, 2e_{(0,1)})}$, $3 e_{(\xi,0)} + e_{(\xi,e_\mathbf{0})} + e_{(0,e_\mathbf{0} + 2 e_{(0,1)})}$, & 8\\ 
		& $3 e_{(\xi,0)} + e_{(\xi, 2 e_\mathbf{0})} + e_{(0, 2e_{(0,1)})}$, $4 e_{(\xi,0)} + e_{(0, 2e_\mathbf{0} + 2e_{(0,1)})}$, & \\
		& $4 e_{(\xi,0)} + e_{(0,2 e_{(0,1)})} + e_{(0, e_\mathbf{0} + 2e_{(0,1)})}$  & \\
		\hline 
	\end{tabular}
	\caption{Spatially symmetric multi-indices of $\mcC$ for \eqref{kpz01_example}.}
	\label{tab:kpz2}
\end{table}
\begin{table}[h!]
	\centering
	\begin{tabular}{c | c | c}
		$|\beta|$ & $\beta$ & $\#\{\beta\}$ \\
		\hline
		$-1\mhyphen$ & $e_{(\xi,0)} + e_{(\xi,e_\mathbf{0})}$, $2e_{(\xi,0)} + e_{(0,2 e_{(0,1)})}$ & 2 \\
		\hline
		& $e_{(\xi,0)} + 3 e_{(\xi, e_\mathbf{0})}$, & \\
		& $2e_{(\xi,0)} + e_{(\xi, e_\mathbf{0})} + e_{(\xi, 2e_\mathbf{0})}$, $2 e_{(\xi,0)} + 2 e_{(\xi,e_\mathbf{0})} + e_{(0, 2 e_{(0,1)})}$, & \\
		$ 0\mhyphen$ & $ 3 e_{(\xi,0)} + e_{(\xi, e_\mathbf{0})} + 2 e_{(0, 2e_{(0,1)})}$, $3 e_{(\xi,0)} + e_{(\xi,e_\mathbf{0})} + e_{(0,e_\mathbf{0} + 2 e_{(0,1)})}$, & 8\\ 
		& $3 e_{(\xi,0)} + e_{(\xi, 2 e_\mathbf{0})} + e_{(0, 2e_{(0,1)})}$, $4 e_{(\xi,0)} + e_{(0, 2e_\mathbf{0} + 2e_{(0,1)})}$, & \\
		& $4 e_{(\xi,0)} + e_{(0,2 e_{(0,1)})} + e_{(0, e_\mathbf{0} + 2e_{(0,1)})}$  & \\
		\hline 
	\end{tabular}
	\caption{Spatially symmetric multi-indices of $\mcC$ for \eqref{kpz01_example} with even noise components.}
	\label{tab:kpz3}
\end{table}	

\medskip

Under these constraints, and according to \referee{\eqref{kpz02_example} to \eqref{kpz05_example}}, the renormalized versions of \eqref{kpz01_example} take the form
\begin{align*}
	(\partial_t - \partial_x^2) u &= f(u) + h(u)(\partial_x u)^2 + \sigma(u) \xi\\
	& \quad + c_{e_{(\xi,0)} + e_{(\xi, e_\mathbf{0})}} \sigma(u) \sigma'(u)\\
	& \quad + c_{2e_{(\xi,0)} + e_{(0, 2e_{(0,1)})}} \sigma(u)^2 h(u) \\
	& \quad + c_{e_{(\xi,0)} + 3e_{(\xi, e_\mathbf{0})}} \sigma(u) \sigma'(u)^3 \\
	& \quad + \tfrac{1}{2} c_{2e_{(\xi,0)} + e_{(\xi,e_\mathbf{0})} + e_{(\xi, 2e_\mathbf{0})}} \sigma(u)^2 \sigma'(u)\sigma''(u)\\
	& \quad + c_{2 e_{(\xi,0)} + 2e_{(\xi,e_\mathbf{0})} + e_{(0, 2e_{(0,1)})}} \sigma(u)^2 \sigma'(u)^2 h(u)\\
	& \quad + c_{3 e_{(\xi,0)} + e_{(\xi, e_\mathbf{0})} + 2 e_{(0, 2e_{(0,1)})}} \sigma(u)^3 \sigma'(u) h(u)^2\\
	& \quad + c_{3 e_{(\xi,0)} + e_{(\xi, e_\mathbf{0})} + e_{(0, e_\mathbf{0} + 2e_{(0,1)})}} \sigma(u)^3 \sigma'(u)h'(u)\\
	& \quad + \tfrac{1}{2}c_{3 e_{(\xi,0)} + e_{(\xi, 2e_\mathbf{0})}+ e_{(0, 2e_{(0,1)})}} \sigma(u)^3 \sigma''(u) h(u)\\
	& \quad + \tfrac{1}{2} c_{4 e_{(\xi,0)} + e_{(0, 2e_\mathbf{0} + 2e_{(0,1)})}} \sigma(u)^4 h''(u)\\
	& \quad + c_{4 e_{(\xi,0)} + e_{(0, 2e_{(0,1)})} + e_{(0, e_\mathbf{0} + 2e_{(0,1)})}} \sigma(u)^4 h(u) h'(u).
\end{align*}
The ten renormalization constants are fixed in the following model equations:
\begin{align*}
	(\partial_t - \partial_x^2)\Pi_{e_{(\xi,0)} + e_{(\xi, e_\mathbf{0})}} = \Pi_{e_{(\xi,0)}} \xi + c_{e_{(\xi,0)} + e_{(\xi, e_\mathbf{0})}},
\end{align*}
\begin{align*}
	(\partial_t - \partial_x^2)\Pi_{2e_{(\xi,0)} + e_{(0, 2e_{(0,1)})}} = (\partial_x \Pi_{e_{(\xi,0)}} )^2 + c_{2e_{(\xi,0)} + e_{(0, 2e_{(0,1)})}},
\end{align*}
\begin{align*}
	(\partial_t - \partial_x^2)\Pi_{e_{(\xi,0)} + 3 e_{(\xi, e_\mathbf{0})}}&= \Pi_{e_{(\xi,0)} + 2e_{(\xi,e_\mathbf{0})}} \xi\\
	& \quad + c_{e_{(\xi,0)}+ 3e_{(\xi, e_\mathbf{0})}}\\
	& \quad + c_{e_{(\xi,0)} + e_{(\xi, e_\mathbf{0})}} \Pi_{e_{(\xi,0)} + e_{(\xi, e_\mathbf{0})}},
\end{align*}
\begin{align*}
	(\partial_t - \partial_x^2)\Pi_{2e_{(\xi,0)} + e_{(\xi,e_\mathbf{0})} + e_{(\xi, 2 e_\mathbf{0})}}&= \Pi_{2 e_{(\xi,0)} + e_{(\xi, 2e_\mathbf{0})}} \xi \\
	& \quad + 2 \Pi_{e_{(\xi,0)}} \Pi_{e_{(\xi,0)} + e_{(\xi, e_\mathbf{0})}} \xi \\
	& \quad + c_{2e_{(\xi,0)} + e_{(\xi,e_\mathbf{0})} + e_{(\xi, 2 e_\mathbf{0})}} \\
	& \quad + c_{e_{(\xi,0)} + e_{(\xi, e_\mathbf{0})}} \Pi_{e_{(\xi,0)}}^2 \\
	& \quad + 2 c_{e_{(\xi,0)} + e_{(\xi, e_\mathbf{0})}} \Pi_{e_{(\xi,0)} + e_{(\xi,e_\mathbf{0})}},
\end{align*}
\begin{align*}
	(\partial_t - \partial_x^2)\Pi_{2 e_{(\xi,0)} + 2 e_{(\xi, e_\mathbf{0})} + e_{(0, 2e_{(0,1)})}}&=  \Pi_{2e_{(\xi,0)} + e_{(\xi, e_\mathbf{0})} + e_{(0, 2 e_{(0,1)})}} \xi \\
	& \quad + (\partial_x \Pi_{e_{(\xi,0)} + e_{(\xi, e_\mathbf{0})}})^2\\
	& \quad + c_{2e_{(\xi,0)} + 2e_{(\xi, e_\mathbf{0})} + e_{(0, 2e_{(0,1)})}}\\
	& \quad + c_{e_{(\xi, 0)} + e_{(\xi, e_\mathbf{0})}} \Pi_{2 e_{(\xi,0)} + e_{(0, 2 e_{(0,1)})}}\\
	& \quad + c_{2 e_{(\xi,0)} + e_{(0, 2e_{(0,1)})}} \Pi_{e_{(\xi,0)}}^2\\
	& \quad + 2 c_{2 e_{(\xi,0)} + e_{(0, 2e_{(0,1)})}}  \Pi_{e_{(\xi,0)} + e_{(\xi, e_\mathbf{0})}},
\end{align*}
\begin{align*}
	(\partial_t - \partial_x^2)\Pi_{3 e_{(\xi,0)} +  e_{(\xi, e_\mathbf{0})} + 2 e_{(0, 2e_{(0,1)})}}&= \Pi_{3 e_{(\xi,0)} + 2 e_{(0, 2 e_{(0,1)})}} \xi\\
	& \quad + 2 \partial_x \Pi_{e_{(\xi,0)}} \partial_x \Pi_{2 e_{(\xi,0)} + e_{(\xi, e_\mathbf{0})} + e_{(0, 2 e_{(0,1)})}}\\
	& \quad + 2 \partial_x \Pi_{e_{(\xi,0)} + e_{(\xi, e_\mathbf{0})}} \partial_x \Pi_{2 e_{(\xi,0)} + e_{(0, 2 e_{(0,1)})}}\\
	& \quad + c_{3 e_{(\xi,0)} + e_{(\xi, e_\mathbf{0})} + 2 e_{(0, 2 e_{(0,1)})}} \\
	& \quad +2 c_{2 e_{(\xi,0)} + e_{(0, 2e_{(0,1)})}} \Pi_{2 e_{(\xi,0)} + e_{(0, 2e_{(0,1)})}},
\end{align*}
\begin{align*}
	(\partial_t - \partial_x^2)\Pi_{3 e_{(\xi,0)} + e_{(\xi, e_\mathbf{0})} + e_{(0, e_\mathbf{0} + 2e_{(0,1)})}}&= \Pi_{3 e_{(\xi,0)} + e_{(0, e_\mathbf{0} + 2 e_{(0,1)})}} \xi\\
	& \quad + 2 \Pi_{e_{(\xi,0)}} \partial_x \Pi_{e_{(\xi,0)}} \partial_x \Pi_{e_{(\xi,0)} + e_{(\xi, e_\mathbf{0})}}\\
	& \quad + \Pi_{e_{(\xi,0)} + e_{(\xi,e_\mathbf{0})}} (\partial_x \Pi_{e_{(\xi,0)}})^2\\
	& \quad + c_{3 e_{(\xi,0)} + e_{(\xi, e_\mathbf{0})} + e_{(0, e_\mathbf{0} + 2 e_{(0,1)})}}\\
	& \quad + c_{2 e_{(\xi,0)} + e_{(0, 2 e_{(0,1)})}} \Pi_{e_{(\xi,0)} + e_{(\xi, e_\mathbf{0})}}\\
	& \quad+ 2 c_{2e_{(\xi,0)} + e_{(0, 2 e_{(0,1)})}} \Pi_{e_{(\xi,0)}}^2,
\end{align*}
\begin{align*}
	(\partial_t - \partial_x^2)\Pi_{3 e_{(\xi,0)} + e_{(\xi, 2 e_\mathbf{0})} + e_{(0, 2e_{(0,1)})}}&= 2 \Pi_{e_{(\xi,0)}} \Pi_{2 e_{(\xi,0)} + e_{(0, 2 e_{(0,1)})}} \xi\\
	& \quad + 2 \partial_x \Pi_{e_{(\xi,0)}} \partial_x \Pi_{2 e_{(\xi,0)} + e_{(\xi, 2 e_\mathbf{0})}}\\
	& \quad + c_{3 e_{(\xi,0)} + e_{(\xi, 2 e_\mathbf{0})} + e_{(0, 2 e_{(0,1)})}}\\
	& \quad +2 c_{e_{(\xi,0)} + e_{(\xi, e_\mathbf{0})}} \Pi_{2 e_{(\xi,0)} + e_{(0,2 e_{(0,1)})}}\\
	& \quad + 2 c_{2 e_{(\xi,0)} + e_{(0, 2 e_{(0,1)})}} \Pi_{e_{(\xi,0)}}^2,
\end{align*}
\begin{align*}
	(\partial_t - \partial_x^2)\Pi_{4 e_{(\xi,0)} + e_{(0, 2 e_\mathbf{0} + 2e_{(0,1)})}}&= \Pi_{e_{(\xi,0)}}^2 (\partial_x \Pi_{e_{(\xi,0)}})^2\\
	& \quad + c_{4 e_{(\xi,0)} + e_{(0, 2e_\mathbf{0} + 2 e_{(0,1)})}}\\
	& \quad + c_{2 e_{(\xi,0) + e_{(0, 2 e_{(0,1)})}}} \Pi_{e(\xi,0)}^2,
\end{align*}
\begin{align*}
	(\partial_t - \partial_x^2)\Pi_{4 e_{(\xi,0)} + e_{(0, 2e_{(0,1)})} +  e_{(\xi, e_\mathbf{0} + 2 e_{(0,1)})}} &= 2 \partial_x \Pi_{e_{(\xi,0)}} \partial_x \Pi_{3 e_{(\xi,0)} + e_{(0, e_\mathbf{0} + 2 e_{(0,1)})}}\\
	& \quad + 2 \Pi_{e_{(\xi,0)}} \partial_x \Pi_{e_{(\xi,0)}} \partial_x \Pi_{2 e_{(\xi,0)} + e_{(0, 2 e_{(0,1)})}}\\
	& \quad + \Pi_{2 e_{(\xi,0)} + e_{(0, 2 e_{(0,1)})}} (\partial_x \Pi_{e_{(\xi,0)}})^2\\
	& \quad + c_{4 e_{(\xi,0)} + e_{(0, 2 e_{(0,1)})} + e_{(0, e_\mathbf{0} + 2 e_{(0,1)})}}\\
	& \quad + c_{2 e_{(\xi,0)} + e_{(0, 2 e_{(0,1)})}} \Pi_{2 e_{(\xi,0)} + e_{(0, 2 e_{(0,1)})}}.		
\end{align*}
\begin{remark}These ten renormalization constants correspond to those of \cite[Proposition 6.2.2]{Bru}, where decorated trees sharing same coefficients  produce ten degrees of freedom at the level of the renormalized equation.  This shows that multi-indices provide in this case an optimal parameterization of the renormalized equations, and thus may also provide a good starting point for understanding the symmetries of the generalized KPZ equation: It was shown in \cite{BGHZ} that, for the geometric stochastic heat equation in sufficiently high spatial dimension, one gets a unique renormalized solution satisfying both the chain rule (Stratonovich) and the Itô isometry. So far, for the flat generalized KPZ, a careful study of the dimension of the space of solutions satisfying both symmetries is missing.
\end{remark}	

\end{document}